\documentclass[11pt,reqno]{amsart}
\usepackage{graphicx}
\usepackage{color}
\usepackage{framed}
\usepackage{amsmath,amssymb,amsthm,amsfonts,
mathrsfs,stmaryrd}
\usepackage{cases,indentfirst}
\usepackage{multicol}
\usepackage[colorlinks=true, linkcolor=black,filecolor=black,urlcolor=black,citecolor=black]{hyperref}
\makeatletter

\newcommand{\Rmnum}[1]{\expandafter\@slowromancap\romannumeral #1@}
\makeatother
\newtheorem{theorem}{Theorem}[section]

\newtheorem{lemma}{Lemma}[section]
\newtheorem{corollary}{Corollary}[section]
\newtheorem{proposition}{Proposition}[section]
\newtheorem{remark}{Remark}[section]
\usepackage[numbers,sort&compress]{natbib}
\theoremstyle{definition}
\numberwithin{equation}{section}

\allowdisplaybreaks[4]
\begin{document}

\author[J.A. Carrillo]{Jos\'{e} A. Carrillo}\address{Jos\'{e} A. Carrillo\newline\indent Mathematical Institute\newline\indent University of Oxford, Oxford OX2 6GG, UK }
\email{carrillo@maths.ox.ac.uk}

\author[G.-Y. Hong]{Guangyi Hong}\address{Guangyi Hong\newline\indent School of Mathematics\newline\indent South China University of Technology\newline\indent Guangzhou 510641, P.R. China}
\email{magyhong@scut.edu.cn}

\author[Z.-A. Wang]{Zhi-an Wang}\address{Zhi-an Wang\newline\indent Department of Applied Mathematics\newline\indent The Hong Kong Polytechnic University\newline\indent Hung Hom,
Kowloon, Hong Kong}
\email{mawza@polyu.edu.hk}

\title[convergence of boundary  layers]{\bf Convergence of boundary layers of chemotaxis models with physical boundary conditions~I: degenerate initial data}

\begin{abstract}
The celebrated experiment of Tuval et al. \cite{tuval2005bacterial} showed that the bacteria living a water drop can form a thin layer near the air-water interface, where a so-called chemotaxis-fluid system with physical boundary conditions was proposed  to interpret the mechanism underlying the pattern formation alongside numerical simulations.  However, the rigorous proof for the existence and convergence of the boundary layer solutions to the proposed model still remains open. This paper shows that the model with physical boundary conditions proposed in \cite{tuval2005bacterial} in one dimension can generate boundary layer solution as the oxygen diffusion rate $\varepsilon>0$ is small.  Specifically, we show that the solution of the model with $\varepsilon>0$ will converge to the solution with $\varepsilon=0$ (outer-layer solution) plus the boundary layer profiles (inner-layer solution) with a sharp transition near the boundary as $ \varepsilon \rightarrow 0$. There are two major difficulties in our analysis.  First, the global well-posedness  of the model is hard to prove since the Dirichlet boundary condition can not contribute to the gradient estimates needed for the cross-diffusion structure in the model. Resorting to the technique of taking anti-derivative, we remove the cross-diffusion structure such that the Dirichlet boundary condition can facilitate the needed estimates.  Second,  the outer-layer profile of bacterial density is required to be degenerate at the boundary as $ t \rightarrow 0 ^{+}$, which makes the traditional cancellation technique incapable. Here we employ the Hardy inequality and delicate weighted energy estimates to overcome this obstacle and derive the requisite uniform-in-$\varepsilon$ estimates allowing us to pass the limit $\varepsilon \to 0$ to achieve our results.

\vspace*{4mm}

\vspace*{1mm}
\end{abstract}
\maketitle

\vspace{4mm}

\section{Introduction}  \label{sec:introduction}
The directional movement of cells in response to a chemical concentration gradient is referred to as chemotaxis, which is said to be endogenous if the chemical is secreted by the cell itself and exogenous if the chemical comes from an external source (like oxygen, light or food). Chemotaxis is a common biological migration strategy occurring in various biological processes, such as aggregation of bacteria (cf. \cite{tyson-lubkin-1999-JMB}), slime mold formation (cf. \cite{hofer1995cellular}), or tumor angiogenesis (cf. \cite{chaplain1993model,Corrias-2003-CRMASParis}). The mathematical models of chemotaxis mostly studied nowadays are of the Keller-Segel type originally proposed in \cite{keller1970initiation-produce, keller1971traveling}. The prototype of Keller-Segel model describing the exogenous chemotaxis reads as
\begin{gather}\label{ks1}
 \displaystyle \begin{cases}
 \displaystyle  u _{t}=\Delta u - \nabla \cdot  \left(u\nabla \phi(v) \right),\\
\displaystyle v _{t}=\varepsilon \Delta v - uv,
 \end{cases}
\end{gather}
where $u$ and $v$ denote the cell density and chemical concentration, respectively, at position $x \in \Omega$ and time $t>0$. $\varepsilon>0$ denotes the chemical diffusivity, and $\phi(v)$ is called the chemotactic sensitivity function which has two prototypes: $\phi(v)=\ln v$ (logarithmic sensitivity) and $\phi(v)=v$ (linear sensitivity). The logarithmic sensitivity was first proposed in \cite{keller1971traveling} based on the Weber-Fechner law (the sensory response to a stimulus is logarithmic) which has various prominent biological applications (cf. \cite{kalinin2009logarithmic, dehaene2003neural,LSN-Mathbio}). It was mentioned in \cite[p.241]{keller1971traveling} that the chemical (i.e. oxygen) diffusion rate $\varepsilon$ is negligible (i.e. $0<\varepsilon\ll1$) compared to the bacterial diffusion rate. The most important application of the logarithmic sensitivity lies in its capability of producing traveling waves to interpret the experiment findings (cf. \cite{KellerOdell75ms}), motivating a great deal of interesting mathematical works on the study of existence and stability of traveling wave solutions  \cite{li2011asymptotic, li2014stability, ChoiK1, ChoiK2, davis2017absolute}, just to mention a few.
The system \eqref{ks1} with linear sensitivity $\phi(v)=v$ was employed in a chemotaxis-fluid model proposed in \cite{tuval2005bacterial} to  interpret the boundary accumulation layer of aerobic bacterial chemotaxis towards the drop edge (air-water interface) in a sessile drop mixed with {\it Bacillus subtilis} bacteria. The model in \cite{tuval2005bacterial} reads as
\begin{gather}\label{fluid-chemo}
\displaystyle \begin{cases}
\displaystyle  u _{t}+ {\bf w} \cdot \nabla u= \Delta u- \nabla \cdot \left( u \nabla v \right)&\mbox{in}\ \ \Omega,\\
  \displaystyle v _{t}+{\bf w} \cdot \nabla v=D\Delta v-uv&\mbox{in}\ \ \Omega,\\
  \displaystyle \rho( {\bf w}_{t}+{\bf w} \cdot \nabla {\bf w})=\mu \Delta {\bf w}+ \nabla p- V_b g u (\rho_b-\rho) {\bf z}&\mbox{in}\ \ \Omega,\\
  \nabla \cdot {\bf w}=0,
\end{cases}
\end{gather}
with the following physical zero flux - Dirichlet-no slip mixed boundary conditions
\begin{gather} \label{bc}
\displaystyle  \left( \nabla u-u \nabla v \right)\cdot \nu =0, \ \ v =v _{\ast}, \ {\bf w}=0 \ \mbox{on}\ \ \partial \Omega,
\end{gather}
where $u$ and $v$ denote the bacterial and oxygen concentrations at $x\in \Omega$ and $t>0$, respectively, and ${\bf w}$ is the fluid velocity governed by the incompressible Navier-Stokes equations with the pure fluid density $\rho$ and viscosity $\mu$. $p$ is a pressure function, $V_b g u (\rho_b-\rho) {\bf z}$ denotes the buoyant force along the upward unit vector ${\bf z}$ where $V_b$ and $\rho_b$ are the bacterial volume and density, respectively, and $g$ is the gravitational constant. In \eqref{bc}, $\nu$ denotes the outward unit normal vector of $\partial \Omega$ and $v_{\ast}>0$ is a constant representing the saturation of oxygen at the air-water interface (i.e. boundary). The numerical simulations in works \cite{chertock2012sinking,lee2015numerical, tuval2005bacterial} have shown that the system \eqref{fluid-chemo} can reproduce the key features of boundary layer formation observed in the experiment of \cite{tuval2005bacterial} in two and three dimensions under the physical boundary conditions \eqref{bc}. Therefore justifying that \eqref{fluid-chemo}-\eqref{bc} admits boundary-layer solutions becomes an imperative question, which has remained open for a long time without good progresses made as we know. Indeed, boundary layer problem has been a fundamental topic arising in the fluid mechanics due to the distortion of non-viscous flow by the surrounding viscous forces observed by Prandtl in 1904 \cite{P} and attracted extensive studies (cf. \cite{F, Frid99, Frid20, JZ09, WX, XY, Varet-Dormy, Yang-JAMS}, just to mention a few). Though the model \eqref{fluid-chemo} contains the fluid dynamics, the boundary layer was formed due to the aggregation of bacteria attracted by the oxygen near the air-water interface (cf. \cite{Dombrowski, tuval2005bacterial}) and thus the fluid dynamics will play minor roles as can be glimpsed from the boundary conditions \eqref{bc}. Since the Dirichlet boundary condition for $v$ can not directly contribute to the estimate of $\nabla v$ required by the first equation of \eqref{ks1} for the estimate of $u$, many basic questions on \eqref{fluid-chemo}-\eqref{bc} like the global well-posedness still remains poorly understood so far apart from the boundary layer solutions. To the best of our knowledge, the only analytical result for problem \eqref{fluid-chemo}-\eqref{bc} was the local existence of weak solutions obtained in \cite{lorz2010coupled}.
If the domain $\Omega$ is radially symmetric (say a ball), then the incompressibility condition  $\nabla \cdot {\bf w}=0$ on $\Omega$ and no slip boundary condition ${\bf w}|_{\partial \Omega}=0$ implies that ${\bf w}=0$, and as a result  \eqref{fluid-chemo}-\eqref{bc} is simplified as
\begin{gather}\label{ks2}
 \displaystyle \begin{cases}
 \displaystyle  u _{t}=\Delta u - \nabla \cdot  \left(u\nabla v \right) &\mbox{in}\ \ \Omega,\\
\displaystyle v _{t}=\varepsilon \Delta v - uv &\mbox{in}\ \ \Omega,
 \end{cases}
\end{gather}
with boundary conditions
\begin{equation}\label{bcn}
\left( \nabla u-u \nabla v \right)\cdot \nu =0, \ \ v =v _{\ast} \ \ \mbox{on}\ \ \partial \Omega.
\end{equation}

Regarding the boundary layer solutions, it was first shown in \cite{lee2019boundary} that the problem \eqref{ks2}-\eqref{bcn} has a unique stationary solution in all dimensions, which possesses a boundary layer profile with thickness of order $\varepsilon^{{1}/{2}}$ as $\varepsilon \to 0$. Subsequently the nonlinear local time-asymptotic stability of stationary solutions of \eqref{ks2}-\eqref{bcn} in one-dimension was established recently in \cite{HW2}. However, whether the time-dependent problem \eqref{ks2}-\eqref{bcn} can develop boundary layer profiles as $\varepsilon \to 0$ remains unknown.
To see the possibility, we integrate the second equation of \eqref{ks2} with $\varepsilon=0$ and get $v(x,t)=v_0(x)\mathop{\mathrm{e}}^{-\int_0^t u(x,\tau)\mathrm{d}\tau}$, which gives rise to
\begin{equation}\label{bv}
v \vert_{\partial \Omega}=v_0 \vert_{ \partial \Omega} \ {\mathop{\mathrm{e}}}^{-\int_0^t u \vert_{\partial \Omega}\mathrm{d}\tau}.
 \end{equation}
This implies that the boundary value of $v$ as $\varepsilon=0$ is intrinsically determined by \eqref{bv}, which may mismatch the prescribed boundary value of $v$ for $\varepsilon>0$. If this occurs, boundary layers will arise as $\varepsilon \to 0$ and the zero-diffusion limit of \eqref{ks2}-\eqref{bcn} as $\varepsilon \to 0$ becomes a singular problem.
However how to justify the convergence of solutions of the singular problem \eqref{ks2}-\eqref{bcn} as $\varepsilon \to 0$ still remains an outstanding open question as far as we know. The goal of this paper is to investigate  the zero-diffusion limit of the problem \eqref{ks2}-\eqref{bcn} in a one-dimensional domain $\mathcal{I}=(0,1)$ as $\varepsilon \to 0$, reading as
\begin{align}\label{eq-orignal}
\displaystyle \begin{cases}
\displaystyle  u _{t}=u _{xx}- \left( u v _{x} \right)_{x}, &\ x\in \mathcal{I},\ t>0,\\
  \displaystyle v _{t}=\varepsilon v _{xx}- uv, &\ x\in \mathcal{I},\ t>0,  \\
  (u,v)(x,0)=(u _{0}, v _{0})(x), & \ x\in \overline{\mathcal{I}},
\end{cases}
\end{align}
with boundary conditions
\begin{eqnarray}\label{intial-bdary}
\begin{cases}
(u _{x}-u v _{x})\vert _{\partial \mathcal{I}}=0,\ \ v \vert_{ \partial \mathcal{I}}=v _{\ast}, & \text{if}\ \varepsilon>0,\\
(u _{x}-u v _{x})\vert _{\partial \mathcal{I}}=0, & \text{if}\ \varepsilon=0,
\end{cases}
\end{eqnarray}
where $\overline{\mathcal{I}}=[0,1]$ and $\partial \mathcal{I}=\{0,1\}$.

The zero-diffusion limit of problem \eqref{eq-orignal}-\eqref{intial-bdary} as $ \varepsilon \rightarrow 0 $ is a multi-scale problem involving sophisticated formal and rigorous analysis with complex compatibility conditions. In this paper, we shall prove that the solution of \eqref{eq-orignal}-\eqref{intial-bdary} is not uniformly convergent in $L^\infty$ with respect to $\varepsilon>0$ but stabilizes to the outer layer profile (solution with $\varepsilon = 0$) plus an inner (boundary) layer profile as $\varepsilon  \to 0$ where governing equations for both outer and inner layer profiles can be precisely derived. There are two major difficulties encountered in our analysis: (1) how to employ the Dirichlet boundary condition of $v$ to obtain the estimates of $v_x$ in order to gain requisite regularity of solutions for the global well-posedenss due to the cross-diffusion structure in the first equation of \eqref{eq-orignal}; (2) how to derive the uniform-in-$\varepsilon$ estimates in order to pass the limit $\varepsilon \to 0$. To overcome the former one, with the mass conservation of $u$ resulting from the zero-flux boundary condition, we make a change of variable (see \eqref{anti-derivatives-transf}) based on the technique of taking anti-derivative as used in our previous works \cite{Carrillo-Li-Wang-PLMS, HW2} to reformulate \eqref{eq-orignal}-\eqref{intial-bdary} into a new Dirichlet problem \eqref{refor-eq}-\eqref{BD-POSITIVE-VE} without cross-diffusion structure, for which the Dirichlet boundary condition on $v$ can contribute to derive desired estimates. In doing so, we pay a price  by requiring $ \inf_{x\in \overline{\mathcal{I}}} u_0(x)=0$ (i.e. the initial value is degenerate) in the compatibility conditions for the reformulated problem, which leads to the failure of cancellation technique used in the existing work \cite{HW2} dealing with the reformulated problem. In this paper, we shall develop a new idea with the help of the Hardy inequality to derive requisite uniform-in-$\varepsilon$ estimates and finally prove our main results. However, our results can not cover the case $ \inf _{x \in \overline{\mathcal{I}}}u _{0}>0$ for which initial layers will be present (see Remark \ref{remark-main}) and new ideas are needed to overcome this barrier. This case will be investigated in a separate work. We stress that the zero-flux boundary condition of $u$ given in \eqref{intial-bdary} can not extrapolate the boundary profile of $u$. While showing that the solution component $v$ has boundary layer profiles as expected, we also prove that $u$ has boundary layer profiles as $\varepsilon  \to 0$ (see Theorem \ref{thm-original}). As far as we know, this is the first result showing that the time-dependent chemotaxis models with physical boundary conditions in \eqref{bcn} have boundary layer profiles for both cell density and oxygen concentration. Our results hence assert that the chemotaxis-fluid model \eqref{fluid-chemo} is capable of generating  boundary layer profiles in one dimension though the higher dimensional case is yet to be proved. Since the technique of taking anti-derivative is not directly applicable in multi-dimensions, the boundary layer problem of \eqref{fluid-chemo}-\eqref{bc} or \eqref{ks2}-\eqref{bcn} in multi-dimensions has to be left out for future efforts with new ideas and techniques.

Apart from the boundary layer problem, when $\Omega$ is a radially symmetric domain in $\mathbb{R}^n(n\geq 2)$, the  existence of global classical solutions of \eqref{ks2}-\eqref{bcn} with $\varepsilon>0$ in two dimension ($n=2$) and global weak solutions in higher dimensions ($n=3,4,5$) were established in \cite{lankeit2021radial}. If $u$ and $v$ satisfy zero-flux and Robin boundary conditions, respectively, the global classical solutions of \eqref{ks2} was obtained in \cite{Lankit-M3AS} for any $n\geq 1$ and the existence of boundary layer solutions as $\varepsilon \to 0$ was established recently in \cite{hou2022boundary}. With homogeneous Neumann boundary conditions,  the global dynamics of \eqref{ks2}  have been well understood (cf. \cite{fan2017global,tao2011boundedness,tao2012eventual}) by employing a clever cancelling idea which is unfortunately inherently restricted to Neumann boundary conditions. For the time-dependent problem \eqref{ks2}-\eqref{bcn}, aside from the local stability of stationary solutions shown in \cite{HW2}, a slightly modified model of \eqref{ks2} subject to \eqref{bc} was recently considered in \cite{WWX-CPDE} where the global generalized (weak) solution was obtained in three dimensional domain ($n=3$). If homogeneous Neumann boundary conditions for $u$ and $v$ and Dirichlet boundary condition for ${\bf w}$ are imposed or the domain is the whole space $\mathbb{R}^n (n\geq 1)$, the chemotaxis-fluid model \eqref{fluid-chemo} and its variants have been widely studied in the literature \cite{Duan, duan2017global, Liu-Lorz, winkler2012global, chae2014global}, just to mention a few due to the limit of spaces.

The rest of the paper is organized as follows: In Section \ref{sec:main_result}, we first reformulate our problem by taking the anti-derivative of a perturbed function against the cell mass and derive the equations for the outer and boundary (or inner) layer profiles. Then we state our main results on the convergence of boundary layer solutions. In Section \ref{sec:study_on_the_inner_outer_layers}, we are devoted to deriving the regularity of outer and boundary layer profiles. Finally, in Section \ref{sec:stability_of_boundary_layers}, we prove our main results.

\vspace{5mm}

\section{Statement of Main results}
\label{sec:main_result}
In this section, we shall first derive the equations that outer- and boundary-layer profiles satisfy by the WKB method~(cf. \cite{Grenier-1998-JDE,Holems-1995-book,Rousset-2005-JDE}), and then state our main results on the convergence of boundary layers as $ \varepsilon \rightarrow 0 $. For clarity, we first introduce some notations used throughout the paper.\\

{\bf Notation.}
\begin{itemize}
	\item Denote $\mathbb{R}_{+}:=(0,\infty)$ and $\mathbb{R}_{-}:=(-\infty,0)$. $ \mathbb{N} $ represents the set of non-negative integers. Let $ L ^{p} $ with $ 1 \leq p \leq \infty $ denote the Lebesgue space $ L ^{p}(\mathcal{I}) $ in which functions are defined with respect to (w.r.t) the variable $ x \in (0,1) $. $ L _{z}^{p} $ denotes the space $L ^{p}(0, \infty)$ for functions defined w.r.t $ z \in (0,\infty) $ and $ L _{\xi}^{p}$ denotes $ L ^{p}(-\infty,0)$ for functions defined w.r.t $ \xi \in (- \infty,0) $, respectively. Accordingly, we denote by $  H ^{k} $, $ H _{z}^{k} $ and $ H _{\xi}^{k} $ the standard Sobolev spaces $ W ^{k,2} $ for functions defined w.r.t $ x \in \mathcal{I} $, $ z \in (0,\infty) $ and $ \xi \in(-\infty,0) $, respectively. We also write $ L _{T}^{p} Y:=L ^{p}(0,T;Y) $~(e.g., $ L _{T}^{\infty}L _{z}^{\infty}:=L ^{\infty}(0,T;L _{z}^{\infty}) $) for convenience when no confusion is caused.

	\item Denote $ \langle z \rangle=\sqrt{1+z ^{2}} $ for $ z \in [0,\infty) $, and $ \langle \xi \rangle=\sqrt{1+\xi ^{2}} $ for $ \xi \in(-\infty,0] $.

    \item $ C(T) >0$ represents a generic constant depending on $ T $ but independent of $ v _{\ast} $ such that $ C(T) \rightarrow 0 $ as $ T \rightarrow 0^+ $. $ C(v _{\ast},T) $ denotes a generic positive constant depending on $ v _{\ast} $ and $ T $ such that $ C(v _{\ast},T) \rightarrow 0 $ as $ (v _{\ast},T) \rightarrow (0,0) $ and $ C(v _{\ast},T) \rightarrow +\infty $ as $ v _{\ast}\rightarrow +\infty $ or $ T  \rightarrow +\infty $. 
        Moreover we denote  $ c(T):= c _{0}+C(T)$ and $c(v_*,T):= c _{0}+C(v _{\ast},T) $, where $ c _{0}>0 $ denotes a generic constant independent of $ v _{\ast} $ and $ T $.

    \item We often use $(\ast)_i$ to denote the $i$-th equation of the system $(\ast)$ for brevity.

\end{itemize}

\subsection{Construction of outer- and boundary(inner)-layer profiles} 
\label{sub:boundary_layer_profiel}
In this subsection, we shall first reformulate our target system \eqref{eq-orignal}-\eqref{intial-bdary} and then derive the equations for the outer- and boundary-layer profiles of the reformulated problem with small $ \varepsilon>0 $ based on the WKB method. Notice that the zero-flux boundary condition for $ u $ gives rise to the conservation of mass:
$$\int _{\mathcal{I}}u(x,t)\mathrm{d}x=\int _{\mathcal{I}}u _{0}\mathrm{d}x=: M,$$
where the constant $M>0$ denotes the cell total mass. By defining
\begin{gather}\label{anti-derivatives-transf}
\displaystyle \varphi(x,t)= \int _{0}^{x}(u(y,t)-M)\mathrm{d}y\
\ \text{with} \ \
\displaystyle \varphi(x,0)= \int _{0}^{x}(u _{0}(y)-M)\mathrm{d}y=:\varphi _{0}(x),
\end{gather}
we reformulate the problem \eqref{eq-orignal}-\eqref{intial-bdary} as
\begin{gather}\label{refor-eq}
 \displaystyle \begin{cases}
   \displaystyle \varphi _{t}=\varphi _{xx}-(\varphi _{x}+M)v _{x}, & x\in \mathcal{I},\\
   \displaystyle v _{t}= \varepsilon v _{xx}-(\varphi _{x}+M)v, & x\in \mathcal{I}, \\
     \displaystyle (\varphi,v)(x,0)=(\varphi _{0},v _{0}),
 \end{cases}
 \end{gather}
 subject to boundary conditions
 \begin{eqnarray}\label{BD-POSITIVE-VE}
 \begin{cases}
\varphi(0,t)=\varphi(1,t)=0, \ \ v(0,t)=v(1,t)=v _{\ast}, & \text{if} \ \varepsilon>0,\\
\varphi(0,t)=\varphi(1,t)=0, & \text{if} \ \varepsilon=0.
\end{cases}
 \end{eqnarray}
We proceed to derive the equations for outer- and boundary-layer profiles of the problem \eqref{refor-eq}-\eqref{BD-POSITIVE-VE} with small $ \varepsilon>0 $. As we will see later, once these profiles are determined, one can easily recover outer- and boundary-layer profiles of the original problem \eqref{eq-orignal}-\eqref{intial-bdary}. To this end, we define the so-called boundary-layer coordinates
\begin{gather}\label{bd-layer-variable}
\displaystyle z= \frac{x}{\sqrt{\varepsilon}},\ \ \ \xi= \frac{x-1}{\sqrt{\varepsilon}},\ \ x \in [0,1],
\end{gather}
 where $ \sqrt{\varepsilon} $ is the thickness of boundary layers which can be determined by the asymptotic matching method (cf. \cite{Holems-1995-book,Hou-wang-2016-JDE}). Clearly $z \in [0, \infty)$ and $\xi \in (-\infty, 0]$. The equations governing outer- and boundary-layer profiles of \eqref{refor-eq} can be derived in four successive steps.

{\bf Step 1.} {\sc Asymptotic expansions.} By the method of perturbation (cf. \cite{Grenier-1998-JDE, Holems-1995-book, Rousset-2005-JDE}), we assume that the solution of problem \eqref{refor-eq}-\eqref{BD-POSITIVE-VE} with $ \varepsilon>0 $ has the following expansions for $ j \in \mathbb{N} $:
\begin{eqnarray}\label{formal-expan}
\begin{cases}
\displaystyle   \varphi ^{\varepsilon}(x,t)= \sum _{j=0}^{\infty}\varepsilon ^{\frac{j}{2}}\left(  \varphi ^{I,j}(x,t)+ \varphi ^{B,j}\left(z,t \right)+ \varphi ^{b,j}\left(\xi,t \right) \right),\\[2mm]
\displaystyle    v ^{\varepsilon}(x,t)= \sum _{j=0}^{\infty}\varepsilon ^{\frac{j}{2}}\left( v ^{I,j}(x,t)+v ^{B,j}\left(z,t \right)+v ^{b,j}\left(\xi,t \right) \right),
\end{cases}
\end{eqnarray}
where the boundary layer profiles $ (\varphi ^{B,j}, v ^{B,j}) $ and $ (\varphi ^{b,j}, v ^{b,j}) $ are smooth and satisfy the following asymptotic behavior for $ j \geq 0 $:
\begin{gather}\label{profile-decay}
\begin{cases}
   \mbox{$ \varphi ^{B,j} $ and $ v ^{ B,j} $ decay to zero exponentially as $ z \rightarrow \infty $,}\\
  \mbox{$ \varphi ^{b,j} $ and $ v ^{b,j} $ decay to zero exponentially as $ \xi \rightarrow - \infty$.} \end{cases}
\end{gather}

{\bf Step 2.} {\sc Initial and boundary conditions.} For initial conditions, setting $t=0$ in \eqref{formal-expan} and noticing that the initial value $(\varphi_0, v_0)$ is independent of $\varepsilon>0$, we immediately get
\begin{gather*}
\displaystyle \begin{cases}
	\displaystyle  \varphi ^{I,0}(x,0)=\varphi _{0}(x),\ \ \varphi ^{B,0}(z,0)=\varphi ^{b,0}(\xi,0)=0,\\
\displaystyle  v ^{I,0}(x,0)=v _{0}(x),\ \ v ^{B,0}(z,0)=v ^{b,0}(\xi,0)=0,
\end{cases}
\end{gather*}
and
\begin{gather*}
\begin{cases}
	\displaystyle \varphi ^{I,j}(x,0) =\varphi ^{B,j}(z,0)=\varphi ^{b,j}(\xi,0)=0,\ \ j \geq 1,\\
\displaystyle v ^{I,j}(x,0) =v ^{B,j}(z,0)=v ^{b,j}(\xi,0)=0,\ \ j \geq 1.
\end{cases}
\end{gather*}

To match boundary conditions, we substitute \eqref{formal-expan} into \eqref{BD-POSITIVE-VE} and use the asymptotic matching method to get
\begin{gather*}
\displaystyle  \begin{cases}
   \displaystyle \varphi ^{I,j}(0,t)+\varphi ^{B,j}(0,t)=0,& \  \varphi ^{I,j}(1,t)+\varphi ^{b,j}(0,t)=0,\ \ \ j \geq 0,\\
   \displaystyle v ^{I,0}(0,t)+v ^{B,0}(0,t)=v _{\ast},& \  v ^{I,0}(1,t)+v ^{b,0}(0,t)=v _{\ast},\\
  \displaystyle v ^{I,j}(0,t)+v ^{B,j}(0,t)=0,& \  v ^{I,j}(1,t)+v ^{b,j}(0,t)=0,\ \ j \geq 1,
\end{cases}
\end{gather*}
where we have neglected $(\varphi ^{b,j}(-\frac{1}{\varepsilon ^{1/2}},t),v ^{b,j}(- \frac{1}{\varepsilon ^{1/2}},t)) $ at $ x=0 $ and $(\varphi ^{B,j}(\frac{1}{\varepsilon ^{1/2}},t),v ^{B,j}(\frac{1}{\varepsilon ^{1/2}},t)) $ at $ x=1 $ based on the decay properties in \eqref{profile-decay} since $\varepsilon>0$ is small.

{\bf Step 3.} {\sc Equations for outer-layer profiles $ (\varphi ^{I,j}, v ^{I,j}) $.} Substituting \eqref{formal-expan} without boundary layer profiles into the equations in \eqref{refor-eq}, we get equations for the outer-layer profiles $ \varphi ^{I,j} $:
\begin{gather}\label{vfi-outer}
\displaystyle  \varphi _{t}^{I,j}=\varphi _{xx}^{I,j}- M v _{x}^{I,j}- \sum _{k=0}^{j}\varphi _{x}^{I,k}v _{x}^{I,j-k},\ \ j \geq 0,
\end{gather}
and equations for the outer-layer profiles $ v ^{I,j} $:
\begin{gather*}
\begin{cases}
	\displaystyle v _{t}^{I,0}=- (\varphi _{x}^{I,0}+M)v ^{I,0},\ &j=0,\\[1mm]
\displaystyle v _{t}^{I,1}=-(\varphi _{x}^{I,0}+M)v ^{I,1}-\varphi _{x}^{I,1}v ^{I,0},\ &j=1,\\[1mm]
\displaystyle v _{t} ^{I,j}= v _{xx} ^{I,j-2}-M v ^{I,j}- \sum _{k=0}^{j}\varphi _{x}^{I,k}v ^{I,j-k},\ &j \geq 2.
\end{cases}
\end{gather*}

{\bf Step 4.} {\sc Equations for boundary layer profiles $(\varphi ^{B,j}, \varphi ^{b,j}, v ^{B,j}, v ^{b,j}) $.} Using \eqref{vfi-outer}, we neglect the right boundary layer profiles $ \varphi ^{b,j} $ and $ v ^{b,j} $, and then insert the remaining terms of \eqref{formal-expan} into the first equation in \eqref{refor-eq} to derive the equations for the left boundary layer profiles $ \varphi ^{B,j} $:
\begin{gather}\label{G-sum}
\displaystyle \sum _{i \geq -2}\varepsilon ^{\frac{i}{2}}G _{i}=0\ \  \mbox{for}\ \ i \geq -2,
\end{gather}
where
\begin{gather}
\begin{cases}
	\displaystyle G _{-2}:= \varphi _{zz}^{B,0}-\varphi _{z}^{B,0}v _{z}^{B,0},\\[1mm]
\displaystyle G _{-1}:=\varphi _{zz}^{B,1}-(\partial _{x}\varphi ^{I,0}(0,t)+M)v _{z}^{B,0}-\varphi _{z}^{B,1}v _{z}^{B,0}
-\varphi _{z}^{B,0}(v _{x} ^{I,0}(0,t)+v _{z}^{B,1}),\\[1mm]
\displaystyle G _{0}:=\varphi _{t}^{B,0}-\varphi _{zz}^{B,2}+v _{z}^{B,0}(\partial _{x}^{2}\varphi ^{I,0}(0,t)z+ \partial _{x}\varphi ^{I,1}(0,t)+\varphi _{z}^{B,2}) + v _{z}^{B,1}(\partial _{x}\varphi ^{I,0}(0,t)+M+\varphi _{z}^{B,1})
    \nonumber \\[1mm]
    \displaystyle \quad \qquad+\varphi _{z}^{B,1}\partial _{x}v ^{I,0}(0,t)+(\partial _{x}^{2}v ^{I,0}(0,t)z+ \partial _{x}v ^{I,1}(0,t)+v _{z}^{B,2}) \varphi _{z}^{B,0},
     \nonumber \\[1mm]
      G _{1}:=\cdots
       \nonumber \\
       \cdots \ \cdots.
\end{cases}
\end{gather}
Similarly, the right boundary layer profiles $  \varphi ^{b,j}  $ satisfy
\begin{gather}\label{tild-G-sum}
\displaystyle \sum _{i \geq -2}\varepsilon ^{\frac{i}{2}}\tilde{G} _{i}=0\ \  \mbox{for}\ \ i \geq -2,
\end{gather}
where, for each $i \geq -2 $, $ \tilde{G}_{i} $ is given by $ G _{i} $ with $ (\partial _{x}^{\ell+1}\varphi ^{I,k}(0,t),\partial _{x}^{\ell} v ^{I,k} (0,t))~(\ell \geq 0) $ and $ (\partial _{z}^{l}\varphi ^{B,k},\partial _{z}^{\lambda} v ^{B,k}) $ replaced by $ (\partial _{x}^{\ell+1}\varphi ^{I,k}(1,t),\partial _{x}^{\ell} v ^{I,k} (1,t))~(\ell \geq 0) $ and $  (\partial _{\xi}^{l}\varphi ^{b,k},\partial _{\xi}^{\lambda} v ^{b,k}) (l,\lambda\geq 0) $, respectively.

By the same procedure as deriving the equations for $ \varphi ^{B,j}$ and $ \varphi ^{b,j} $ above, we obtain the equations for the left boundary layer profiles $  v ^{B,j} $ as
\begin{gather}
\displaystyle \begin{cases}
\varphi _{z}^{B,0}(v ^{B,0}+v ^{I,0}(0,t))=0,\\[1mm]
v _{t}^{B,0}-v _{zz}^{B,0}+\varphi _{z}^{B,0}(v _{x}^{I,0}(0,t)z+v ^{I,1}(0,t)+v ^{B,1})
 \nonumber \\[2mm]
\displaystyle \qquad +(\varphi _{x}^{I,0}(0,t)+M)v ^{B,0}+\varphi _{z}^{B,1}(v ^{B,0}+v ^{I,0}(0,t))=0,\\[1mm]
 \displaystyle v _{t}^{B,1}-v _{zz}^{B,1}+ (\varphi _{x}^{I,0}(0,t)+M)v ^{B,1}
+\varphi _{z}^{B,1}(v _{x}^{I,0}(0,t)z+v ^{I,1}(0,t)+v ^{B,1})     \nonumber \\[2mm]
     \qquad+(\varphi _{xx}^{I,0}(0,t) z+\varphi _{x}^{I,1}(0,t))v ^{B,0}+ \varphi _{z}^{B,2}(v ^{I,0}(0,t)+v ^{B,0})\nonumber \\[2mm]
     \qquad+\varphi_z^{B,0}(\frac{z^2}{2}v_{xx}^{I,0}(0,t)+ v_x^{I,1}(0,t)z+v^{B,2}+v^{I,2}(0,t))=0,\\
\cdots\ \cdots,
\end{cases}
\end{gather}
and the equations for the right boundary layer profiles $ v ^{b,j}$ as
\begin{gather}
\displaystyle \begin{cases}
\varphi _{\xi}^{b,0}(v ^{b,0}+v ^{I,0}(1,t))=0,\\[1mm]
v _{t}^{b,0}-v _{\xi \xi}^{b,0}+\varphi _{\xi}^{b,0}(v _{x}^{I,0}(1,t)\xi+v ^{I,1}(1,t)+v ^{b,1})
 \nonumber \\
  \displaystyle\qquad+(\varphi _{x}^{I,0}(1,t)+M)v ^{b,0}+\varphi _{\xi}^{b,1}(v ^{b,0}+v ^{I,0}(1,t))=0,\\[1mm]
 \displaystyle v _{t}^{b,1}-v _{\xi \xi}^{b,1}+ (\varphi _{x}^{I,0}(1,t)+M)v ^{b,1}
+\varphi _{\xi}^{b,1}(v _{x}^{I,0}(1,t)\xi+v ^{I,1}(1,t)+v ^{b,1})
     \nonumber \\[1mm]
     \qquad+(\varphi _{xx}^{I,0}(1,t) \xi+\varphi _{x}^{I,1}(1,t))v ^{b,0}+ \varphi _{\xi}^{b,2}(v ^{I,0}(1,t)+v ^{b,0})\nonumber \\[2mm]
     \qquad+\varphi_z^{b,0}(\frac{z^2}{2}v_{xx}^{I,0}(1,t)+ v_x^{I,1}(1,t)z+v^{b,2}+v^{I,2}(1,t))=0,\\
    \cdots\ \cdots.
\end{cases}
\end{gather}

Finally, from the above Step 1 to Step 4, we derive initial-boundary value problems satisfied by the profiles $ (\varphi ^{I,j},\varphi ^{B,j},\varphi ^{b,j})\, (0 \leq j \leq 2)$ and $  (v ^{I,j},v ^{B,j},v ^{b,j})\,(0 \leq j \leq 1)$ for later use. First the leading-order outer-layer profile $ (\varphi ^{I,0}, v ^{I,0}) $ satisfies the problem
\begin{align}\label{eq-outer-0}
\displaystyle \begin{cases}
	\displaystyle\varphi _{t}^{I,0}=\varphi _{xx}^{I,0}-(\varphi _{x}^{I,0}+M)v _{x}^{I,0},& x\in \mathcal{I}, t>0\\[1mm]
	\displaystyle v _{t}^{I,0}=-(\varphi _{x}^{I,0}+M)v ^{I,0},& x\in \mathcal{I}, t>0\\[1mm]
	\varphi ^{I,0}(0,t)=\varphi ^{I,0}(1,t)=0,\\[1mm]
	\displaystyle (\varphi ^{I,0}, v ^{I,0})(x,0)=(\varphi _{0}, v _{0}),
	\end{cases}
\end{align}
which is nothing but the zero-diffusion problem of \eqref{refor-eq}. We note that the stability of the unique non-constant steady state to the problem \eqref{eq-outer-0} has been established in our previous work \cite{HW2}. We further remark that, as will be stated in Section \ref{sec:study_on_the_inner_outer_layers}, if the initial value is compatible with boundary conditions and smooth enough, one can prove the global existence of unique classical solutions to \eqref{eq-outer-0} with large initial data due to the dissipation effect. The first-order outer-layer profile $ (\varphi^{I,1},v ^{I,1}) $ satisfies the following problem:
\begin{align}\label{first-outer-problem}
\displaystyle \begin{cases}
	\displaystyle\varphi _{t}^{I,1}= \varphi _{xx}^{I,1}-(\varphi _{x}^{I,0}+M)v _{x}^{I,1}-\varphi _{x}^{I,1}v _{x}^{I,0}, & x\in \mathcal{I}, t>0\\[1mm]
	\displaystyle v _{t}^{I,1}=-(\varphi _{x}^{I,0}+M)v ^{I,1}-\varphi _{x}^{I,1}v ^{I,0},& x\in \mathcal{I}, t>0\\[1mm]
	\displaystyle \varphi ^{I,1}(0,t) =- \varphi ^{B,1}(0,t), \ \ \varphi ^{I,1}(1,t)=-\varphi ^{b,1}(1,t),\\[1mm]
	\displaystyle (\varphi ^{I,1}, v ^{I,1})(x,0)=(0,0).
\end{cases}
\end{align}
The leading-order boundary layer profile $ \varphi ^{B,0} $ near the left boundary solves
\begin{gather*}
\begin{cases}	
\displaystyle  \varphi _{zz}^{B,0}-\varphi _{z}^{B,0}v _{z}^{B,0}=0, &z\in \mathbb{R}_+,\\[1mm]
\varphi ^{B,0}(0,t)=0,\ \ \varphi ^{B,0}(+\infty,t)=0,\\
\displaystyle \varphi ^{B,0}(z,0)=0,
\end{cases}
\end{gather*}
and thus $ \varphi ^{B,0}\equiv 0 $. The boundary layer profile $ v ^{B,0} $ near the left boundary solves
\begin{gather}\label{first-bd-layer-pro}
\displaystyle \begin{cases}
	\displaystyle v _{t}^{B,0}= v ^{B,0}_{zz}- (\varphi _{x} ^{I,0}(0,t)+M)v ^{I,0}(0,t)    ({\mathop{\mathrm{e}}}^{v ^{B,0}}-1)-(\varphi _{x} ^{I,0}(0,t)+M)    {\mathop{\mathrm{e}}}^{v ^{B,0}}v ^{B,0}, &z\in \mathbb{R}_+,\\
	\displaystyle v ^{B,0}(0,t)= v _{\ast}-v ^{I,0}(0,t),\ \ v ^{B,0}(+\infty,t)=0,\\
		\displaystyle v ^{B,0}(z,0)=0,
\end{cases}
\end{gather}
and $ \varphi ^{B,1} $ is determined by $ v ^{B,0} $ through
\begin{align}\label{vfi-bd-1ord-lt}
   \displaystyle \varphi ^{B,1}=- \int _{z}^{\infty}(\varphi _{x} ^{I,0}(0,t)+M)\left({\mathop{\mathrm{e}}}^{v ^{B,0}(y,t)}-1 \right) \mathrm{d}y.
   \end{align}
The boundary layer profile $ v ^{b,0} $ near the right boundary satisfies
\begin{gather}\label{first-bd-pro-rt}
\displaystyle \begin{cases}
	\displaystyle v _{t}^{b,0}= v ^{b,0}_{\xi \xi}- (\varphi _{x} ^{I,0}(1,t)+M)v ^{I,0}(1,t)    ({\mathop{\mathrm{e}}}^{v ^{b,0}}-1)-(\varphi _{x} ^{I,0}(1,t)+M)    {\mathop{\mathrm{e}}}^{v ^{b,0}}v ^{b,0}, & \xi\in \mathbb{R}_-,\\
	\displaystyle v ^{b,0}(0,t)= v _{\ast}-v ^{I,0}(1,t),\ \ v ^{b,0}(-\infty,t)=0,\\
		\displaystyle v ^{b,0}(\xi,0)=0.
\end{cases}
\end{gather}
Furthermore, we have $ \varphi ^{b,0}\equiv 0 $, and $ \varphi ^{b,1} $ is given by
\begin{gather}\label{firs-bd-1-rt}
\displaystyle  \varphi ^{b,1}=\int _{- \infty}^{\xi}(\varphi _{x} ^{I,0}(1,t)+M)\left( 		{\mathop{\mathrm{e}}}^{v ^{b,0}(y,t)}-1 \right)\mathrm{d}y.
\end{gather}
Although we focus only on the convergence result for leading-order approximation, some estimates of the higher-order outer- and boundary layer profiles are also needed in our analysis. The problem formed by equations for $ \varphi ^{B,2} $ and $ v ^{B,1} $ reads
\begin{align}\label{second-bd-eq}
\displaystyle \begin{cases}
\displaystyle	 -\varphi _{zz}^{B,2}+v _{z}^{B,0}(\varphi _{xx} ^{I,0}(0,t)z+\varphi _{x} ^{I,1}(0,t)+\varphi _{z}^{B,2})\\[1mm]
\displaystyle \qquad + v _{z}^{B,1}(\varphi _{x} ^{I,0}(0,t)+M+\varphi _{z}^{B,1})+\varphi _{z}^{B,1}v _{x} ^{I,0}(0,t)=0,& z\in \mathbb{R}_+,\\[1mm]
 \displaystyle  v _{t}^{B,1}-v _{zz}^{B,1}+ (\varphi _{x} ^{I,0}(0,t)+M)v ^{B,1}
+\varphi _{z}^{B,1}(v _{x} ^{I,0}(0,t)z+v ^{I,1}(0,t)+v ^{B,1}) \\[1mm]
     \qquad+(\varphi _{xx}^{I,0}(0,t) z+\varphi _{x} ^{I,1}(0,t))v ^{B,0}+ \varphi _{z}^{B,2}(v ^{I,0}(0,t)+v ^{B,0}) =0 , & z\in \mathbb{R}_+,\\
     \displaystyle v ^{B,1}(0,t)=-v ^{I,1}(0,t),\ \  \varphi ^{B,2}(+\infty, t)=v ^{B,1}(+\infty,t)=0,\\
     \displaystyle (\varphi ^{B,2},v ^{B,1})(z,0)=(0,0),
\end{cases}
\end{align}
and the problem for $ (\varphi ^{b,2},v ^{b,1}) $ can be stated as
\begin{align}\label{sec-bd-eq-rt}
\displaystyle \begin{cases}
\displaystyle	 -\varphi _{\xi \xi}^{b,2}+v _{\xi}^{b,0}(\varphi _{xx} ^{I,0}(1,t)\xi+ \varphi _{x}^{I,1}(1,t)+\varphi _{\xi}^{b,2})
     \\[1mm]
    \displaystyle \quad + v _{\xi}^{b,1}(\varphi _{x}^{I,0}(1,t)+M+\varphi _{\xi}^{b,1})+\varphi _{\xi}^{b,1}v _{x} ^{I,0}(1,t)=0, &\xi\in \mathbb{R}_-,\\[1mm]
    \displaystyle  v _{t}^{b,1}-v _{\xi \xi}^{b,1}+ (\varphi _{x}^{I,0}(1,t)+M)v ^{b,1}
+\varphi _{\xi}^{b,1}(v _{x} ^{I,0}(1,t)\xi+v ^{I,1}(1,t)+v ^{b,1})
    \\[1mm]
     \qquad+(\varphi _{xx}^{I,0}(1,t) \xi+\varphi _{x}^{I,1}(1,t))v ^{b,0}+ \varphi _{\xi}^{b,2} (v ^{I,0}(1,t)+v ^{b,0}) =0 , &\xi\in \mathbb{R}_-,\\
     \displaystyle v ^{b,1}(0,t)=-v ^{I,1}(1,t),\ \  \varphi ^{b,2}(-\infty, t)=v ^{b,1}(-\infty,t)=0,\\
     \displaystyle (\varphi ^{b,2},v ^{b,1})(\xi,0)=(0,0).
\end{cases}
\end{align}
Finally, we remark that the global existence and regularity of solutions to problems \eqref{first-outer-problem}, \eqref{first-bd-layer-pro}, \eqref{first-bd-pro-rt}, \eqref{second-bd-eq} and \eqref{sec-bd-eq-rt} will be detailed in Section \ref{sec:study_on_the_inner_outer_layers}.


 \subsection{Statement of main results} 
 \label{sec:main_results}
 To prove the convergence of boundary -layer profiles deduced in the preceding subsection, we require that the initial data $(\varphi _{0}, v _{0})$ satisfy compatibility conditions at the boundary as follows
\begin{subequations}\label{compatibility-simple}
\begin{align}
&\partial _{t}^{i}\varphi ^{I,0}\vert _{t=0}=0, \ i=1,2,3, & \text{on}\ \partial \mathcal{I}, \label{compatibility-simple-1}\\
&v _{0}=v _{\ast}, \ \partial _{t}^{j}v ^{I,0}\vert _{t=0}=0, \ j=1,2, & \text{on}\ \partial \mathcal{I},\label{compatibility-simple-2}
 \end{align}
\end{subequations}
where $\partial _{t}^{i}\varphi ^{I,0}\vert _{t=0}$ and $\partial _{t}^{i}v ^{I,0}\vert _{t=0}$ can be inductively determined from the equations in \eqref{eq-outer-0} as
 \begin{gather}\label{compatibility-vfi0}
\begin{cases}
\partial _{t}\varphi ^{I,0}\vert _{t=0}:=\varphi _{0xx}-(\varphi _{0x}+M)v _{0x}, \\[1mm]
\partial _{t}^2\varphi ^{I,0}\vert _{t=0}:=(\partial _{t}\varphi^{I,0}\vert _{t=0})_{xx}+ (\varphi _{0x}+M)((\varphi _{0x}+M)v _{0}) _{x}-(\partial _{t}\varphi^{I,0}\vert _{t=0})_{x}v _{0x}, \\[1mm]
\partial _{t}^3\varphi ^{I,0}\vert _{t=0}:=(\partial _{t}^{2}\varphi ^{I,0}\vert _{t=0})_{xx}-(\partial _{t}^{2}\varphi ^{I,0} \vert _{t=0}) _{x}v _{0x}+2(\partial _{t}\varphi ^{I,0}\vert _{t=0}) _{x}((\varphi _{0x}+M )v _{0})_{x}
     \\
    \qquad \qquad \  \ \ \ \ +(\varphi _{0x}+M)((\partial _{t}\varphi ^{I,0}\vert_{t=0})_{x}v _{0}) _{x}-(\varphi _{0x}+M)((\varphi _{0x}+M)^{2}v _{0}) _{x}\makebox[-2pt]{~}=0,\\
   \displaystyle   \partial _{t}v ^{I,0}\vert _{t=0}= (\varphi _{0x}+M)v_0,\\[1mm]
   \displaystyle \partial _{t}^2v ^{I,0}\vert _{t=0}=[-\varphi _{0xxx}+((\varphi _{0x}+M)v _{0x}) _{x}+ (\varphi _{0x}+M)^{2}]v_0.
 \end{cases}
 \end{gather}
We say that the initial value $\varphi ^{I,0}|_{t=0}$ of the problem \eqref{eq-outer-0} is compatible with boundary conditions up to order three if it fulfills \eqref{compatibility-simple-1}, while the initial values of problem \eqref{first-bd-layer-pro} and \eqref{first-bd-pro-rt} are compatible with boundary conditions up to order two if the conditions in \eqref{compatibility-simple-2} hold. The compatibility conditions for other initial-boundary value problem mentioned in the sequel are defined similarly.
In terms of the initial data $(\varphi_0, v_0)$, we can write the compatibility conditions given by \eqref{compatibility-simple}-\eqref{compatibility-vfi0} more explicitly as
 \begin{gather}\label{compatibility-vfi}
 \displaystyle \begin{cases}
        \displaystyle v _{0}=v _{\ast},\ \varphi _{0x}+M=0, \ \varphi_{0xx}v _{0x}- \varphi_{0xxx}=0, & \text{on}\ \partial \mathcal{I},\\[1mm]
   \displaystyle (\partial _{t}\varphi^{I,0}\vert _{t=0})_{xx}-(\partial _{t}\varphi^{I,0}\vert _{t=0})_{x}v _{0x} =0, & \text{on}\ \partial \mathcal{I},\\[1mm]
   \displaystyle  (\partial _{t}^{2}\varphi ^{I,0}\vert _{t=0})_{xx}-(\partial _{t}^{2}\varphi ^{I,0} \vert _{t=0}) _{x}v _{0x}+2(\partial _{t}\varphi ^{I,0}\vert _{t=0}) _{x}(\varphi _{0x}+M )v _{0x}=0, & \text{on}\ \partial \mathcal{I},
 \end{cases}
 \end{gather}
 where for brevity we have not explicitly expressed $\partial _{t}\varphi^{I,0}\vert _{t=0}$ and $\partial _{t}^2\varphi^{I,0}\vert _{t=0}$ that are  given in \eqref{compatibility-vfi0}.

 We underline that the condition $ (\varphi _{0x}+M) \vert _{\partial \mathcal{I}}=0 $ in \eqref{compatibility-vfi} implies that $ \inf _{x\in \mathcal{I}} u_0=0$ (i.e. the initial value $u_0$ is degenerate on $\mathcal{\bar{I}}$) and hence $ \inf _{(x,t)\in \mathcal{I}\times (0,T]} u ^{I,0}(x,t)=0$, where $ u ^{I,0} $ is the leading outer-layer profile of $ u $ satisfying $u ^{I,0}(x,0)=u_0(x)$, see \eqref{original-zero-diffu}.

\vspace{2mm}

The main results of this paper concerning the convergence of boundary layers for the reformulated problem \eqref{refor-eq}--\eqref{BD-POSITIVE-VE} as $\varepsilon \to 0$ are stated in the following.
\begin{theorem}\label{thm-stabi-refor}
Assume that $ (\varphi  _{0},v _{0})\in H ^{7} \times  H ^{7} $ and $ (\sqrt{v _{0}}) _{x} \in L ^{2}$ with $ \varphi _{0x}+M \geq, \not\equiv 0 $ satisfying \eqref{compatibility-vfi}. Then for any $ v _{\ast}>0 $, there exists constants $ T _{0}(v _{\ast})>0 $ and $ \varepsilon _{0}>0 $, where $ T _{0}(v _{\ast}) \rightarrow \infty$ as $ v _{\ast}\rightarrow 0 $, such that for any $ \varepsilon \in (0,\varepsilon _{0}) $, the problem \eqref{refor-eq}--\eqref{BD-POSITIVE-VE} admits a unique solution $ (\varphi ^{\varepsilon},v ^{\varepsilon}) \in  L ^{\infty}(0,T _{0} ;H ^{2}\times H ^{2}) $ satisfying the following asymptotic expansions for any $x\in [0,1]$
\begin{subequations}\label{conver-thm1}
\begin{align}
\displaystyle   \displaystyle  \varphi ^{\varepsilon}(x,t)&=\varphi ^{I,0}(x,t)+ \varepsilon ^{1/2}\left[ \varphi ^{I,1}(x,t) +\varphi ^{B,1} \left( z,t\right) +\varphi ^{b,1}\left(\xi ,t\right) \right] +O(\varepsilon ^{5/8}),\\
  \displaystyle \varphi _{x}^{\varepsilon}(x,t)&=\varphi _{x}^{I,0}(x,t)+ \left[ \varphi _{z}^{B,1}(z,t)+\varphi _{\xi}^{b,1}(\xi,t) \right]+O(\varepsilon ^{1/4}),
   \\
     \displaystyle v ^{\varepsilon}(x,t)&=v ^{I,0}(x,t)+v ^{B,0} \left( z,t \right) +v ^{ b,0}\left( \xi,t \right)+ O(\varepsilon ^{1/2}),
\end{align}
\end{subequations}
with $z:=\frac{x}{\varepsilon ^{1/2}}$ and $\xi:= \frac{x-1}{\varepsilon ^{1/2}} $,
where $ (\varphi^{I,0}, v ^{I,0}) $, $ v ^{B,0} $ and $ v ^{b,0} $ are solutions of problems \eqref{eq-outer-0}, \eqref{first-bd-layer-pro} and \eqref{first-bd-pro-rt}, respectively, $\varphi ^{I,1}  $ is determined by \eqref{first-outer-problem}, $ \varphi ^{B,1} $ and $ \varphi ^{b,1} $ are given by \eqref{vfi-bd-1ord-lt} and \eqref{firs-bd-1-rt}, respectively.
\end{theorem}

\vspace{2mm}

With the transformation \eqref{anti-derivatives-transf}, we can transfer the results of \eqref{refor-eq}--\eqref{BD-POSITIVE-VE} stated in Theorem \ref{thm-stabi-refor} to the original problem \eqref{eq-orignal}-\eqref{intial-bdary}. Indeed from \eqref{anti-derivatives-transf}, we have
\begin{equation}\label{UI0}
u ^{\varepsilon} = \varphi _{x}^{\varepsilon}+M,\ \ u ^{I,0}=\varphi _{x}^{I, 0}+M
\end{equation}
with $ \varphi ^{\varepsilon} $ and $ \varphi ^{I,0} $ being the solutions to the problem \eqref{refor-eq}--\eqref{BD-POSITIVE-VE} and the problem \eqref{eq-outer-0}, respectively. Then $ (u ^{\varepsilon}, v ^{\varepsilon}) $ and $ (u ^{I, 0}, v ^{I, 0}) $ solve the problem \eqref{eq-orignal}-\eqref{intial-bdary} for $ \varepsilon>0 $ and $ \varepsilon=0 $, respectively. With \eqref{vfi-bd-1ord-lt} and \eqref{firs-bd-1-rt}, we have
\begin{eqnarray}\label{BV}
\begin{aligned}
\displaystyle  &u ^{B,0}\left(z,t \right)=\varphi _{z}^{B,1}(z,t)=(\varphi _{x}^{I,0}(0,t)+M)(		{\mathop{\mathrm{e}}}^{v ^{B,0}(z,t)}-1) ,\\
\displaystyle &u ^{b,0}\left(\xi,t \right)=\varphi _{\xi}^{b,1}(\xi,t)=
(\varphi _{x}^{I,0}(1,t)+M)(		{\mathop{\mathrm{e}}}^{v ^{b,0}(\xi,t)}-1).
\end{aligned}
\end{eqnarray}
Then the convergence of boundary layer solutions of the original problem \eqref{eq-orignal}-\eqref{intial-bdary} is stated in the following theorem.

\begin{theorem}\label{thm-original}
Assume that $ (u _{0},v _{0})\in H ^{6} \times  H ^{7} $ with $ u _{0}\geq,\not\equiv 0, v_0\geq 0$ and $ (\sqrt{v _{0}}) _{x} \in L ^{2}$ satisfying the compatibility conditions \eqref{compatibility-vfi} with $\varphi_{0x}=u_0-M$. Then for any $ v _{\ast}>0 $, there exists constants $ T _{0}(v _{\ast})>0 $ and $ \varepsilon _{0}>0 $, where $ T _{0}(v _{\ast}) \rightarrow \infty$ as $ v _{\ast}\rightarrow 0 $, such that for any $ \varepsilon \in (0,\varepsilon _{0})$, the problem \eqref{eq-orignal}-\eqref{intial-bdary} admits a unique solution $ (u ^{\varepsilon},v ^{\varepsilon}) \in L ^{\infty}(0,T _{0} ;H ^{1}\times H ^{2})  $ which satisfies for any $x\in [0,1]$
\begin{eqnarray}\label{u-converg}
\begin{aligned}
&u ^{\varepsilon}(x,t)=u ^{I,0}(x,t)+   u^{B,0}\Big(\frac{x}{\sqrt{\varepsilon}},t\Big)+u ^{b,0}\Big(\frac{1-x}{\sqrt{\varepsilon}},t\Big) +O(\varepsilon ^{1/4}), \\[1mm]
&\displaystyle v ^{\varepsilon}(x,t)=v ^{I,0}(x,t)+v ^{B,0}\Big(\frac{x}{\sqrt{\varepsilon}},t\Big) +v ^{ b,0}\Big(\frac{1-x}{\sqrt{\varepsilon}},t)+ O(\varepsilon ^{1/2}),
\end{aligned}
\end{eqnarray}
where $u^{I,0}$ and $(u^{B,0}, b^{b,0})$ are given in \eqref{UI0} and \eqref{BV}, respectively, while $(\varphi^{I,0}, v^{I,0})$, $v ^{B,0}$ and $ v ^{b,0}$ are solutions of problems \eqref{eq-outer-0}, \eqref{first-bd-layer-pro} and \eqref{first-bd-pro-rt}, respectively.
\end{theorem}

\begin{remark}\label{remark-main}
\em{ We give several remarks to enhance the understanding of our results.
\begin{itemize}
\item The $O(\varepsilon^r)$, for some $r>0$, notation used in the main results is a shortcut whose exact meaning is that the difference of the two sides of the identities \eqref{conver-thm1} and \eqref{u-converg} in $L _{T}^{\infty}L _{x}^{\infty}$, for any $0<T<T_0$, is bounded by $\varepsilon^r$ modulo a constant depending only on the initial data and $v_\ast$.

\item The conditions of $ (u _{0}, v _{0}) $ assumed in Theorem \ref{thm-original} can be fulfilled by many functions, for instance $ u _{0}=x^{8}(1-x)^{8} $ and $ v _{0}=v _{\ast}+x ^{6}(1-x)^{6} $. Furthermore, if $ (u _{0}, v _{0}) $ satisfies some higher-order compatibility conditions, by the standard energy method (cf. \cite[Chap. 7]{evans-book}), one can prove that the solutions $ (\varphi ^{\varepsilon},v ^{\varepsilon}) $ and $ (u ^{\varepsilon},v ^{\varepsilon}) $ obtained in Theorem \ref{thm-stabi-refor} and Theorem \ref{thm-original} are indeed classical. We skip the details here since this is not the main goal of this paper.

\item From the refined solution structure given in \eqref{u-converg}, without difficulty we can show for any $\delta=O(\varepsilon ^{\alpha})>0~(0< \alpha<1/2)$, it holds that
\begin{eqnarray*}
\begin{aligned}
& \lim _{\varepsilon \rightarrow 0}\left\|u^{\varepsilon}-u^{I,0}\right\|_{L^{\infty}([\delta, 1-\delta] \times[0, T_0])}=0,\ \liminf _{\varepsilon \rightarrow 0}\left\|u^{\varepsilon}-u^{I,0}\right\|_{L^{\infty}([0,1] \times[0, T_0])}>0,\\
& \lim _{\varepsilon \rightarrow 0}\left\|v^{\varepsilon}-v^{I,0}\right\|_{L^{\infty}([\delta, 1-\delta] \times[0, T_0])}=0,\ \liminf _{\varepsilon \rightarrow 0}\left\|v^{\varepsilon}-v^{I,0}\right\|_{L^{\infty}([0,1] \times[0, T_0])}>0,
\end{aligned}
\end{eqnarray*}
which indicates that the solution $(u^\varepsilon, v^\varepsilon)$ of \eqref{eq-orignal}-\eqref{intial-bdary} will develop a boundary layer profile with thickness of order $ \varepsilon ^{1/2} $ as $\varepsilon \to 0$, which consists of out-layer profile $(u^{I,0}, v^{I,0})$ (i.e. the solution of \eqref{eq-orignal}-\eqref{intial-bdary} with $\varepsilon=0$) and boundary (inner) layer profiles $(u^{B,0}, v^{B,0})$ at the left boundary $x=0$ and $(u^{b,0}, v^{b,0})$ at the right boundary $x=1$, with an error at the order of $\varepsilon ^{1/4}$ for $u^\varepsilon$ and of $\varepsilon ^{1/2}$ for $v^\varepsilon$ as $\varepsilon \to 0$.

\item Though the boundary values of $u ^{\varepsilon}$ are elusive in the zero-flux boundary condition of $ u $ prescribed for $u$ in \eqref{intial-bdary}, the expansion \eqref{u-converg} not only indicates that $u ^{\varepsilon}(x,t)$ has boundary layer profiles $u^{B,0}(z,t)$ near $x=0$ and $u^{b,0}(z,t)$ near $x=1$, but also gives the approximate boundary value of $u$ for $0<\varepsilon \ll 1$
\begin{eqnarray*}
\begin{aligned}
&u ^{\varepsilon}(0,t)=u^{I,0}(0,t)\exp\big(v_*-v^{I,0}(0,t)\big)+O(\varepsilon ^{1/4}), \\
&u ^{\varepsilon}(1,t)=u^{I,0}(1,t)\exp\big(v_*-v^{I,0}(1,t)\big)+O(\varepsilon ^{1/2}),
\end{aligned}
\end{eqnarray*}
where $u^{I,0}(x,t)=\varphi _{x}^{I, 0}+M$, see \eqref{BV}.
\item When $ v _{\ast}=0 $, according to our analysis, the boundary layer profiles in \eqref{u-converg} will vanish, which leads to $ (u ^{\varepsilon},v ^{\varepsilon}) \rightarrow (u ^{I,0}, v ^{I,0}) $ in $ L ^{\infty} $ as $ \varepsilon \rightarrow 0 $, where $(u ^{I,0}, v ^{I,0})  $ is the solution of the problem \eqref{eq-orignal}--\eqref{intial-bdary} with $ \varepsilon=0 $.
\item The compatibility condition $(\varphi _{0x}+M)\vert _{\partial \mathcal{I}}=0$ implies $\min \limits_{x \in \bar{\mathcal{I}}} u _{0}=0$. If we assume $\min \limits_{x \in \bar{\mathcal{I}}} u _{0}>0$, by the maximum principle  we can find some constant $c>0 $ which may depend on $ T_0$ such that $ 0<c^{-1} \leq u ^{I,0}(x,t) \leq c $ for any $ t \in [0,T_0] $ and $x\in (0,1)$. In this case the condition $ (\varphi _{0x}+M) \vert _{\partial \mathcal{I}}=0 $ in \eqref{compatibility-vfi} will fail, and consequently the initial values of \eqref{first-bd-layer-pro} and \eqref{first-bd-pro-rt} only satisfy the zero-order compatibility conditions, for which initial layers will be present and the key analyses in this paper are inapplicable. We shall investigate this case in a separate paper using different approaches.
\end{itemize}
}
\end{remark}

\vspace{4mm}

\section{Regularity of the outer/boundary layer profiles} 
\label{sec:study_on_the_inner_outer_layers}
In this section, we shall derive the regularity of solutions to  problems \eqref{eq-outer-0}, \eqref{first-outer-problem}, \eqref{first-bd-layer-pro}, \eqref{first-bd-pro-rt}, \eqref{second-bd-eq} and \eqref{sec-bd-eq-rt}, respectively. Let us begin with the problem \eqref{eq-outer-0} for the leading-order outer-layer profile $ (\varphi ^{I,0}, v ^{I,0})$. As mentioned before, this problem is exactly the zero-diffusion problem of \eqref{refor-eq} which, in the sense of classical solutions, is equivalent to the zero-diffusion problem of \eqref{eq-orignal}-\eqref{intial-bdary}. Denote by $ (u ^{I,0},v ^{I,0}) $ the solution to the zero-diffusion problem of \eqref{eq-orignal}-\eqref{intial-bdary}. Then we have
\begin{gather}\label{original-zero-diffu}
\displaystyle \begin{cases}
	\displaystyle u _{t}^{I,0}=\left( u _{x}^{I,0}-u ^{I,0}v _{x}^{I,0} \right)_{x}, & x \in \mathcal{I}, \\
	\displaystyle v _{t}^{I,0}=-u ^{I,0}v ^{I,0}, & x \in \mathcal{I},\\
	\displaystyle  (u _{x} ^{I,0}-u ^{I,0} v _{x} ^{I,0})\vert _{\partial \mathcal{I}}=0, \\
	\displaystyle (u ^{I,0}, v ^{I,0})(x,0)=(u _{0}, v _{0})(x).
\end{cases}
\end{gather}
We will first establish the global existence of solutions to the problem \eqref{original-zero-diffu}, and then transfer the result to problem \eqref{eq-outer-0}.
\begin{lemma}\label{lem-orig-zero-dif-problem}
Assume that $ (u _{0},v _{0})\in H ^{6} \times  H ^{7} $ with $ u _{0}\geq,\not\equiv 0, v_0\geq 0$ and $ (\sqrt{v _{0}}) _{x} \in L ^{2}$ subject to compatibility conditions in  \eqref{compatibility-vfi} with $ \varphi _{0}= \int _{0}^{x}(u _{0}-M)\mathrm{d}y$ and $ M= \int _{\mathcal{I}}u _{0}\mathrm{d}x $. Then for any $ T>0 $, the problem \eqref{original-zero-diffu} admits a unique classical solution on $ [0,T] $ such that
\begin{subequations}\label{con-u-I-0-regula}
\begin{gather}
u ^{I,0} \geq 0,\ \
\|\partial _{t}^{k}u ^{I,0}\|_{L _{T}^{2}H ^{7-2k}}\leq c(T), \ \  k=0,1,2,3,4,\label{u-I-0a} \\
\displaystyle \|v ^{I,0}\|_{L _{T} ^{\infty}H ^{7}} + \|\partial _{t}^{k}v ^{I,0}\|_{L _{T}^{2}H ^{9-2k}} \leq c(T), \ \ k=1,2,3,4.
\end{gather}
\end{subequations}
\end{lemma}
\begin{proof}
  The proof of local existence and uniqueness of classical solutions to the problem \eqref{original-zero-diffu} is standard based on the fixed point theorem, so is the property $ u ^{I,0} \geq 0 $ in its lifespan if $  u _{0} \geq 0 $~(cf. \cite{fan2017global}). In the following, we are devoted to deriving the \emph{a priori} estimates of solutions by which the local solutions can be extended to global ones. To begin with, for any $ T>0 $, we assume that $ (u ^{I,0}, v ^{I,0}) $ is a classical solution to the problem \eqref{original-zero-diffu} on $ [0,T] $ satisfying the following \emph{a priori} assumption
  \begin{gather}\label{u-i-0-apprio}
  \displaystyle \int _{0}^{t}\|v _{x}^{I,0}\|_{L ^{\infty}}^{2}\mathrm{d}\tau \leq C _{1},\ \ t \in [0,T]
  \end{gather}
  for some constant $ C _{1}>0 $ to be determined later. Testing the equation $ \eqref{original-zero-diffu} _{1} $ against $ u _{-}^{I,0}:=-\max \{-u ^{I,0},0 \} $, we get
  \begin{align*}
  &\displaystyle  \frac{1}{2}\frac{\mathrm{d}}{\mathrm{d}t} \int _{\mathcal{I}}\vert u _{-}^{I,0}\vert ^{2}\mathrm{d}x + \int _{\mathcal{I}}\vert( u _{-} ^{I,0})_{x}\vert ^{2}\mathrm{d}x = \int _{ \{ u ^{I,0}<0 \} }u ^{I,0}v _{x}^{I,0}u _{x}^{I,0}\mathrm{d}x
   \nonumber \\
   &~\displaystyle \leq \frac{1}{2}\int _{\mathcal{I}}\vert( u _{-} ^{I,0})_{x}\vert ^{2}\mathrm{d}x +c _{0} \|v _{x}^{I,0}\|_{L ^{\infty}}^{2}\int _{\mathcal{I}}\vert u _{-}^{I,0}\vert ^{2}\mathrm{d}x,
  \end{align*}
  where the Cauchy-Schwarz inequality has been used, and the constant $ c _{0}>0 $ is independent of $C_{1}$. This along with \eqref{u-i-0-apprio} and the Gronwall inequality gives
  \begin{align*}
  \displaystyle \int _{\mathcal{I}}\vert u _{-}^{I,0}\vert ^{2}\mathrm{d}x \leq 		{\mathop{\mathrm{e}}}^{C _{1}t} \int _{\{ u _{0}<0 \}}\vert u _{0}\vert ^{2}\mathrm{d}x =0
  \end{align*}
  for any $ t \in (0,T] $, where $ u _{0}\geq 0 $ has been used. Therefore it holds that
  \begin{align}\label{u-I-0-positive}
  \displaystyle u ^{I,0}(x,t) \geq 0,\ \ \ t \in (0,T].
  \end{align}
  With \eqref{u-I-0-positive}, we have from $ \eqref{original-zero-diffu}_{2} $ that $ v ^{I,0} \leq v _{0} $. Testing $ \eqref{original-zero-diffu}_{1} $ against $ \ln u ^{I,0} $, one has
  \begin{align}\label{u-i-0-entropy}
   \displaystyle \frac{\mathrm{d}}{\mathrm{d}t} \int _{\mathcal{I}} u ^{I,0} \ln u^{I,0} \mathrm{d}x+ \int _{\mathcal{I}} \frac{\vert u _{x}^{I,0}\vert^{2}}{u ^{I,0}} \mathrm{d}x= \int _{\mathcal{I}} u _{x}^{I,0} v  _{x}^{I,0} \mathrm{d}x,
   \end{align}
    where $ \int _{\mathcal{I}}u ^{I,0} \mathrm{d}x=\int _{\mathcal{I}}u _{0}\mathrm{d}x=M $ due to the zero-flux boundary condition. Differentiating $ \eqref{original-zero-diffu}_{2} $ with respect to $ x $, and testing the resulting equation against ${v _{x}^{I,0}}/{v ^{I,0}} $, we get
    \begin{align}\label{v-i-0-zero}
    \displaystyle \frac{1}{2}\frac{\mathrm{d}}{\mathrm{d}t}\int _{\mathcal{I}} \frac{\vert v _{x}^{I,0}\vert ^{2}}{v ^{I,0}} \mathrm{d}x + \frac{1}{2}\int _{\mathcal{I}} \frac{u ^{I,0}\vert v _{x}^{I,0}\vert ^{2}}{v ^{I,0}} \mathrm{d}x=-\int _{\mathcal{I}} u _{x}^{I,0} v  _{x}^{I,0} \mathrm{d}x.
    \end{align}
    Combining \eqref{u-i-0-entropy} with \eqref{v-i-0-zero}, and integrating the resulting identity over $ [0,t] $ for any $ t \in (0,T] $, we have
    \begin{align*}
    \displaystyle   \int _{\mathcal{I}} u ^{I,0} \ln u^{I,0} \mathrm{d}x+\frac{1}{2}\int _{\mathcal{I}} \frac{\vert v _{x}^{I,0}\vert ^{2}}{v ^{I,0}} \mathrm{d}x+\int _{0}^{t} \int _{\mathcal{I}} \left( \frac{\vert u _{x}^{I,0}\vert^{2}}{u ^{I,0}}+\frac{1}{2}\frac{u ^{I,0}\vert v _{x}^{I,0}\vert ^{2}}{v ^{I,0}} \right)  \mathrm{d}x \mathrm{d}\tau \leq c _{0},
    \end{align*}
    which, along with the basic inequality $ -x \ln x \leq \mathop{\mathrm{e}} ^{-1}$ for $ x \geq 0 $, and $ v ^{I,0} \leq v _{0} $, gives
    \begin{gather}\label{v-i-0-esti-con}
    \displaystyle \frac{1}{2} \int _{\mathcal{I}}\vert v _{x}^{I,0}\vert ^{2} \mathrm{d}x+\int _{0}^{t} \int _{\mathcal{I}} \left( \frac{\vert u _{x}^{I,0}\vert^{2}}{u ^{I,0}}+\frac{1}{2}\frac{u ^{I,0}\vert v _{x}^{I,0}\vert ^{2}}{v ^{I,0}} \right)  \mathrm{d}x \mathrm{d}\tau  \leq c _{0}
    \end{gather}
    for any $ t \in [0,T] $, where the constant $ c _{0}>0 $ is independent of $C_{1}$. Furthermore, it holds from \eqref{u-I-0-positive}, \eqref{u-i0-inf-multi}, the basic inequality $ \|f\|_{L ^{\infty}} \leq c _{0} \|f\|_{W ^{1,1}} $ and the H\"older inequality that
\begin{align}\label{u-i0-inf-multi}
\displaystyle \int _{0}^{T} \|u ^{I,0}\|_{L ^{\infty}} \mathrm{d}\tau &\leq c _{0}\int _{0}^{T}\|u _{x}^{I,0}\|_{L ^{1}}\mathrm{d}\tau+ c _{0}\int _{0}^{T}\|u ^{I,0}\|_{L ^{1}}\mathrm{d}\tau
 \nonumber \\
 &\displaystyle \leq   \int _{0}^{T}\left( \int _{\mathcal{I}}\frac{\vert u _{x}^{I,0}\vert^{2}}{u ^{I,0}} \mathrm{d}x\right)^{1/2}\left( \int _{\mathcal{I}}u ^{I,0} \mathrm{d}x \right)^{1/2}\mathrm{d}\tau +c(T) \leq  c(T),
\end{align}
where $ c(T)> 0$ is as stated in Section \ref{sec:main_result}, and it is independent of $C_{1}$. To proceed, multiplying $ \eqref{original-zero-diffu}_{1} $ by $ u _{t}^{I,0}$ followed by an integration over $ \mathcal{I} $, we have
\begin{align}\label{u-i-0-two}
&\displaystyle \frac{1}{2}\frac{\mathrm{d}}{\mathrm{d}t}\int _{\mathcal{I}} \vert u _{x}^{I,0}\vert^{2} \mathrm{d}x + \int _{\mathcal{I}}\vert u _{t}^{I,0}\vert^{2}\mathrm{d}x
 \nonumber \\
 & \displaystyle= \int _{\mathcal{I}} u ^{I,0} v _{x}^{I,0} u _{tx}^{I,0} \mathrm{d}x= \frac{\mathrm{d}}{\mathrm{d}t}\int _{\mathcal{I}} u ^{I,0} v _{x}^{I,0} u _{x}^{I,0} \mathrm{d}x
 - \int _{\mathcal{I}} u _{t}^{I,0} v _{x}^{I,0}u _{x}^{I,0} \mathrm{d}x- \int _{\mathcal{I}} u^{I,0} v _{xt}^{I,0}u _{x}^{I,0} \mathrm{d}x.
\end{align}
By \eqref{v-i-0-esti-con}, Sobolev inequality \eqref{Sobolev-infty} and equation $ \eqref{original-zero-diffu} _{1} $, we deduce that
\begin{align*}
\displaystyle \int _{\mathcal{I}} |u _{x}^{I,0}|^{2}|v _{x}^{I,0}|^{2} \mathrm{d}x &
  \leq \|u _{x}^{I,0}\|_{L ^{\infty}}^{2}\|v _{x}^{I,0}\|_{L ^{2}}^{2}
   \nonumber \\
   & \displaystyle \leq c_{0}\|u _{x}^{I,0}\|_{L ^{2}}\|u _{xx}^{I,0}\|_{L ^{2}}+c _{0} \|u _{x}^{I,0}\|_{L ^{2}}^{2}
    \nonumber \\
    & \displaystyle \leq c _{0}\|u _{x}^{I,0}\|_{L ^{2}}(\|u _{t}^{I,0}\|_{L ^{2}}+\|u _{x}^{I,0} v _{x}^{I,0}\|_{L ^{2}}+\|u ^{I,0} v _{xx} ^{I,0}\|_{L ^{2}})+c _{0}\|u _{x}^{I,0}\|_{L ^{2}}^{2}
     \nonumber \\
     & \displaystyle \leq \frac{1}{2}\|u _{x}^{I,0} v _{x}^{I,0}\|_{L ^{2}}^{2}+ \frac{1}{32}\|u _{t}^{I,0}\|_{L ^{2}}^{2}+c _{0} \|u ^{I,0}\|_{L ^{\infty}}\|v _{xx}^{I,0}\| _{L ^{2}}\|u _{x}^{I,0}\| _{L ^{2}}+c_{0} \|u _{x}^{I,0}\|_{L ^{2}}^{2}.
\end{align*}
That is,
\begin{align}\label{nonliear-esti}
\displaystyle  \int _{\mathcal{I}} \vert u _{x}^{I,0}\vert^{2}\vert v _{x}^{I,0}\vert^{2} \mathrm{d}x  \leq \frac{1}{16}\|u _{t}^{I,0}\|_{L ^{2}}^{2}+c _{0}\|u ^{I,0}\|_{L ^{\infty}}\|v _{xx}^{I,0}\| _{L ^{2}}\|u _{x}^{I,0}\| _{L ^{2}}+c _{0}\|u _{x}^{I,0}\|_{L ^{2}}^{2}
\end{align}
for some constant $ c _{0}>0 $ independent of $C_{1}$.
This along with the Cauchy-Schwarz inequality gives
\begin{align*}
\displaystyle  - \int _{\mathcal{I}} u _{t}^{I,0}v _{x}^{I,0}u _{x}^{I,0} \mathrm{d}x &\leq \frac{1}{4}\int _{\mathcal{I}} \vert u _{t}^{I,0}\vert ^{2}\mathrm{d}x+4 \int _{\mathcal{I}}  |u _{x}^{I,0}|^{2}|v _{x}^{I,0}|^{2}\mathrm{d}x
 \nonumber \\
 &\displaystyle\leq \frac{1}{2}\|u _{t}^{I,0}\|_{L ^{2}}^{2}+c _{0}\|u ^{I,0}\|_{L ^{\infty}}\left( \|v _{xx}^{I,0}\| _{L ^{2}}^{2}+ \|u _{x}^{I,0}\| _{L ^{2}}^{2} \right) +c _{0}\|u _{x}^{I,0}\|_{L ^{2}}^{2}.
\end{align*}
Noticing from $ \eqref{original-zero-diffu} _{2} $ that $ v _{tx}^{I,0}=-u _{x}^{I,0}v^{I,0} -u^{I,0} v _{x}^{I,0} $, we estimate the last term on the right-hand side of \eqref{u-i-0-two} as follows
\begin{align*}
\displaystyle - \int _{\mathcal{I}} u^{I,0} v _{xt}^{I,0}u _{x}^{I,0} \mathrm{d}x &= \int _{\mathcal{I}} u^{I,0}(u _{x}^{I,0}v^{I,0}+v _{x}^{I,0}u^{I,0}) u _{x}^{I,0}\mathrm{d}x
 \nonumber \\
 & \displaystyle\leq  \|v^{I,0}\|_{L ^{\infty}}\|u^{I,0}\|_{L ^{\infty}}\|u _{x}^{I,0}\|_{L ^{2}}^{2}+ \|u^{I,0}\|_{L ^{\infty}} \left( \|u^{I,0}\|_{L ^{\infty}}^{2}\|v _{x}^{I,0}\| _{L ^{2}}^{2}+\|u _{x}^{I,0}\|_{L ^{2}}^{2} \right)
 \nonumber \\
 & \displaystyle \leq c_{0} \|u^{I,0}\|_{L ^{\infty}} \|u _{x}^{I,0}\|_{L ^{2}}^{2}+c_{0} \|u ^{I,0}\|_{L ^{\infty}}\left[ ( \|u _{x}^{I,0}\|_{L ^{2}}+\|u ^{I,0}\|_{L ^{1}} )^{2}+\|u _{x}^{I,0}\|_{L ^{2}}^{2} \right]
  \nonumber \\
   & \displaystyle \leq c_{0} \|u^{I,0}\|_{L ^{\infty}} \left( \|u _{x}^{I,0}\|_{L ^{2}}^{2}+1 \right),
\end{align*}
where we have used \eqref{v-i-0-esti-con}, $ v ^{I,0} \leq v _{0} $, $ \|u ^{I,0}\|_{L ^{1}}=M $, \eqref{l-infty-Sobolev} and the Cauchy-Schwarz inequality. Therefore we have from \eqref{u-i-0-two} that
\begin{align}\label{esti-con-u-I-0-X}
&\displaystyle  \frac{1}{2}\int _{\mathcal{I}} \vert u _{x}^{I,0}\vert^{2} \mathrm{d}x -\int _{\mathcal{I}} u ^{I,0} v _{x}^{I,0} u _{x}^{I,0} \mathrm{d}x+ \frac{1}{2}\int _{0}^{t}\int _{\mathcal{I}} \vert u _{\tau}^{I,0}\vert^{2} \mathrm{d}x \mathrm{d}\tau
 \nonumber \\
 &~\displaystyle \leq c (T)+c _{0} \int _{0}^{t}\|u ^{I,0}\|_{L ^{\infty}}\left( \|v _{xx}^{I,0}\| _{L ^{2}}^{2}+ \|u _{x}^{I,0}\| _{L ^{2}}^{2} \right)\mathrm{d}\tau + c_{0}\int _{0}^{t}\|u _{x}^{I,0}\|_{L ^{2}}^{2}\mathrm{d}\tau,
\end{align}
where \eqref{u-i0-inf-multi} has been used, and the constant $  c _{0}>0$ is independent of $C_{1}$. Noting that
\begin{align*}
\displaystyle  \int _{\mathcal{I}} u ^{I,0}v _{x}^{I,0}u _{x}^{I,0} \mathrm{d}x &\leq \frac{1}{8}\int _{\mathcal{I}} \vert u _{x}^{I,0}\vert^{2} \mathrm{d}x+ \int _{\mathcal{I}} \vert u^{I,0}\vert ^{2}\vert v _{x}^{I,0}\vert^{2} \mathrm{d}x
 \nonumber \\
 & \displaystyle \leq \frac{1}{8}\int _{\mathcal{I}} \vert u _{x}^{I,0}\vert^{2} \mathrm{d}x+  c _{0}\|u ^{I,0}\|_{L ^{\infty}}^{2}\leq \frac{1}{4}\int _{\mathcal{I}} \vert u _{x}^{I,0}\vert^{2} \mathrm{d}x+ c _{0}
\end{align*}
due to \eqref{v-i-0-esti-con}, \eqref{l-infty-Sobolev} and $ \|u ^{I,0}\|_{L ^{1}}=M $,
we further update \eqref{esti-con-u-I-0-X} as
\begin{align}\label{u-x-I-0-inte-con}
&\displaystyle  \int _{\mathcal{I}} \vert u _{x}^{I,0}\vert^{2} \mathrm{d}x +\int _{0}^{t}\int _{\mathcal{I}} \vert u _{\tau}^{I,0}\vert^{2} \mathrm{d}x \mathrm{d}\tau
 \nonumber \\
 &\displaystyle~\leq
 c (T)+ c _{0} \int _{0}^{t}\|u ^{I,0}\|_{L ^{\infty}}\left( \|v _{xx}^{I,0}\| _{L ^{2}}^{2}+ \|u _{x}^{I,0}\| _{L ^{2}}^{2} \right)\mathrm{d}\tau +  c _{0}\int _{0}^{t}\|u _{x}^{I,0}\|_{L ^{2}}^{2}\mathrm{d}\tau.
\end{align}
On the other hand, differentiating the equation $ \eqref{original-zero-diffu}_{2} $ with respect to $ x $ twice gives
\begin{align*}
\displaystyle v _{txx}^{I,0}=-u _{xx}^{I,0}v^{I,0}-2u _{x}^{I,0}v _{x}^{I,0}-u^{I,0} v _{xx}^{I,0}.
\end{align*}
Testing the above equation against $ v _{xx}^{I,0} $,  thanks to $ \eqref{original-zero-diffu}_{1} $, \eqref{nonliear-esti}, the fact $ v ^{I,0}\leq v _{0} $ and the Cauchy-Schwarz inequality, it follows that
\begin{align}\label{v-xx-I-0-diff}
&\displaystyle \frac{1}{2}\frac{\mathrm{d}}{\mathrm{d}t}\int _{\mathcal{I}} \vert v _{xx}^{I,0}\vert ^{2} \mathrm{d}x+ \int _{\mathcal{I}} u ^{I,0}\vert v _{xx}^{I,0}\vert ^{2} \mathrm{d}x=- \int _{\mathcal{I}} u _{xx}^{I,0} v^{I,0}v _{xx}^{I,0}\mathrm{d}x - \int _{\mathcal{I}}2u _{x}^{I,0}v _{x}^{I,0}v _{xx}^{I,0}  \mathrm{d}x
 \nonumber \\
 &~ \displaystyle=- \int _{\mathcal{I}} (u _{t}^{I,0}+u _{x}^{I,0}v _{x}^{I,0}+u ^{I,0}v _{xx}^{I,0})v _{xx}^{I,0} v^{I,0} \mathrm{d}x + \|v _{xx}^{I,0}\|_{L ^{2}}\|u _{x}^{I,0}v _{x}^{I,0}\|_{L ^{2}}
  \nonumber \\
  &~ \displaystyle \leq \frac{1}{8}\int _{\mathcal{I}} \vert u _{t}^{I,0}\vert ^{2} \mathrm{d}x+ c _{0} \|u ^{I,0}\|_{L ^{\infty}}(\|u _{x}^{I,0}\|_{L ^{2}}^{2}+\|v _{xx}^{I,0}\|_{L ^{2}}^{2})+ c _{0}\|u _{x}^{I,0}\|_{L ^{2}}^{2}+ c _{0} \|v _{xx}^{I,0}\|_{L ^{2}}^{2},
\end{align}
where $  c _{0}>0 $ is independent of $C_{1}$. Integrating \eqref{v-xx-I-0-diff} over $ (0,t) $ for any $ t \in(0,T ]$ yields that
\begin{align*}
&\displaystyle  \int _{\mathcal{I}} \vert v _{xx}^{I,0}\vert ^{2} \mathrm{d}x+\int _{0}^{t}\int _{\mathcal{I}} u ^{I,0}\vert v _{xx}^{I,0}\vert ^{2} \mathrm{d}x \mathrm{d}\tau
 \nonumber \\
 &~\displaystyle \leq \frac{1}{8}\int _{0}^{t}\int _{\mathcal{I}} \vert u _{\tau}^{I,0}\vert ^{2} \mathrm{d}x \mathrm{d}\tau+ c _{0} \int _{0}^{t}(\|u ^{I,0}\|_{L ^{\infty}}+1)(\|u _{x}^{I,0}\|_{L ^{2}}^{2}+\|v _{xx}^{I,0}\|_{L ^{2}}^{2})\mathrm{d}\tau.
\end{align*}
This, combined with \eqref{u-I-0-positive} and \eqref{u-x-I-0-inte-con}, implies that
\begin{align}\label{esti-v-I0-xx}
&\displaystyle  \int _{\mathcal{I}} \left( \vert u _{x}^{I,0}\vert^{2} + \vert v _{xx}^{I,0}\vert^{2} \right) \mathrm{d}x +\int _{0}^{t}\int _{\mathcal{I}} \vert u _{\tau}^{I,0}\vert^{2} \mathrm{d}x \mathrm{d}\tau
 \nonumber \\
 &~\displaystyle \leq   c _{0}\int _{0}^{t}(\|u ^{I,0}\|_{L ^{\infty}}+1)(\|u _{x}^{I,0}\|_{L ^{2}}^{2}+\|v _{xx}^{I,0}\|_{L ^{2}}^{2})\mathrm{d}\tau+c(T).
\end{align}
Therefore an application of the Gronwall inequality along with \eqref{u-i0-inf-multi} gives
\begin{align}\label{u-I-0-x-final-con}
\displaystyle  \int _{\mathcal{I}} \left( \vert u _{x}^{I,0}\vert^{2} +|v _{xx}^{I,0}|_{L ^{2}}^{2} \right) \mathrm{d}x +\int _{0}^{t}\int _{\mathcal{I}} \vert u _{\tau}^{I,0}\vert^{2} \mathrm{d}x \mathrm{d}\tau \leq c(T)
\end{align}
for any $ t \in (0,T] $, where the constant $  c(T)> 0 $ is independent of $C_{1}$. Furthermore, in virtue of \eqref{v-i-0-esti-con}, \eqref{u-i0-inf-multi}, \eqref{u-I-0-x-final-con}, \eqref{Sobolev-infty} and the equations in \eqref{original-zero-diffu}, we have
\begin{align}\label{u-I-0-xx}
\displaystyle \int _{0}^{T}\left(\|u _{xx}^{I,0}\|_{L ^{2}}^{2}+\|v _{t}^{I,0}\|_{H ^{2}}^{2} \right)  \mathrm{d}t \leq c(T) .
\end{align}
Using \eqref{v-i-0-esti-con}, \eqref{esti-v-I0-xx} and the Sobolev inequality $ \|f\|_{L ^{\infty}} \leq c _{0} \|f\|_{W ^{1,2}} $, we get
\begin{align*}
\displaystyle  \int _{0}^{t}\|v _{x}^{I,0}\|_{L ^{\infty}}^{2}\mathrm{d}t  \leq c(T)
\end{align*}
where the constant $c(T)>0 $ depends on the initial data and $ T $ but independent of $C_{1}$. Therefore the \emph{a priori} assumption \eqref{u-i-0-apprio} is closed provided that $ C _{1}>0 $ is chosen to be large such that $C_1>c(T)$, and thus the estimates \eqref{u-I-0-positive}, \eqref{v-i-0-esti-con}, \eqref{u-i0-inf-multi}, \eqref{u-I-0-x-final-con} and \eqref{u-I-0-xx} subsequently follow. Next we shall derive some higher-order estimates for the solution. The proof is based on the standard energy method (cf. \cite[pp. 387-388]{evans-book}), namely, recovering the estimates on spatial derivatives from those on time derivatives. For brevity, we will establish the estimates on the second-order time derivatives of the solution only and their implications in the estimates of spatial derivatives, while estimates on the higher-order time derivatives can be obtained in the same spirit. To this end, we differentiate the equations in \eqref{original-zero-diffu} with respect to $ t $ and get
\begin{align}\label{eQ-U-I-0-t}
\displaystyle \begin{cases}
\displaystyle	u _{tt}^{I,0}=\left( u _{tx}^{I,0}-u ^{I,0}v _{tx}^{I,0}-u _{t}^{I,0}v _{x}^{I,0} \right)_{x},\\[2mm]
	\displaystyle v _{tt}^{I,0}=-u _{t}^{I,0}v ^{I,0}-u ^{I,0}v _{t}^{I,0}.
\end{cases}
\end{align}
Multiplying $ \eqref{eQ-U-I-0-t}_{1} $ by $ u _{t}^{I,0} $, and integrating the resulting equation over $ \mathcal{I} $, we have
\begin{align*}
&\displaystyle \frac{1}{2}\frac{\mathrm{d}}{\mathrm{d}t}\int _{\mathcal{I}} \vert u _{t}^{I,0}\vert^{2} \mathrm{d}x +\int _{\mathcal{I}} \vert u _{xt}^{I,0}\vert ^{2}\mathrm{d}x= \int _{\mathcal{I}} u ^{I,0}v _{tx}^{I,0}u _{xt}^{I,0} \mathrm{d}x+\int _{\mathcal{I}}u _{t}^{I,0}v _{x}^{I,0} u _{xt}^{I,0}   \mathrm{d}x
 \nonumber \\
 &~ \leq \frac{1}{4}\int _{\mathcal{I}} \vert u _{xt}^{I,0}\vert ^{2}\mathrm{d}x+c _{0}\int _{\mathcal{I}}\vert  u ^{I,0}v _{tx}^{I,0} \vert ^{2}\mathrm{d}x+c _{0} \int _{\mathcal{I}}\vert u _{t}^{I,0}v _{x}^{I,0}\vert ^{2} \mathrm{d}x
  \nonumber \\
  &~\displaystyle \leq \frac{1}{4}\int _{\mathcal{I}} \vert u _{xt}^{I,0}\vert ^{2}\mathrm{d}x+c _{0} \|u ^{I,0}\|_{L ^{\infty}}^{2}\|v _{tx}^{I,0}\|_{L ^{2}}^{2}+c _{0} \|v _{x}^{I,0}\|_{L ^{\infty}}^{2}\|u _{t}^{I,0}\|_{L ^{2}}^{2}
   \nonumber \\
   &~\displaystyle \leq \frac{1}{4}\int _{\mathcal{I}} \vert u _{xt}^{I,0}\vert ^{2}\mathrm{d}x+c(T)\left(  \|v _{tx}^{I,0}\|_{L ^{2}}^{2}+\|u_{t}^{I,0}\|_{L ^{2}}^{2}\right),
\end{align*}
where we have used the Cauchy-Schwarz inequality and $ \|u ^{I,0}\|_{L _{T} ^{\infty}L ^{\infty}}+ \|v _{x}^{I,0}\|_{L _{T} ^{\infty}L ^{\infty}} \leq c(T)$ ensured by \eqref{v-i-0-esti-con}, \eqref{u-I-0-x-final-con} and \eqref{Sobolev-infty}. Therefore we get, thanks to \eqref{u-I-0-x-final-con} and \eqref{u-I-0-xx},
\begin{align}\label{u-I-0-TX}
\displaystyle  \int _{\mathcal{I}}\vert  u _{t}^{I,0}\vert^{2}(\cdot,t) \mathrm{d}x +\int _{0}^{t}\int _{\mathcal{I}} \vert u _{x \tau}^{I,0}\vert ^{2}\mathrm{d}x \mathrm{d}\tau \leq c(T)
\end{align}
for any $ t \in [0,T] $. This along with \eqref{u-I-0-x-final-con}, \eqref{u-I-0-xx} and the equations in \eqref{original-zero-diffu} further implies that
\begin{align}\label{v-tt-x0}
\displaystyle  \|u _{xx}^{I,0}\|_{L _{T}^{\infty}L ^{2}}^{2}+\|v _{t}^{I,0}\|_{L _{T}^{\infty}H ^{2}} +\|v _{tt}^{I,0}\|_{L _{T}^{\infty}L^{2}} +\|v _{ttx}^{I,0}\|_{L _{T}^{2}L ^{2}} \leq c(T).
\end{align}
Next testing $ \eqref{eQ-U-I-0-t}_{1} $ against $ u _{tt}^{I,0} $, we have
\begin{align}\label{u-t-x-esti0}
&\displaystyle \frac{1}{2}\frac{\mathrm{d}}{\mathrm{d}t}\int _{\mathcal{I}} \vert u _{tx}^{I,0}\vert ^{2} \mathrm{d}x+ \int _{\mathcal{I}} \vert u _{tt}^{I,0}\vert  ^{2}\mathrm{d}x = \int _{\mathcal{I}} \left( u ^{I,0}v _{tx}^{I,0}+u _{t}^{I,0}v _{x}^{I,0}  \right)u _{ttx}^{I,0}  \mathrm{d}x
 \nonumber \\
 &~\displaystyle = \frac{\mathrm{d}}{\mathrm{d}t} \int _{\mathcal{I}} \left( u ^{I,0}v _{tx}^{I,0}+u _{t}^{I,0}v _{x}^{I,0}  \right)u _{tx}^{I,0} \mathrm{d}x- \int _{\mathcal{I}}  \left( u _{t}^{I,0}v _{tx}^{I,0}+u ^{I,0}v _{ttx}^{I,0}+u _{tt}^{I,0}v _{x}^{I,0}+u _{t}^{I,0}v _{xt}^{I,0}  \right)u _{tx}^{I,0} \mathrm{d}x
  \nonumber \\
  &\displaystyle~ \leq \frac{\mathrm{d}}{\mathrm{d}t} \int _{\mathcal{I}} \left( u ^{I,0}v _{tx}^{I,0}+u _{t}^{I,0}v _{x}^{I,0}  \right)u _{tx}^{I,0} \mathrm{d}x + \left( \|u _{t}^{I,0}\|_{L ^{\infty}}\|v _{tx}^{I,0}\|_{L ^{2}}+\|u ^{I,0}\|_{L ^{\infty}}\|v _{ttx}^{I,0}\|_{L ^{2}} \right) \|u _{tx}^{I,0}\|_{L ^{2}}
   \nonumber \\
   &~\displaystyle \quad +c _{0} \left(\|u _{tt}^{I,0}\|_{L ^{2}}\|v _{x}^{I,0}\|_{L ^{\infty}}+ \|u _{t}^{I,0}\|_{L ^{\infty}}\|v _{tx}^{I,0}\|_{L ^{2}} \right) \|u _{tx}^{I,0}\|_{L ^{2}}
    \nonumber \\
    &~\leq \frac{\mathrm{d}}{\mathrm{d}t} \int _{\mathcal{I}} \left( u ^{I,0}v _{tx}^{I,0}+u _{t}^{I,0}v _{x}^{I,0}  \right)u _{tx}^{I,0} \mathrm{d}x + \frac{1}{8}\|u _{tt}^{I,0}\|_{L ^{2}}^{2}+c(T) \|u _{tx}^{I,0}\|_{L ^{2}}^{2}+c(T)\|v _{ttx}^{I,0}\|_{L ^{2}}^{2} ,
\end{align}
where we have used \eqref{u-I-0-x-final-con}, \eqref{u-I-0-xx}, \eqref{u-I-0-TX}, \eqref{v-tt-x0}, the fact $ \|u ^{I,0}\|_{L _{T} ^{\infty}L ^{\infty}}+\|v _{x}^{I,0}\|_{L _{T} ^{\infty}L ^{\infty}} \leq c(T) $ and the Sobolev inequality \eqref{Sobolev-infty}. Noting that
\begin{align*}
&\displaystyle  \int _{\mathcal{I}} \left( u ^{I,0}v _{tx}^{I,0}-u _{t}^{I,0}v _{x}^{I,0}  \right)u _{tx}^{I,0} \mathrm{d}x
 \nonumber \\
 &~\displaystyle\leq \frac{1}{4}\|u _{tx}^{I,0}\|_{L ^{2}}^{2}+c _{0} \left( \|u ^{I,0}\|_{L ^{\infty}}^{2}\|v _{tx}^{I,0}\|_{L ^{2}}^{2}+\|v _{x}^{I,0}\|_{L ^{\infty}}^{2}\|u _{t}^{I,0}\|_{L ^{2}}^{2} \right)  \leq \frac{1}{4}\|u _{tx}^{I,0}\|_{L ^{2}}^{2}+c(T)
\end{align*}
due to \eqref{u-i0-inf-multi}, \eqref{u-I-0-x-final-con}, \eqref{u-I-0-TX} and \eqref{v-tt-x0}, we get after integrating \eqref{u-t-x-esti0} over $ [0,t] $ for any $ t \in (0,T] $
\begin{align*}
\displaystyle  \int _{\mathcal{I}} \vert u _{tx}^{I,0}\vert ^{2}(\cdot,t) \mathrm{d}x+ \int _{0}^{t} \int _{\mathcal{I}} \vert u _{\tau \tau}^{I,0}\vert  ^{2}\mathrm{d}x \mathrm{d}\tau \leq c(T) ,
\end{align*}
where \eqref{u-I-0-TX} has been used. This combined with \eqref{u-I-0-xx}, $ \eqref{eQ-U-I-0-t}_{1} $, \eqref{u-I-0-TX} and \eqref{v-tt-x0} entails that
\begin{align}\label{u-t-I0-multiH2}
 \displaystyle \|u _{t}^{I,0}\|_{L _{T}^{\infty}L ^{\infty}}+ \|u _{t}^{I,0}\|_{L _{T}^{2}H ^{2}} \leq c(T) .
 \end{align}
 Applying $ \partial _{x}^{3} $ to the equation $ \eqref{eQ-U-I-0-t}_{2} $, we get
 \begin{align*}
 \displaystyle \partial _{x}^{3}v _{t}^{I,0}=- \sum _{k=0}^{3}\partial _{x}^{k}u ^{I,0}\partial _{x}^{3-k}v ^{I,0}.
 \end{align*}
 Multiplying this equation by $ \partial _{x}^{3}v ^{I,0} $ followed by an integration over $ \mathcal{I} $, we have
 \begin{align}\label{v-xxx-esti}
 &\displaystyle \frac{1}{2}\frac{\mathrm{d}}{\mathrm{d}t}\int _{\mathcal{I}} \vert \partial _{x}^{3}v ^{I,0}\vert ^{2} \mathrm{d}x + \int _{\mathcal{I}} u ^{I,0}\vert  \partial _{x}^{3}v ^{I,0}\vert ^{2} \mathrm{d}x
  \nonumber \\
  &~\displaystyle \leq  \sum  _{k=0}^{1}\|\partial _{x}^{k}u ^{I,0}\|_{L ^{\infty}}\|\partial _{x}^{3-k}v ^{I,0}\|_{L ^{2}}\|\partial _{x}^{3}v ^{I,0}\|_{L ^{2}}+c _{0} \sum _{k=2}^{3} \|\partial _{x}^{k}u ^{I,0}\|_{L ^{2}}\|\partial _{x}^{3-k}v ^{I,0}\|_{L ^{\infty}}\|\partial _{x}^{3}v ^{I,0}\|_{L ^{2}}
  \nonumber \\
  &~\displaystyle \leq c(T) \|\partial _{x}^{3}u ^{I,0}\|_{L ^{2}} ^{2}+\|\partial _{x}^{3}v ^{I,0}\| _{L ^{2}}^{2}+c(T),
\end{align}
where we have used \eqref{v-i-0-esti-con}, \eqref{u-I-0-x-final-con}, \eqref{v-tt-x0} and the Cauchy-Schwarz inequality. On the other hand, by $ \eqref{original-zero-diffu} _{1} $, \eqref{v-i-0-esti-con}, \eqref{u-I-0-x-final-con} and \eqref{u-t-I0-multiH2}, we get
\begin{align}\label{u-xxx-I-0}
\displaystyle \|\partial _{x}^{3}u ^{I,0}\|_{L ^{2}} ^{2} &\leq \|u _{tx}^{I,0}\|_{L ^{2}}^{2}+ \sum _{k=0}^{1}\|\partial _{x}^{k} u ^{I,0}\|_{L ^{\infty}}^{2}\| \partial _{x}^{3-k}v ^{I,0}\|_{L ^{2}}^{2}+\| (u _{x}^{I,0}v _{x}^{I,0})_{x}\|_{L ^{2}}^{2}
 \nonumber \\
 &\displaystyle \leq c(T) \|\partial _{x}^{3}v ^{I,0}\|_{L ^{2}}^{2}+c(T).
\end{align}
Therefore we update \eqref{v-xxx-esti} as
\begin{align}
\displaystyle  \frac{1}{2}\frac{\mathrm{d}}{\mathrm{d}t}\int _{\mathcal{I}} \vert \partial _{x}^{3}v ^{I,0}\vert ^{2} \mathrm{d}x + \int _{\mathcal{I}} u ^{I,0}\vert  \partial _{x}^{3}v ^{I,0}\vert ^{2} \mathrm{d}x  \leq c(T)\|\partial _{x}^{3}v ^{I,0}\| _{L ^{2}}^{2}+c(T) ,
\end{align}
which along with the Gronwall inequality, \eqref{u-xxx-I-0} and the fact $ u ^{I,0} \geq 0 $ entails that for any $ t \in [0,T] $,
\begin{align}\label{v-I-0-H3}
\displaystyle \|\partial _{x}^{3} v ^{I,0}(\cdot,t)\|_{L ^{2}}^{2}+  \|\partial _{x}^{3}u ^{I,0}(\cdot,t)\|_{L ^{2}}^{2}  \leq c(T).
\end{align}
By the analogous arguments, one can also get
\begin{align}\label{v-I-0-H4}
\displaystyle  \|\partial _{x}^{4} v ^{I,0}(\cdot,t)\|_{L ^{2}}^{2}+ \int _{0}^{t} \|\partial _{x}^{4}u ^{I,0}\|_{L ^{2}}^{2}  \mathrm{d}\tau \leq c(T)
\end{align}
for any $ t \in [0,T] $. Now combining $ \eqref{original-zero-diffu}_{2} $, \eqref{v-i-0-esti-con}, \eqref{u-I-0-x-final-con}, \eqref{u-I-0-xx}, \eqref{v-tt-x0}, \eqref{u-t-I0-multiH2}, \eqref{v-I-0-H3} and \eqref{v-I-0-H4} yields
\begin{align*}
\displaystyle  \|v _{t}^{I,0}\|_{L _{T}^{2}H ^{4}}+\|v _{tt}^{I,0}\|_{L _{T}^{2}H ^{2}} \leq c(T).
\end{align*}
The rest of the estimates in \eqref{con-u-I-0-regula} can be proved in a similar manner by applying $ \partial _{t} $ and $ \partial _{t}^{2} $ to the equations in \eqref{eQ-U-I-0-t}, and the details are omitted here for brevity.
\end{proof}

With the solution obtained in Lemma \ref{lem-orig-zero-dif-problem} for the problem \eqref{original-zero-diffu}, recalling the transformation \eqref{anti-derivatives-transf}, one can easily show the existence of unique classical solutions to \eqref{eq-outer-0}. Precisely, we have
\begin{lemma}\label{lem-regul-outer-layer-0}
Assume that $ (\varphi _{0}, v _{0}) \in H ^{7}\times H ^{7} $ and $ (\sqrt{v _{0}}) _{x} \in L ^{2}	 $ satisfying \eqref{compatibility-vfi} and $\varphi _{0x}+M\geq  0 $. Then for any $ T>0 $, there exists a unique solution $ (\varphi ^{I,0}, v ^{I,0}) $ to the problem \eqref{eq-outer-0} on $ [0,T] $ satisfying
\begin{subequations}\label{con-vfi-v-I-0-regula}
\begin{gather}
\displaystyle \displaystyle \varphi _{x}^{I,0}+M  \geq 0, \ \ \|\partial _{t}^{k}\varphi ^{I,0}\|_{L _{T}^{2}H ^{8-2k}} \leq c(T)  \ \ \mbox{for } k=0,1,2,3,4, \label{con-vfi-I-0-only}\\
\displaystyle \|v ^{I,0}\| _{L _{T}^{\infty}H ^{7}}+  \|\partial _{t}^{k}v ^{I,0}\| _{L _{T}^{2}H ^{9-2k}} \leq c(T)   \ \ \mbox{for }k=1,2,3,4. \label{con-v-I-0}
\end{gather}
\end{subequations}
\end{lemma}

The next lemma gives the regularity of boundary layer profiles $ v ^{B,0} $ and $ \varphi ^{B,1} $.
\begin{lemma}\label{lem-v-B-0}
Let $ (\varphi ^{I,0}, v ^{I,0}) $ be the solution of \eqref{eq-outer-0} obtained in Lemma \ref{lem-regul-outer-layer-0}. Then for any $ T>0 $, the problem \eqref{first-bd-layer-pro}--\eqref{vfi-bd-1ord-lt} admits a unique solution $ v ^{B,0} $ on $[0,T]$ such that for any $ l \in \mathbb{N} $,
\begin{gather}\label{v-B-0-regularity}
\displaystyle 0 \leq v ^{B,0} \leq v _{\ast},\ \ \ \langle z \rangle ^{l}\partial _{t}^{k} v ^{B,0} \in L _{T}^{2} H _{z}^{6-2k},\ \langle z \rangle ^{l}\partial _{t}^{k} \varphi ^{B,1} \in L _{T}^{2} H _{z}^{7-2k} \ \mbox{for\ } k=0,1,2,3.
\end{gather}
Furthermore, it holds that
\begin{align}\label{con-vfi-B-1}
\displaystyle   \| \langle z \rangle ^{l}\partial _{t}^{k} v ^{B,0}\|_{L _{T}^{2} H _{z}^{6-2k}}\leq K _{0}(T,v _{\ast}) v _{\ast},\ \  \|\langle z \rangle ^{l}\partial _{t}^{k} \varphi ^{B,1}\|_{L _{T}^{2} H _{z}^{7-2k}}\leq c(v _{\ast}, T) v _{\ast},\ \ k=0,1,2,3,\\
\displaystyle \sum _{k=0}^{2}\|\langle z \rangle ^{l} \partial _{t}^{k} v ^{B,0}\|_{L _{T}^{\infty}H _{z}^{4-2k}}+\sum _{\lambda=0}^{1}\sum _{\ell=0}^{3- 2\lambda}\|\langle z \rangle ^{l}\partial _{t}^{\lambda}\partial _{z}^{\ell} v ^{B,0}\|_{L _{T}^{\infty}L _{z}^{\infty}} \leq K _{0}(T, v _{\ast}) v _{\ast},
\end{align}
where the constant $ K_{0}(T, v _{\ast}):=C(T)		{\mathop{\mathrm{e}}}^{c(v _{\ast}, T)}>0 $ with $ c(v _{\ast}, T) $ and $ C(T) $ being as stated in Section  \ref{sec:main_result}. Clearly, $ K _{0}(T, v _{\ast}) $ is increasing in $ T $ and $ v _{\ast} $ with $ \lim \limits_{T \rightarrow 0}K  _{0}(T,v _{\ast})=0 $ and $ \lim\limits _{T \rightarrow +\infty}K _{0}(T,v _{\ast})=+\infty$.
\end{lemma}
\begin{proof}
  The local existence and uniqueness of solutions to the problem \eqref{first-bd-layer-pro} with regularity given in \eqref{v-B-0-regularity} can be proved by routine procedures: first, we study the linearized problem by the reflection method; second, we derive suitable estimates for solutions of the linearized problem and then prove the existence of solutions for the original nonlinear problem by the fixed point theorem. For completeness, we detail the proof in Appendix A. Below we are devoted to deriving the \emph{a priori} estimates of solutions, which are used not only for the global existence of solutions but also for the convergence of boundary  layers. We first prove that the solution of \eqref{first-bd-layer-pro} is bounded and satisfies
\begin{gather}\label{v-B-0-bds}
\displaystyle  0 \leq v ^{B,0} \leq v _{\ast}.
\end{gather}
To this end, we test the equation in \eqref{first-bd-layer-pro} against $ v ^{-}:=-\max\{0,-v ^{B,0}\} $ to derive that
\begin{align*}
&\displaystyle \frac{1}{2}\frac{\mathrm{d}}{\mathrm{d}t}\int _{\mathbb{R}_{+}}\vert v ^{-}\vert ^{2} \mathrm{d}z+\int _{\mathbb{R}_{+}}\vert \partial _{z}v ^{-}\vert ^{2}\mathrm{d}z +\int _{\mathbb{R}_{+}}(\varphi _{x}^{I,0}(0,t)+M)    {\mathop{\mathrm{e}}}^{v ^{B,0}} \vert v ^{-}\vert ^{2} \mathrm{d}z
 \nonumber \\
 &~\displaystyle \quad +\int _{\{v ^{B,0}<0\}}(\varphi _{x}^{I,0}(0,t)+M)v ^{I,0}(0,t)(    {\mathop{\mathrm{e}}}^{v ^{B,0}}-1)v ^{B,0} \mathrm{d}z=0,
\end{align*}
where, to ensure the validity of integration by parts, we have used the fact $ v _{\ast} \geq v ^{I,0}(0,t)\geq 0$ due to $ \varphi _{x}^{I,0}+M \geq 0 $ and
\begin{gather}\label{v-I-0-positive}
\displaystyle v ^{I,0}(0,t)=v _{\ast} \mathop{\mathrm{exp}}\nolimits \left( - \int _{0}^{t}(\varphi _{x}^{I,0}(0,t)+M)\mathrm{d}\tau \right) .
\end{gather}
This entails that
\begin{gather*}
\displaystyle \int _{\mathbb{R}_{+}} \vert v ^{-}\vert ^{2} \mathrm{d}z \leq 0,
\end{gather*}
which implies $ v ^{-}=0 $ and $ v ^{B,0} \geq 0 $. Similarly, testing the equation \eqref{first-bd-layer-pro} against $ v ^{+}:=\max\{v ^{B,0}-v _{\ast},0\} $, we can show that $ v ^{B,0} \leq v _{\ast}  $. Therefore \eqref{v-B-0-bds} is proved.

Next we shall derive some weighted estimates for $ v ^{B,0} $. Let $ \eta(z) \in C ^{\infty}([0,\infty)) $ such that
\begin{gather}\label{eta-defi}
\displaystyle \eta (0)=1,\ \ \eta(z)=0\ \ \mbox{for}\ z \geq 1,
\end{gather}
and denote by $ \overline{u ^{I,0}}:=\varphi _{x}^{I,0}(0,t)+M $. Then if we take $ \vartheta= v ^{B,0}- \eta(z)(v _{\ast} - v ^{I,0}(0,t))=:v ^{B,0}-\phi(z,t)  $, it follows that $ \vartheta $ solves
\begin{align}\label{eq-for-esti-v-B-0}
 \displaystyle \begin{cases}
  \displaystyle \vartheta_{t}=\vartheta_{zz} -\overline{u ^{I,0}}    {\mathop{\mathrm{e}}}^{\vartheta+ \phi}(\vartheta+ \phi)- \overline{u ^{I,0}}\,v ^{I,0}(0,t)(    {\mathop{\mathrm{e}}}^{\vartheta+ \phi}-   1) + \varrho,\\
\displaystyle \vartheta(0,t)=0,\ \ \vartheta(+\infty,t)=0,\\
\displaystyle \vartheta(z,0)=0,
 \end{cases}
  \end{align}
where
\begin{gather*}
\displaystyle \varrho =  \eta _{zz}(z)(v _{\ast}- v ^{I,0}(0,t)) - \eta(z)(v _{\ast}- v ^{I,0}(0,t)) _{t}.
\end{gather*}
By \eqref{con-vfi-v-I-0-regula} and \eqref{Sobolev-infty}, we get that
  \begin{align}\label{vfi-bd-I-0}
 \begin{cases}
 	\displaystyle  \|\partial _{t}^{k}\varphi _{x}^{I,0}(0,t)\|_{L ^{2}(0,T)} \leq \|\partial _{t}^{k}\varphi _{x}^{I,0}\|_{L _{T}^{2}H ^{1}} \leq c(T) \ \ \mbox{for}\ \ 0 \leq k \leq 3,\\[1mm]
  \displaystyle \|\partial _{t}^{k}v ^{I,0}(0,t)\|_{L ^{2}(0,T)} \leq \|\partial _{t}^{k}v ^{I,0}\|_{L _{T}^{2}H ^{1}} \leq c(T)\ \ \mbox{for}\ \ 0 \leq k \leq 4,
 \end{cases}
  \end{align}
  which gives rise to
\begin{gather}\label{vfi-I-0-x-bd-infty}
\displaystyle   \|\partial _{t}^{k} \varphi _{x}^{I,0}(0,t)\|_{L ^{\infty}(0,T)} \leq c(T)\ \mbox{for} \ 0 \leq k \leq 2 \ \ \mbox{and}\ \ \|\partial _{t}^{k} v^{I,0}(0,t)\|_{L ^{\infty}(0,T)} \leq c(T) \ \mbox{for}  \ 0 \leq k \leq 3.
\end{gather}
Thanks to \eqref{compatibility-vfi}, \eqref{v-I-0-positive} and \eqref{vfi-I-0-x-bd-infty}, it holds for $ l \in \mathbb{N} $ that $\langle z \rangle ^{l} \partial _{t}^{k} \varrho \in L _{T}^{2}H _{z}^{4-2k}\,(k=0,1,2)  $ with
\begin{align}\label{tild-rho-esti}
\displaystyle \|\langle z \rangle ^{l} \partial _{t}^{k} \varrho\|_{L _{T}^{2}H _{z}^{4-2k}} \leq C(T) v _{\ast},\ \  k=0,1,2,
\end{align}
where the constant $  C(T) >0$ is as stated in Section  \ref{sec:main_result}. Similarly, we get for $ l \in \mathbb{N} $ that $\langle z \rangle ^{l} \partial _{t}^{k} \phi \in L _{T}^{2}H _{z}^{4-2k}\,(k=0,1,2)  $ with
\begin{gather}\label{fip-esti}
\displaystyle  \|\langle z \rangle ^{l} \partial _{t}^{k+\lambda} \phi\|_{L _{T}^{2}H _{z}^{4-2k}} \leq C(T) v _{\ast},\ \  k=0,1,2,\ \ \lambda=0,1.
\end{gather}

Multiplying the equation $ \eqref{eq-for-esti-v-B-0}_{1} $ by $ \langle z \rangle ^{2l}\vartheta$ followed by an integration over $ \mathbb{R}_{+} $, we have
\begin{align}\label{v-B-0-esti-0}
   \displaystyle & \frac{1}{2}\frac{\mathrm{d}}{\mathrm{d}t}\int _{\mathbb{R}_{+}}\langle z \rangle ^{2l}\vartheta^{2} \mathrm{d}z+\int _{\mathbb{R}_{+}}\langle z \rangle ^{2l}\vartheta_{z}^{2}\mathrm{d}z +\int _{\mathbb{R}_{+}}\langle z \rangle ^{2l}\overline{u ^{I,0}} {\mathop{\mathrm{e}}}^{\vartheta+ \phi}\vartheta^{2}\mathrm{d}z
    \nonumber \\
    & \displaystyle  = \int _{\mathbb{R}_{+}}\langle z \rangle ^{2l}\vartheta  \varrho \mathrm{d}z- 2l \int _{\mathbb{R}_{+}}\langle z \rangle ^{2l-2}z \vartheta _{z} \vartheta\mathrm{d}z-\int _{\mathbb{R}_{+}}\langle z \rangle ^{2l}\overline{u ^{I,0}} {\mathop{\mathrm{e}}}^{\vartheta+ \phi}\vartheta  \phi\mathrm{d}z
     \nonumber \\
     &~\displaystyle \quad -\int _{\mathbb{R}_{+}}\langle z \rangle ^{2l} \overline{u ^{I,0}}\,{v ^{I,0}}(0,t)\left(     {\mathop{\mathrm{e}}}^{\vartheta+ \phi}-   1 \right) \vartheta\mathrm{d}z=:\mathcal{A},
   \end{align}
   where, due to $ \overline{u ^{I,0}}=\varphi _{x}^{I,0}(0,t)+M  \geq 0 $ and $  0 \leq v ^{B,0} \leq v _{\ast} $, it holds that
   \begin{align}\label{v-b-0-good}
   \displaystyle  \int _{\mathbb{R}_{+}}\langle z \rangle ^{2l}\overline{u ^{I,0}} {\mathop{\mathrm{e}}}^{\vartheta+ \phi}\vartheta^{2}\mathrm{d}z \geq 0.
   \end{align}
   We now turn to estimate the terms on the right hand side of \eqref{v-B-0-esti-0}. By \eqref{v-I-0-positive}, \eqref{vfi-I-0-x-bd-infty}, $  0 \leq v ^{B,0} \leq v _{\ast} $ and the Cauchy-Schwarz inequality, we get
   \begin{align}\label{cal-A-esti}
   \mathcal{A}
    &\displaystyle  \leq \|\langle z \rangle ^{l}\vartheta\|_{L  _{z}^{2}}\|\langle z \rangle ^{l}\varrho\|_{L _{z} ^{2}}+c _{0}  \|\langle z \rangle ^{l}\vartheta\|_{L  _{z}^{2}}\|\langle z \rangle ^{l}\vartheta _{z}\|_{L _{z} ^{2}}+ c _{0}\|\langle z \rangle ^{l}\vartheta\|_{L  _{z}^{2}}\|\langle z \rangle ^{l}\phi\|_{L _{z} ^{2}}
     \nonumber \\
     & \displaystyle \quad+c(v _{\ast}, T)  \int _{\mathbb{R}_{+}} \langle z \rangle ^{2l} \overline{u ^{I,0}}v ^{I,0}(0,t) \left(\left\vert \vartheta\right\vert+ \left\vert \phi\right\vert  \right)\vartheta \mathrm{d}z
     \nonumber \\
          & \displaystyle \leq  \frac{1}{4}\int _{\mathbb{R}_{+}}\langle z \rangle ^{2l}\vartheta _{z}^{2} \mathrm{d}z+c(v _{\ast}, T)  \int _{\mathbb{R}_{+}}\langle z \rangle ^{2l} \vartheta ^{2} \mathrm{d}z + c _{0} \|\langle z \rangle ^{l} \varrho\|_{L _{z}^{2}}^{2}+c _{0} \|\langle z \rangle ^{l}\phi\|_{L _{z}^{2}}^{2} ,
   \end{align}
   where the constant $ c(v _{\ast}, T)>0 $ is as stated in Section  \ref{sec:main_result}. Inserting \eqref{v-b-0-good}-\eqref{cal-A-esti} into \eqref{v-B-0-esti-0}, and integrating the result for any $ t \in (0,T] $, we get
  \begin{align}\label{con-vfi-B-0-0-pre}
   \displaystyle  \int _{\mathbb{R}_{+}}\langle z \rangle ^{2l}\vartheta ^{2} (\cdot,t)\mathrm{d}z+\int _{0}^{t} \int _{\mathbb{R}_{+}}\langle z \rangle ^{2l}\vartheta _{z} ^{2}   \mathrm{d}z \mathrm{d}\tau \leq C(T) v _{\ast}^{2}+ c(v _{\ast}, T) \int _{0}^{t}\int _{\mathbb{R}_{+}}\langle z \rangle ^{2l} \vartheta ^{2} \mathrm{d}z \mathrm{d}\tau,
   \end{align}
   where $ C(T) $ and $ c(v _{\ast}, T) $ are constants as stated in Section  \ref{sec:main_result}. Applying the Gronwall inequality to \eqref{con-vfi-B-0-0-pre}, we get
   \begin{align}\label{con-vfi-B-0-0}
   \displaystyle  \int _{\mathbb{R}_{+}}\langle z \rangle ^{2l}\vartheta ^{2} (\cdot,t)\mathrm{d}z+\int _{0}^{t} \int _{\mathbb{R}_{+}}\langle z \rangle ^{2l} \vartheta _{z} ^{2}   \mathrm{d}z \mathrm{d}\tau \leq C(T){\mathop{\mathrm{e}}}^{c(v _{\ast}, T)}v _{\ast}^{2}.
   \end{align}
   Multiplying  $ \eqref{eq-for-esti-v-B-0}_{1} $ by $ \langle z \rangle ^{2l}\vartheta _{t} $ and integrating the resulting equation over $ \mathbb{R}_{+} $, we have
\begin{align}\label{vte-t-esti}
&\displaystyle \frac{1}{2}\frac{\mathrm{d}}{\mathrm{d}t}\int _{\mathbb{R}_{+}}\langle z \rangle ^{2l}\left( \vartheta_{z} ^{2}+\overline{u ^{I,0}} \vartheta ^{2}     {\mathop{\mathrm{e}}}^{\vartheta+ \phi}\right)  \mathrm{d}z + \int _{\mathbb{R}_{+}}\langle z \rangle ^{2l}\vartheta_{t} ^{2}\mathrm{d}z
 \nonumber \\
 &~\displaystyle =\frac{1}{2}\int _{\mathbb{R}_{+}} \langle z \rangle ^{2l} \partial _{t}\overline{u ^{I,0}} \vartheta^{2}    {\mathop{\mathrm{e}}}^{\vartheta+ \phi}\mathrm{d}z+ \frac{1}{2}\int _{\mathbb{R}_{+}}\langle z \rangle ^{2l} \overline{u ^{I,0}}\vartheta^{2}\left( \vartheta_{t}+ \phi _{t} \right)     {\mathop{\mathrm{e}}}^{\vartheta+ \phi} \mathrm{d}z  -2 l \int _{\mathbb{R}_{+}}\langle z \rangle ^{2l-2}z \vartheta _{t}\vartheta _{z} \mathrm{d}z
  \nonumber \\
  &~\displaystyle \quad- \int _{\mathbb{R}_{+}}\langle z \rangle ^{2l}\overline{u ^{I,0}}    {\mathop{\mathrm{e}}}^{\vartheta+ \phi}  \phi \vartheta_{t} \mathrm{d}z-\overline{u ^{I,0}}v ^{I,0}(0,t)\int _{\mathbb{R}_{+}}\langle z \rangle ^{2l}(    {\mathop{\mathrm{e}}}^{\vartheta+ \phi}-    1)\vartheta_{t} \mathrm{d}z- \int _{\mathcal{I} }\langle z \rangle ^{2l} \varrho\vartheta_{t} \mathrm{d}z
   \nonumber \\
   &~\displaystyle \leq  c(v _{\ast}, T)  \vert  \overline{\partial _{t}u ^{I,0}}\vert  \int _{\mathbb{R}_{+}}\langle z \rangle ^{2l}\vartheta^{2} \mathrm{d}z+c(v _{\ast}, T)  \overline{u ^{I,0}} \int _{\mathbb{R}_{+}}\langle z \rangle ^{2l} \vartheta ^{2}(\left\vert \vartheta _{t}\right\vert+\left\vert \phi _{t}\right\vert)  \mathrm{d}z + c _{0} \int _{\mathbb{R}_{+}}\langle z \rangle ^{2l-1}\left\vert \vartheta _{t}\right\vert \left\vert \vartheta _{z}\right\vert \mathrm{d}z
    \nonumber \\
    &\displaystyle~ \quad +c(v _{\ast}, T)  \int _{\mathbb{R}_{+}}\langle z \rangle ^{2l}\left\vert \phi\right\vert \left\vert \vartheta _{t}\right\vert\mathrm{d}z+c(v _{\ast}, T)   \overline{u ^{I,0}} \vert v ^{I,0} (0,t)\vert\int _{\mathbb{R}_{+}}\langle z \rangle ^{2l}\left( \left\vert \vartheta\right\vert+\left\vert \phi\right\vert \right) \left\vert \vartheta_{t}\right\vert  \mathrm{d}z - \int _{\mathbb{R}_{+} }\langle z \rangle ^{2l} \varrho\vartheta_{t} \mathrm{d}z
      \nonumber \\
        &~\displaystyle \leq \frac{1}{8}\int _{\mathbb{R}_{+}}\langle z \rangle ^{2l}\vartheta_{t} ^{2}\mathrm{d}z+c(v _{\ast}, T)  \int _{\mathbb{R}_{+}}\langle z \rangle ^{2l}\vartheta _{z}^{2}  \mathrm{d}z+ c(v _{\ast}, T)  \int _{\mathbb{R}_{+}}\langle z \rangle ^{2l} (\phi ^{2}+\phi _{t}^{2})\mathrm{d}z + c_{0}\int _{\mathbb{R}_{+}}\langle z \rangle ^{2l} \varrho^{2}\mathrm{d}z
         \nonumber \\
         & ~\displaystyle \quad + c(v _{\ast}, T)  (\vert  \overline{\partial _{t}u ^{I,0}}\vert+ \vert \overline{u ^{I,0}}\vert ^{2}+ \vert \overline{u ^{I,0}}\vert^{2}\vert v ^{I,0} (0,t)\vert^{2})\int _{\mathbb{R}_{+}}\langle z \rangle ^{2l}\vartheta^{2} \mathrm{d}z,
\end{align}
where we have used \eqref{con-vfi-v-I-0-regula}, \eqref{con-vfi-B-0-0}, $  0 \leq v ^{B,0} \leq v _{\ast} $ and the Cauchy-Schwarz inequality. By \eqref{vfi-I-0-x-bd-infty}, we further update \eqref{vte-t-esti} as
\begin{align*}
\displaystyle  &\displaystyle \frac{1}{2}\frac{\mathrm{d}}{\mathrm{d}t}\int _{\mathbb{R}_{+}}\langle z \rangle ^{2l}\left( \vartheta_{z} ^{2}+\overline{u ^{I,0}} \vartheta ^{2}     {\mathop{\mathrm{e}}}^{\vartheta+ \phi}\right)  \mathrm{d}z + \int _{\mathbb{R}_{+}}\langle z \rangle ^{2l}\vartheta_{t} ^{2}\mathrm{d}z
 \nonumber \\
 &~\displaystyle \leq c(v _{\ast}, T) \int _{\mathbb{R}_{+}}\langle z \rangle ^{2l}\left( \vartheta^{2} +\vartheta _{z}^{2} \right)  \mathrm{d}z+ c(v _{\ast}, T) \int _{\mathbb{R}_{+}}\langle z \rangle ^{2l} (\phi ^{2}+\phi _{t}^{2})\mathrm{d}z +c _{0}\int _{\mathbb{R}_{+}}\langle z \rangle ^{2l} \varrho^{2}\mathrm{d}z .
\end{align*}
This along with \eqref{con-vfi-v-I-0-regula}, \eqref{tild-rho-esti}, \eqref{fip-esti}, \eqref{con-vfi-B-0-0} and the Gronwall inequality yields for any $ t \in (0,T] $ that
\begin{align}\label{con-vfi-B-0-1}
\displaystyle  \int _{\mathbb{R}_{+}}\langle z \rangle ^{2l}\vartheta _{z}^{2} (\cdot,t) \mathrm{d}z+ \int _{0}^{t}\int _{\mathbb{R}_{+}}\langle  z\rangle ^{2l}\vartheta _{\tau}^{2} \mathrm{d}z \mathrm{d}\tau \leq C (T)		{\mathop{\mathrm{e}}}^{c(v _{\ast}, T) } v _{\ast}^{2}.
\end{align}
With \eqref{vfi-I-0-x-bd-infty}, \eqref{tild-rho-esti} and \eqref{con-vfi-B-0-1}, we get from  $ \eqref{eq-for-esti-v-B-0}_{1} $ that
\begin{gather}\label{vte-xx}
\displaystyle \int _{0}^{T}\int _{\mathbb{R}_{+}}\langle z \rangle ^{2l}\vartheta _{zz}^{2} \mathrm{d}z \mathrm{d}t \leq C(T){\mathop{\mathrm{e}}}^{c(v _{\ast}, T) } v _{\ast}^{2}.
\end{gather}
Denote $ \tilde{\vartheta}=\vartheta_{t} $. Then by \eqref{eq-for-esti-v-B-0} and the compatibility condition \eqref{compatibility-vfi}, we find that $ \tilde{\vartheta} $ satisfies
\begin{align}\label{tild-vte-eq}
  \displaystyle \begin{cases}
    \displaystyle  \tilde{\vartheta}_{t}= \tilde{\vartheta}_{zz} -\overline{u ^{I,0}}    {\mathop{\mathrm{e}}}^{\vartheta+ \phi}\tilde{\vartheta}-\overline{u ^{I,0}}    {\mathop{\mathrm{e}}}^{\vartheta+ \phi}(\vartheta + \phi) \tilde{\vartheta}- \overline{u ^{I,0}}\,v ^{I,0}(0,t){\mathop{\mathrm{e}}}^{\vartheta+ \phi} \tilde{\vartheta}+\tilde{\varrho},\\
\displaystyle \tilde{\vartheta}(0,t)=0,\ \ \tilde{\vartheta}(+\infty,t)=0,\\
\displaystyle \tilde{\vartheta}(z,0)=0,
  \end{cases}
  \end{align}
  where $ \tilde{\varrho} $ is given by
  \begin{align*}
  \displaystyle  \tilde{\varrho}&=- \partial _{t}\overline{u ^{I,0}}    {\mathop{\mathrm{e}}}^{\vartheta+ \phi}(\vartheta + \phi)-\overline{u ^{I,0}}    {\mathop{\mathrm{e}}}^{\vartheta+ \phi} \phi _{t}(1+\vartheta+ \phi)-\overline{u ^{I,0}}\,v ^{I,0}(0,t){\mathop{\mathrm{e}}}^{\vartheta+ \phi}\phi _{t}
   \nonumber \\
   & \displaystyle \quad -\partial _{t} \left( \overline{u ^{I,0}}\,v ^{I,0}(0,t) \right) (    {\mathop{\mathrm{e}}}^{\vartheta+ \phi}-   1)+  \partial _{t}\varrho.
  \end{align*}
From \eqref{con-vfi-v-I-0-regula}, \eqref{tild-rho-esti}, \eqref{con-vfi-B-0-0},
  \eqref{vfi-I-0-x-bd-infty}--\eqref{vte-xx}, it follows for $ l \in \mathbb{N} $ that $ \langle z \rangle ^{l} \partial _{t}^{k}  \tilde{\varrho} \in L _{T}^{2}H _{z}^{2-2k}\,(k=0,1) $ with
  \begin{align}\label{tild-varrho-esti-1}
  \displaystyle \|\langle z \rangle ^{l} \partial _{t}^{k}  \tilde{\varrho}\|_{L _{T}^{2}H _{z}^{2-2k}} \leq C(T) {\mathop{\mathrm{e}}}^{c(v _{\ast}, T) }v _{\ast},\ \ \ k=0,1.
  \end{align}
  With \eqref{tild-varrho-esti-1}, by repeating the procedures in the proof of \eqref{con-vfi-B-0-0}, \eqref{con-vfi-B-0-1} and \eqref{vte-xx}, we have
  \begin{align*}
 \displaystyle  \int _{\mathbb{R}_{+}}\langle z \rangle ^{2l}\left(\tilde{ \vartheta}^{2} +\tilde{\vartheta} _{z}^{2} \right) (\cdot,t) \mathrm{d}z+ \int _{0}^{t}\int _{\mathbb{R}_{+}}\langle  z\rangle ^{2l}\left( \tilde{\vartheta} _{\tau}^{2} +\tilde{\vartheta}_{z}^{2} +\tilde{\vartheta}_{zz}^{2}\right) \mathrm{d}z \mathrm{d}\tau \leq C(T) {\mathop{\mathrm{e}}}^{c(v _{\ast}, T) }v _{\ast}^{2}
  \end{align*}
  for any $ t \in (0,T] $. This along with \eqref{con-vfi-v-I-0-regula}, $ \eqref{eq-for-esti-v-B-0}_{1} $, $ \eqref{tild-vte-eq} _{1} $ and the fact $ \tilde{v}=v _{t} $ implies that
  \begin{align*}
  \displaystyle  \displaystyle \int _{0}^{T}\int _{\mathbb{R}_{+}} \langle z \rangle ^{2l} \left( \vartheta_{zzz}^{2}+ \vartheta_{zzzz}^{2}+ \vartheta_{t zz}^{2} +\vartheta_{tt}^{2}\right) \mathrm{d}z \mathrm{d}t \leq C(T){\mathop{\mathrm{e}}}^{c(v _{\ast}, T) } v _{\ast}^{2},
  \end{align*}
  where we have used $ \|\langle z \rangle ^{l} \partial _{t}^{k} \varrho\|_{L _{T}^{2}H _{z}^{2-2k}}\leq C(T) v _{\ast}~(k=0,1) $ from \eqref{tild-rho-esti} and the estimate $ \|\langle z \rangle ^{l}\tilde{\varrho}\|_{L _{T}^{2}L _{z}^{2}} \leq C(T){\mathop{\mathrm{e}}}^{c(v _{\ast}, T) } v _{\ast}  $ from \eqref{tild-varrho-esti-1}. Thus we conclude for the problem \eqref{eq-for-esti-v-B-0} that
  \begin{gather}\label{vete-fisrt-conclu}
  \displaystyle  \|\langle z \rangle ^{l}\partial _{t}^{k} \vartheta\|_{L _{T}^{2}H _{z}^{4-2k}} \leq C(T){\mathop{\mathrm{e}}}^{c(v _{\ast}, T) } v _{\ast} ,\ \ k=0,1,2,
  \end{gather}
  provided $ \| \langle z \rangle ^{l} \partial _{t}^{k} \varrho\| _{L _{T}^{2}H _{z}^{2-2k}} \leq C(T) v _{\ast}$ with $ k=0,1 $. Notice that the initial value for the problem \eqref{tild-vte-eq} is compatible up to order one, and that $ \|\langle z \rangle ^{l}\partial _{t}^{k}\tilde{\varrho}\| _{L _{T}^{2}H _{z}^{2-2k}} \leq C(T){\mathop{\mathrm{e}}}^{c(v _{\ast}, T) } v _{\ast} $ with $ k=0,1$. Therefore, by the same arguments as proving \eqref{vete-fisrt-conclu}, we have for the problem \eqref{tild-vte-eq} that
\begin{align}\label{vete-sec-conclud}
  \displaystyle  \|\langle z \rangle ^{l}\partial _{t} ^{k} \tilde{\vartheta}\| _{L _{T}^{2}H _{z}^{4-2k}} \leq C(T){\mathop{\mathrm{e}}}^{c(v _{\ast}, T) } v _{\ast} ,\ \ k=0,1,2.
  \end{align}
  This along with \eqref{tild-rho-esti} and \eqref{vete-fisrt-conclu} further gives
  \begin{gather}\label{vete-highest}
  \displaystyle \int _{0}^{T} \|\partial _{z}^{5}\vartheta(\cdot,t)\|_{H ^{1}}^{2}\mathrm{d}t \leq C(T){\mathop{\mathrm{e}}}^{c(v _{\ast}, T) } v _{\ast}^{2}.
  \end{gather}
  Collecting \eqref{vete-fisrt-conclu}--\eqref{vete-highest}, we have
  \begin{align}\label{v-B-0-esti}
  \displaystyle  \langle z \rangle ^{l}\partial _{t}^{k} v ^{B,0} \in L _{T}^{2} H _{z}^{6-2k},\ \ k=0,1,2,3.
  \end{align}
 By \eqref{v-B-0-esti} and Proposition \ref{prop-embeding-spacetime}, we get for $ k=0,1,2,\ \ell=0,1,\cdots, 4-2k $ that $ \langle z \rangle ^{l} \partial _{t}^{k} v ^{B,0} \in C([0,T];H _{z}^{5-2k}) $ and $  \langle z \rangle ^{l}\partial _{t}^{k}\partial _{z}^{\ell} v ^{B,0} \in L _{T} ^{\infty}L _{z}^{\infty}$ with
  \begin{align}\label{V-B-0-l-infty-bd}
  \displaystyle \|\langle z \rangle ^{l} \partial _{t}^{k} v ^{B,0}\|_{L _{T}^{\infty}H _{z}^{4-2k}}+\sum _{\lambda=0}^{3}\|\langle z \rangle ^{l}\partial _{z}^{\lambda} v ^{B,0}\|_{L _{T}^{\infty}L _{z}^{\infty}} +\sum _{\lambda=0}^{1}\|\langle z \rangle ^{l}\partial _{z}^{\lambda} \partial _{t} v ^{B,0}\|_{L _{T}^{\infty}L _{z}^{\infty}}\leq C(T){\mathop{\mathrm{e}}}^{c(v _{\ast}, T) } v _{\ast}.
  \end{align}
  Now let us derive estimates for $ \varphi ^{B,1} $. Since
  \begin{gather*}
  \displaystyle  \partial _{z}^{\ell}      {\mathop{\mathrm{e}}}^{ v ^{B,0}}= \sum _{\substack{\ell _{1}+\cdots +\ell _{r}=\ell \\
  1 \leq \ell _{1} \leq \cdots \leq\ell _{r} ,\,
  1 \leq r \leq l}} C _{r}  {\mathop{\mathrm{e}}}^{v ^{B,0}} \partial _{z}^{\ell _{1}}v ^{B,0}\cdots \partial _{z}^{\ell _{r}}v ^{B,0},\ \  \ell \geq 1
  \end{gather*}
  for some constant $ C _{r} $ independent of $ v _{\ast} $ and $ T $, we get, thanks to \eqref{v-B-0-regularity} and \eqref{V-B-0-l-infty-bd},
  \begin{align}\label{V-B-0-exp}
  &\displaystyle \left\|   \langle z \rangle ^{l} ({\mathop{\mathrm{e}}}^{ v ^{B,0}}-1)\right\|_{L _{T}^{2}H _{z}^{6}} ^{2}
   \nonumber \\
   &\displaystyle~ \leq  \sum _{\substack{  \ell _{1}+\cdots+ \ell _{r} \leq 6 \\
  1 \leq \ell _{1} \leq \cdots \leq\ell _{r} ,\,
  1 \leq r \leq 6}}C _{r}\|\langle z \rangle ^{l} \partial _{z}^{\ell _{1}}v ^{B,0}\cdots \partial _{z}^{\ell _{r}}v ^{B,0}\|_{L _{T}^{2}L _{z}^{2}}^{2}+c _{0}\|\langle z \rangle ^{l} ({\mathop{\mathrm{e}}}^{ v ^{B,0}}-1)\|_{L _{T}^{2}L _{z}^{2}}^{2}
   \nonumber \\
   & \displaystyle~ \leq  \sum _{\substack{  \ell _{1}+\cdots+ \ell _{r} \leq 6 \\
  1 \leq \ell _{1} \leq \cdots \leq\ell _{r} ,\,
  1 \leq r \leq 6}} \makebox[-4pt]{~}\int _{0}^{T}
  \|\partial _{z} ^{\ell _{1}}v ^{B,0}\|_{L _{z}^{\infty}}^{2} \cdots \|\partial _{z} ^{\ell _{r-1}}v ^{B,0}\|_{L _{z}^{\infty}}^{2}\|\langle z \rangle ^{l}\partial _{z} ^{\ell _{r}}v ^{B,0}\|_{L _{z}^{2}}^{2}\mathrm{d}t +c _{0}\|\langle z \rangle ^{l} v ^{B,0}\|_{L _{T}^{2}L _{z}^{2}}^{2}
   \nonumber \\
   &~\displaystyle \leq c(v _{\ast}, T) v _{\ast} \sum _{\ell=1}^{6}\| \langle z \rangle ^{l}\partial _{z}^{\ell}v ^{B,0}\|_{L _{T}^{2}L _{z}^{2}  }^{2}+\|\langle z \rangle ^{l} v ^{B,0}\|_{L _{T}^{2}L _{z}^{2}}^{2}  \leq c(v _{\ast}, T) \|\langle z \rangle ^{l} v ^{B,0}\|_{L _{T}^{2}H_{z}^{6}} ^{2} \leq c(v _{\ast}, T) v _{\ast}^{2}.
  \end{align}
  Similarly, we have for any $ l \in \mathbb{N} $ that $  \langle z \rangle ^{l} ({\mathop{\mathrm{e}}}^{ v ^{B,0}}-1) \in L _{T}^{\infty}H ^{5}$ with
  \begin{gather}\label{exp-v-B-0-linfty}
  \displaystyle  \|\langle z \rangle ^{l} ({\mathop{\mathrm{e}}}^{ v ^{B,0}}-1)\|_{L _{T}^{\infty}H _{z}^{4}} \leq c(v _{\ast}, T) v _{\ast}.
  \end{gather}
  Noting that
  \begin{gather*}
  \displaystyle \partial _{t}^{k}      {\mathop{\mathrm{e}}}^{ v ^{B,0}}= \sum _{\substack{\ell _{1}+\cdots \ell _{r}=k \\ 1 \leq \ell _{1} \leq \cdots \leq\ell _{r} ,\,
  1 \leq r \leq k}} C _{r}  {\mathop{\mathrm{e}}}^{v ^{B,0}} \partial _{t}^{\ell _{1}}v ^{B,0}\cdots \partial _{t}^{\ell _{r}}v ^{B,0}\ \ \mbox{for } k \geq 1,
  \end{gather*}
  with $ C _{r} $ being a constant independent of $ v _{\ast} $ and $ \varepsilon $, similar to \eqref{V-B-0-exp}, we get for $ k = 1,2,3 $ that
  \begin{align*}
  \displaystyle  \|\langle z \rangle ^{l}\partial _{t}^{k}      {\mathop{\mathrm{e}}}^{ v ^{B,0}}\|_{L _{T}^{2}H ^{6-2k}} ^{2}&  \leq \sum _{\substack{\ell _{1}+\cdots \ell _{r}=k \\ 1 \leq \ell _{1} \leq \cdots \leq\ell _{r} ,\,
  1 \leq r \leq k}} c(v _{\ast}, T)\|\langle z \rangle ^{l}\partial _{z}(     {\mathop{\mathrm{e}}}^{v ^{B,0}} )\partial _{t}^{\ell _{1}}v ^{B,0}\cdots \partial _{t}^{\ell _{r}}v ^{B,0}\|_{L _{T}^{2}H _{z}^{5-2k}}^{2}
   \nonumber \\
   & \displaystyle  \quad
   +  \sum _{\substack{\ell _{1}+\cdots \ell _{r}=k \\ 1 \leq \ell _{1} \leq \cdots \leq\ell _{r} ,\,
  1 \leq r \leq k}} c(v _{\ast}, T)\|\langle z \rangle ^{l}\partial _{t}^{\ell _{1}}v ^{B,0}\cdots \partial _{t}^{\ell _{r}}v ^{B,0}\|_{L _{T}^{2}H _{z}^{6-2k}}^{2}
   \nonumber \\
   & \displaystyle \leq  \sum _{j=1}^{k}c(v _{\ast}, T)\|\langle z \rangle ^{l}\partial _{t}^{j}v ^{B,0}\|_{L _{T}^{2}H _{z}^{5-2k}}^{2} \|\langle z \rangle ^{l}\partial _{z} (\mathrm{e}^{v ^{B,0}})\|_{L _{T}^{\infty}H _{z}^{5-2k}}^{2}
    \nonumber \\
    & \displaystyle \quad +\sum _{j=1}^{k} c(v _{\ast}, T) \|\langle z \rangle ^{l}\partial _{t}^{j}v ^{B,0}\|_{L _{T}^{2}H _{z}^{6-2k}}^{2}  \leq \sum _{j=1}^{k}c(v _{\ast}, T)\|\langle z \rangle ^{l}\partial _{t}^{j}v ^{B,0}\|_{L _{T}^{2}H _{z}^{6-2j}}^{2}
     \nonumber \\
     & \displaystyle
     \leq c(v _{\ast}, T) v _{\ast}^{2},
  \end{align*}
  where we have used \eqref{v-B-0-regularity}, \eqref{V-B-0-l-infty-bd}, \eqref{V-B-0-exp}, \eqref{exp-v-B-0-linfty} and the fact
\begin{gather}\label{product-law}
\displaystyle  \|\langle x \rangle ^{l} fg\|_{H ^{k}(\mathbb{R}_{+})} \leq c _{0}\|\langle x \rangle ^{l} f\|_{H ^{k}(\mathbb{R}_{+})} \|\langle x \rangle ^{l} g\|_{H ^{k}(\mathbb{R}_{+})}
\end{gather}
for any $ l \in \mathbb{N} $ and any integer $ k \geq 1 $, provided $ \langle x \rangle ^{l}f ,\, \langle x \rangle ^{l}g \in H ^{k}(\mathbb{R}_{+})$. Therefore we now have for $ l \in \mathbb{N} $ that
  \begin{gather}\label{con-exp-v-B-0-final}
  \displaystyle \|   \langle z \rangle ^{l}\partial _{t}^{k}( {\mathop{\mathrm{e}}}^{v ^{B,0}}-1)\| _{L _{T}^{2}H _{z}^{6-2k}} ^{2} \leq c(v _{\ast}, T) \sum _{j=0}^{k}\|\langle z \rangle ^{l}\partial _{t}^{j}v ^{B,0}\|_{L _{T}^{2}H _{z}^{6-2k}}^{2} \leq c(v _{\ast}, T) v _{\ast}^{2},\ \ \ 0 \leq k \leq 3,
  \end{gather}
  where $ c(v _{\ast}, T) >0$ is a constant as stated in Section  \ref{sec:main_result}. With \eqref{vfi-bd-1ord-lt}, \eqref{vfi-I-0-x-bd-infty}, \eqref{con-exp-v-B-0-final} and the H\"older inequality, we derive for $ k=0,1,2 $ that
\begin{align}\label{vfi-B-1-est-1}
&\displaystyle    \| \langle z \rangle ^{l} \partial _{t}^{k}\varphi ^{B,1} \|_{L _{T}^{2}H _{z} ^{7-2k}}^{2}\leq c _{0} \sum _{i=0}^{k}\left\|\langle z \rangle ^{l}\int _{z}^{\infty}\partial _{t}^{k-i}(\varphi _{x}^{I,0}(0,t)+M) \partial _{t}^{i}\left(     {\mathop{\mathrm{e}}}^{v ^{B,0}}-1 \right) \mathrm{d}y \right\|_{L _{T}^{2}H _{z} ^{7-2k}}^{2}
 \nonumber \\
 &~\displaystyle \leq c _{0}\sum _{i=0}^{k} \|\partial _{t}^{k-i}\left( \varphi _{x}^{I,0}(0,t)+M\right) \|_{L ^{\infty}(0,T)}^{2} \left\|\int _{z}^{\infty}\partial _{t}^{i}\left(     {\mathop{\mathrm{e}}}^{v ^{B,0}}-1 \right) \mathrm{d}y \right\|_{L _{T}^{2}H _{z} ^{7-2k}}^{2}
  \nonumber \\
  &~\displaystyle \leq c(v _{\ast}, T) \sum _{i=0}^{k} \left( 1+\int _{\mathbb{R}_{+}}\int _{z}^{\infty}\langle y \rangle ^{-4}\mathrm{d}y \mathrm{d}z  \right) \|\langle z \rangle ^{l+2}\partial _{t}^{i} ({\mathop{\mathrm{e}}}^{v ^{B,0}}-1)\|_{L _{T}^{2}H _{z} ^{6-2k}}^{2} \leq c(v _{\ast}, T) v _{\ast} ^{2}.
\end{align}
For $ k=3 $, we get
\begin{align*}
&\displaystyle  \| \langle z \rangle ^{l} \partial _{t}^{3}\varphi ^{B,1} \|_{L _{T}^{2}H _{z}^{1}}^{2} \leq c _{0}\sum _{i=0}^{3}\left\|\langle z \rangle ^{l}\int _{z}^{\infty}\partial _{t}^{3-i}(\varphi _{x}^{I,0}(0,t)+M) \partial _{t}^{i}\left(     {\mathop{\mathrm{e}}}^{v ^{B,0}}-1 \right)\mathrm{d}y  \right\|_{L _{T} ^{2}H _{z}^{1}}^{2}
 \nonumber \\
 &~\displaystyle \leq c _{0} \sum _{i=1}^{3} \|\partial _{t}^{3-i}\left( \varphi _{x}^{I,0}(0,t)+M\right) \|_{L ^{\infty}(0,T)}^{2} \left\|\int _{z}^{\infty}\partial _{t}^{i}\left(     {\mathop{\mathrm{e}}}^{v ^{B,0}}-1 \right) \mathrm{d}y \right\|_{L _{T} ^{2}H _{z} ^{1}}^{2}
  \nonumber \\
  & ~ \displaystyle \quad+c _{0}\left\|\langle z \rangle ^{l}\int _{z}^{\infty}\partial _{t}^{3}(\varphi _{x}^{I,0}(0,t)+M) \left(     {\mathop{\mathrm{e}}}^{v ^{B,0}}-1 \right) \mathrm{d}y \right\|_{L _{T}^{2}H _{z}^{1}}^{2}=:A _{1}+A _{2},
\end{align*}
where $ A _{1} $ can be estimated by the similar arguments as proving \eqref{vfi-B-1-est-1}:
\begin{align*}
\displaystyle A_{1} \leq  \sum _{i=1}^{3} c(v _{\ast}, T)\left( 1+\int _{\mathbb{R}_{+}}\int _{z}^{\infty}\langle y \rangle ^{-4}\mathrm{d}y \mathrm{d}z  \right) \|\langle z \rangle ^{l+2}\partial _{t}^{i} ({\mathop{\mathrm{e}}}^{v ^{B,0}}-1)\|_{L _{T}^{2}L_{z} ^{2}}^{2} \leq c(v _{\ast}, T) v _{\ast}^{2},
\end{align*}
where \eqref{vfi-I-0-x-bd-infty} and \eqref{con-exp-v-B-0-final} have been used. We proceed to estimate $ A _{2} $. It follows from \eqref{vfi-bd-I-0}, \eqref{exp-v-B-0-linfty} and the H\"older inequality that
\begin{align*}
\displaystyle A _{2}& \leq   c(v _{\ast}, T)\left( 1+\int _{\mathbb{R}_{+}}\int _{z}^{\infty}\langle y \rangle ^{-4}\mathrm{d}y \mathrm{d}z  \right) \|\langle z \rangle ^{l+2}({\mathop{\mathrm{e}}}^{v ^{B,0}}-1)\|_{L _{T}^{\infty}H _{z}^{1}}^{2} \|\partial _{t}^{3}\varphi _{x}^{I,0}(0,t)\|_{L ^{2}(0,T)}^{2} \leq c(v _{\ast}, T)v _{\ast}^{2} .
\end{align*}
    Therefore we get for any $ l \in \mathbb{N} $ that
  \begin{gather*}
  \displaystyle \|\langle z \rangle ^{l}\partial _{t}^{k}\varphi ^{B,1}\|_{L _{T}^{2}H _{z}^{7-2k}}\leq c(v _{\ast}, T) v _{\ast},\ \ \ k=0,1,2,3.
  \end{gather*}
  The proof is complete.
\end{proof}
The following lemma gives the regularity of $ (\varphi ^{b,1}, v ^{b,1}) $ which can be proved by similar arguments as proving Lemma \ref{lem-v-B-0}.
\begin{lemma}\label{lem-first-bd-rt}
Assume the conditions in Lemma \ref{lem-regul-outer-layer-0} hold. Then for any $ T>0 $, the problem \eqref{first-bd-pro-rt}, \eqref{firs-bd-1-rt} admits a unique solution $ (v ^{b,0},\varphi ^{b,1})$ on $[0,T]$ such that $  0 \leq v ^{b,0} \leq v _{\ast} $,
\begin{gather}\label{con-b-0}
\displaystyle \ \
\displaystyle   \| \langle z \rangle ^{l}\partial _{t}^{k} v ^{b,0}\|_{L _{T}^{2} H _{\xi}^{6-2k}}\leq K _{0}(T,v _{\ast}) v _{\ast},\ \  \|\langle z \rangle ^{l}\partial _{t}^{k} \varphi ^{b,1}\|_{L _{T}^{2} H _{\xi}^{7-2k}}\leq c(v _{\ast}, T) v _{\ast},\ \ k=0,1,2,3,\\
\displaystyle \|\langle \xi \rangle ^{l} \partial _{t}^{k}  v ^{b,0}\|_{L _{T}^{\infty}H _{\xi}^{4-2k}}+\sum _{\lambda=0}^{1}\sum _{\ell=0}^{3- 2\lambda}\|\langle \xi \rangle ^{l}\partial _{t}^{\lambda}\partial _{\xi}^{\ell} v ^{b,0}\|_{L _{T}^{\infty}L _{\xi}^{\infty}} \leq K _{0}(v _{\ast}, T) v _{\ast},
\end{gather}
where $ K _{0}(T,v _{\ast})>0 $ is as in Lemma \ref{lem-v-B-0}, $ c(v _{\ast}, T) $ is as stated in Section  \ref{sec:main_result}.
\end{lemma}
We next turn to the existence and regularity of the outer-layer profile $ (\varphi ^{I,1}, v ^{I,1}) $.
\begin{lemma}\label{lemf-vfi-V-I-1}
Assume the conditions in Lemma \ref{lem-regul-outer-layer-0} hold, and let $ (v ^{B,0}, \varphi ^{B,1}) $ and $ (v ^{b,0},\varphi ^{b,1}) $ be the solutions obtained in Lemmas \ref{lem-v-B-0} and \ref{lem-first-bd-rt}, respectively. Then for any $ T>0 $, the problem \eqref{first-outer-problem} admits a unique classical solution $ (\varphi ^{I,1}, v ^{I,1}) $ on $ [0,T]$ satisfying
\begin{subequations}\label{vfi-I-1-V-i-1}
\begin{gather}
\displaystyle  \|\partial _{t}^{k}\varphi^{I,1}\|_{L _{T} ^{2}H ^{6-2k}} \leq c(v _{\ast}, T) \ \ \ \mbox{for } k=0,1,2,3,\label{con-vfi-I-1}\\
\displaystyle \|v ^{I,1}\| _{L _{T}^{\infty}H ^{5}}+ \|\partial _{t}^{k}v ^{I,1}\| _{L _{T}^{2}H ^{7-2k}}\leq c(v _{\ast}, T)  \ \ \mbox{for}\ k=1,2,3\label{con-v -I-1}.
\end{gather}
\end{subequations}
\end{lemma}
\begin{proof}
 The local existence and uniqueness of solutions to the problem \eqref{first-outer-problem} on $ (\varphi ^{I,1}, v ^{I,1}) $ can be proved by the classical PDE theory for linear parabolic equations (cf. \cite[Section  7.1]{evans-book}) along with the fixed point theorem. In the following, we will devote ourselves to establishing some \emph{a priori} estimates from which the global existence and the desired regularity of the solution follow.

 Denote $ b(x,t):=x \varphi ^{b,1}(1,t)+(1-x)\varphi ^{B,1}(0,t) $ and $ \tilde{\varphi}:=\varphi ^{I,1}+b(x,t) $ with
\begin{gather}\label{bdry-value-vfi-B-b-1}
\begin{cases}
  \displaystyle  \varphi ^{B,1}(0,t)= -\int _{0}^{\infty}(\varphi _{x} ^{I,0}(0,t)+M)\left(    {\mathop{\mathrm{e}}}^{v ^{B,0}(y,t)}-1 \right) \mathrm{d}y,
   \nonumber \\[3mm]
   \displaystyle \varphi ^{b,1}(0,t)=\int _{- \infty}^{0}(\varphi _{x} ^{I,0}(1,t)+M)\left(    {\mathop{\mathrm{e}}}^{v ^{b,0}(y,t)}-1 \right)\mathrm{d}y.
\end{cases}
\end{gather}
Then we deduce from \eqref{first-outer-problem} that
  \begin{gather}\label{eq-vfi-I-1-refor}
  \displaystyle  \begin{cases}
    \displaystyle \tilde{\varphi}_{t}=\tilde{\varphi}_{xx}-(\varphi _{x}^{I,0}+M)v _{x}^{I,1}-\tilde{\varphi}_{x}v _{x}^{I,0}+f _{1}(x,t),\\
    \displaystyle   v_{t} ^{I,1}=-\left( \varphi _{x}^{I,0}+M \right)v ^{I,1}-\tilde{\varphi}_{x}v ^{I,0}+f _{2}(x,t),\\
    \displaystyle \tilde{\varphi}(0,t)=\tilde{\varphi}(1,t)=0,\\
    \displaystyle (\tilde{\varphi}, v ^{I,1})(x,0)=(0,0),\\
  \end{cases}
  \end{gather}
  where the fact $ v ^{B,0}(z,0)=v ^{b,0}(\xi,0)=0 $ has been used, and $ f _{i}(x,t)~(i=1,2) $ are given by
  \begin{gather}\label{f-1-2-defi}
  \displaystyle f _{1}(x,t):=b _{t}+b _{x}v _{x}^{I,0},\ \  f _{2}(x,t):= b _{x}v ^{I,0},\ \ k=0,1.
  \end{gather}
 To ensure the desired regularity of the solution, it is necessary to derive some estimates for the source terms involved. By \eqref{con-vfi-v-I-0-regula}, \eqref{v-B-0-regularity}, \eqref{vfi-bd-I-0}, \eqref{vfi-I-0-x-bd-infty} and \eqref{con-exp-v-B-0-final}, we deduce for $ k=0,1,2 $ that
  \begin{align*}
  \displaystyle \|\partial _{t}^{k}\varphi ^{B,1}(0,t)\|_{L ^{2}(0,T)}^{2}& \leq c _{0}\sum _{j=0}^{k}\int _{0}^{T}\left\vert \int _{0}^{\infty} \partial _{t}^{k-j}( \varphi_{x} ^{I,0}(0,t)+M)\partial _{t}^{j}\left(    {\mathop{\mathrm{e}}}^{v ^{B,0}(y,t)}-1 \right) \mathrm{d}y\right\vert ^{2}\mathrm{d}t
   \nonumber \\
   &\displaystyle \leq c _{0}\sum _{j=0}^{k}\|\partial _{t}^{k-j}(\varphi _{x}^{I,0}(0,t)+M)\|_{L ^{\infty}(0,T)} \int _{0}^{T}\left\vert \int _{0}^{\infty} \partial _{t}^{j}\left(    {\mathop{\mathrm{e}}}^{v ^{B,0}(y,t)}-1 \right) \mathrm{d}y\right\vert ^{2}\mathrm{d}t
    \nonumber \\
       & \displaystyle\leq c(v _{\ast}, T)\sum _{j=0}^{k}\int _{0}^{\infty}\langle y \rangle ^{-2}\mathrm{d}y \|\langle z \rangle \partial _{t}^{j}( {\mathop{\mathrm{e}}}^{v ^{B,0}}-1)\|_{L _{T}^{2}L _{z}^{2}}^{2}
        \nonumber \\
        &\leq c(v _{\ast}, T),
  \end{align*}
  and for $ k=3 $ that
  \begin{align*}
  \displaystyle  \|\partial _{t}^{3}\varphi ^{B,1}(0,t)\|_{L ^{2}(0,T)}^{2}&
      \leq c _{0}\sum _{j=1}^{3}\int _{0}^{T}\left\vert \int _{0}^{\infty} \partial _{t}^{3-j}(\varphi _{x}^{I,0}(0,t)+M)\partial _{t}^{j}\left(    {\mathop{\mathrm{e}}}^{v ^{B,0}(y,t)}-1 \right) \mathrm{d}y\right\vert ^{2}\mathrm{d}t
   \nonumber \\
   & \displaystyle \quad + c _{0}\int _{0}^{T}\left\vert \int _{0}^{\infty}\partial _{t}^{3}\varphi _{x}^{I,0}(0,t)\left(     {\mathop{\mathrm{e}}}^{v ^{B,0}(y,t)}-1 \right)\mathrm{d}y  \right\vert ^{2}\mathrm{d}t
   \nonumber \\
   & \leq c(v _{\ast}, T) \int _{\mathbb{R}_{+}}\langle y \rangle ^{-2}\mathrm{d}y\|\partial _{t}^{3}\varphi _{x}^{I,0}(0,t)\|_{L ^{2}(0,T)}^{2}\|\langle z \rangle( {\mathop{\mathrm{e}}}^{v ^{B,0}}-1)\|_{L _{T}^{\infty}L _{z}^{2}}^{2}
    \nonumber \\
    & \displaystyle \quad+
   c(v _{\ast}, T)\sum _{j=1}^{3}\int _{\mathbb{R}_{+}}\langle y \rangle ^{-2}\mathrm{d}y \|\langle z \rangle\partial _{t}^{j}( {\mathop{\mathrm{e}}}^{v ^{B,0}}-1)\|_{L _{T}^{2}L _{z}^{2}}^{2}
    \nonumber \\
    & \displaystyle\leq c(v _{\ast}, T).
  \end{align*}
Thus it holds for $ k=0,1,2 ,3$ that
  \begin{gather}\label{esti-vfi-B-0-BDY}
  \displaystyle  \|\partial _{t}^{k}\varphi ^{B,1}(0,t)\|_{L ^{2}(0,T)}^{2} \leq c(v _{\ast}, T).
  \end{gather}
Similarly, by \eqref{con-vfi-v-I-0-regula}, \eqref{vfi-bd-I-0} and \eqref{con-b-0}, we have for $ \varphi ^{b,0}(0,t) $ that
  \begin{gather}\label{esti-vfi-b-0-bd}
  \displaystyle   \|\partial _{t}^{k}\varphi ^{b,1}(0,t)\|_{L ^{2}(0,T)}^{2} \leq c(v _{\ast}, T)\ \ \ \mbox{for}\ k=0,1,2,3.
  \end{gather}
  With \eqref{con-vfi-v-I-0-regula}, \eqref{esti-vfi-B-0-BDY} and \eqref{esti-vfi-b-0-bd}, recalling the definitions of $ f _{1} $ and $ f _{2} $ in \eqref{f-1-2-defi}, we get for $ k=0,1,2 $ that
  \begin{subequations}\label{f-1-f2-esti}
  \begin{align}
  \displaystyle \|\partial _{t}^{k}f _{1}\|_{L _{T}^{2}H ^{4-2k}}^{2}
     &\displaystyle \leq c(v _{\ast}, T) \left(  \|\partial _{t}^{k+1}\varphi ^{B,1}(0,t)\|_{L ^{2}(0,T)}^{2}+ \|\partial _{t}^{k+1}\varphi ^{b,1}(0,t)\|_{L ^{2}(0,T)}^{2} \right)
   \nonumber \\
   & \displaystyle \quad + c _{0}\sum _{j=0}^{k}\left(  \|\partial _{t}^{j}\varphi ^{B,1}(0,t)\|_{L ^{2}(0,T)}^{2}+ \|\partial _{t}^{j}\varphi ^{b,1}(0,t)\|_{L ^{2}(0,T)}^{2}  \right)
    \nonumber \\
    & \displaystyle \quad\quad \quad~\quad \quad\times\|\partial _{t}^{k-j}v ^{I,0}\| _{L _{T}^{\infty}H ^{5-2k}}^{2} \leq c(v _{\ast}, T), \label{f-1-esti}\\
    \displaystyle  \|\partial _{t}^{k}f _{2}\|_{L ^{2}(0,T;H ^{5-2k})}^{2} & \leq c _{0}\sum _{j=0}^{k}\left(  \|\partial _{t}^{j}\varphi ^{B,1}(0,t)\|_{L ^{2}(0,T)}^{2}+ \|\partial _{t}^{j}\varphi ^{b,1}(0,t)\|_{L ^{2}(0,T)}^{2}  \right)
   \nonumber \\
   & \displaystyle \quad\quad \quad~ \quad\times\|\partial _{t}^{k-j}v ^{I,0}\| _{L _{T}^{\infty}H ^{5-2k}}^{2} \leq c(v _{\ast}, T). \label{f-2-esti}
  \end{align}
  \end{subequations}
  Now we are ready to establish estimates for the solution. Multiplying the first equation in \eqref{eq-vfi-I-1-refor} by $ \tilde{\varphi} $ and integrating the resulting equation over $ \mathcal{I} $, we have
  \begin{align}\label{vfi-outer-1-esti0}
  &\displaystyle \frac{1}{2}\frac{\mathrm{d}}{\mathrm{d}t}\int _{\mathcal{I} } \tilde{\varphi}^{2}\mathrm{d}x +\int _{\mathcal{I} }\tilde{\varphi}_{x}^{2} \mathrm{d}x
   \nonumber \\
   &~\displaystyle=-\int _{\mathcal{I} }(\varphi _{x}^{I,0}+M)v _{x}^{I,1} \tilde{\varphi} \mathrm{d}x- \int _{\mathcal{I} }\tilde{\varphi}_{x}v _{x}^{I,0} \tilde{\varphi} \mathrm{d}x+ \int _{\mathcal{I} }f _{1} \tilde{\varphi}\mathrm{d}x
   \nonumber \\
   &~ \displaystyle = \int _{\mathcal{I} }\varphi _{xx}^{I,0}v ^{I,1}\tilde{\varphi} \mathrm{d}x+\int _{\mathcal{I} }(\varphi _{x}^{I,0}+M)v ^{I,1} \tilde{\varphi}_{x} \mathrm{d}x +  \|v _{x}^{I,0}\|_{L ^{\infty}}\|\tilde{\varphi}_{x}\|_{L ^{2}}\|\tilde{\varphi}\|_{L ^{2}}+\|f _{1}\|_{L ^{2}}\|\tilde{\varphi}\|_{L ^{2}}      \nonumber \\
     &~\displaystyle \leq  \|\varphi _{xx}^{I,0}\|_{L ^{\infty}} \|v ^{I,1}\|_{L ^{2}}\|\tilde{\varphi}\|_{L ^{2}}+\frac{1}{16}\|\tilde{\varphi}_{x}\|_{L ^{2}}^{2}+ c(v _{\ast}, T) \left(  \|\tilde{\varphi}_{x}\|_{L ^{2}}\|v ^{I,1}\|_{L ^{2}}+ \|f _{1}\|_{L ^{2}}^{2}+\|\tilde{\varphi}\|_{L ^{2}}^{2}\right)
       \nonumber \\
       &~  \displaystyle \leq \frac{1}{8}\|\tilde{\varphi}_{x}\|_{L ^{2}}^{2} + c(v _{\ast}, T) \left( \|\tilde{\varphi}\|_{L ^{2}}^{2}+\|v ^{I,1}\|_{L ^{2}} ^{2} \right)+c(v _{\ast}, T) \|f _{1}\|_{L ^{2}}^{2},
  \end{align}
  where \eqref{con-vfi-v-I-0-regula}, integration by parts and the Cauchy-Schwarz inequality have been used. On the other hand, testing the second equation in \eqref{eq-vfi-I-1-refor} against $ v ^{I,1} $, we have
  \begin{align}\label{v-outer-1-esti0}
  &\displaystyle \frac{1}{2}\frac{\mathrm{d}}{\mathrm{d}t}\int _{\mathcal{I} }\vert v ^{I,1}\vert ^{2} \mathrm{d}x + \int _{\mathcal{I} } \left( \varphi _{x}^{I,0}+M \right)\vert v ^{I,1}\vert ^{2}\mathrm{d}x
  = \int _{\mathcal{I} }\left( -\tilde{\varphi}_{x}v ^{I,0}+f _{2} \right)v ^{I,1} \mathrm{d}x
   \nonumber \\
   &~ \displaystyle \leq c _{0}\int _{\mathcal{I} }\vert v ^{I,1}\vert ^{2} \mathrm{d}x+c _{0} \|v ^{I,0}\|_{L ^{\infty}}^{2} \|\tilde{\varphi}_{x}\|_{L ^{2}}^{2} +c _{0} \|f _{2}\|_{L ^{2}}^{2}
    \nonumber \\
    &~ \displaystyle \leq c  _{0}\int _{\mathcal{I} }\vert v ^{I,1}\vert ^{2} \mathrm{d}x+ c(v _{\ast}, T)\|\tilde{\varphi}_{x}\|_{L ^{2}}^{2} +c _{0} \|f _{2}\|_{L ^{2}}^{2},
  \end{align}
  where we have used \eqref{con-vfi-v-I-0-regula} and the Cauchy-Schwarz inequality. Combining \eqref{vfi-outer-1-esti0} with \eqref{v-outer-1-esti0} implies that
  \begin{align*}
  \displaystyle \frac{\mathrm{d}}{\mathrm{d}t}\int _{\mathcal{I} }\left( \tilde{\varphi}^{2}+\vert v ^{I,1}\vert ^{2} \right)  \mathrm{d}x+\int _{\mathcal{I} } \tilde{\varphi}_{x}^{2} \mathrm{d}x \leq c(v _{\ast}, T)\left( \|\tilde{\varphi}\|_{L ^{2}}^{2}+\|v ^{I,1}\|_{L ^{2}} ^{2} \right)+ c(v _{\ast}, T)\left(\|f _{1}\|_{L ^{2}}^{2}+ \|f _{2}\|_{L ^{2}}^{2}\right),
  \end{align*}
  where we have used the fact $ \varphi _{x}^{I,0} +M \geq 0 $ from \eqref{con-vfi-v-I-0-regula}.
  This along with \eqref{f-1-f2-esti} and the Gronwall inequality immediately yields for any $ t \in (0,T] $ that
  \begin{align}\label{CON-vfi-I-1-esti0}
    \displaystyle  \int _{\mathcal{I} }\left( \tilde{\varphi}^{2}+\vert v ^{I,1}\vert ^{2} \right) (\cdot,t) \mathrm{d}x+ \int _{0}^{t}\int _{\mathcal{I} }  \tilde{\varphi}_{x}^{2} \mathrm{d}x  \mathrm{d}\tau \leq c(v _{\ast}, T).
    \end{align}
    Multiplying the first equation in \eqref{eq-vfi-I-1-refor} by $ \tilde{ \varphi}_{t} $
    followed by an integration over $ \mathcal{I} $, it holds that
    \begin{align}\label{vfi-OUT-1-T}
   &\displaystyle \frac{1}{2}\frac{\mathrm{d}}{\mathrm{d}t}\int _{\mathcal{I} }\tilde{\varphi}_{x}^{2}  \mathrm{d}x+ \int _{\mathcal{I} }\tilde{\varphi}_{t}^{2} \mathrm{d}x
     =
    - \int _{\mathcal{I} }\tilde{\varphi}_{x}v _{x}^{I,0} \tilde{\varphi}_{t} \mathrm{d}x
    + \int _{\mathcal{I} }f _{1} \tilde{\varphi}_{t}\mathrm{d}x
    -\int _{\mathcal{I} }(\varphi _{x}^{I,0}+M)v _{x}^{I,1} \tilde{\varphi}_{t} \mathrm{d}x.
    \end{align}
    By \eqref{con-vfi-v-I-0-regula} and the Cauchy-Schwarz inequality, we have
      \begin{align}\label{vfi-out-t-etsi}
      \displaystyle
    - \int _{\mathcal{I} }\tilde{\varphi}_{x}v _{x}^{I,0} \tilde{\varphi}_{t} \mathrm{d}x
    + \int _{\mathcal{I} }f _{1} \tilde{\varphi}_{t}\mathrm{d}x
     &\leq \frac{1}{4}\int _{\mathcal{I} }\tilde{ \varphi}_{t}^{2} \mathrm{d}x
     +c _{0} \|v _{x}^{I,0}\|_{L ^{\infty}}^{2}\|\tilde{\varphi}_{x}\|_{L ^{2}}^{2}+c  _{0}\|f _{1}\|_{L ^{2}}^{2}
       \nonumber \\
       & \displaystyle \leq  \frac{1}{4} \|\tilde{ \varphi}_{t}\|_{L ^{2}}^{2}
       +c(v _{\ast}, T)\|\tilde{\varphi }_{x}\|_{L ^{2}}^{2}
       + c_{0} \|f _{1}\|_{L ^{2}}^{2},
      \end{align}
      where we have used the fact $ \|v _{x}^{I,0}\|_{L _{T} ^{\infty}L ^{\infty}} \leq c(v _{\ast}, T) $
      due to \eqref{con-vfi-v-I-0-regula}. For the last term on the right hand
      side of \eqref{vfi-OUT-1-T}, we get by virtue of integration by parts
      and the Cauchy-Schwarz inequality that
      \begin{align}\label{vfi-out-t-etsi-1}
      &\displaystyle  -\int _{\mathcal{I} }(\varphi _{x}^{I,0}+M)v _{x}^{I,1} \tilde{\varphi}_{t} \mathrm{d}x
      = \int _{\mathcal{I} }(\varphi _{x}^{I,0}+M)v ^{I,1}\tilde{\varphi}_{tx} \mathrm{d}x
      +\int _{\mathcal{I} }\varphi _{xx}^{I,0}v ^{I,1}\tilde{\varphi}_{t} \mathrm{d}x
       \nonumber \\
       &~ \displaystyle =\frac{\mathrm{d}}{\mathrm{d}t}\int _{\mathcal{I} }(\varphi _{x}^{I,0}+M)v ^{I,1}\tilde{\varphi}_{x}  \mathrm{d}x
       - \int _{\mathcal{I} }\varphi _{xt}^{I,0}v ^{I,1}\tilde{\varphi }_{x} \mathrm{d}x
       - \int _{\mathcal{I} } (\varphi _{x}^{I,0}+M)v _{t}^{I,1}\tilde{\varphi}_{x} \mathrm{d}x
       +\int _{\mathcal{I} }\varphi _{xx}^{I,0}v ^{I,1}\tilde{\varphi}_{t} \mathrm{d}x
        \nonumber \\
        &~ \displaystyle \leq \frac{\mathrm{d}}{\mathrm{d}t}\int _{\mathcal{I} }(\varphi _{x}^{I,0}+M)v ^{I,1}\tilde{\varphi}_{x}  \mathrm{d}x
        + \|\varphi _{xt}^{I,0}\|_{L ^{\infty}}\|v ^{I,1}\|_{L ^{2}}\|\tilde{\varphi}_{x}\|_{L ^{2}}
        +c(v _{\ast}, T) \|v _{t}^{I,1}\|_{L ^{2}}\|\tilde{\varphi}_{x}\|_{L ^{2}}
         \nonumber \\
         &~\displaystyle \quad+\|\varphi _{xx}^{I,0}\|_{L ^{\infty}}\|v ^{I,1}\|_{L ^{2} }\|\tilde{\varphi}_{t}\|_{L ^{2}}
          \nonumber \\
          &~\displaystyle \leq \frac{\mathrm{d}}{\mathrm{d}t}\int _{\mathcal{I} }(\varphi _{x}^{I,0}+M)v ^{I,1}\tilde{\varphi}_{x}  \mathrm{d}x+\frac{1}{8}\|\tilde{\varphi}_{t}\|_{L ^{2}}^{2}+c(v _{\ast}, T) \|\tilde{\varphi}_{x}\|_{L ^{2}}^{2}+c(v _{\ast}, T) \left( \|v _{t}^{I,1}\|_{L ^{2}}^{2}+ \|v^{I,1}\|_{L ^{2}}^{2}\right),
      \end{align}
            where we have used $ \|\varphi _{xx}^{I,0}\|_{L _{T}^{\infty}L ^{\infty}}+\|\varphi _{xt}^{I,0}\|_{L _{T}^{\infty}L ^{\infty}}\leq c(v _{\ast}, T)$ due to \eqref{con-vfi-v-I-0-regula}
            and Proposition \ref{prop-embeding-spacetime}.
            Collecting \eqref{vfi-out-t-etsi} and \eqref{vfi-out-t-etsi-1},
            we thus have from \eqref{vfi-OUT-1-T} that
      \begin{align}\label{pre-con-vfi-1-out}
      &\displaystyle  \frac{1}{2}\frac{\mathrm{d}}{\mathrm{d}t}\int _{\mathcal{I} }\tilde{\varphi}_{x}^{2}  \mathrm{d}x
      -\frac{\mathrm{d}}{\mathrm{d}t}\int _{\mathcal{I} }
      (\varphi _{x}^{I,0}+M)v ^{I,1}\tilde{\varphi}_{x}  \mathrm{d}x
      + \frac{1}{2}\int _{\mathcal{I} }\tilde{\varphi}_{t}^{2} \mathrm{d}x
       \nonumber \\
       &~\displaystyle \leq c(v _{\ast}, T)(\|\tilde{\varphi }_{x}\|_{L ^{2}}^{2}+\|v^{I,1}\|_{L ^{2}}^{2})
       + c(v _{\ast}, T) \|f _{1}\|_{L ^{2}}^{2}+c(v _{\ast}, T)\|v _{t}^{I,1}\|_{L ^{2}}^{2}.
      \end{align}
      To control the term on $v _{t}^{I,1} $ on the right hand side of \eqref{pre-con-vfi-1-out}, we test the second equation in \eqref{eq-vfi-I-1-refor} against $ v _{t}^{I,1} $ and deduce that
      \begin{align}\label{multi-v-t-I-1}
      &\displaystyle \frac{1}{2}\frac{\mathrm{d}}{\mathrm{d}t}\int _{\mathcal{I} }\left( \varphi _{x}^{I,0}+M \right) \vert v ^{I,1}\vert ^{2} \mathrm{d}x+\int _{\mathcal{I} } \vert   v _{t}^{I,1}\vert ^{2}\mathrm{d}x
       \nonumber \\
       &~\displaystyle= \frac{1}{2}\int _{\mathcal{I} }\varphi _{xt}^{I,0}\vert v ^{I,1}\vert ^{2} \mathrm{d}x+\int _{\mathcal{I} }\left( -\tilde{\varphi}_{x}v ^{I,0}+ f _{2} \right)v _{t}^{I,1}  \mathrm{d}x
       \nonumber \\
       & ~\displaystyle \leq \frac{1}{2} \int _{\mathcal{I} }\vert v _{t}^{I,1}\vert ^{2} \mathrm{d}x+c _{0} \|\varphi _{xt}^{I,0}\|_{L ^{\infty}}\|v ^{I,1}\|_{L ^{2}}^{2}+c_{0}\|v ^{I,0}\|_{L ^{\infty}} ^{2} \|\tilde{\varphi}_{x}\|_{L ^{2}}^{2}+ c_{0}\|f _{2}\|_{L ^{2}}^{2}
        \nonumber \\
        & ~\displaystyle \leq  \frac{1}{2} \int _{\mathcal{I} }\vert v_{t} ^{I,1}\vert ^{2} \mathrm{d}x+c(v _{\ast}, T) \|v ^{I,1}\|_{L ^{2}}^{2}+c(v _{\ast}, T) \|\tilde{\varphi}_{x}\|_{L ^{2}}^{2}+c_{0}\|f _{2}\|_{L ^{2}} ^{2},
      \end{align}
      where we have used \eqref{vfi-I-0-x-bd-infty}. Therefore we get from \eqref{pre-con-vfi-1-out} and \eqref{multi-v-t-I-1} that
      \begin{align*}
     & \displaystyle    \frac{\mathrm{d}}{\mathrm{d}t}\int _{\mathcal{I} }\left( \tilde{\varphi}_{x}^{2}+( \varphi _{x}^{I,0}+M) \vert v ^{I,1}\vert ^{2} \right) (\cdot,t)  \mathrm{d}x+ \int _{\mathcal{I} }\left( \tilde{\varphi}_{t}^{2}+\vert v _{t}^{I,1}\vert ^{2} \right)  \mathrm{d}x
      \nonumber \\
      &~ \displaystyle \leq c(v _{\ast}, T)\left( \|\tilde{\varphi }_{x}\|_{L ^{2}}^{2}+\|v ^{I,1}\|_{L ^{2}}^{2} \right)
       + c(v _{\ast}, T) \left( \|f _{1}\|_{L ^{2}} ^{2}+\|f _{2}\|_{L ^{2}} ^{2} \right)+c(v _{\ast}, T)\frac{\mathrm{d}}{\mathrm{d}t}\int _{\mathcal{I} }(\varphi _{x}^{I,0}+M)v ^{I,1}\tilde{\varphi}_{x}  \mathrm{d}x.
      \end{align*}
      Integrating the above inequality over $ (0,t) $ for any $ t \in (0,T] $ yields that
      \begin{align*}
      &\displaystyle  \int _{\mathcal{I} }\left( \tilde{\varphi}_{x}^{2}+( \varphi _{x}^{I,0}+M) \vert v ^{I,1}\vert ^{2} \right)  (\cdot,t) \mathrm{d}x+\int _{0}^{t}\int _{\mathcal{I} }\left( \tilde{\varphi}_{\tau}^{2}+\vert  v _{\tau}^{I,1}\vert ^{2} \right) \mathrm{d}x \mathrm{d}\tau
       \nonumber \\
       &~ \displaystyle \leq c(v _{\ast}, T)+c(v _{\ast}, T)\int _{\mathcal{I} }(\varphi _{x}^{I,0}+M)v ^{I,1}\tilde{\varphi}_{x}  \mathrm{d}x+ c(v _{\ast}, T) \int _{0}^{t}  \left( \|f _{1}\|_{L ^{2}} ^{2}+\|f _{2}\|_{L ^{2}} ^{2} \right) \mathrm{d}\tau \nonumber \\
        &~\displaystyle \quad +c(v _{\ast}, T)\int _{0}^{t}\left( \|\tilde{\varphi }_{x}\|_{L ^{2}}^{2}+\|v ^{I,1}\|_{L ^{2}}^{2} \right) \mathrm{d}\tau
         \nonumber \\
         &~ \displaystyle \leq  \frac{1}{2}\int _{\mathcal{I} }\tilde{\varphi}_{x}^{2} \mathrm{d}x+c(v _{\ast}, T)+c(v _{\ast}, T)\int _{\mathcal{I} }\vert v ^{I,1}\vert ^{2} \mathrm{d}x+c(v _{\ast}, T)\int _{0}^{t}\left( \|\tilde{\varphi }_{x}\|_{L ^{2}}^{2}+\|v ^{I,1}\|_{L ^{2}}^{2} \right) \mathrm{d}\tau
          \nonumber \\
          &~ \displaystyle \leq  \frac{1}{2}\int _{\mathcal{I} }\tilde{\varphi}_{x}^{2} \mathrm{d}x+c(v _{\ast}, T) +c(v _{\ast}, T)\int _{0}^{t}\|\tilde{\varphi }_{x}\|_{L ^{2}}^{2} \mathrm{d}\tau,
      \end{align*}
      where we have used \eqref{f-1-f2-esti}, \eqref{CON-vfi-I-1-esti0} and the Cauchy-Schwarz inequality. We thus have
     \begin{align}\label{con-vfi-I-1-ESTI1}
     & \displaystyle  \int _{\mathcal{I} } \tilde{\varphi}_{x}^{2}(\cdot,t)   \mathrm{d}x+\int _{0}^{t}\int _{\mathcal{I} }\left( \tilde{\varphi}_{\tau}^{2}+\vert  v_{\tau} ^{I,1}\vert ^{2} \right) \mathrm{d}x \mathrm{d}\tau \leq c(v _{\ast}, T)+c(v _{\ast}, T)\int _{0}^{t}\|\tilde{\varphi }_{x}\|_{L ^{2}}^{2} \mathrm{d}\tau \leq c(v _{\ast}, T)  ,
      \end{align}
      where we have used \eqref{con-vfi-v-I-0-regula} and \eqref{CON-vfi-I-1-esti0}. We proceed to derive estimate for $ v _{x}^{I,1} $. Differentiating the second equation in \eqref{eq-vfi-I-1-refor} with respect to $ x $ leads to
\begin{align}\label{v-I-X-eq}
\displaystyle v _{tx}^{I,1}=-(\varphi _{x}^{I,0}+M)v _{x}^{I,1}-\varphi _{xx}^{I,0}v  ^{I,1} -\tilde{\varphi}_{xx}v ^{I,0}-\tilde{\varphi}_{x}v _{x}^{I,0}+ \partial _{x} f _{2}.
\end{align}
Multiplying \eqref{v-I-X-eq} by $ v _{x}^{I,1} $ followed by an integration over $ \mathcal{I} $, we have
\begin{align}\label{diff-ineq-vif-I-1}
&\displaystyle  \frac{1}{2}\frac{\mathrm{d}}{\mathrm{d}t}\int _{\mathcal{I} }\vert v _{x}^{I,1}\vert ^{2} \mathrm{d}x
+\int _{\mathcal{I} }(\varphi _{x}^{I,0}+M)\vert v _{x}^{I,1}\vert ^{2} \mathrm{d}x
 \nonumber \\
 &~\displaystyle = -\int _{\mathcal{I} }\varphi _{xx}^{I,0}v  ^{I,1} v _{x}^{I,1} \mathrm{d}x
 - \int _{\mathcal{I} }\tilde{\varphi}_{xx}v ^{I,0} v _{x}^{I,1} \mathrm{d}x
 -\int _{\mathcal{I} }\tilde{\varphi}_{x}v _{x}^{I,0} v _{x}^{I,1} \mathrm{d}x
 + \int _{\mathcal{I} } \partial _{x}f _{2} v _{x}^{I,1} \mathrm{d}x
  \nonumber \\
  & ~\displaystyle \leq \|\varphi _{xx}^{I,0}\|_{L ^{\infty}}\|v ^{I,1}\|_{L ^{2}}\|v _{x}^{I,1}\|_{L ^{2}}
  + \|v ^{I,0}\|_{L ^{\infty}}\|\tilde{\varphi}_{xx}\|_{L ^{2}}\|v _{x}^{I,1}\|_{L ^{2}}
   \nonumber \\
   &~\displaystyle \quad+\|v _{x}^{I,0}\|_{L ^{\infty}}\|\tilde{\varphi}_{x}\|_{L ^{2}}\|v _{x}^{I,1}\|_{L ^{2}}
   + \|\partial _{x}f _{2}\|_{L ^{2}} \|v _{x}^{I,1}\|_{L ^{2}}
   \nonumber \\
   &~ \leq \frac{1}{8}\|\tilde{\varphi}_{xx}\|_{L ^{2}}^{2}
   +c _{0} \left( \|\varphi _{xx}^{I,0}\|_{L ^{\infty}}^{2}+\|v ^{I,0}\|_{L ^{\infty}} ^{2}
   +\|v _{x}^{I,0}\|_{L ^{\infty}} ^{2}+\|v ^{I,0}\|_{L ^{2}}^{2} \right) \|v _{x}^{I,1}\|_{L ^{2}} ^{2}
    \nonumber \\
    &~\displaystyle \quad+c_{0}\left( \|v ^{I,1}\|_{L ^{2}}^{2}+\|\tilde{\varphi}_{x}\|_{L ^{2}}^{2}
    + \|\partial _{x}f _{2}\|_{L ^{2}}^{2} \right)
     \nonumber \\
     & ~\displaystyle \leq \frac{1}{8}\|\tilde{\varphi}_{xx}\|_{L ^{2}}^{2}
     +c(v _{\ast}, T) \|v _{x}^{I,1}\|_{L ^{2}}^{2}+c(v _{\ast}, T)
     +  c(v _{\ast}, T)\|\partial _{x}f _{2}\|_{L ^{2}}^{2},
\end{align}
where we have used \eqref{con-vfi-v-I-0-regula}, \eqref{con-vfi-I-1-ESTI1}, Proposition \ref{prop-embeding-spacetime} and the Cauchy-Schwarz inequality. On the other hand, with \eqref{con-vfi-v-I-0-regula}, \eqref{CON-vfi-I-1-esti0} and \eqref{con-vfi-I-1-ESTI1}, we deduce from $ \eqref{eq-vfi-I-1-refor} _{1} $ that
\begin{align*}
\displaystyle  \|\tilde{\varphi}_{xx}\|_{L ^{2}}^{2} &\leq c _{0}\|\tilde{\varphi}_{t}\|_{L ^{2}}^{2}+ c_{0}\|(\varphi _{x}^{I,0}+M)v _{x}^{I,1}\|_{L ^{2}}^{2}+ c_{0}\|v _{x}^{I,0}\|_{L ^{\infty}}^{2}\|\tilde{\varphi}_{x}\|_{L ^{2}}^{2}+  c_{0}\|f _{1}\|_{L ^{2}}^{2}
 \nonumber \\
   & \displaystyle \leq c(v _{\ast}, T) \left( 1+\|\tilde{\varphi}_{t}\|_{L ^{2}}^{2}+\|v _{x}^{I,1}\| _{L ^{2}}^{2}+\|f _{1}\|_{L ^{2}}^{2} \right),
\end{align*}
which together with \eqref{con-vfi-v-I-0-regula} and \eqref{diff-ineq-vif-I-1} yields that
\begin{align}\label{v-i-1-x-final}
\displaystyle  \frac{1}{2}\frac{\mathrm{d}}{\mathrm{d}t}\int _{\mathcal{I} }\vert v _{x}^{I,1}\vert ^{2} \mathrm{d}x+ \| \tilde{\varphi}_{xx}\|_{L ^{2}}^{2}
  \leq c(v _{\ast}, T)\left(1+ \|\tilde{\varphi}_{t}\|_{L ^{2}}^{2}+\|v _{x}^{I,1}\| _{L ^{2}}^{2}+ \|f _{1}\|_{L ^{2}}^{2} +\|\partial _{x}f _{2}\|_{L ^{2}}^{2}\right).
\end{align}
Applying the Gronwall inequality to \eqref{v-i-1-x-final}, by virtue of \eqref{f-1-f2-esti} and \eqref{con-vfi-I-1-ESTI1}, we then arrive at
\begin{align}\label{con-v-i-1-x}
\displaystyle  \int _{\mathcal{I} }\vert v _{x}^{I,1}\vert ^{2}(\cdot,t) \mathrm{d}x + \int _{0}^{t}\| \tilde{\varphi}_{xx}\|_{L ^{2}}^{2} \mathrm{d}\tau \leq c(v _{\ast}, T)
\end{align}
for any $ t \in (0,T] $. This along with \eqref{con-vfi-v-I-0-regula}, \eqref{f-2-esti} and \eqref{v-I-X-eq} further gives that $  \|\partial _{t}v _{x}^{I,1}\|_{L _{T}^{2}L ^{2}}\leq c(v _{\ast}, T) $. Denote by $ \psi: =\tilde{\varphi}_{t} $ and $ w:=v _{t}^{I,1} $. Then in view of \eqref{eq-vfi-I-1-refor} and the compatibility conditions of initial data, we have
\begin{align}\label{v-I-1-AUX-EQ}
\displaystyle  \begin{cases}
    \displaystyle \psi_{t}=\psi_{xx}-(\varphi _{x}^{I,0}+M)w _{x}^{I,1}- \psi_{x}v _{x}^{I,0}+\tilde{f} _{1}(x,t),\\[1mm]
    \displaystyle w _{t}=-\left( \varphi _{x}^{I,0}+M \right)w ^{I,1}- \psi_{x}v ^{I,0}+\tilde{f} _{2}(x,t),\\[1mm]
    \displaystyle \psi(0,t)= \psi(1,t)=0,\\[1mm]
    \displaystyle (\psi, w)(x,0)=(\tilde{\varphi}_{t} , v _{t}) \vert _{t=0}=(0,0),
  \end{cases}
\end{align}
where $ \tilde{f}_{i}(x,t)~(i=1,2) $ are given by
\begin{gather*}
\displaystyle  \tilde{f} _{1}(x,t)=- \varphi _{xt}^{I,0}v _{x}^{I,1} - \tilde{\varphi}_{x}v _{xt}^{I,0}+\partial _{t}f _{1}(x,t),\ \ \ \tilde{f} _{2}(x,t)=-\varphi _{xt}^{I,0} v ^{I,1}-\tilde{\varphi}_{x}v _{t}^{I,0}+ \partial _{t}f _{2}(x,t).
\end{gather*}
Thanks to \eqref{con-vfi-v-I-0-regula}, \eqref{f-1-f2-esti}, \eqref{con-vfi-I-1-ESTI1} and \eqref{con-v-i-1-x} and Proposition \ref{prop-embeding-spacetime}, we deduce that
\begin{align}\label{hti-f-esti-first-1}
\displaystyle \|\tilde{f}_{1}\|_{L _{T}^{2}L ^{2}}^{2}&  \leq \int _{0}^{T} \|\varphi _{xt}^{I,0}\|_{L ^{\infty}} ^{2} \|v _{x}^{I,1}\|_{L ^{2}}^{2}\mathrm{d}t+ \int _{0}^{T} \|v _{xt}^{I,0}\|_{L ^{\infty}}^{2}\|\tilde{\varphi}_{x}\|_{L ^{2}}^{2}\mathrm{d}t + \int _{0}^{T}\|\partial _{t}f _{1}(x,t)\|_{L ^{2}}^{2}\mathrm{d}t
 \nonumber \\
 & \displaystyle \leq c(v _{\ast}, T) \int _{0}^{T}\|\varphi _{xt}^{I,0}\|_{L ^{\infty}} ^{2} \mathrm{d}t+c(v _{\ast}, T) \int _{0}^{T}\|v _{xt}^{I,0}\| _{L ^{\infty}}^{2}\mathrm{d}t+ \int _{0}^{T}\|\partial _{t}f _{1}(x,t)\|_{L ^{2}}^{2}\mathrm{d}t  \nonumber \\
  &\displaystyle\leq c(v _{\ast}, T),
\end{align}
\begin{align}\label{hti-f-esti-first-2}
\displaystyle \|\tilde{f}_{2}\|_{L _{T}^{2}L ^{2}}^{2}&  \leq \int _{0}^{T} \|\varphi _{xt}^{I,0}\|_{L ^{2}}^{2} \|v^{I,1}\|_{L ^{\infty}} ^{2}\mathrm{d}t+ \int _{0}^{T} \|v _{t}^{I,0}\|_{L ^{\infty}}^{2}\|\tilde{\varphi}_{x}\|_{L ^{2}}^{2}\mathrm{d}t + \int _{0}^{T}\|\partial _{t} f _{2}(x,t)\|_{L ^{2}}^{2}\mathrm{d}t
 \nonumber \\
 & \displaystyle \leq c(v _{\ast}, T) \int _{0}^{T}\|\varphi _{xt}^{I,0}\|_{L ^{2}} ^{2} \mathrm{d}t+c(v _{\ast}, T) \int _{0}^{T}\|v _{t}^{I,0}\| _{L ^{\infty}}^{2}\mathrm{d}t+c(v _{\ast}, T) \leq c(v _{\ast}, T)
\end{align}
and
\begin{align*}
\displaystyle \|\partial _{x} \tilde{f}_{2}\| _{L _{T}^{2}L ^{2}}^{2} &\leq \int _{0}^{T}\left( \|\varphi _{xxt}^{I,0}\| _{L ^{2}}^{2}\|v ^{I,1}\|_{L ^{\infty}}^{2}+\|\varphi _{xt}^{I,0}\|_{L ^{\infty}}^{2}\|v _{x}^{I,1}\| _{L ^{2}}^{2}\right)\mathrm{d}t
 \nonumber \\
 & \displaystyle \quad+\int _{0}^{T} \left( \|\tilde{\varphi}_{xx}\|_{L ^{2}}^{2}\|v _{t}^{I,0}\|_{L ^{\infty}}^{2}+ |\tilde{\varphi}_{x}\|_{L ^{\infty}}^{2}\|v _{xt}^{I,0}\|_{L ^{2}}^{2}+\|\partial _{t}\partial _{x}f _{2}\|_{L ^{2}}^{2}\right)  \mathrm{d}t
  \nonumber \\
  & \leq c(v _{\ast}, T) \int _{0}^{T}\left( \|\varphi _{xxt}^{I,0}\|_{L ^{2}}^{2}+\|\varphi _{x}^{I,0}\|_{L ^{\infty}}^{2}+\|\tilde{\varphi}_{x}\|_{H ^{1}}^{2}+\|\partial _{t}\partial _{x}f _{2}\|_{L ^{2}}^{2} \right) \mathrm{d}t \leq c(v _{\ast}, T).
\end{align*}
Therefore by the above procedure for estimates on $ (\tilde{\varphi},v ^{I,1}) $, we conclude for any $ t \in (0,T] $ that
\begin{align*}
\displaystyle \left( \|\psi(\cdot,t)\|_{H ^{1}}^{2}+\|w(\cdot,t)\|_{H ^{1}}^{2} \right)+ \int _{0}^{t}\left( \left\|\psi _{x}\right\|_{H ^{1}}^{2}+\left\|\psi _{\tau}\right\|_{L ^{2}}^{2}+\|w _{\tau}\| _{H ^{1}}^{2}\right)\mathrm{d}\tau  \leq c(v _{\ast}, T).
\end{align*}
That is,
\begin{align}\label{tld-v-t-all}
\displaystyle  \Big( \|\tilde{\varphi}_{t}(\cdot,t)\|_{H ^{1}}^{2}+\|v _{t} ^{I,1}(\cdot,t)\|_{H ^{1}}^{2} \Big)+ \int _{0}^{t}\Big( \|\tilde{\varphi} _{x \tau}\|_{H ^{1}}^{2}+\|\tilde{\varphi}_{\tau \tau} \|_{L ^{2}}^{2}+\|v _{\tau \tau}^{I,1}\| _{H^{1}}^{2}\Big)\mathrm{d}\tau  \leq c(v _{\ast}, T) .
\end{align}
With \eqref{con-vfi-v-I-0-regula}, \eqref{f-1-esti}, \eqref{CON-vfi-I-1-esti0}, \eqref{con-vfi-I-1-ESTI1}, \eqref{con-v-i-1-x} and \eqref{tld-v-t-all}, we deduce from \eqref{eq-vfi-I-1-refor} that
\begin{align}\label{esti-hti-v-xxx}
 \displaystyle \int _{0}^{t}\|\partial _{x}^{3}\tilde{\varphi}\|_{L ^{2}}^{2}\mathrm{d}\tau  &\leq  c _{0}\int _{0}^{t}\|\tilde{\varphi}_{x \tau}\|_{L ^{2}}^{2}\mathrm{d}\tau + c(v _{\ast}, T)\int _{0}^{t}\|v _{xx}^{I,1}\|_{L ^{2}}^{2}\mathrm{d}\tau+c _{0}\int _{0}^{t}\|\varphi _{xx}^{I,0}\|_{L ^{\infty}}^{2}\|v _{x}^{I,1}\|_{L ^{2}}^{2}\mathrm{d}\tau
  \nonumber \\
  & \displaystyle \quad+ c _{0}\int _{0}^{t} \|v _{x}^{I,0}\|_{L ^{\infty}}^{2}\|\tilde{\varphi}_{xx}\|_{L ^{2}}^{2}\mathrm{d}\tau +  c_{0}\int _{0}^{t}\|\tilde{\varphi}_{x}\|_{L ^{2}}^{2}\|v _{x}^{I,0}\|_{L ^{\infty}}^{2}\mathrm{d}\tau+c_{0}
  \int _{0}^{t}\| \partial _{x} f _{1}\|_{L ^{2}}^{2}\mathrm{d}\tau
   \nonumber \\
   & \displaystyle \leq c(v _{\ast}, T)+c(v _{\ast}, T)  \int _{0}^{t} \|v_{xx}^{I,1}\|_{L ^{2}}^{2}\mathrm{d}\tau
 \end{align}
 for any $ t \in (0,T] $. Differentiating \eqref{v-I-X-eq} with respect to $ x $ gives
 \begin{gather}\label{v-xx-I-1}
 \displaystyle  \displaystyle v _{txx}^{I,1}=-(\varphi _{x}^{I,0}+M)v _{xx}^{I,1}-2\varphi _{xx}^{I,0}v  _{x}^{I,1}- \partial _{x}^{3}\varphi^{I,0}v ^{I,1} - \partial _{x}^{3}\tilde{\varphi}v ^{I,0}-2\tilde{\varphi}_{xx}v _{x}^{I,0}-\tilde{\varphi}_{x}v _{xx}^{I,0}+ \partial _{x}^{2} f _{2}.
 \end{gather}
 Testing \eqref{v-xx-I-1} against $ v _{xx}^{I,1} $, we have
 \begin{align}\label{v-xx-I-1-diff}
 &\displaystyle \frac{1}{2}\frac{\mathrm{d}}{\mathrm{d}t}\int _{\mathcal{I} }\vert v _{xx} ^{I,1}\vert ^{2} \mathrm{d}x+\int _{\mathcal{I} }(\varphi _{x}^{I,0}+M) \vert v _{xx} ^{I,1}\vert ^{2}  \mathrm{d}x
  \nonumber \\
  &~\displaystyle =- \int _{\mathcal{I} } 2\varphi _{xx}^{I,0}v  _{x}^{I,1} v _{xx}^{I,1}\mathrm{d}x- \int _{\mathcal{I} } \varphi _{xxx}^{I,0}v ^{I,1}  v _{xx}^{I,1} \mathrm{d}x- \int _{\mathcal{I} }\tilde{\varphi}_{xxx}v ^{I,0} v _{xx}^{I,1} \mathrm{d}x
   \nonumber \\
   &~\displaystyle \quad - \int _{\mathcal{I} }2\tilde{\varphi}_{xx}v _{x}^{I,0} v _{xx}^{I,1} \mathrm{d}x- \int _{\mathcal{I} }\tilde{\varphi}_{x}v _{xx}^{I,0}  v _{xx}^{I,1}\mathrm{d}x+ \int _{\mathcal{I} }\partial _{x}^{2} f _{2} v _{xx}^{I,1} \mathrm{d}x
    \nonumber \\
    &~\displaystyle \leq 2 \|\varphi _{xx}^{I,0}\|_{L ^{\infty}}\|v _{x}^{I,1}\|_{L ^{2}}\|v _{xx}^{I,1}\|_{L ^{2}}+\|v ^{I,1}\|_{L ^{\infty}}\| \partial _{x}^{3}\varphi ^{I,0}\|_{L ^{2}}^{2}\|v _{xx}^{I,1}\|_{L ^{2}}+\|\partial _{x}^{3}\tilde{\varphi}\|_{L ^{2}}\|v ^{I,0}\|_{L ^{\infty}}\|v _{xx}^{I,1}\|_{L ^{2}}
     \nonumber \\
     &~\displaystyle \quad +2 \|v _{x}^{I,0}\|_{L ^{\infty}} \|\tilde{\varphi}_{xx}\|_{L ^{2}}\|v _{xx}^{I,1}\|_{L ^{2}}+
     \|v _{x}^{I,0}\|_{L ^{\infty}}\|\tilde{\varphi}_{x}\|_{L ^{2}}\|v _{x}^{I,1}\|_{L ^{2}}+\|\partial _{x}^{2}f _{2}\|_{L ^{2}}\|v _{xx}^{I,1}\|_{L ^{2}}
      \nonumber \\
      & ~ \displaystyle\leq c(v _{\ast}, T) \|v _{xx}^{I,1}\|_{L ^{2}}^{2}+c(v _{\ast}, T)\left( 1+\|\tilde{\varphi}_{xx}\|_{L ^{2}} ^{2}+ \|\partial _{x}^{3}\tilde{\varphi}\|_{L ^{2}}^{2}+\|\partial _{x}^{2}f _{2}\|_{L ^{2}} ^{2} \right),
 \end{align}
 where we have used \eqref{con-vfi-v-I-0-regula}, \eqref{con-vfi-I-1-ESTI1}, \eqref{tld-v-t-all}, Proposition \ref{prop-embeding-spacetime} and the Cauchy-Schwarz inequality. Integrating \eqref{v-xx-I-1-diff} over $ (0,t) $ gives
 \begin{align*}
 \displaystyle  \int _{\mathcal{I} }\vert v _{xx} ^{I,1}\vert ^{2} (\cdot,t)\mathrm{d}x+ \int _{0}^{t}\int _{\mathcal{I} }(\varphi _{x}^{I,0}+M)\vert v _{xx} ^{I,1}\vert ^{2}  \mathrm{d}x \mathrm{d}\tau  \leq c(v _{\ast}, T) \int _{0}^{t}\int _{\mathcal{I} }\left( \vert v _{xx} ^{I,1}\vert ^{2}+ \vert \partial _{x}^{3}\tilde{\varphi}\vert^{2} \right) \mathrm{d}x \mathrm{d}\tau +c(v _{\ast}, T),
 \end{align*}
 where we have used \eqref{f-2-esti} and \eqref{con-v-i-1-x}. This combined with \eqref{esti-hti-v-xxx} yields that
 \begin{align*}
 \displaystyle  \int _{\mathcal{I} }\vert v _{xx} ^{I,1}\vert ^{2} (\cdot,t)\mathrm{d}x+ \int _{0}^{t}\int _{\mathcal{I} }\left( (\varphi _{x}^{I,0}+M)\vert v _{xx} ^{I,1}\vert ^{2}+ \vert \partial _{x}^{3}\tilde{\varphi}\vert^{2} \right) \mathrm{d}x \mathrm{d}\tau \leq c(v _{\ast}, T)\int _{0}^{t}\int _{\mathcal{I} }\vert v _{xx} ^{I,1}\vert ^{2} \mathrm{d}x \mathrm{d}\tau +c(v _{\ast}, T).
 \end{align*}
 Applying the Gronwall inequality to the above inequality, we have
 \begin{align}\label{con-v-xx-I-1}
 \displaystyle  \int _{\mathcal{I} }\vert v _{xx} ^{I,1}\vert ^{2} (\cdot,t)\mathrm{d}x+ \int _{0}^{t}\int _{\mathcal{I} }\vert \partial _{x}^{3}\tilde{\varphi}\vert^{2}  \mathrm{d}x \mathrm{d}\tau  \leq c(v _{\ast}, T),
 \end{align}
 where $ \varphi _{x}^{I,0}+M \geq 0 $ from \eqref{con-vfi-v-I-0-regula} has been used. Similar to the proof of \eqref{con-v-xx-I-1}, we can derive that
 \begin{align}\label{con-v-xxx-I-1}
 \displaystyle  \int _{\mathcal{I} }\vert \partial _{x}^{3}  v  ^{I,1}\vert ^{2} (\cdot,t)\mathrm{d}x+ \int _{0}^{t}\int _{\mathcal{I} }\left\vert \partial _{x}^{4}\tilde{\varphi}\right\vert ^{2}\mathrm{d}x \mathrm{d}\tau  \leq c(v _{\ast}, T),
 \end{align}
 where we have used $ \|f _{2}\|_{L _{T}^{2}H ^{3}} \leq c(v _{\ast}, T) $ by \eqref{f-2-esti}. Furthermore, by \eqref{con-vfi-v-I-0-regula}, \eqref{f-2-esti}, \eqref{CON-vfi-I-1-esti0}, \eqref{con-vfi-I-1-ESTI1}, \eqref{con-v-i-1-x}, \eqref{tld-v-t-all}, \eqref{con-v-xx-I-1} and \eqref{con-v-xxx-I-1}, we deduce from $ \eqref{eq-vfi-I-1-refor}_{2} $ that $ \|\partial _{t}^{k} v ^{I,1}\|_{L _{T} ^{2}H ^{5-2k}}\leq c(v _{\ast}, T) $ for $ k=1,2 $. Hence, we conclude for the problem \eqref{eq-vfi-I-1-refor} that
\begin{gather}
\displaystyle  \|\partial _{t}^{k}\varphi^{I,1}\|_{L _{T}^{2}H ^{4-2k}}\leq c(v _{\ast}, T) \ \  \mbox{for } k=0,1,2,\label{summarize-vfi-I-1}\\
\displaystyle \| v ^{I,1}\|_{ L _{T}^{\infty}H ^{3}} \leq c(v _{\ast}, T) , \ \ \ \|\partial _{t}^{k}v ^{I,1}\|_{L _{T}^{2}H ^{5-2k}} \leq c(v _{\ast}, T) \ \ \mbox{for}\ k=1,2,\label{summarize-v-I1-1}
\end{gather}
provided $\| \partial _{t}^{k} f _{1} \|_{L _{T}^{2}H ^{2-2k}} \leq c(v _{\ast}, T) $, $ \|\partial _{t}^{k}f _{2} \|_{L _{T}^{2}H ^{3-2k}} \leq c(v _{\ast}, T) $ for $ k=0,1$. With \eqref{con-vfi-v-I-0-regula}, \eqref{f-1-f2-esti}, \eqref{summarize-vfi-I-1}, \eqref{summarize-v-I1-1} and Proposition \ref{prop-embeding-spacetime}, we can update the estimates in \eqref{hti-f-esti-first-1} and \eqref{hti-f-esti-first-2}, respectively, for $ \tilde{f}_{1} $ and $ \tilde{f}_{2} $ as
\begin{align*}
\displaystyle  \|\partial _{t}^{k} \tilde{f} _{1}\|_{L _{T}^{2}H ^{2-2k}} \leq c(v _{\ast}, T), \ \, \|\partial _{t}^{k}\tilde{f} _{2}\|_{L _{T}^{2}H ^{3-2k}} \leq c(v _{\ast}, T)
,\ \ \ k=0,1.
\end{align*}
Indeed, it holds for $ \tilde{f}_{1} $ that
\begin{align*}
 \displaystyle  \|\tilde{f}_{1}\|_{L _{T}^{2}H ^{2}}^{2}& \leq c_{0}\int _{0}^{T} \|\varphi _{xt}^{I,0}\|_{H ^{2}} ^{2} \|v _{x}^{I,1}\|_{H ^{2}}^{2}\mathrm{d}t
 + c_{0} \int _{0}^{T} \|v _{xt}^{I,0}\|_{H^{2}}^{2}\|\tilde{\varphi}_{x}\|_{H ^{2}}^{2}\mathrm{d}t
 + c_{0}\int _{0}^{T}\|\partial _{t}f _{1}(x,t)\|_{H^{2}}^{2}\mathrm{d}t
  \nonumber \\
  & \displaystyle \leq c_{0} \|\varphi _{t}^{I,0}\|_{L _{T} ^{\infty}H ^{3}}^{2} \int _{0}^{T}\|v ^{I,1}\|_{H ^{3}}^{2}\mathrm{d}t
  +c_{0}\|v _{t}^{I,0}\| _{L _{T} ^{\infty}H ^{3}}^{2}\int _{0}^{T} \|\tilde{\varphi}\|_{H ^{3}}^{2}\mathrm{d}t
   +c(v _{\ast}, T)
   \nonumber \\
   & \displaystyle \leq c(v _{\ast}, T),
 \end{align*}
 and
 \begin{align*}
 \displaystyle   \|\partial _{t}\tilde{f}_{1}\|_{L _{T}^{2}L ^{2}}^{2}
 & \displaystyle \leq c _{0}\int _{0}^{T} \left(
 \|\varphi _{ttx}^{I,0}\|_{L ^{\infty}}^{2}\|v _{x}^{I,1}\|_{L ^{2}}^{2}
 + \|\varphi _{xt}^{I,0}\|_{L ^{\infty}} ^{2} \|v _{xt}^{I,1}\|_{L ^{2}}^{2}
 \right) \mathrm{d}t
 + c_{0}\int _{0}^{T} \|\partial _{t}f _{1}\|_{L ^{2}}^{2}\mathrm{d}t
  \nonumber \\
  & \displaystyle \quad +c_{0}\int _{0}^{T}
  \left( \|v _{ttx}^{I,0}\|_{L ^{\infty}}^{2}\|\tilde{\varphi}_{x}\|_{L ^{2}}^{2}
  +\|v _{xt}^{I,0}\|_{L ^{\infty}}^{2}\|\tilde{\varphi}_{xt}\|_{L ^{2}}^{2}
  \right) \mathrm{d}t
   \nonumber \\
   & \displaystyle \leq c_{0}\int _{0}^{T}\left( \|v _{x}^{I,1}\|_{L ^{2}}^{2}
   +\|v _{xt}^{I,1}\|_{L ^{2}}^{2}+\|\tilde{\varphi}_{x}\|_{L ^{2}}^{2}
   +\|\tilde{\varphi}_{x t}\|_{L ^{2}}^{2}
   \right)\mathrm{d}t +c(v _{\ast}, T)
   \leq c(v _{\ast}, T),
 \end{align*}
For $ \tilde{f}_{2} $, it follows that
 \begin{align*}
 \displaystyle  \|\tilde{f}_{2}\|_{L _{T}^{2}H ^{3}}^{2}
 & \displaystyle \leq c_{0} \int _{0}^{T} \|\varphi _{x}^{I,0}\|_{H ^{3}}^{2}\|v ^{I,1}\|_{H ^{3}}^{2}\mathrm{d}t +c_{0} \int _{0}^{T}\|\tilde{\varphi}_{x}\|_{H ^{3}}^{2}\|v _{t}^{I,0}\|_{H ^{3}}^{2}\mathrm{d}t
  +c _{0} \int _{0}^{T}\|\partial _{t}f _{2}\|_{H^{3}}^{2}\mathrm{d}t
  \nonumber \\
  & \displaystyle \leq c_{0} \|\varphi ^{I,0}\| _{L _{T}^{\infty} H ^{4}}^{2}
   \int _{0}^{T}\|v ^{I,1}\|_{H ^{3}}^{2  }\mathrm{d}t
  +c_{0} \int _{0}^{T}\|\partial _{t}f _{2}\|_{H^{3}}^{2}\mathrm{d}t
   \nonumber \\
   & \displaystyle \quad+ c_{0} \|v _{t}^{I,0}\|_{L _{T}^{\infty}H ^{3}}^{2}
   \int _{0}^{T}\|\tilde{\varphi}\|_{H ^{4}}^{2}\mathrm{d}t
    \leq c(v _{\ast}, T),
 \end{align*}
 and
 \begin{align*}
 \displaystyle \|\partial _{t}\tilde{f}_{2}\|_{L _{T}^{2}H ^{1}}^{2}& \displaystyle \leq c_{0} \int _{0}^{T}\left( \|\varphi _{x tt}^{I,0}\|_{H ^{1}}^{2} \|v ^{I,1}\|_{H ^{1}}^{2}+ \|\varphi _{x t }^{I,0}\|_{H ^{1}}^{2} \|v _{t}^{I,1}\|_{H ^{1}}^{2} \right)\mathrm{d}t+c _{0} \int _{0}^{T}\|\partial _{t}^{2}f _{2}\|_{H ^{1}}^{2}\mathrm{d}t
  \nonumber \\
  & \displaystyle \quad + c_{0} \int _{0}^{T}\left(\left\|\tilde{\varphi}_{x t}\right\| _{ H ^{1}}^{2}\|v _{t}^{I,0}\|_{H ^{1}}^{2}+\|\tilde{\varphi}_{x}\|_{H ^{1}}^{2}\|v _{tt}^{I,0}\|_{H ^{1}}^{2}  \right) \mathrm{d}t
   \nonumber \\
   & \displaystyle \leq  c_{0}\int _{0}^{T}\left( \|v ^{I,1}\|_{H ^{1}}^{2}+ \|v _{t} ^{I,1}\|_{H ^{1}}^{2}
   +\left\|\tilde{\varphi}_{x t}\right\| _{ H ^{1}}^{2}+\|\tilde{\varphi}_{x}\|_{H ^{1}}^{2}
   \right)\mathrm{d}t+c(v _{\ast}, T) \leq c(v _{\ast}, T).
 \end{align*}
 Here we have used the Sobolev inequality  $ \|fg\|_{H ^{k}(\mathbb{R}_{+})} \leq C _{k}\|f\|_{H ^{k}(\mathbb{R}_{+})} \|g\|_{H ^{k}(\mathbb{R}_{+})}$ for any integer $ k \geq 1 $. On the other hand, it can be verified that the initial value of the problem \eqref{v-I-1-AUX-EQ} is compatible up to order one. Therefore, by the similar arguments as proving \eqref{summarize-vfi-I-1} and \eqref{summarize-v-I1-1}, we have for the problem \eqref{v-I-1-AUX-EQ} that
 \begin{gather}
\displaystyle  \|\partial _{t}^{k}\psi\|_{L _{T}^{2}H ^{4-2k}} \leq c(v _{\ast}, T) \  \mbox{ for } k=0,1,2,\label{summarize-vfi-I-2}\\
\displaystyle \|w\|_{L _{T}^{\infty}H ^{3}}+ \|\partial _{t}^{k}w \|_{L _{T}^{2}H ^{5-2k}}\leq c(v _{\ast}, T) \ \ \mbox{for}\ k=1,2.\label{summarize-v-I1-2}
\end{gather}
Collecting estimates \eqref{summarize-vfi-I-1}, \eqref{summarize-v-I1-1}, \eqref{summarize-vfi-I-2} and \eqref{summarize-v-I1-2}, making also use of \eqref{con-vfi-v-I-0-regula} and \eqref{f-1-f2-esti}, one can deduce
\begin{align*}
\displaystyle \|\varphi ^{I,1}\|_{L _{T}^{2} H ^{6}} + \|v ^{I,1}\|_{L  _{T}^{\infty} H ^{5}}+ \|\partial _{t}^{k}v ^{I,1}\| _{L _{T}^{2}H ^{7-2k}} \leq c(v _{\ast}, T)  \ \mbox{ for }\   k=1,2,3,
\end{align*}
and ultimately obtain \eqref{vfi-I-1-V-i-1}. The proof of Lemma \ref{lemf-vfi-V-I-1} is complete.
\end{proof}

With Lemmas \ref{lem-regul-outer-layer-0}--\ref{lemf-vfi-V-I-1} at hand, we proceed to study the problems \eqref{second-bd-eq} and \eqref{sec-bd-eq-rt}.
\begin{lemma}\label{lem-v-B-1}
Assume the conditions in Lemmas \ref{lem-regul-outer-layer-0}, \ref{lem-v-B-0} and \ref{lemf-vfi-V-I-1} hold. Then the problem \eqref{second-bd-eq} admits a unique solution $ (v ^{B,1}, \varphi ^{B,2}) $ on $ [0,T] $ for any $ T \in (0,\infty) $ which satisfies, for any $ l \in \mathbb{N} $,
\begin{gather}\label{con-vfi-v-B-2}
\|\langle z \rangle ^{l}\partial _{t}^{k}v ^{B,1}\|_{L _{T}^{2}H _{z}^{6-2k}} +  \|\langle z \rangle ^{l}\partial _{t}^{j}\varphi ^{B,2}\|_{L _{T}^{2}H _{z}^{6-2j}} \leq c(v _{\ast}, T),
\end{gather}
 where $ k=0,1,2,3$, and $ j=0,1,2 $.
\end{lemma}
\begin{proof}
 From $  \eqref{second-bd-eq} _{1} $, we have
\begin{align}\label{vfi-B-2-formul-foruse}
\displaystyle \varphi _{z}^{B,2}&=  - {\mathop{\mathrm{e}}}^{v ^{B,0}}\int _{z}^{\infty}v _{y}^{B,1}(\varphi _{x} ^{I,0}(0,t)+M+\varphi _{y}^{B,1})   {\mathop{\mathrm{e}}}^{-v ^{B,0}} \mathrm{d}y
 \nonumber \\
 & \displaystyle \quad- {\mathop{\mathrm{e}}}^{v ^{B,0}}\int _{z}^{\infty}\left[  v _{y}^{B,0}(\varphi _{xx}^{I,0}(0,t)y+\varphi _{x}^{I,1}(0,t))+\varphi _{y}^{B,1}v _{x} ^{I,0}(0,t) \right]   {\mathop{\mathrm{e}}}^{-v ^{B,0}} \mathrm{d}y
  \nonumber \\
   & \displaystyle =- {\mathop{\mathrm{e}}}^{v ^{B,0}}\int _{z}^{\infty}\left[  v _{y}^{B,0}(\varphi _{xx}^{I,0}(0,t) y+\varphi _{x}^{I,1}(0,t))+\varphi _{y}^{B,1}v _{x}^{I,0}(0,t) \right]   {\mathop{\mathrm{e}}}^{-v ^{B,0}} \mathrm{d}y
    \nonumber \\
    & \displaystyle \quad+{\mathop{\mathrm{e}}}^{v ^{B,0}}\int _{z}^{\infty}v ^{B,1}\partial _{y}\left[(\varphi _{x}^{I,0}(0,t)+M+\varphi _{y}^{B,1}) {\mathop{\mathrm{e}}}^{-v ^{B,0}}\right] \mathrm{d}y
  \nonumber \\
  & \displaystyle \quad +v ^{B,1}(\varphi _{x}^{I,0}(0,t)+M+\varphi _{z}^{B,1}) .
\end{align}
This together with $ \eqref{second-bd-eq}_{2} $ gives
\begin{align}\label{v-B-1-final-eq}
\displaystyle
v _{t}^{B,1}&=v _{zz}^{B,1} - (\varphi _{x} ^{I,0}(0,t)+M)v ^{B,1}-v ^{B,1}(\varphi _{x} ^{I,0}(0,t)+M+\varphi _{z}^{B,1})(v ^{I,0}(0,t)+v ^{B,0})
 \nonumber \\
 & \displaystyle \quad- {\mathop{\mathrm{e}}}^{v ^{B,0}}\int _{z}^{\infty}v ^{B,1}\partial _{y}\left[(\varphi _{x}^{I,0}(0,t)+M+\varphi _{y}^{B,1}) {\mathop{\mathrm{e}}}^{-v ^{B,0}}\right] \mathrm{d}y (v ^{I,0}(0,t)+v ^{B,0})
  \nonumber \\
  &\displaystyle \quad+{\mathop{\mathrm{e}}}^{v ^{B,0}}\int _{z}^{\infty}\left[  v _{y}^{B,0}(\varphi _{xx} ^{I,0}(0,t)y+\varphi _{x}^{I,1}(0,t))+\varphi _{y}^{B,1}v _{x} ^{I,0}(0,t) \right]   {\mathop{\mathrm{e}}}^{-v ^{B,0}} \mathrm{d}y (v ^{I,0}(0,t)+v ^{B,0})
   \nonumber \\
   & \displaystyle \quad- \varphi _{z}^{B,1}(v _{x}^{I,0}(0,t)z+v ^{I,1}(0,t)+v ^{B,1})-(\varphi _{xx}^{I,0}(0,t)z+\varphi _{x}^{I,1}(0,t))v ^{B,0} .
\end{align}
Take
\begin{gather*}
\displaystyle  \tilde{v}= v ^{B,1}+ \eta(z)v ^{I,1}(0,t)
\end{gather*}
with $ \eta(z) $ as in \eqref{eta-defi}. Then we deduce from $ \eqref{second-bd-eq} _{3} $, $ \eqref{second-bd-eq} _{4} $ and \eqref{v-B-1-final-eq} that $ \tilde{v} $ solves
\begin{align}\label{pro-v-b-1-refor}
\begin{cases}
   \displaystyle \tilde{v} _{t} = \tilde{v}_{zz}- (\varphi _{x}^{I,0}(0,t)+M)		\tilde{v}-\tilde{v}(\varphi _{x}^{I,0}(0,t)+M+\varphi _{z}^{B,1})(v ^{I,0}(0,t)+v ^{B,0})- \tilde{v}\varphi _{z}^{B,1}
  \\[1mm]
  \displaystyle \quad \quad - {\mathop{\mathrm{e}}}^{v ^{B,0}}\int _{z}^{\infty}\tilde{v}\partial _{y}\left[(\varphi _{x}^{I,0}(0,t)+M+\varphi _{y}^{B,1}) {\mathop{\mathrm{e}}}^{-v ^{B,0}}\right] \mathrm{d}y (v ^{I,0}(0,t)+v ^{B,0}) +g,\\
 \displaystyle \tilde{v}(0,t)=0,\ \ \tilde{v}(+\infty,t)=0,\\
 \displaystyle \tilde{v}(z,0)=0,
 \end{cases}
\end{align}
where $ g $ is given by
\begin{align}\label{g-definition}
\displaystyle
  g
  &\displaystyle = {\mathop{\mathrm{e}}}^{v ^{B,0}}\int _{z}^{\infty} \eta v ^{I,1}(0,t)\partial _{y}\left[(\varphi _{x}^{I,0}(0,t)+M+\varphi _{y}^{B,1}) {\mathop{\mathrm{e}}}^{-v ^{B,0}}\right] \mathrm{d}y (v ^{I,0}(0,t)+v ^{B,0})
   \nonumber \\
  &\displaystyle \quad+{\mathop{\mathrm{e}}}^{v ^{B,0}}\int _{z}^{\infty}\left[  v _{y}^{B,0}(\varphi _{xx}^{I,0}(0,t)y+\varphi _{x}^{I,1}(0,t))+\varphi _{y}^{B,1}v _{x}^{I,0}(0,t) \right]   {\mathop{\mathrm{e}}}^{-v ^{B,0}} \mathrm{d}y (v ^{I,0}(0,t)+v ^{B,0})
   \nonumber \\
   & \displaystyle \quad
 + (\varphi _{x} ^{I,0}(0,t)+M) \eta(z)v ^{I,1}(0,t)+\eta(z)v _{t}^{I,1}(0,t)- \varphi _{z}^{B,1}(v _{x}^{I,0}(0,t)z+v ^{I,1}(0,t))
  \nonumber \\
  & \displaystyle \quad -(\varphi _{xx}^{I,0}(0,t)z+\varphi _{x}^{I,1}(0,t))v ^{B,0} + \eta(z)v ^{I,1}(0,t) (\varphi _{x}^{I,0}(0,t)+M+\varphi _{z}^{B,1})(v ^{I,0}(0,t)+v ^{B,0})
   \nonumber \\
   & \displaystyle \quad - \eta''(z)v ^{I,1}(0,t)+\eta(z)v ^{I,1}(0,t)\varphi _{z}^{B,1}.
\end{align}
The existence of solutions to the problem \eqref{pro-v-b-1-refor} can be proved by using the reflection method along with the fixed point theorem. Since the argument is similar to that in Appendix A for the linearized problem of \eqref{first-bd-layer-pro}, we omit the details here. In the following, we are devoted to deriving some weighted estimates for the solution. It can be verified that the initial datum for the problem \eqref{pro-v-b-1-refor} is compatible up to order two. That is, if we define $ \partial _{t}^{k}\tilde{v}\vert _{t=0}~(k=1,2) $ through the first equation in \eqref{pro-v-b-1-refor}, then $ \partial _{t}^{k}\tilde{v}\vert _{t=0}~(k=0,1,2) $ vanish at the boundary. Furthermore, we have for $ k=0,1,2 $ and $ l \in \mathbb{N} $ that
\begin{align}\label{g-esti}
\displaystyle  \|\langle z \rangle ^{l} \partial _{t}^{k}g\|_{L _{T} ^{2}H _{z}^{4-2k}} \leq c(v _{\ast}, T).
\end{align}
The proof of \eqref{g-esti} will be detailed in Appendix B. We proceed to prove for $ m=1,2,3 $ and $ l \in \mathbb{N} $ that
\begin{align}\label{state-htild-v}
\displaystyle  \|\langle z \rangle ^{l}\partial _{t}^{k}\tilde{v} \| _{ L _{T}^{2} H _{z}^{2m-2k}} \leq c(v _{\ast}, T)\ \ \text{for} \ k=0,1, \cdots,m.
\end{align}
Indeed, for the case $ m=1 $, multiplying the first equation in \eqref{pro-v-b-1-refor} by $ \langle z \rangle  ^{2l}\tilde{v} $ followed by an integration over $ \mathbb{R}_{+} $, we have
\begin{align}\label{tid-v-first-esti}
&\displaystyle \frac{1}{2}\frac{\mathrm{d}}{\mathrm{d}t}\int _{\mathbb{R}_{+}}\langle z \rangle^{2l} \tilde {v}^{2} \mathrm{d}z+\int _{\mathbb{R}_{+}}\langle z \rangle^{2l}\tilde{v}_{z}^{2} \mathrm{d}z + \int _{\mathbb{R}_{+}}\langle z \rangle^{2l}(\varphi _{x}^{I,0}(0,t)+M)(v ^{I,0}(0,t)+v ^{B,0})  \tilde{v} ^{2} \mathrm{d}z
 \nonumber \\
 & ~\displaystyle =-\int _{\mathbb{R}_{+}}\varphi _{z}^{B,1}(v ^{I,0}(0,t)+v ^{B,0}+1)\langle z \rangle^{2l}\tilde{v}^{2} \mathrm{d}z +\int _{\mathbb{R}_{+}}g \langle z \rangle^{2l}\tilde{v} \mathrm{d}z- 2l \int _{\mathbb{R}_{+}}\langle z \rangle^{2l-2}z \tilde{v} \tilde{v}_{z} \mathrm{d}z
  \nonumber \\
  &~\displaystyle \quad- \int _{\mathbb{R}_{+}} {\mathop{\mathrm{e}}}^{v ^{B,0}}\int _{z}^{\infty}\tilde{v}\partial _{y}\left[(\varphi _{x} ^{I,0}(0,t)+M+\varphi _{y}^{B,1}) {\mathop{\mathrm{e}}}^{-v ^{B,0}}\right] \mathrm{d}y (v ^{I,0}(0,t)+v ^{B,0})\langle z \rangle^{2l}\tilde{v}\mathrm{d}z,
\end{align}
where by \eqref{con-vfi-v-I-0-regula}--\eqref{con-vfi-B-1}, the Sobolev inequality $ \|f\|_{L _{z}^\infty} \leq c _{0} \|f\|_{H _{z}^{1}} $ and the Cauchy-Schwarz inequality, it holds that
\begin{align*}
\displaystyle \int _{\mathbb{R}_{+}}\varphi _{z}^{B,1}(v ^{I,0}(0,t)+v ^{B,0}+1)\langle z \rangle^{2l}\tilde{v}^{2} \mathrm{d}z
 &\leq c_{0}\left\|\varphi _{z}^{B,1}\right\|_{L _{z}^{\infty}}
 \left( \|v ^{I,0}\|_{L ^{\infty}}+\|v ^{B,0}\|_{L _{z}^{\infty}}+1
  \right) \int _{\mathbb{R}_{+} }\langle z \rangle ^{2 l}\tilde{v}^{2} \mathrm{d}z
 \nonumber \\
 & \displaystyle \leq c(v _{\ast}, T)  \left\|\varphi _{z}^{B,1}\right\|_{L _{z}^{\infty}}\int _{\mathbb{R}_{+} }\langle z \rangle ^{2 l}\tilde{v}^{2} \mathrm{d}z
 \leq c(v _{\ast}, T)\int _{\mathbb{R}_{+} }\langle z \rangle ^{2 l}\tilde{v}^{2} \mathrm{d}z,
 \\
\displaystyle  \int _{\mathbb{R}_{+}}g \langle z \rangle^{2l}\tilde{v} \mathrm{d}z- 2l \int _{\mathbb{R}_{+}}\langle z \rangle^{2l-2}z \tilde{v} \tilde{v}_{z} \mathrm{d}z
& \leq \frac{1}{8}\int _{\mathbb{R}_{+}}\langle z \rangle ^{2l}\tilde{v}_{z} ^{2} \mathrm{d}z
+ c_{0}\int _{\mathbb{R}_{+} } \langle z \rangle ^{2l} \tilde{v}^{2}\mathrm{d}z
+c_{0}\int _{\mathbb{R}_{+}}\langle z \rangle ^{2l}g ^{2} \mathrm{d}z
\end{align*}
and
\begin{align*}
&\displaystyle  - \int _{\mathbb{R}_{+}} {\mathop{\mathrm{e}}}^{v ^{B,0}}\int _{z}^{\infty}\tilde{v}\partial _{y}\left[(\varphi _{x}^{I,0}(0,t)+M+\varphi _{y}^{B,1}) {\mathop{\mathrm{e}}}^{-v ^{B,0}}\right] \mathrm{d}y (v ^{I,0}(0,t)+v ^{B,0})\langle z \rangle^{2l}\tilde{v}\mathrm{d}z
 \nonumber \\
 &~\displaystyle \leq
 c(v _{\ast}, T) \left( \vert v ^{I,0}(0,t)\vert+\|v ^{B,0}\|_{L _{z}^{\infty}} \right)
 \int _{\mathbb{R}_{+}}\langle z \rangle ^{2l}\left\vert \int _{z}^{\infty}\vert \tilde{v}\vert \Big( \vert v _{y}^{B,0}\vert + \vert \varphi _{yy}^{B,1}\vert+\vert \varphi _{y}^{B,1}\vert \vert v _{y}^{B,0}\vert \Big) \mathrm{d}y\right\vert ^{2} \mathrm{d}z
  \nonumber \\
  &~\displaystyle \quad+ c(v _{\ast}, T)\|\langle z \rangle ^{l}\tilde{v}\| _{L ^{2} }^{2}
   \nonumber \\
     & ~\displaystyle \leq c(v _{\ast}, T)\|\langle z \rangle ^{l}\tilde{v}\| _{L _{z}^{2} }^{2}+c(v _{\ast}, T) \|\langle z \rangle ^{l}\tilde{v}\| _{L _{z}^{2} }^{2} \int _{\mathbb{R}_{+}}\langle z \rangle ^{2l} \int _{z}^{\infty}\langle y \rangle ^{-2l} \left( \vert v _{y}^{B,0}\vert + \vert \varphi _{yy}^{B,1}\vert+\vert \varphi _{y}^{B,1}\vert\vert v _{y}^{B,0}\vert\right) ^{2}\mathrm{d}y  \mathrm{d}z
   \nonumber \\
   &~\displaystyle \leq c(v _{\ast}, T)\|\langle z \rangle ^{l}\tilde{v}\| _{L _{z}^{2} }^{2}+c(v _{\ast}, T) \|\langle z \rangle ^{l}\tilde{v}\| _{L _{z}^{2} }^{2}  \left( \|\langle z \rangle ^{l}  v _{z}^{B,0}\|_{L _{z}^{2}}^{2}+\|\langle z \rangle ^{l}\varphi _{zz}^{B,1}\|_{L _{z}^{2}}^{2}+\|\varphi _{z}^{B,1}\|_{L _{z}^{\infty}}^{2} \|\langle z \rangle ^{l}v _{z}^{B,0}\|_{L _{z}^{2}}^{2} \right)
    \nonumber \\
    &~\displaystyle \leq c(v _{\ast}, T) \|\langle z \rangle ^{l}\tilde{v}\| _{L _{z}^{2} }^{2}.
\end{align*}
Therefore we update \eqref{tid-v-first-esti} as
\begin{align*}
\displaystyle  &\displaystyle \frac{1}{2}\frac{\mathrm{d}}{\mathrm{d}t}\int _{\mathbb{R}_{+}}\langle z \rangle^{2l} \tilde {v}^{2} \mathrm{d}z+\frac{1}{2}\int _{\mathbb{R}_{+}}\langle z \rangle^{2l}\tilde{v}_{z}^{2} \mathrm{d}z   \leq c(v _{\ast}, T) \|\langle z \rangle ^{l}\tilde{v}\|_{L _{z}^{2}}^{2}+c(v _{\ast}, T) \int _{\mathbb{R}_{+}}\langle z \rangle ^{2l}g ^{2} \mathrm{d}z,
\end{align*}
where we have used \eqref{con-vfi-v-I-0-regula} and \eqref{v-I-0-positive}. This along with \eqref{g-esti} and the Gronwall inequality yields that
\begin{align}\label{con-tild-v-0}
\displaystyle  \|\langle z \rangle ^{l}\tilde{v}(\cdot,t)\|_{L _{z}^{2}}^{2} + \int _{0}^{t} \|\langle z \rangle ^{l}\tilde{v}_{z}\|_{L _{z}^{2}}^{2} \mathrm{d}\tau \leq c(v _{\ast}, T)
\end{align}
for any $ t \in (0,T] $. Multiplying the first equation in \eqref{pro-v-b-1-refor} by $ \langle z \rangle  ^{2l}\tilde{v}_{t} $ and then integrating the resulting equation over $ \mathbb{R}_{+} $, we get
\begin{align}\label{tild-v-z-esti}
&\displaystyle  \frac{1}{2}\frac{\mathrm{d}}{\mathrm{d}t}\left\{ \int _{\mathbb{R}_{+}}\langle z \rangle ^{2l}\tilde{v}_{z}^{2}\mathrm{d}z+ \int _{\mathbb{R}_{+}}\langle z \rangle^{2l}(\varphi _{x} ^{I,0}(0,t)+M)(v ^{I,0}(0,t)+v ^{B,0})  \tilde{v} ^{2} \mathrm{d}z \right\}+\int _{\mathbb{R}_{+}}\langle z \rangle ^{2l}\tilde{v}_{t}^{2}\mathrm{d}z
 \nonumber \\
 &~\displaystyle =\frac{1}{2}\int _{\mathbb{R}_{+}}\langle z \rangle^{2l}\left[ (\varphi _{x}^{I,0}(0,t)+M)(v ^{I,0}(0,t)+v ^{B,0}) \right]_{t}   \tilde{v} ^{2} \mathrm{d}z
  \nonumber \\
  &~\displaystyle \quad  -\int _{\mathbb{R}_{+}}\varphi _{z}^{B,1}(v ^{I,0}(0,t)+v ^{B,0}+1)\langle z \rangle^{2l}\tilde{v} \tilde{v}_{t} \mathrm{d}z +\int _{\mathbb{R}_{+}}g \langle z \rangle^{2l}\tilde{v} _{t}\mathrm{d}z- 2l \int _{\mathbb{R}_{+}}\langle z \rangle^{2l-2}z \tilde{v}_{t} \tilde{v}_{z} \mathrm{d}z
   \nonumber \\
     &~\displaystyle \quad- \int _{\mathbb{R}_{+}} {\mathop{\mathrm{e}}}^{v ^{B,0}}\int _{z}^{\infty}\tilde{v}\partial _{y}\left[(\varphi _{x}^{I,0}(0,t)+M+\varphi _{y}^{B,1}) {\mathop{\mathrm{e}}}^{-v ^{B,0}}\right] \mathrm{d}y (v ^{I,0}(0,t)+v ^{B,0})\langle z \rangle^{2l}\tilde{v} _{t}\mathrm{d}z
      \nonumber \\
      &~ \displaystyle= \sum _{i=1}^{5}\mathcal{I}_{i}.
     \end{align}
     We now estimate $ \mathcal{I}_{i}~(1 \leq i \leq 5) $ term by term. By \eqref{con-vfi-v-I-0-regula}, \eqref{v-B-0-regularity} and Proposition \ref{prop-embeding-spacetime}, we have
     \begin{align*}
     \displaystyle \mathcal{I}_{1} \leq  c(v _{\ast}, T) \left( \|\varphi _{xt}^{I,0}\|_{L ^{\infty}}+\|v _{t}^{I,0}\| _{L ^{\infty}}+ \|v _{t}^{B,0}\|_{L _{z}^{\infty}} \right) \|\langle z \rangle ^{l}\tilde{v}\|_{L _{z}^{2}}^{2} \leq  c(v _{\ast}, T)\|\langle z \rangle ^{l}\tilde{v}\|_{L _{z}^{2}}^{2}.
     \end{align*}
     Similarly, for $ \mathcal{I}_{2} $, we get
     \begin{align*}
     \displaystyle \mathcal{I}_{2} & \leq   \|\varphi _{z}^{B,1}\|_{L _{z}^{\infty}} \left( \|v ^{I,0}\|_{L ^{\infty}} +\|v ^{B,0}\| _{L _{z}^{\infty}}+1\right) \|\langle z \rangle ^{l}\tilde{v}\|_{L _{z}^{2}}\|\langle z \rangle ^{l}\tilde{v}_{t}\|_{L _{z}^{2}}
     \leq c(v _{\ast}, T)\|\langle z \rangle ^{l}\tilde{v}\|_{L _{z}^{2}}\|\langle z \rangle ^{l}\tilde{v}_{t}\|_{L _{z}^{2}}
      \nonumber \\
      & \displaystyle \leq \frac{1}{8}\|\langle z \rangle ^{l}\tilde{v}_{t}\|_{L _{z}^{2}} ^{2}
      +c(v _{\ast}, T)\|\langle z \rangle ^{l}\tilde{v}\|_{L _{z}^{2}}^{2}.
     \end{align*}
     By the Cauchy-Schwarz inequality, we have
     \begin{align*}
     \displaystyle \mathcal{I}_{3}+\mathcal{I}_{4}& \leq \frac{1}{8}\int _{\mathbb{R}_{+}}\langle z \rangle ^{2l}\tilde{v}_{t}^{2}\mathrm{d}z+c(v _{\ast}, T) \int _{\mathbb{R}_{+}} \langle z \rangle ^{2l} g ^{2}\mathrm{d}z + c(v _{\ast}, T) \int _{\mathbb{R}_{+}} \langle z \rangle ^{2l} \tilde{v}_{z} ^{2}\mathrm{d}z.
     \end{align*}
     Finally, in view of \eqref{con-vfi-v-I-0-regula}, \eqref{v-B-0-regularity}, \eqref{con-vfi-B-1} and the Cauchy-Schwarz inequality, we get for $ \mathcal{I}_{5} $ that
     \begin{align*}
     \displaystyle  \mathcal{I}_{5}& \leq \frac{1}{8}\int _{\mathbb{R}_{+}}\langle z \rangle ^{2l}\tilde{v}_{t}^{2} \mathrm{d}z+c(v _{\ast}, T) \int _{\mathbb{R}_{+}}\langle z \rangle ^{2l} \left\vert \int _{z}^{\infty}\tilde{v}\partial _{y}\left[(\varphi _{x}^{I,0}(0,t)+M+\varphi _{y}^{B,1}) {\mathop{\mathrm{e}}}^{-v ^{B,0}}\right] \mathrm{d}y\right\vert ^{2}\mathrm{d}z
      \nonumber \\
      & \displaystyle \leq \frac{1}{8}\int _{\mathbb{R}_{+}}\langle z \rangle ^{2l}\tilde{v}_{t}^{2} \mathrm{d}z+c(v _{\ast}, T) \|\langle z \rangle ^{l}\tilde{v}\|_{L _{z}^{2}}^{2} \left( \|\langle z \rangle ^{l+2}\varphi _{zz}^{B,1}\|_{L _{z}^{2}}^{2}+\|\langle z \rangle ^{l+2}v _{z}^{B,0}\|_{L _{z}^{2}}^{2} \right) \int _{\mathbb{R}_{+}}\langle z \rangle ^{-2} \mathrm{d}z
       \nonumber \\
       & \displaystyle \leq \frac{1}{8}\int _{\mathbb{R}_{+}}\langle z \rangle ^{2l}\tilde{v}_{t}^{2} \mathrm{d}z+c(v _{\ast}, T)  \|\langle z \rangle ^{l}\tilde{v}\|_{L _{z}^{2}}^{2}.
     \end{align*}
     Collecting estimates for $ \mathcal{I}_{i}~(1 \leq i \leq 5) $, we have from \eqref{tid-v-first-esti} that
     \begin{align*}
     \displaystyle  \int _{\mathbb{R}_{+}}\langle z \rangle ^{2l}  \tilde{v}_{z}^{2}(\cdot,t) \mathrm{d}z+ \int _{0}^{t} \int _{\mathbb{R}_{+}} \langle z \rangle ^{2l}\tilde{v}_{\tau}^{2}\mathrm{d}z \mathrm{d}\tau\leq c(v _{\ast}, T)\int _{0}^{t}\int _{\mathbb{R}_{+}}\langle z \rangle ^{2l} \left( \tilde{v}_{z}^{2}+ \tilde{v}^{2} \right) \mathrm{d}z \mathrm{d}\tau
     \end{align*}
     for any $ t \in (0,T] $, where we have used the facts $ \varphi _{x}^{I,0}(0,t)+M \geq 0 $ and $ v ^{B,0} \geq 0 $. Therefore we utilize \eqref{con-tild-v-0} and the Gronwall inequality to deduce that
\begin{align}\label{con-htild-v}
\displaystyle  \int _{\mathbb{R}_{+}}\langle z \rangle ^{2l}  \tilde{v}_{z}^{2}(\cdot,t) \mathrm{d}z+\int _{0}^{t} \int _{\mathbb{R}_{+}} \langle z \rangle ^{2l}\tilde{v}_{\tau}^{2}\mathrm{d}z \mathrm{d}\tau \leq c(v _{\ast}, T).
\end{align}
This along with  $ \eqref{pro-v-b-1-refor}_{1} $, \eqref{con-vfi-v-I-0-regula}, \eqref{v-B-0-regularity}, \eqref{con-vfi-B-1} and \eqref{g-esti} leads to
\begin{align*}
\displaystyle \int _{0}^{T}\|\langle z \rangle ^{l}\tilde{v}_{zz}\|_{L _{z}^{2}}^{2}\mathrm{d}t \leq c(v _{\ast}, T).
\end{align*}
Then we finish the proof of \eqref{state-htild-v} for $ m=1 $. To proceed, set $ \hat{v}=\tilde{v}_{t} $. Then $ \hat{v} $ satisfies
\begin{align}\label{hat-v-eq}
\displaystyle  \begin{cases}
\displaystyle   \displaystyle \hat{v} _{t} = \hat{v}_{zz}- (\varphi _{x} ^{I,0}(0,t)+M)		\hat{v}- \hat{v}(\varphi _{x}^{I,0}(0,t)+M+\varphi _{z}^{B,1})(v ^{I,0}(0,t)+v ^{B,0})-\hat{v}\varphi _{z}^{B,1} \\[2mm]
\displaystyle \quad \quad
- {\mathop{\mathrm{e}}}^{v ^{B,0}}\int _{z}^{\infty}\hat{v}\partial _{y}\left[(\varphi _{x} ^{I,0}(0,t)+M+\varphi _{y}^{B,1}) {\mathop{\mathrm{e}}}^{-v ^{B,0}}\right] \mathrm{d}y (v ^{I,0}(0,t)+v ^{B,0})+ \tilde{g},
\\[2mm]
      \displaystyle \hat{v}(0,t)=0,\ \ \hat{v}(+\infty,t)=0,\\
 \displaystyle \hat{v}(z,0)=\tilde{v}_{t}\vert _{t=0} ,
\end{cases}
\end{align}
where $ \tilde{v}_{t}\vert _{t=0} $ is defined through the equation $ \eqref{pro-v-b-1-refor}_{1} $, and $ \tilde{g} $ is given by
\begin{align*}
\displaystyle  \tilde{g}&= - \partial _{t}[(\varphi _{x} ^{I,0}(0,t)+M)		]\tilde{v}
 - \tilde{v}(\partial _{t}\varphi _{x} ^{I,0}(0,t)+\varphi _{zt}^{B,1})(v ^{I,0}(0,t)+v ^{B,0})
  \nonumber\\[2mm]
  & \displaystyle \quad  - \tilde{v}(\varphi _{x}^{I,0}(0,t)+M+\varphi _{z}^{B,1})(v _{t}^{I,0}(0,t)+v _{t}^{B,0})  -\tilde{v}\varphi _{zt}^{B,1}+g _{t}\nonumber\\[2mm]
 & \displaystyle \quad   - \int _{z}^{\infty}\tilde{v}\partial _{y}\left[(\varphi _{x}^{I,0}(0,t)+M+\varphi _{y}^{B,1}) {\mathop{\mathrm{e}}}^{-v ^{B,0}}\right] \mathrm{d}y \left[  {\mathop{\mathrm{e}}}^{v ^{B,0}}(v ^{I,0}(0,t)+v ^{B,0}) \right] _{t}\nonumber
\\[2mm]
   & \displaystyle  \quad-{\mathop{\mathrm{e}}}^{v ^{B,0}}(v ^{I,0}(0,t)+v ^{B,0})\int _{z}^{\infty}\tilde{v}\partial _{yt}\left[(\varphi _{x}^{I,0}(0,t)+M+\varphi _{y}^{B,1}) {\mathop{\mathrm{e}}}^{-v ^{B,0}}\right] \mathrm{d}y  .
   \end{align*}
In virtue of \eqref{con-vfi-v-I-0-regula}, \eqref{v-B-0-regularity}, \eqref{con-vfi-B-1}, \eqref{g-esti}, \eqref{con-tild-v-0}, \eqref{con-htild-v} and similar arguments to proving \eqref{g-esti}, it holds for $ k=0,1 $ that
\begin{align}\label{til-g-esti}
\displaystyle  \|\langle z \rangle ^{l}\partial _{t}^{k}\tilde{g}\|_{L _{T}^{2}H _{z}^{2-2k}} \leq c(v _{\ast}, T).
\end{align}
Repeating the argument in the proof of \eqref{con-tild-v-0} and \eqref{con-htild-v}, we then arrive at
\begin{align}\label{con-hat-v}
\displaystyle  \displaystyle  \int _{\mathbb{R}_{+}}\langle z \rangle ^{2l}  \hat{v}_{z}^{2}(\cdot,t) \mathrm{d}z+\int _{0}^{t} \int _{\mathbb{R}_{+}} \langle z \rangle ^{2l} \left( \hat{v}_{z}^{2}+\hat{v}_{\tau}^{2} \right) \mathrm{d}z \mathrm{d}\tau \leq c(v _{\ast}, T)
\end{align}
for any $ t \in(0,T] $. Furthermore, from $ \eqref{hat-v-eq} _{1} $ and \eqref{con-hat-v}, we get
\begin{align}\label{tild-v-tzz}
\displaystyle \int _{0}^{T}\int _{\mathbb{R}_{+}}\langle z \rangle ^{2l}\tilde{v}_{t zz}^{2} \mathrm{d}z \mathrm{d}t \leq c(v _{\ast}, T).
\end{align}
In view of $ \eqref{pro-v-b-1-refor}_{1} $, \eqref{con-tild-v-0}, \eqref{con-htild-v}, $ \eqref{hat-v-eq} _{1} $ and \eqref{tild-v-tzz}, we also have that
\begin{align*}
\displaystyle \int _{0}^{T}\int _{\mathbb{R}_{+}} \langle z \rangle ^{2l} \left( \vert \partial _{z}^{3}\tilde{v}\vert^{2}+ \vert \partial _{z}^{4}\tilde{v}\vert^{2} \right) \mathrm{d}z \mathrm{d}t \leq c(v _{\ast}, T).
\end{align*}
Thus we finish the proof of \eqref{state-htild-v} for $ m=2 $. Now let us turn to the proof of  \eqref{state-htild-v} for the case $ m=3 $. Based on \eqref{til-g-esti} and the fact that the initial datum of the problem \eqref{hat-v-eq} is compatible up to order one, we get apply the procedure in the proof of the cases $ m=1,2 $ to the problem \eqref{hat-v-eq} and get that $ \|\langle z \rangle ^{l}\partial _{t}^{k}\hat{v}\|_{L _{T}^{2} H _{z}^{4-2k}} \leq C  $ for any $ l \in \mathbb{N} $ and $ k=0,1,2 $. That is,
\begin{align*}
\displaystyle  \| \langle z \rangle ^{l}\partial _{t}^{k}\tilde{v}\|_{L _{T}^{2}H _{z}^{6-2k}}  \leq c(v _{\ast}, T) \ \ \mbox{for } k=1,2,3.
\end{align*}
This along with \eqref{con-vfi-v-I-0-regula}, \eqref{v-B-0-regularity}, \eqref{con-vfi-B-1}, $ \eqref{pro-v-b-1-refor}_{1} $ and \eqref{g-esti} further gives that
\begin{align*}
\displaystyle \int _{0}^{T}\left( \|\langle z \rangle ^{l}\partial _{z}^{5}\tilde{v}\|_{L _{z}^{2}}^{2}+\|\langle z \rangle ^{l}\partial _{z}^{6}\tilde{v}\|_{L _{z}^{2}}^{2} \right)\mathrm{d}t\leq  c(v _{\ast},T).
\end{align*}
Then \eqref{state-htild-v} is proved. With the definition of $ \tilde{v} $ and \eqref{state-htild-v}, one can immediately obtain the estimate for $ v ^{B,1} $ in \eqref{con-vfi-v-B-2}. The estimates for $ \varphi ^{B,2} $ follow from \eqref{con-vfi-v-I-0-regula}, \eqref{v-B-0-regularity}, \eqref{con-vfi-B-1}, \eqref{con-vfi-I-1}, \eqref{con-v -I-1}, $ \eqref{con-vfi-v-B-2} _{1} $ and \eqref{vfi-B-2-formul-foruse} along with similar arguments as proving \eqref{g-esti}. We thus finish the proof of Lemma \ref{lem-v-B-1}.
\end{proof}

By analogous arguments as proving Lemma \ref{lem-v-B-1}, we have the following existence and regularity result on $ (\varphi ^{b,2},v ^{b,1}) $.
\begin{lemma}\label{lem-v-b-1}
Assume the conditions in Lemmas \ref{lem-regul-outer-layer-0}, \ref{lem-v-B-0} and \ref{lemf-vfi-V-I-1} hold. Then there exists a unique solution $ (\varphi ^{b,2}, v ^{b,1}) $ to the problem \eqref{sec-bd-eq-rt}  on $ [0,T] $ for any $ T \in (0,\infty) $ such that for any $ l \in \mathbb{N} $,
\begin{gather}\label{con-vfi-b-2-v}
\|\langle \xi \rangle ^{l}\partial _{t}^{k}v ^{b,1}\|_{L _{T}^{2}H _{\xi}^{6-2k}}+ \|\langle \xi \rangle ^{l}\partial _{t}^{j}\varphi ^{b,2}\|_{L _{T}^{2}H _{\xi}^{6-2j}} \leq c(v _{\ast},T),
 \end{gather}
 where $  k=0,1,2,3$, and $ j=0,1,2 $.
\end{lemma}

\vspace{4mm}

\section{convergence of boundary  layers} 
\label{sec:stability_of_boundary_layers}
\subsection{Reformulation of the problem} 
\label{sub:reformulation_of_the_problem}
Denote by $ (\varphi ^{\varepsilon}, v ^{\varepsilon}) $ the solution to problem \eqref{refor-eq}--\eqref{BD-POSITIVE-VE}. To prove Theorem \ref{thm-stabi-refor}, normally we shall construct a perturbation as
\begin{gather}\label{ansaz}
\begin{cases}
  \displaystyle \varphi ^{\varepsilon}=\varphi ^{I,0}+\varepsilon ^{1/2}\left( \varphi ^{I,1}(x,t)+\varphi ^{B,1}(z,t)+\varphi ^{b,1}(\xi,t) \right)+\mathcal{E}_{1}^{\varepsilon} ,\\[2mm]
\displaystyle v ^{\varepsilon}= v ^{I,0}+v ^{B,0}+v ^{b,0}+ \mathcal{E}_{2}^{\varepsilon}
\end{cases}
\end{gather}
and estimate the remainder $ (\mathcal{E}_{1}^{\varepsilon},\mathcal{E}_{2}^{\varepsilon})$ to show that
\begin{align}\label{R-1-R-2}
 \displaystyle \| \mathcal{E} _{1}^{\varepsilon}\|_{L _{T}^{\infty}L ^{\infty}} = O(\varepsilon ^{5/8}),\ \  \|\mathcal{E} _{2}^{\varepsilon}\|_{L _{T}^{\infty}L ^{\infty}} = O(\varepsilon ^{1/2}),\ \
 \displaystyle \|\partial _{x} \mathcal{E}_{1}^{\varepsilon}\|_{L _{T}^{\infty}L ^{\infty}}  =O (\varepsilon ^{1/4}).
 \end{align}
for some $ T>0 $. However, if we substitute \eqref{ansaz} into \eqref{refor-eq}, we shall find that the equations for $ (\mathcal{E} _{1}^{\varepsilon},\mathcal{E} _{2}^{\varepsilon}) $ involves terms that converge to non-zero constants as $ \varepsilon \rightarrow 0 $, but we need estimates in \eqref{R-1-R-2}, where $ \mathcal{E} _{1}^{\varepsilon} $ behaves like $ o(\varepsilon ^{1/2}) $. This gap causes troubles to the analysis. To circumvent this difficulty, we resort to higher-order outer- and boundary layer profiles by introducing an approximate solution to the problem \eqref{refor-eq}-\eqref{BD-POSITIVE-VE} as follows
\begin{subequations}\label{approximate}
\begin{align}
\displaystyle \Phi ^{A}(x,t)&:=\varphi ^{I,0}+\varepsilon ^{1/2}\left( \varphi ^{I,1}(x,t)+\varphi ^{B,1}(z,t)+\varphi ^{b,1}(\xi,t) \right)
 \nonumber \\
 &\quad  +\varepsilon \left( \varphi ^{B,2}(z,t)+\varphi ^{b,2}(\xi,t) \right)
 +b _{\varphi}^{\varepsilon}(x,t), \label{approximate-fida}
  \\
  \displaystyle V ^{A}(x,t)&:=  v ^{I,0}+v ^{B,0}+v ^{b,0}+\varepsilon ^{1/2}\left( v ^{I,1}(x,t)+v ^{B,1}(z,t)+v ^{b,1}(\xi,t) \right)
 +b _{v}^{\varepsilon}(x,t),\label{approximate-va}
\end{align}
\end{subequations}
where the functions $ b _{\varphi}^{\varepsilon}(x,t) $ and $ b _{v}^{\varepsilon}(x,t) $ are constructed below to homogenize the boundary values of $ (\Phi^{A}, V^{A})  $:
\begin{subequations}\label{corre-b}
\begin{align}
\displaystyle \displaystyle\makebox[-0.5pt]{~} b _{\varphi}^{\varepsilon}(x,t)&=-(1-x)\left[ \varepsilon ^{1/2}\varphi ^{b,1}(- \frac{1}{\varepsilon ^{1/2}} ,t)+\varepsilon \varphi ^{b,2}(-  \frac{1}{\varepsilon ^{1/2}},t)+ \varepsilon \varphi^{B,2}(0,t) \right]
 \nonumber \\
  & \displaystyle \quad-x \left[  \varepsilon ^{1/2}\varphi ^{B,1}(\frac{1}{\varepsilon ^{1/2}},t)+\varepsilon \varphi ^{B,2}(\frac{1}{\varepsilon ^{1/2}},t)+ \varepsilon \varphi ^{b,2}(1,t)\right],   \label{b-vfi-ve}
 \\
 \displaystyle \displaystyle\makebox[-0.5pt]{~} b _{v}^{\varepsilon}(x,t)&=(x-1)\left[v ^{b,0}(- \frac{1}{\varepsilon ^{1/2}},t)+ \varepsilon ^{ 1/2}v^{b,1}(-\frac{1}{\varepsilon ^{ 1/2}},t) \right]
  -x \left[v ^{B,0}( \frac{1}{\varepsilon ^{ 1/2}},t)+  \varepsilon ^{1/2}v ^{B,1}( \frac{1}{\varepsilon ^{1/2}},t)\right]. \label{b-v-ve}
\end{align}
\end{subequations}
Then we can write $ ( \varphi ^{\varepsilon},v ^{\varepsilon}) $ as
\begin{align}\label{solution-recover-esti}
\displaystyle \varphi ^{\varepsilon}=\Phi ^{A}+\varepsilon ^{1/2}\Phi ^{\varepsilon},\ \ v ^{\varepsilon}= V ^{A}+\varepsilon ^{1/2}V ^{\varepsilon}
\end{align}
with $ (\Phi ^{\varepsilon},V ^{\varepsilon}) $ being the perturbation functions, which along with \eqref{ansaz} implies that
\begin{align}
\displaystyle \mathcal{E} _{1}^{\varepsilon}&= \varepsilon ^{1/2}\Phi ^{\varepsilon}+\varepsilon \left( \varphi ^{B,2}(z,t)+\varphi ^{b,2}(\xi,t) \right)
 +b _{\varphi}^{\varepsilon}(x,t),\label{R-1-precise}\\
 \displaystyle \mathcal{E} _{2}^{\varepsilon}&=\varepsilon ^{1/2}V ^{\varepsilon}++\varepsilon ^{1/2}\left( v ^{I,1}(x,t)+v ^{B,1}(z,t)+v ^{b,1}(\xi,t) \right)
 +b _{v}^{\varepsilon}(x,t). \label{R-2-pres}
\end{align}
We remark that we have omitted the term $\varepsilon \varphi^{I,2} $ in the above construction of $ \Phi ^{A}$. Indeed, this term is of order $ \varepsilon $, and is unnecessary for our analysis. On the other hand, if this term is included, then the upper bound on $ \|\partial _{t}\varphi _{x}^{I,2}\|_{L _{T}^{2}L ^{2}} $ is needed for the estimate of $ \mathcal{R}_{1}^{\varepsilon} $ in the subsequent analysis. This will require higher-order regularities on the initial data $ (\varphi _{0},v _{0})$. Substituting \eqref{solution-recover-esti} into \eqref{refor-eq}--\eqref{BD-POSITIVE-VE}, we see that the perturbation functions $ (\Phi ^{\varepsilon}, V ^{\varepsilon}) $ satisfy
\begin{gather}\label{eq-perturbation}
\displaystyle  \begin{cases}
\displaystyle \Phi _{t}^{\varepsilon}=\Phi _{xx}^{\varepsilon}-\varepsilon ^{1/2}\Phi _{x}^{\varepsilon}V _{x}^{\varepsilon}-\Phi _{x}^{\varepsilon}V_{x}^{A}- V _{x}^{\varepsilon}(\Phi_{x}^{A}+M)+\varepsilon ^{-1/2}\mathcal{R}_{1}^{\varepsilon},\\[2mm]
  \displaystyle V _{t}^{\varepsilon}=\varepsilon V _{xx}^{\varepsilon}-\varepsilon ^{1/2}\Phi _{x}^{\varepsilon}V ^{\varepsilon}-\Phi _{x}^{\varepsilon}V^{A}-(\Phi_{x}^{A}+M)V ^{\varepsilon}
  +\varepsilon ^{-1/2}\mathcal{R}_{2}^{\varepsilon},\\[2mm]
  \displaystyle (\Phi ^{\varepsilon},V ^{\varepsilon})(x,0)=(0,0),\\[2mm]
  \displaystyle (\Phi ^{\varepsilon},V ^{\varepsilon})(0,t)=(\Phi ^{\varepsilon},V ^{\varepsilon})(1,t)=(0,0),
  \end{cases}
\end{gather}
where
\begin{align}\label{E-rror-fomula}
\displaystyle \mathcal{R}_{1}^{\varepsilon}=\Phi_{xx}^{A} -(\Phi_{x}^{A}+M)V_{x}^{A} -\Phi_{t}^{A}, \ \ \ \mathcal{R}_{2}^{\varepsilon}=\varepsilon V_{xx}^{A}-(\Phi_{x}^{A}+M)V ^{A}-V_{t}^{A}.
\end{align}
Notice that the coefficients and source terms in \eqref{eq-perturbation} involve only the outer- and boundary layer profiles studied in the previous section. By standard arguments (e.g. \cite{nishida-1978,zhaokun-2015-JDE}), one can prove the local-in-time existence and uniqueness of solutions to the problem \eqref{eq-perturbation} with $ \varepsilon>0 $ in the time interval $ [0,T _{\varepsilon}] $ for some $ T _{\varepsilon}>0 $ which may be small. Now the key is to establish some uniform-in-$ \varepsilon $ estimates for $ (\Phi ^{\varepsilon}, V ^{\varepsilon}) $ so that the $ \varepsilon $-independent lifespan of the solution and the convergence of boundary layers can be extracted. To this end, we present the following results for the problem \eqref{eq-perturbation}, which will be proved in the next subsection.
\begin{proposition}\label{prop-pertur}
Assume the conditions in Theorem \ref{thm-stabi-refor} hold. Then for any $ v _{\ast}>0 $, there exists constants $T>0 $ and $ \varepsilon _{0}>0 $ such that for any $ \varepsilon \in (0,\varepsilon _{0}) $, the problem \eqref{eq-perturbation} admits a unique solution $ (\Phi ^{\varepsilon}, V ^{\varepsilon}) \in L ^{\infty}(0,T;H ^{2}\times H ^{2}) $ which satisfies for any $ t \in [0,T] $,
\begin{gather*}
\displaystyle  \|\Phi^{\varepsilon}(\cdot,t)\|_{L ^{2}}^{2}+ \varepsilon ^{1/2}\|\Phi _{x}^{\varepsilon}(\cdot,t)\|_{L ^{2}}^{2}+\varepsilon ^{3/2}\|\Phi _{xx}^{\varepsilon}\|_{L ^{2}}^{2}+\varepsilon ^{\ell}\|\partial _{x}^{\ell}V ^{\varepsilon}(\cdot,t)\|_{L ^{2}}^{2} \leq c(v _{\ast},T) \varepsilon ^{1/2}
\end{gather*}
and
\begin{gather*}
\displaystyle \int _{0}^{t}\left( \|\Phi _{x}^{\varepsilon}\|_{L ^{2}}^{2}+ \varepsilon ^{1/2}\|\Phi _{\tau} ^{\varepsilon}\|_{L ^{2}}^{2}+\varepsilon \|\Phi _{x \tau}^{\varepsilon}\|_{L ^{2}}^{2}+\varepsilon \|V _{x}^{2}\|_{L ^{2}}^{2}+\|V _{\tau}^{\varepsilon}\|_{L ^{2}}^{2}+\varepsilon ^{5/2}\|V _{x \tau}\|_{L ^{2}}^{2} \right) \mathrm{d}\tau \leq c(v _{\ast},T) \varepsilon ^{1/2},
\end{gather*}
where $ \ell=0,1,2 $, $ c(v _{\ast},T) >0$ is a constant depending on $ T $ but independent of $ \varepsilon $.
\end{proposition}

\subsection{\emph{A priori} estimates} 
\label{subsec:em}
\subsubsection{Preliminaries} 
\label{ssub:some_basic_facts}
We introduce some basic facts for later use. By \eqref{v-B-0-regularity}, \eqref{con-vfi-B-1}, \eqref{con-vfi-v-B-2}, \eqref{con-vfi-b-2-v} and Proposition \ref{prop-embeding-spacetime}, we have for $ l \in \mathbb{N} $ that
   \begin{subequations}\label{some-l-infty-layer-for-vfi-b-1-stab}
    \begin{numcases}
\displaystyle\makebox[-4pt]{~}
     \|\langle z \rangle ^{l} \partial _{t}^{k}  \varphi ^{B,1}\|_{L _{T}^{\infty}H _{z}^{5-2k}}+\| \langle z \rangle ^{l} \partial _{t}^{k} \partial _{z}^{j}  \varphi ^{B,1}\|_{L _{T}^{\infty}L _{z}^{\infty}}
      \leq K _{0}(T,v _{\ast}) v _{\ast},\\
\displaystyle\makebox[-4pt]{~} \|\langle \xi \rangle ^{l} \partial _{t}^{k}  \varphi ^{b,1}\|_{L _{T}^{\infty}H _{\xi}^{5-2k}}+\| \langle \xi \rangle ^{l} \partial _{t}^{k} \partial _{\xi}^{j}  \varphi ^{b,1}\|_{L _{T}^{\infty}L_{\xi}^{\infty}} \leq K _{0}(T,v _{\ast}) v _{\ast}
    \end{numcases}
   \end{subequations}
for $ k=0,1,2,\ j=0,1,\cdots, 4-2k $, and that
  \begin{subequations}\label{some-l-infty-layer-for-stab}
      \begin{numcases}
\displaystyle\makebox[-4pt]{~}
    \displaystyle \|\langle z \rangle ^{l} \partial _{t}^{k}  \varphi ^{B,2}\|_{L _{T}^{\infty}H _{z}^{4-2k}}+ \|\langle z \rangle ^{l} \partial _{t}^{k} \partial _{z}^{j}  \varphi ^{B,2}\|_{L _{T} ^{\infty}L _{z}^{\infty}} \leq c(v _{\ast},T),\\
  \displaystyle\makebox[-4pt]{~} \|\langle \xi \rangle ^{l} \partial _{t}^{k}  \varphi ^{b,2}\|_{L _{T}^{\infty}H _{\xi}^{4-2k}}+ \|\langle \xi \rangle ^{l} \partial _{t}^{k} \partial _{\xi}^{j}  \varphi ^{b,2}\|_{L _{T}^{\infty}L _{\xi}^{\infty}} \leq c(v _{\ast},T)
    \end{numcases}
   \end{subequations}
   for $ k=0,1 $, $ j=0,1,\cdots, 4-2k $. Hereafter the constant $ c(v _{\ast},T)>0 $ is as stated in Section  \ref{sec:main_result}, $ K_{0}(T,v _{\ast})>0 $ is as in Lemma \ref{lem-v-B-0}. Also, we collect some basic estimates on the boundary layer profiles of $ v $ as follows.
\begin{subequations}\label{l-infty-v-bd-stat}
\begin{numcases}
\displaystyle\makebox[-4pt]{~}
	\displaystyle\|\langle z \rangle ^{l} \partial _{t}^{k}  v ^{B,1}\|_{L _{T}^{\infty}H _{z}^{4-2k}}+\sum _{\lambda=0}^{1}\sum _{\ell=0}^{3- 2\lambda}\|\langle z \rangle ^{l}\partial _{t}^{\lambda}\partial _{z}^{\ell} v ^{B,1}\|_{L _{T}^{\infty}L _{z}^{\infty}} \leq c(v _{\ast},T),\label{l-infty-v-bd-stat-a}\\
	\displaystyle\makebox[-4pt]{~} \|\langle \xi \rangle ^{l} \partial _{t}^{k}  v ^{b,1}\|_{L _{T}^{\infty}H _{\xi}^{4-2k}}+\sum _{\lambda=0}^{1}\sum _{\ell=0}^{3- 2\lambda}\|\langle \xi\rangle ^{l}\partial _{t}^{\lambda}\partial _{z}^{\ell} v ^{b,1}\|_{L _{T}^{\infty}L _{\xi}^{\infty}} \leq c(v _{\ast},T) , \label{l-infty-v-bd-stat-b}
    \end{numcases}
   \end{subequations}
 and
    \begin{subequations}\label{some-l-infty-layer-for-v-b-0-stab}
     \begin{numcases}
     \displaystyle\makebox[-4pt]{~}
     \|\langle z \rangle ^{l} \partial _{t}^{k}  v ^{B,0}\|_{L _{T}^{\infty}H _{z}^{4-2k}}+\sum _{\lambda=0}^{1}\sum _{\ell=0}^{3- 2\lambda}\|\langle z \rangle ^{l}\partial _{t}^{\lambda}\partial _{z}^{\ell} v ^{B,0}\|_{L _{T}^{\infty}L _{z}^{\infty}}
      \leq K_{0}(T,v _{\ast}) v _{\ast},\\
    \displaystyle\makebox[-4pt]{~} \|\langle \xi \rangle ^{l} \partial _{t}^{k}  v ^{b,0}\|_{L _{T}^{\infty}H _{\xi}^{4-2k}}+\sum _{\lambda=0}^{1}\sum _{\ell=0}^{3- 2\lambda}\|\langle \xi \rangle ^{l}\partial _{t}^{\lambda}\partial _{\xi}^{\ell} v ^{b,0}\|_{L _{T}^{\infty}L _{\xi}^{\infty}} \leq K _{0}(T,v _{\ast}) v _{\ast}
    \end{numcases}
   \end{subequations}
for $ k=0,1,2,\ j=0,1,\cdots, 4-2k $, due to Lemmas \ref{lem-v-B-0}, \ref{lem-first-bd-rt}, \ref{lem-v-B-1}, \ref{lem-v-b-1} and Proposition \ref{prop-embeding-spacetime}. From \eqref{con-vfi-v-I-0-regula} and  \eqref{some-l-infty-layer-for-vfi-b-1-stab}--\eqref{l-infty-v-bd-stat}, one can deduce some estimates on the approximate solution $ (\Phi^{A}, V^{A}) $:
\begin{subequations}\label{esti-approx-fida-2}
\begin{gather}
\displaystyle  \|\partial _{x}^{l}\Phi^{A}\|_{L _{T}^{\infty}L ^{\infty}}+\|\partial _{t}\partial _{x}^{l}\Phi^{A}\|_{L _{T}^{\infty}L ^{\infty}} \leq c(v _{\ast},T), \ \ l=0,1, \label{esti-approx-fida-2a}\\
\displaystyle \|V^{A}\|_{L _{T}^{\infty}L ^{\infty}}+ \|V_{t}^{A}\|_{L _{T}^{\infty}L ^{\infty}}+\varepsilon ^{1/2}\|\partial _{t}^{l}\partial _{x} V ^{A}\|_{L _{T}^{\infty}L ^{\infty}} \leq c(v _{\ast},T), \ \ l=0,1.\label{esti-approx-fida-2b}
\end{gather}
\end{subequations}

\subsubsection{Estimates on the error terms} 
\label{sub:estimates_on_the_error_terms}

Now let us turn to estimates on the error terms $ \mathcal{R}_{1}^{\varepsilon} $ and $ \mathcal{R}_{2}^{\varepsilon} $.
\begin{lemma}\label{lem-R-1-VE}
Let $ 0< \varepsilon <1 $. It holds for any $ T>0 $ that
\begin{gather}\label{con-R-1-in-lem}
\displaystyle
 \displaystyle \| \mathcal{R}_{1}^{\varepsilon}\|_{L _{T}^{\infty}L ^{\infty}} \leq c(v _{\ast},T) \varepsilon ^{1/2},\ \ \ \|\mathcal{R}_{1}^{\varepsilon}\|_{L _{T}^{\infty}L ^{2}} +\|\partial _{t}\mathcal{R}_{1}^{\varepsilon}\|_{L _{T}^{2}L ^{2}}\leq c(v _{\ast},T) \varepsilon ^{3/4}.
\end{gather}
\end{lemma}

\begin{proof}
First recalling the definition of $ G _{i}~(i=-1,0) $ and $ \tilde{G} _{i}~(i=-1,0) $ in \eqref{G-sum} and \eqref{tild-G-sum}, respectively, using \eqref{E-rror-fomula} and the first equation in \eqref{eq-outer-0} and in \eqref{first-outer-problem}, we get from a direct computation that
\begin{align*}
\displaystyle \mathcal{R}_{1}^{\varepsilon}&=-   v _{x}^{B,0}\left[\varphi _{x}^{I,0}- \varphi _{x}^{I,0}(0,t)-x\varphi _{xx}^{I,0}(0,t)\right]+ v _{x}^{b,0}\left[\varphi _{x}^{I,0}-\varphi _{x}^{I,0}(1,t)-(x-1)\varphi _{xx} ^{I,0}(1,t)) \right]
   \nonumber \\
   &\displaystyle \quad- \left[\varepsilon ^{1/2}v _{x}^{B,1}(\varphi _{x}^{I,0}(x,t)-\varphi_{x} ^{I,0}(0,t))+  \varepsilon ^{1/2}v _{x}^{b,1}(\varphi _{x}^{I,0}(x,t)-\varphi _{x}^{I,0}(1,t))\right]
     \nonumber \\
   & \displaystyle \quad - \left[\varepsilon ^{1/2}v _{x}^{B,0}(\varphi_{x}^{I,1}(x,t)-\varphi _{x}^{I,1}(0,t))+  \varepsilon ^{1/2}v _{x}^{b,0}(\varphi_{x}^{I,1}(x,t)-\varphi_{x} ^{I,1}(1,t))\right]
    \nonumber \\
    & \displaystyle \quad  - \varepsilon ^{1/2}\left[ \varphi _{x}^{B,1} (v _{x}^{I,0}(x,t)-v _{x}^{I,0}(0,t))+\varphi _{x}^{b,1} (v _{x}^{I,0}(x,t)-v _{x}^{I,0}(1,t)) \right] - \varepsilon ^{1/2}( v _{x}^{B,0}\varphi _{x}^{b,1}+v _{x}^{b,0}\varphi _{x}^{B,1} )
     \nonumber \\
              &\displaystyle \quad  -\varepsilon(v _{x}^{B,1}\varphi _{x}^{b,1}+v _{x}^{b,1}\varphi _{x}^{B,1}+\varphi _{x} ^{B,2} v _{x}^{b,0}+\varphi _{x}^{b,2}v _{x}^{B,0})- \varepsilon \left[\varphi_{x}^{I,1}(v _{x}^{B,1}+v _{x}^{b,1})  +  v _{x}^{I,1}(\varphi _{x}^{B,1}+\varphi _{x}^{b,1})\right]
 \nonumber \\
               & \displaystyle \quad -\varepsilon\varphi _{x}^{I,1}v _{x}^{I,1}
            - \varepsilon \varphi _{x}^{B,2}\left[ v ^{I,0}
       +\varepsilon (v ^{I,1}+v ^{B,1}+v ^{b,1}) \right] _{x}- \varepsilon \varphi _{x}^{b,2}\left[ v ^{I,0}
       +\varepsilon (v ^{I,1}+v ^{B,1}+v ^{b,1}) \right] _{x}
      \nonumber \\
       & \displaystyle \quad
       - \left[ \varepsilon ^{1/2}( \varphi _{t}^{B,1}+\varphi _{t}^{b,1} )+ \varepsilon ( \varphi _{t}^{B,2}+\varphi _{t} ^{b,2} ) \right]     +F ^{\varepsilon}=: \sum _{i=1} ^{12}\mathcal{P}_{i}+F ^{\varepsilon},
\end{align*}
where
\begin{align}\label{F-defi}
\displaystyle F ^{\varepsilon}&= - \partial _{x} b _{v}^{\varepsilon}\left( \varphi _{x}^{I,0}+M+\varepsilon ^{1/2}(\varphi_{x} ^{I,1}+\varphi _{x}^{B,1}+\varphi_{x} ^{b,1}) +\varepsilon(\varphi_{x} ^{B,2}+\varphi_{x} ^{b,2})   \right)
 \nonumber \\
 & \displaystyle \quad  -\partial _{x} b _{\varphi}^{\varepsilon}\left( v ^{I,0}+v ^{B,0}+v ^{b,0}+\varepsilon ^{1/2}(v ^{I,1}+v ^{B,1}+v ^{b,1}) \right)_{x}   -(\partial _{x} b _{\varphi}^{\varepsilon}\partial _{x} b _{v}^{\varepsilon})-\partial _{t} b _{\varphi}^{\varepsilon}.
\end{align}
Now we are ready to estimate $ \|\mathcal{R}_{1} ^{\varepsilon}\|_{L _{T}^{2}L _{z}^{2}} $. By \eqref{con-vfi-v-I-0-regula}, \eqref{l-infty-v-bd-stat-a}, \eqref{z-transfer} and Taylor's formula, we have
\begin{align}\label{P-1-esti}
&\displaystyle \|\mathcal{P}_{1}\|_{L _{T}^{\infty}L ^{2}} =\left\|\frac{\varphi _{x}^{I,0}(x,t)-\varphi _{x}^{I,0}-x\varphi _{xx} ^{I,0}(0,t)}{x ^{2}}x ^{2}v _{x}^{B,0}\right\|_{L _{T}^{\infty}L ^{2}}
 \nonumber \\
 &~\displaystyle \leq  \frac{1}{2} \|\partial _{x}^{3}\varphi ^{I,0}\|_{L _{T}^{\infty}L ^{\infty}} \|x ^{2}v _{x}^{B,0}\|_{L _{T}^{\infty}L  ^{2}} \leq c _{0} \varepsilon \|\varphi ^{I,0}\|_{L _{T}^{\infty}H ^{4}}\|z ^{2}v _{x}^{B,0}\|_{L _{T}^{\infty}L  ^{2}}
  \nonumber \\
  & ~\displaystyle \leq c _{0} \varepsilon ^{3/4}\|\varphi ^{I,0}\|_{L _{T}^{\infty}H ^{4}}\|z ^{2}v _{z}^{B,0}\|_{L _{T}^{\infty}L  _{z}^{2}} \leq c(v _{\ast},T) \varepsilon ^{3/4}.
\end{align}
The same argument as above yields
\begin{gather*}
\displaystyle \|\mathcal{P}_{2}\|_{L _{T}^{\infty}L  ^{2}} \leq  \frac{1}{2} \|\partial _{x}^{3}\varphi ^{I,0}\|_{L _{T}^{\infty} L ^{\infty}} \|(x-1)^{2}v _{x}^{b,0}\|_{L _{T}^{\infty}L  ^{2}} \leq c(v _{\ast},T) \varepsilon ^{3/4}.
\end{gather*}
Furthermore, in the same manner, we get from \eqref{con-vfi-v-I-0-regula}, \eqref{some-l-infty-layer-for-vfi-b-1-stab}--\eqref{l-infty-v-bd-stat} that
\begin{gather*}
\displaystyle \|\mathcal{P}_{i}\| _{L _{T}^{\infty}L  ^{2}} \leq c(v _{\ast},T)\varepsilon ^{3/4},\ \ i=3,4,5.
\end{gather*}
Notice that $ \frac{1}{2 \varepsilon ^{1/2}} < z=\frac{x}{\varepsilon ^{1/2}}< \frac{1}{\varepsilon ^{1/2}} $ for $1/2 \leq x \leq 1 $, and that $ - \frac{1}{\varepsilon ^{1/2}} \leq \xi=\frac{x-1}{\varepsilon ^{1/2}} \leq -\frac{1}{2\varepsilon ^{1/2}} $ for $ 0 \leq x \leq 1/2$. This along with \eqref{con-b-0}, \eqref{con-vfi-v-B-2}, \eqref{con-vfi-b-2-v}, \eqref{l-infty-v-bd-stat} and \eqref{inte-transfer} implies for $ m \in \mathbb{N} $ and $  k=0,1,2 $ that
\begin{align}\label{half-l-infty}
&\displaystyle \varepsilon ^{-\frac{m}{2}}\|\partial _{t}^{k}\partial _{x}^{i} v ^{B,j}\|_{L ^{\infty}((1/2, 1)\times(0,T))}+\varepsilon ^{-\frac{m}{2}}\|\partial _{t}^{k}\partial _{x}^{i} v ^{b,j}\|_{L ^{\infty}((0,\frac{1}{2})\times(0,T))}
 \nonumber \\
    &~\displaystyle\leq c _{0} \|z ^{m+i} \partial _{t}^{k} \partial _{z}^{i}v ^{B,j} \|_{L ^{\infty}(0,T;L _{z}  ^{\infty}(0, \varepsilon ^{-1/2}))}+c _{0} \|\xi ^{m+i} \partial _{t}^{k} \partial _{\xi}^{i}v ^{b,j} \|_{L ^{\infty}(0,T;L _{\xi}  ^{\infty}(-\varepsilon ^{-1/2},0))}
     \nonumber \\
     &~ \displaystyle
    \leq c _{0} \|\langle z \rangle ^{m+i}\partial _{t}^{k}\partial _{z}^{i}v ^{B,j}\|_{L _{z}  ^{\infty}}
+ c _{0} \|\langle \xi \rangle ^{m+i}\partial _{t}^{k}\partial _{\xi}^{i}v ^{b,j}\|_{L _{\xi}  ^{\infty}} \leq c(v _{\ast},T),
\end{align}
where $ j=0,1 $, $ i=0,1,\cdots,4-2k $. Similarly, we have
\begin{gather}\label{vfi-B-1-HALF-INTY}
\displaystyle  \varepsilon ^{-m/2}\left( \|\partial _{t}^{k} \partial _{x}^{i}\varphi ^{B,1}\|_{L ^{\infty}((1/2, 1)\times(0,T))}+\|\partial _{t}^{k} \partial _{x}^{i}\varphi ^{b,1}\|_{L ^{\infty}((0,1/2)\times(0,T))}  \right) \leq c(v _{\ast},T)
\end{gather}
for $ k=0,1,2$, $i=0,1,\cdots,4-2k $, and
\begin{gather}\label{vfi-B-2-half}
\displaystyle  \displaystyle  \varepsilon ^{-m/2}\left( \|\partial _{t}^{k} \partial _{x}^{i}\varphi ^{B,2}\|_{L ^{\infty}((1/2, 1)\times(0,T))}+ \|\partial _{t}^{k} \partial _{x}^{i}\varphi ^{b,2}\|_{L ^{\infty}((0,1/2)\times(0,T))} \right)  \leq c(v _{\ast},T)
\end{gather}
if $  k=0,1$, $i=0,1,\cdots,2-2k $. Therefore, we deduce for $m \in \mathbb{N}_{+} $ that
\begin{align}\label{P-6-1}
\displaystyle  \varepsilon ^{1/2}\|v _{x}^{B,0}\varphi _{x}^{b,1}\|_{L _{T}^{\infty}L ^{2}}
 &\displaystyle \leq \varepsilon ^{1/2} \|v _{x}^{B,0}\varphi _{x}^{b,1}\|_{L ^{\infty}(0,T;L ^{2}(0,1/2))}+\varepsilon ^{1/2} \|v _{x}^{B,0}\varphi _{x}^{b,1}\|_{L ^{\infty}(0,T;L ^{2}(1/2,1))}
  \nonumber \\
  &\displaystyle \leq \varepsilon ^{\frac{m+1}{2}}\left( \|\varepsilon ^{-m/2}\varphi _{x}^{b,1}\|_{L ^{\infty}((0,1/2 )\times(0,T))}\|v _{x}^{B,0}\|_{L _{T}^{\infty}L ^{2}} \right.
   \nonumber \\
   &\displaystyle  \qquad \qquad \  \left.+\|\varepsilon ^{-m/2}v _{x}^{B,0}\|_{L ^{\infty}((1/2,1)\times (0,T))}\|\varphi _{x}^{b,1}\|_{L _{T}^{\infty}L ^{2}} \right)
    \nonumber \\
    &\displaystyle \leq c _{0} \varepsilon ^{\frac{2m+1}{4}} \left(  \|v _{z}^{B,0}\|_{L _{T}^{\infty}L _{z}^{2}}+\|v _{\xi}^{b,1}\|_{L _{T}^{\infty}L _{\xi}^{2}}\right)
     \nonumber \\
     &\displaystyle \leq c(v _{\ast},T)  \varepsilon ^{\frac{2m+1}{4}} ,
\end{align}
and
\begin{align}\label{p-6-2}
\displaystyle \varepsilon ^{1/2} \|v _{x}^{b,0}\varphi _{x}^{B,1}\|_{L _{T}^{\infty}L ^{2}}
 &\displaystyle \leq \varepsilon ^{1/2} \|v _{x}^{b,0}\varphi _{x}^{B,1}\|_{L ^{\infty}(0,T;L ^{2}(0,1/2))}+\varepsilon ^{1/2} \|v _{x}^{b,0}\varphi _{x}^{B,1}\|_{L ^{\infty}(0,T;L ^{2}(1/2,1))}
  \nonumber \\
  &\displaystyle \leq \varepsilon ^{\frac{m+1}{2}}\left( \|\varepsilon ^{-m/2}v _{x}^{b,0}\|_{L ^{\infty}((0,1/2)\times (0,T))}\|\varphi _{x}^{B,1}\|_{L _{T}^{\infty}L ^{2}} \right.
   \nonumber \\
   &\displaystyle  \qquad \qquad \  \left.+\|\varepsilon ^{-m/2}\varphi _{x}^{B,1}\|_{L ^{\infty}((1/2,1)\times (0,T))}\|v_{x}^{b,0}\|_{L _{T}^{\infty}L ^{2}} \right)
    \nonumber \\
    &\displaystyle \leq c _{0} \varepsilon ^{\frac{2m+1}{4}} \left(  \|\varphi _{z}^{B,1}\|_{L _{T}^{\infty}L _{z}^{2}}+\|v _{\xi}^{b,0}\|_{L _{T}^{\infty}L _{\xi}^{2}}\right)
     \nonumber \\
     &\displaystyle \leq c(v _{\ast},T)  \varepsilon ^{\frac{2m+1}{4}} .
     \end{align}
     where we have used \eqref{some-l-infty-layer-for-vfi-b-1-stab}, \eqref{l-infty-v-bd-stat} and \eqref{inte-transfer}. Thus $ \|\mathcal{P }  _{6}\| _{L _{T}^{\infty}L ^ 2}\leq c(v _{\ast},T) \varepsilon ^{3/4}$. By the same argument as proving estimates for $ \mathcal{P}_{6} $, one can infer that
\begin{align*}
\displaystyle  \left\|\mathcal{P}_{7}\right\| _{L _{T}^{\infty}L _{z}^{2}} \leq c(v _{\ast},T)\varepsilon ^{3/4}.
\end{align*}
With \eqref{vfi-I-1-V-i-1}, \eqref{some-l-infty-layer-for-vfi-b-1-stab}, \eqref{l-infty-v-bd-stat} and \eqref{inte-transfer}, we obtain
\begin{align*}
\displaystyle  \|\mathcal{P}_{8}\|_{L _{T}^{\infty}L ^{2}}& \leq c(v _{\ast},T) \varepsilon ^{3/4} \|\varphi _{x}^{I,1}\|_{L _{T}^{\infty}L ^{\infty}} \left( \|v _{z}^{B,1}\|_{L _{T}^{\infty} L _{z}^{2}}+\|v _{\xi}^{b,1}\|_{L _{T}^{\infty} L _{\xi}^{2}} \right)
 \nonumber \\
 & \displaystyle \quad +c(v _{\ast},T) \varepsilon ^{3/4} \|v _{x}^{I,1}\|_{L _{T}^{\infty}L ^{\infty}} \left( \|\varphi _{z}^{B,1}\|_{L _{T}^{\infty}L _{z}^{2}}+\|\varphi _{\xi}^{b,1}\|_{L _{T} ^{\infty}L _{\xi}^{2}} \right)  \leq c(v _{\ast},T) \varepsilon ^{3/4}.
\end{align*}
Similarly, we also have
\begin{align*}
\displaystyle  \|\mathcal{P}_{i}\|_{L _{T}^{\infty}L ^{2}} \leq c(v _{\ast},T) \varepsilon ^{3/4},\ \ \ i =9,10,11,12.
\end{align*}
For the last term $ F ^{\varepsilon} $, we first deduce from \eqref{b-vfi-ve} and \eqref{some-l-infty-layer-for-vfi-b-1-stab}--\eqref{l-infty-v-bd-stat} that
\begin{align}\label{b-v-ve-esti-1}
&\displaystyle \|\partial _{t}^{k}b _{\varphi}^{\varepsilon}\|_{L _{T}^{\infty}H ^{1}}
 \nonumber \\
 & ~\displaystyle \leq c(v _{\ast},T) \varepsilon \Big( \|\varepsilon ^{-1/2} \partial _{t}^{k}\varphi ^{b,1}(- \varepsilon ^{-1/2} ,t)
\|_{L ^{\infty}(0,T)}+\| \partial _{t}^{k}\varphi ^{b,2}(- \varepsilon ^{-1/2},t)\|_{L ^{\infty}(0,T)}
+\|\partial _{t}^{k}\varphi^{B,2}(0,t)\|_{L ^{\infty}(0,T)}\Big)
  \nonumber \\
   & ~ \displaystyle \leq c(v _{\ast},T)\varepsilon \left( \|\langle \xi \rangle \partial _{t}^{k}\varphi ^{b,1} \|_{L _{T}^{\infty}L _{\xi}^{\infty}}+\| \partial _{t}^{k}\varphi ^{b,2} \|_{L _{T}^{\infty}L _{\xi}^{\infty}}+\|\partial _{t}^{k}\varphi ^{B,2} \| _{L _{T}^{\infty}L _{z}^{\infty}} \right)
  \nonumber \\
     &~\displaystyle\leq c(v _{\ast},T) \varepsilon
\end{align}
for $ k=0,1 $. By similar arguments, we have from \eqref{b-v-ve} that
\begin{align}\label{b-v-ve-esti}
\displaystyle  \|\partial _{t}^{k} b _{v}^{\varepsilon}\|_{L _{T}^{\infty}H ^{1}} \leq c(v _{\ast},T) \varepsilon, \ \ \ k=0,1.
\end{align}
Similar arguments along with \eqref{v-B-0-regularity}, \eqref{con-b-0}, \eqref{con-vfi-v-B-2} and \eqref{con-vfi-b-2-v} further imply that
\begin{align}\label{b-v-ve-l-2-esti}
\displaystyle \|\partial _{t}^{2}b _{\varphi}^{\varepsilon}\|_{L _{T}^{2}H ^{1}}+\left\|\partial _{t}^{2} b _{v}^{\varepsilon}\right\| _{L _{T}^{2}H ^{1}} \leq c(v _{\ast},T) \varepsilon.
\end{align}
Notice also that $ \partial _{x}b _{\varphi}^{\varepsilon} $ and $ \partial _{x}b _{v}^{\varepsilon} $ are independent of $ x $. Thus it holds that
\begin{align}\label{b-time-esti}
\displaystyle  \|\partial _{t}^{2}\partial _{x} b _{\varphi}^{\varepsilon}\|_{L ^{2}(0,T)}+\|\partial _{t}^{2}\partial _{x} b _{v}^{\varepsilon} \|_{L ^{2}(0,T)}+\|\partial _{t}^{k}\partial _{x} b _{\varphi}^{\varepsilon}\|_{L ^{\infty}(0,T)}+\|\partial _{t} ^{k}\partial _{x} b _{v}^{\varepsilon} \|_{L ^{\infty}(0,T)} \leq c(v _{\ast},T) \varepsilon,
\end{align}
where $ k=0,1 $. With \eqref{con-vfi-v-I-0-regula},  \eqref{some-l-infty-layer-for-vfi-b-1-stab}--\eqref{l-infty-v-bd-stat} and \eqref{b-v-ve-esti-1}--\eqref{b-time-esti}, recalling the definition of $ F ^{\varepsilon} $ in \eqref{F-defi}, we have
\begin{align}\label{F-es-t-1}
&\displaystyle  \|F ^{\varepsilon}\|_{L _{T}^{\infty} L ^{2}}
 \nonumber \\
 &~ \leq c(v _{\ast},T) \|\partial _{x}b _{v}^{\varepsilon}\|_{L ^{\infty}}\left( 1+\left\|\varphi ^{I,0}\right\|_{ L _{T}^{\infty}H ^{1}}+\varepsilon ^{1/2}\|\varphi ^{I,1}\|_{L _{T}^{\infty}H ^{1}}+\varepsilon ^{1/4}\|\varphi ^{B,1}\|_{L _{T}^{\infty}H _{z}^{1}}\right.
 \nonumber \\
 & ~\displaystyle \left. \qquad \qquad \qquad \quad+ \varepsilon ^{1/4}\|\varphi ^{b,1}\| _{L _{T}^{\infty}H _{\xi}^{1}}+\varepsilon ^{3/4}\|\varphi ^{B,2}\|_{L _{T}^{\infty}H _{z}^{1}}+\varepsilon ^{3/4}\|\varphi ^{b,2}\|_{L _{T}^{\infty}H _{\xi}^{1}}\right)
  \nonumber \\
  &~\displaystyle \quad +\|\partial _{x}b _{\varphi}^{\varepsilon}\|_{L ^{\infty}(0,T)}\left( \|v ^{I,0}\|_{L _{T}^{\infty}H ^{1}}+\varepsilon ^{1/4}\|v ^{B,1}\|_{L _{T}^{\infty}H _{z}^{1}}+ \varepsilon ^{1/4}\|v ^{b,1}\| _{L _{T}^{\infty}H _{\xi}^{1}}\right)
   \nonumber \\
   &~\displaystyle \quad+ \|\partial _{x} b _{\varphi}^{\varepsilon}\|_{L ^{\infty}(0,T)}\|\partial _{x} b _{v}^{\varepsilon}\| _{L ^{\infty}(0,T)}+ \|\partial _{t}b _{\varphi}^{\varepsilon}\|_{L _{T}^{\infty}L^{2}}
    \nonumber \\
    &~\displaystyle \leq c(v _{\ast},T)\varepsilon ^{5/4}.
\end{align}
In summary, we now have for $ 0< \varepsilon<1 $ that
\begin{gather}\label{R-1-first-esti}
\displaystyle \|\mathcal{R}_{1}^{\varepsilon}\|_{L _{T}^{\infty}L ^{2}} \leq \sum _{i=1}^{12}\|\mathcal{P}_{i}\|_{L _{T}^{\infty}L ^{2}}+\|F ^{\varepsilon}\|_{L _{T}^{\infty}L ^{2}} \leq c(v _{\ast},T) \varepsilon ^{3/4}.
\end{gather}
Repeating the above procedure with $ L ^{2} $-norm replaced by $ L ^{\infty} $-norm, we have that
\begin{gather}\label{R-1-l-infty}
\displaystyle \|\mathcal{R}_{1} ^{\varepsilon}\| _{L _{T}^{\infty}L ^{\infty}}\leq c(v _{\ast},T) \varepsilon ^{1/2}.
\end{gather}
We proceed to estimate $ \|\partial _{t} \mathcal{R}_{1}^{\varepsilon}\|_{L _{T} ^{2}L ^{2} } $. Notice that if $ \|h \chi\|_{Z} \leq c _{0} \|h\|_{X}\| \chi\|_{Y} $ for $ h \in X $ and $ \chi \in Y $ with $ X $, $ Y $ and $ Z $ being Banach spaces, then
\begin{gather}\label{product-tim}
\displaystyle \|\partial _{t}(h \chi)\|_{Z} \leq c _{0}\|\partial _{t}h\|_{X}\|\chi\|_{Y}+ c _{0} \|h\|_{X}\| \partial _{t}\chi\|_{Y},
\end{gather}
provided that $ \partial _{t}f \in X $ and $ \partial _{t}\chi \in Y $. Therefore, by \eqref{con-vfi-v-I-0-regula}, \eqref{l-infty-v-bd-stat-a} and similar arguments to proving \eqref{P-1-esti}, we have
\begin{align}\label{P-1-t-esti}
&\displaystyle \|\partial _{t}\mathcal{P}_{1}\|_{L _{T}^{\infty}L ^{2}}
 \nonumber \\
 &~ \displaystyle \leq  \left\|\frac{\partial _{t} \varphi_{x} ^{I,0}(x,t)-\partial _{t}\varphi_{x} ^{I,0}(0,t)-x\partial _{t}\varphi _{xx} ^{I,0}(0,t)}{x ^{2}}x ^{2}v _{x}^{B,0}\right\|_{L _{T}^{\infty}L ^{2}}
 \nonumber \\
 & ~\displaystyle \quad+ \left\|\frac{\varphi_{x} ^{I,0}(x,t)-\varphi _{x}^{I,0}(0,t)-x\varphi _{xx}^{I,0}(0,t)}{x ^{2}}x ^{2} v _{xt}^{B,0}\right\|_{L _{T} ^{\infty}L ^{2}}
  \nonumber \\
  & ~\displaystyle \leq  \|\partial _{t}\partial _{x}^{3}\varphi ^{I,0}\|_{L _{T}^{\infty} L ^{\infty}} \|x ^{2}v _{x}^{B,0}\|_{L _{T}^{\infty}L  ^{2}}+ \|\partial _{x}^{3}\varphi ^{I,0}\|_{L _{T}^{\infty}L ^{\infty}} \|x ^{2}v _{xt}^{B,0}\|_{L _{T}^{\infty}L  ^{2}}
   \nonumber \\
   & ~\displaystyle \leq c(v _{\ast},T)\varepsilon ^{3/4}\|\varphi  _{t}^{I,0}\|_{L _{T}^{\infty}H ^{4}}\|z ^{2}v _{z}^{B,0}\|_{L _{T}^{\infty}L  _{z}^{2}}+c(v _{\ast},T) \varepsilon ^{3/4}\|\varphi ^{I,0}\|_{L _{T}^{\infty}H ^{4}}\|z ^{2} v _{zt}^{B,0}\|_{L _{T}^{\infty}L  _{z}^{2}}
    \nonumber \\
    &~\displaystyle \leq c(v _{\ast},T) \varepsilon ^{3/4}.
\end{align}
Similar arguments further yield that
\begin{align}\label{P-i-t-esti}
\displaystyle  \|\partial _{t}\mathcal{P}_{i}\|_{L _{T}^{\infty}L ^{2}} \leq c(v _{\ast},T) \varepsilon ^{3/4},\ \ i=2,3,\cdots,11.
\end{align}
Now it remains to prove $ \| \partial _{t}\mathcal{P}_{12}\|_{L _{T}^{2}L ^{2}} \leq c(v _{\ast},T) \varepsilon ^{3/4}$ and $ \|\partial _{t}F ^{\varepsilon}\|_{L _{T} ^{2}L ^{2}}\leq c(v _{\ast},T) \varepsilon ^{3/4} $. For the former, it follows from \eqref{con-vfi-B-1}, \eqref{con-b-0}, \eqref{con-vfi-v-B-2}, \eqref{con-vfi-b-2-v} and $ 0<\varepsilon<1 $ that
\begin{align}\label{P-t-12-esti}
 \| \partial _{t}\mathcal{P}_{12}\|_{L _{T} ^{2}L ^{2}}
  &\displaystyle\leq c _{0}\varepsilon ^{\frac{3}{4}}( \|\partial _{t}^{2} \varphi ^{B,1}\|_{L _{T}^{2}L _{z} ^{2}}+\|\partial _{t}^{2} \varphi ^{b,1}\|_{L _{T}^{2}L _{\xi}^{2}})  \nonumber \\
   &\displaystyle \quad+ c _{0}\varepsilon ^{5/4} (\| \partial_{t}^{2}\varphi ^{B,2}\|_{L _{T}^{2}L _{z} ^{2}}+\| \partial_{t}^{2}\varphi  ^{b,2}\| _{L _{T}^{2}L _{\xi} ^{2}}) \leq c(v _{\ast},T) \varepsilon ^{3/4}.
\end{align}
For the latter, we split $\partial _{t} F ^{\varepsilon}$ into two parts:
\begin{align*}
\displaystyle \partial _{t}F ^{\varepsilon} &= - \partial _{t}\left[\partial _{x} b _{v}^{\varepsilon}\left( \partial _{x}\varphi ^{I,0}+M+\varepsilon ^{1/2}(\partial _{x}\varphi ^{I,1}+\partial _{x}\varphi ^{B,1}+\partial _{x}\varphi ^{b,1}) +\varepsilon(\partial _{x}\varphi ^{B,2}+\partial _{x}\varphi ^{b,2})   \right) \right.
 \nonumber \\
  & \displaystyle \quad  -\left.\partial _{x} b _{\varphi}^{\varepsilon}\left( v ^{I,0}+v ^{B,0}+v ^{b,0}+\varepsilon ^{1/2}(v ^{I,1}+v ^{B,1}+v ^{b,1}) \right)_{x}  +(\partial _{x} b _{\varphi}^{\varepsilon}\partial _{x} b _{v}^{\varepsilon}) \right]-\partial _{t}^{2} b _{\varphi}^{\varepsilon}
   \nonumber \\
   &\displaystyle =:\tilde{F}-\partial _{t}^{2} b _{\varphi}^{\varepsilon},
\end{align*}
where $ \|\partial _{t}^{2} b _{\varphi}^{\varepsilon}\|_{L _{T}^{2}L ^{2}} \leq c(v _{\ast},T) \varepsilon  $ due to \eqref{b-v-ve-l-2-esti}. In view of \eqref{product-tim} along with a modification of the arguments in \eqref{F-es-t-1}, it holds that $ \|\tilde{F}\|_{L _{T}^{2}L ^{2}} \leq c(v _{\ast},T) \varepsilon ^{3/4} $. Therefore we have
\begin{gather*}
\displaystyle \|\partial _{t}F ^{\varepsilon}\|_{L _{T}^{2}L ^{2}} \leq c(v _{\ast},T) \varepsilon ^{3/4}.
\end{gather*}
This alongside \eqref{R-1-first-esti}, \eqref{R-1-l-infty}, \eqref{P-1-t-esti}--\eqref{P-t-12-esti} gives rise to \eqref{con-R-1-in-lem}, and thus complete the proof of Lemma \ref{lem-R-1-VE}.
\end{proof}

\begin{lemma}\label{lem-R-2-ESTI}
For any $ 0<T<\infty $ and $ 0< \varepsilon<1 $, it holds that
\begin{align}\label{con-R-2-ESTI-in-lem}
\displaystyle   \|\mathcal{R}_{2}^{\varepsilon}\|_{L _{T}^{\infty}L ^{2}}+\|\partial _{t}\mathcal{R}_{2}^{\varepsilon}\|_{L _{T}^{2}L ^{2}} \leq c(v _{\ast},T) \varepsilon ^{3/4},\ \ \ \|\mathcal{R}_{2}^{\varepsilon}\|_{L _{T}^{\infty}L ^{\infty}} \leq c(v _{\ast},T) \varepsilon ^{1/2}.
\end{align}
\end{lemma}
\begin{proof}
From \eqref{first-bd-layer-pro}--\eqref{firs-bd-1-rt}, we know that
\begin{gather}\label{v-ZZ}
\displaystyle \begin{cases}
 \displaystyle v _{zz}^{B,0}=v _{t}^{B,0} +(\partial _{x}\varphi ^{I,0}(0,t)+M)v ^{B,0}+\varphi _{z}^{B,1}(v ^{B,0}+v ^{I,0}(0,t)),\\[2mm]
 \displaystyle v _{\xi \xi}^{b,0} =v _{t}^{b,0} +(\partial _{x}\varphi ^{I,0}(1,t)+M)v ^{b,0}+\varphi _{z}^{b,1}(v ^{b,0}+v ^{I,0}(1,t)).
\end{cases}
\end{gather}
Plugging \eqref{v-ZZ} into $ \mathcal{R}_{2}^{\varepsilon} $ in \eqref{E-rror-fomula}, recalling the definition of $ \Phi^{A} $ and $ V^{A} $, we have
\begin{align*}
\displaystyle \mathcal{R}_{2}^{\varepsilon}&=    -\left[ v ^{ B,0}( \varphi _{x}^{I,0}(x,t)-\varphi _{x} ^{I,0}(0,t)-x\varphi _{xx} ^{I,0}(0,t))+v ^{ b,0}( \varphi _{x} ^{I,0}(x,t)- \varphi _{x}^{I,0}(1,t)- (x-1)\varphi _{xx}^{I,0}(1,t))  \right]
   \nonumber \\
   & \displaystyle  \quad - \varepsilon ^{1/2}\left[v ^{ B,0}( \varphi _{x}^{I,1}(x,t)-\varphi _{x} ^{I,1}(0,t))  +v ^{ b,0}( \varphi _{x} ^{I,1}(x,t)- \varphi _{x}^{I,1}(1,t))\right]
    \nonumber \\
    & \displaystyle  \quad  - \varepsilon ^{1/2}\left[v ^{B,1}( \varphi _{x}^{I,0}(x,t)-\varphi _{x} ^{I,0}(0,t))  +v ^{b,1}( \varphi _{x} ^{I,0}(x,t)- \varphi _{x}^{I,0}(1,t))\right]
    \nonumber \\
        &\displaystyle \quad -\varepsilon ^{1/2} \left[ \varphi _{x}^{ B,1}(v ^{I,0}-v ^{I,0}(0,t)-xv _{x}^{I,0}(0,t))+\varphi _{x}^{ b,1}\left( v ^{I,0}-v ^{I,0}(1,t)-(x-1)v _{x}^{I,0}(1,t) \right) \right]
               \nonumber \\
               & \displaystyle \quad -\varepsilon \left[ \varphi _{x}^{ B,1}(v ^{I,1}(x,t)-v ^{I,1}(0,t))+\varphi _{x}^{ b,1}(v ^{I,1}(x,t)-v ^{I,1}(1,t)) \right]
                \nonumber \\
&\displaystyle  \quad-\varepsilon \left[ \varphi _{x}^{B,2}(v ^{I,0}(x,t)-v ^{I,0}(0,t))+ \varphi _{x}^{b,2}(v ^{I,0}(x,t)-v ^{I,0}(1,t))\right]
                 \nonumber \\
        & \quad -\varepsilon \varphi _{x} ^{I,1}(v ^{I,1}+ v ^{B,1}+v ^{b,1}) -\varepsilon ^{1/2}( \varphi _{x}^{ B,1}v ^{ b,0}+\varphi _{x}^{ b,1}v ^{ B,0} ) - \varepsilon ( \varphi _{x}^{ B,1}v ^{b,1}+\varphi _{x}^{ b,1}v ^{B,1} )
        \nonumber \\
& \displaystyle \quad - \varepsilon (  \varphi _{x}^{B,2} v ^{ b,0}+\varphi _{x}^{b,2} v ^{\varepsilon B,0} )  -\varepsilon ^{3/2}  \varphi _{x}^{B,2}(v ^{I,1}+v ^{B,1}+v ^{b,1})  - \varepsilon ^{3/2} \varphi _{x}^{b,2}  (v ^{I,1}+v ^{B,1}+v ^{b,1})
         \nonumber \\
&\displaystyle \quad-b _{v}^{\varepsilon}[\varphi _{x}^{I,0}+M+\varepsilon ^{1/2}(\varphi_{x} ^{I,1}+\varphi_{x}^{ B,1}+\varphi_{x} ^{b,1})+\partial _{x}b _{\varphi}^{\varepsilon}  ]
 \nonumber \\
 &\displaystyle \quad
 - \partial _{x}b _{\varphi}^{\varepsilon}(v ^{I,0}+v ^{B,0}+v ^{b,0}+\varepsilon ^{\frac{1}{2}}(v ^{I,1}+ v ^{B,1}+v ^{b,1}))
           +\left[ \varepsilon v _{xx}^{I,0}+\varepsilon ^{3/2} v _{xx}^{I,1} \right] -\partial _{t} b _{v}^{\varepsilon}  =:\sum _{i=1}^{16}\mathcal{K}_{i}.
\end{align*}
To prove \eqref{con-R-2-ESTI-in-lem}, it suffices to establish estimates for $ \mathcal{K}_{i}~(1 \leq i \leq 15) $. The proof is quite similar to the one for Lemma \ref{lem-R-1-VE}. We first prove $ \|\mathcal{R}_{2}^{\varepsilon}\|_{L _{T}^{\infty}L ^{2}} \leq c(v _{\ast},T) \varepsilon ^{3/4} $. By \eqref{half-l-infty}, \eqref{inte-transfer} and Taylor's formula, we get
 \begin{align*}
 \displaystyle \|\mathcal{K}_{1}\|_{L _{T}^{\infty}L ^{2}}& \leq c _{0} \varepsilon \|\partial _{x}^{3}\varphi ^{I,0}\|_{L _{T}^{\infty}L ^{\infty}}\left( \|v ^{B,0}\|_{L _{T}^{\infty}L _{z}^{2}}+\|v ^{b,0}\|_{L _{T}^{\infty}L _{\xi}^{2}} \right)  \leq c(v _{\ast},T) \varepsilon ^{3/4}.
 \end{align*}
 Similar arguments imply that $ \|\mathcal{K}_{i}\| _{L _{T}^{\infty}L ^{2}}\leq c(v _{\ast},T) \varepsilon ^{3/4} $ for $ i=2,3,4,5,6 $. From \eqref{con-vfi-v-I-0-regula}, \eqref{some-l-infty-layer-for-vfi-b-1-stab}--\eqref{l-infty-v-bd-stat} and \eqref{inte-transfer}, we get
 \begin{align*}
 \displaystyle \| \mathcal{K} _{7}\|_{L _{T}^{\infty}L ^{2}}&  \leq \varepsilon \|\varphi _{x}^{I,1}\|_{L _{T}^{\infty}L ^{\infty}}\left( \|v ^{I,1}\|_{L _{T}^{\infty}L ^{2}}+\|v ^{B,1}\|_{L _{T}L ^{2}}+\|v ^{b,1}\|_{L _{T}^{\infty}L ^{2}} \right) \leq c(v _{\ast},T) \varepsilon ,
 \end{align*}
 where the constraint $ 0< \varepsilon<1 $ has been used. Analogously, we further have that
\begin{align*}
 \displaystyle   \|\mathcal{K}_{i}\| _{L _{T}^{\infty}L ^{2}} \leq c(v _{\ast},T) \varepsilon ^{5/4}, \ \ \ i=11, 12,
\end{align*}
and
\begin{align*}
\displaystyle \|\mathcal{K}_{13}\|_{L _{T}^{\infty}L ^{2}} &\leq c _{0}\left(1+\|\varphi _{x}^{I,0}\|_{L _{T}^{\infty} L ^{2}}+\varepsilon ^{\frac{1}{2}}\|\varphi _{x}^{I,1}\|_{L _{T}^{\infty}L ^{2}}+\varepsilon ^{\frac{1}{4}}\|\varphi _{z}^{ B,1}\|_{L _{T}^{\infty}L _{z}^{2}} \right.
        \nonumber \\
        &\qquad \qquad ~\qquad \qquad\left. +\varepsilon ^{\frac{1}{4}}\|\varphi _{\xi}^{ b,1}\|_{L _{T}^{\infty}L _{\xi}^{2}}+\|\partial _{x}b _{\varphi}^{\varepsilon}\|_{L _{T}^{\infty}L ^{2}}\right)\|b _{v}^{\varepsilon}\|_{L _{T}^{\infty}L ^{\infty}}  \leq c(v _{\ast},T) \varepsilon,\\ \displaystyle  \|\mathcal{K}_{14}\|_{L _{T}^{\infty}L ^{2}} &\leq c(v _{\ast},T)\|\partial _{x}b _{\varphi}^{\varepsilon}\|_{L _{T}^{\infty}L ^{2}}\left( 1+ \varepsilon ^{\frac{1}{2}}( \|v ^{I,1}\|_{L _{T}^{\infty}L ^{\infty}}+\|v ^{B,1}\|_{L _{T}^{\infty}L ^{\infty}}+\|v ^{b,1}\|_{L _{T}^{\infty}L ^{\infty}}) \right)
         \nonumber \\
         & \displaystyle \leq c(v _{\ast},T)  \varepsilon .
\end{align*}
Recalling the arguments in \eqref{P-6-1} and \eqref{p-6-2}, we proceed to estimate $ \|\mathcal{K}_{8}\|_{L _{T} ^{\infty}L ^{2}} $ as follows:
\begin{align*}
\displaystyle  \|\mathcal{K}_{8}\|_{L _{T}^{\infty}L ^{2}}& \leq \varepsilon ^{1/2}\|v ^{b,0}\varphi _{x}^{B,1}\|_{L _{T}^{\infty}L ^{2}}+\varepsilon ^{1/2}\|v ^{B,0}\varphi _{x}^{b,1}\|_{L _{T}^{\infty}L ^{2}}
            \nonumber \\
&  \displaystyle \leq  c _{0}\varepsilon ^{\frac{2m+1}{4}}\left( \|\varepsilon ^{-m/2}v ^{b,0}\|_{L ^{\infty}((0,1/2)\times (0,T))}\|\varphi _{z}^{B,1}\|_{L _{T}^{\infty}L _{z} ^{2}} \right.
   \nonumber \\
&\displaystyle  \qquad \qquad \  \left. \quad+\|\varepsilon ^{-m/2}\varphi _{x}^{B,1}\|_{L ^{\infty}((1/2,1)\times (0,T))}\|v^{b,0}\|_{L _{T}^{\infty}L _{\xi}^{2}} \right)
    \nonumber \\
    & \displaystyle \quad  + c _{0}\varepsilon ^{\frac{2m+1}{4}}\left( \|\varepsilon ^{-m/2}\varphi _{x}^{b,1}\|_{L ^{\infty}((0,1/2 )\times(0,T))}\|v ^{B,0}\|_{L _{T}^{\infty}L _{z} ^{2}} \right.
   \nonumber \\
   &\displaystyle  \qquad \qquad \quad \quad \left.+\|\varepsilon ^{-m/2}v ^{B,0}\|_{L ^{\infty}((1/2,1)\times (0,T))}\|\varphi _{\xi}^{b,1}\|_{L _{T}^{\infty}L _{\xi}^{2}} \right)
    \nonumber \\
    & \displaystyle \leq c(v _{\ast},T)\varepsilon ^{\frac{2m+1}{4}}
\end{align*}
for any integer $ m \geq 1 $, where we have used \eqref{half-l-infty}, \eqref{vfi-B-1-HALF-INTY}, \eqref{inte-transfer} and $ 0< \varepsilon<1 $. Thus, $ \|\mathcal{K}_{8}\|_{L _{T}^{\infty}L ^{2}} \leq c(v _{\ast},T) \varepsilon ^{3/4} $. Similarly, we have also $ \|\mathcal{K}_{i}\|_{L _{T} ^{\infty}L ^{2}} \leq c(v _{\ast},T) \varepsilon ^{3/4} $ for $ i=9,10 $. For $ \mathcal{K}_{15} $ and $ \mathcal{K}_{16} $, it follows from \eqref{con-vfi-v-I-0-regula}, \eqref{vfi-I-1-V-i-1} and \eqref{b-v-ve-esti} that
\begin{align*}
\displaystyle \|\mathcal{K}_{15}\| _{L _{T}^{\infty}L ^{2}} \leq \varepsilon \|v _{xx}^{I,0}\|_{L _{T}^{\infty}L ^{2}}+ \varepsilon ^{\frac{3}{2}}\|v _{xx}^{I,1}\|_{L _{T}^{\infty}L ^{2}} \leq c(v _{\ast},T) \varepsilon, \ \  \|\mathcal{K}_{16}\| _{L _{T}^{\infty}L ^{2}} \leq c(v _{\ast},T) \varepsilon,
\end{align*}
where we have used $ 0< \varepsilon<1 $. Therefore we conclude that
\begin{gather*}
 \displaystyle \|\mathcal{R}_{2}^{\varepsilon}\|_{L _{T} ^{\infty}L ^{2}} \leq c _{0} \sum _{i=1}^{16}\|\mathcal{K}_{i}\|_{L _{T}^{\infty}L ^{2}} \leq c(v _{\ast},T) \varepsilon ^{3/4}.
 \end{gather*}
Repeating the above arguments with $ \|\cdot\|_{L _{T}^{\infty}L ^{2}} $ replaced by $ \|\cdot\|_{L _{T}^{\infty}L ^{\infty}} $, from \eqref{con-vfi-v-I-0-regula}, \eqref{some-l-infty-layer-for-vfi-b-1-stab}--\eqref{l-infty-v-bd-stat}, \eqref{half-l-infty}, \eqref{b-v-ve-esti-1}--\eqref{b-time-esti} and \eqref{inte-transfer}, one can deduce that
 \begin{gather*}
 \displaystyle \|\mathcal{R}_{2} ^{\varepsilon}\|_{L _{T}^{\infty}L ^{\infty}} \leq c(v _{\ast},T) \varepsilon ^{1/2}.
 \end{gather*}
Finally, in view of \eqref{product-tim}, the above estimates for $ \mathcal{K}_{i} ~(1 \leq i \leq 16)$ and Lemmas \ref{lem-regul-outer-layer-0}--\ref{lem-v-b-1}, we have that $ \|\partial _{t}\mathcal{R}_{2}^{\varepsilon} \| _{L _{T}^{\infty}L ^{2}} \leq c(v _{\ast},T) \varepsilon ^{3/4}$. This ends the proof of Lemma \ref{lem-R-2-ESTI}.
\end{proof}

\subsubsection{Lower-order estimates} 
\label{ssub:lower_order_estimates}
From now on, we shall establish some uniform-in-$ \varepsilon $ estimates for $ (\Phi ^{\varepsilon}, V ^{\varepsilon}) $. Throughout this section, we assume that $ (\Phi ^{\varepsilon}, V^{\varepsilon}) $ satisfies for any $ T>0 $,
\begin{gather}\label{apriori-assumption}
\displaystyle \sup _{t \in [0,T]} \|\Phi ^{\varepsilon}(\cdot, t)\|_{L ^{\infty}}^{2} \leq \delta,
\end{gather}
where $ \delta>0 $ is a small constant to be determined later, and may depend on $ T $. The results in Lemmas \ref{lem-regul-outer-layer-0}--\ref{lem-v-b-1}, \ref{lem-R-1-VE}, \ref{lem-R-2-ESTI} will be frequently used in the subsequent analysis without further clarification. We emphasize that these estimates are all independent of $ \delta $. We begin with the $ L ^{2} $ estimates of $ (\Phi ^{\varepsilon}, V ^{\varepsilon}) $.

\begin{lemma}\label{lem-L2-pertur}
Let the conditions in Proposition \ref{prop-pertur} hold. Assume $ 0< \varepsilon<1 $ and that the solution $ (\Phi ^{\varepsilon},V ^{\varepsilon}) $ to \eqref{eq-perturbation} on $ [0,T] $ satisfies \eqref{apriori-assumption}. Then there exist a positive constant $ \delta _{1}>0 $  independent of $ \varepsilon $ and $ \delta $ such that for any $ t \in [0,T] $,
\begin{gather}\label{con-L2-PERT}
\|V ^{\varepsilon}(\cdot,t)\|_{L ^{2}}^{2}+ \|\Phi ^{\varepsilon}(\cdot,t)\|_{L ^{2}}^{2}+ \int _{0}^{t}\left( \|V ^{\varepsilon}\|_{L ^{2}}^{2}+\varepsilon \|V _{x}^{\varepsilon}\|_{L ^{2}}^{2}+\|\Phi _{x}^{\varepsilon}\|_{L ^{2}}^{2} \right) \mathrm{d}\tau \leq c(v _{\ast},T) \varepsilon ^{1/2},
\end{gather}
provided $ \delta \leq \delta _{1} $ and $ K _{1}(T,v _{\ast})v _{\ast} \leq 1/16 $, where $ K _{1}(T,v _{\ast}) $ is given in \eqref{five-l2}, $ c(v _{\ast},T)>0 $ is a constant depending on $ T $ but independent of $ \varepsilon $ and $ \delta $.
\end{lemma}
\begin{proof}
  Multiplying the first equation in \eqref{eq-perturbation} by $\Phi ^{\varepsilon} $, followed by an integration over $ \mathcal{I} $ and integration by parts, one deduces that
\begin{align}\label{L-2-diff-id}
&\displaystyle \frac{1}{2}\frac{\mathrm{d}}{\mathrm{d}t}\int _{\mathcal{I} }\vert \Phi ^{\varepsilon}\vert ^{2} \mathrm{d}x +\int _{\mathcal{I} }\vert \Phi _{x}^{\varepsilon}\vert ^{2}\mathrm{d}x
 \nonumber \\
 &~\displaystyle = - \varepsilon ^{1/2}\int _{\mathcal{I} }\Phi ^{\varepsilon}\Phi _{x}^{\varepsilon}V _{x}^{\varepsilon} \mathrm{d}x +\int _{\mathcal{I} }\varepsilon ^{-1/2}\mathcal{R}_{1}^{\varepsilon}\Phi^{\varepsilon} \mathrm{d}x- \int _{\mathcal{I} }(\Phi_{x}^{A}+M)V _{x}^{\varepsilon}\Phi ^{\varepsilon} \mathrm{d}x-\int _{\mathcal{I} }\Phi ^{\varepsilon}\Phi _{x}^{\varepsilon}V_{x}^{A}  \mathrm{d}x.
\end{align}
The terms on the right hand side of \eqref{L-2-diff-id} can be treated as follows. Thanks to \eqref{apriori-assumption}, Lemma \ref{lem-R-1-VE} and the Cauchy-Schwarz inequality, it holds that
\begin{align}\label{firs-l2-nonlinear}
\displaystyle -\varepsilon ^{1/2}\int _{\mathcal{I} } \Phi ^{\varepsilon}\Phi _{x}^{\varepsilon}V _{x}^{\varepsilon}\mathrm{d}x &\leq \frac{\varepsilon}{8}\int _{\mathcal{I} }\left\vert V _{x}^{\varepsilon}\right\vert ^{2}\mathrm{d}x+ c _{0}\|\Phi ^{\varepsilon}\|_{L ^{\infty}}^{2} \int _{\mathcal{I} }\vert \Phi _{x}^{\varepsilon}\vert ^{2} \mathrm{d}x
 \nonumber \\
 \displaystyle&
\leq \frac{\varepsilon}{8}\int _{\mathcal{I} }\left\vert V _{x}^{\varepsilon}\right\vert ^{2}\mathrm{d}x
+c_{0}\delta \int _{\mathcal{I} }\vert \Phi _{x}^{\varepsilon}\vert ^{2} \mathrm{d}x,
\end{align}
and that
\begin{align*}
\displaystyle \int _{\mathcal{I} }\varepsilon ^{-1/2}\mathcal{R}_{1}^{\varepsilon}\Phi^{\varepsilon} \mathrm{d}x \leq c _{0} \|\Phi ^{\varepsilon}\|_{L ^{2}}^{2}+ c _{0}\varepsilon ^{-1}\|\mathcal{R}_{1}^{\varepsilon}\|_{L ^{2}}^{2} \leq  c _{0} \|\Phi ^{\varepsilon}\|_{L ^{2}}^{2}+ c(v _{\ast},T) \varepsilon ^{1/2}.
\end{align*}
Hereafter the constant $ c(v _{\ast},T)>0 $ is independent of $ \varepsilon $ and $ \delta $. The integration by parts along with \eqref{con-vfi-v-I-0-regula}, \eqref{some-l-infty-layer-for-vfi-b-1-stab}, \eqref{some-l-infty-layer-for-stab}, \eqref{esti-approx-fida-2a}, the fact $ \partial _{x}^{2}b  _{\varphi}^{\varepsilon}=0 $, the Hardy inequality \eqref{hardy} and the Cauchy-Schwarz inequality gives
\begin{align*}
&\displaystyle - \int _{\mathcal{I} }(\Phi_{x}^{A}+M)V _{x}^{\varepsilon}\Phi ^{\varepsilon} \mathrm{d}x = \int _{\mathcal{I}}V ^{\varepsilon}\Phi ^{\varepsilon}\Phi _{xx}^{A}\mathrm{d}x +\int _{\mathcal{I}} (\Phi_{x}^{A}+M)V ^{\varepsilon}\Phi _{x}^{\varepsilon} \mathrm{d}x
 \nonumber \\
 &~\displaystyle \leq \varepsilon ^{1/2}\int _{\mathcal{I}}V ^{\varepsilon}\Phi ^{\varepsilon}(\varphi _{xx}^{B,1}+\varphi _{xx}^{b,1})\mathrm{d}x + \int _{\mathcal{I}}V ^{\varepsilon}\Phi ^{\varepsilon}\left( \varphi _{xx}^{I,0}+\varepsilon ^{1/2}\varphi _{xx}^{I,1}+\varepsilon \varphi _{xx}^{B,2}+\varepsilon \varphi _{xx}^{b,2}+\partial _{x}^{2} b _{\varphi}^{\varepsilon}\right) \mathrm{d}x
  \nonumber \\
  &~\displaystyle \quad + c _{0} \|\Phi_{x}^{A}+M\|_{L ^{\infty}} \|V ^{\varepsilon}\|_{L ^{2}}\|\Phi _{x}^{\varepsilon}\|_{L ^{2}}
  \nonumber \\
  &~\displaystyle \leq \|V ^{\varepsilon}\|_{L ^{2}} \left\|\frac{\Phi ^{\varepsilon}}{x(1-x)}\right\|_{L ^{2}}\left( \left\|\frac{x(1-x)}{\varepsilon ^{1/2}}\varphi _{zz}^{B,1}\right\|_{L ^{\infty}}+\left\|\frac{x(1-x)}{\varepsilon ^{1/2}}\varphi _{\xi \xi}^{b,1}\right\|_{L ^{\infty}} \right)
    +c(v _{\ast},T) \|V ^{\varepsilon}\|_{L ^{2}}\|\Phi _{x}^{\varepsilon}\|_{L ^{2}}
   \nonumber \\
   &~\displaystyle \quad+c(v _{\ast},T) \|V ^{\varepsilon}\|_{L ^{2}}\|\Phi ^{\varepsilon}\|_{L ^{2}}\left( \|\varphi _{xx}^{I,0}\|_{L ^{\infty}}+\varepsilon ^{1/2}\|\varphi _{xx}^{I,1}\|_{L ^{\infty}}+\|\varphi _{zz}^{B,2}\|_{L _{z}^{\infty}}+\|\varphi _{\xi \xi}^{b,2}\|_{L _{\xi}^{\infty}}+\|\partial _{x}^{2}b _{\varphi}^{\varepsilon}\|_{L ^{\infty}} \right)
     \nonumber \\
     &~ \displaystyle \leq c(v _{\ast},T) \|V ^{\varepsilon}\|_{L ^{2}}\|\Phi _{x}^{\varepsilon}\|_{L ^{2}} \left( \|\langle z \rangle \varphi _{zz}^{B,1}\|_{L _{z}^{\infty}}+\|\langle \xi \rangle \varphi _{\xi \xi}^{b,1}\|_{L _{\xi}^{\infty}} \right)
     +c(v _{\ast},T) \|V ^{\varepsilon}\|_{L ^{2}}\left( \|\Phi _{x}^{\varepsilon}\|_{L ^{2}}+\|\Phi ^{\varepsilon}\|_{L ^{2}}\right)
      \nonumber \\
      & ~\displaystyle
     \leq \frac{1}{8}\|\Phi _{x}\|_{L ^{2}}^{2}+c(v _{\ast},T) (\|\Phi ^{\varepsilon}\|_{L ^{2}}^{2}+\|V ^{\varepsilon}\|_{L ^{2}}^{2}).
\end{align*}
For the last term on the right hand side of \eqref{L-2-diff-id}, from \eqref{con-vfi-v-I-0-regula}, \eqref{vfi-I-1-V-i-1}, \eqref{approximate-va}, \eqref{l-infty-v-bd-stat}, \eqref{some-l-infty-layer-for-v-b-0-stab}, \eqref{Sobolev-z-xi}, the Hardy inequality \eqref{hardy} and the Cauchy-Schwarz inequality, we get
\begin{align}\label{five-l2}
&\displaystyle -\int _{\mathcal{I} }\Phi ^{\varepsilon}\Phi _{x}^{\varepsilon}V_{x}^{A}  \mathrm{d}x = -\int _{\mathcal{I}}\Phi ^{\varepsilon}\Phi _{x}^{\varepsilon} \left( v _{x}^{B,0}+ v _{x}^{b,0} \right)\mathrm{d}x -\int _{\mathcal{I}}   \Phi ^{\varepsilon}\Phi _{x}^{\varepsilon}\left[v _{x} ^{I,0} +\varepsilon ^{1/2}(v _{x}^{I,1}+ v _{x}^{B,1}+ v _{x}^{b,1}) +\partial _{x}b _{v}^{\varepsilon}\right]
 \nonumber \\
 &~\displaystyle \leq \left\|\frac{\Phi ^{\varepsilon}}{x(1-x)}\right\|_{L ^{2}}\|\Phi _{x}^{\varepsilon}\|_{L ^{2}}\left( \left\|\frac{x(1-x)}{\varepsilon ^{1/2}}v _{z}^{B,0}\right\|_{L ^{\infty}}+ \left\|\frac{x(1-x)}{\varepsilon ^{1/2}}v _{\xi}^{b,0}\right\|_{L ^{\infty}} \right)
  \nonumber \\
  &~\displaystyle \quad+ \|\Phi ^{\varepsilon}\|_{L ^{2}}\|\Phi _{x}^{\varepsilon}\|_{L ^{2}} \left( \|v _{x}^{I,0}\|_{L ^{\infty}}+\|v _{z}^{B,1}\|_{L _{z}^{\infty}}+\|v _{\xi}^{b,1}\|_{L _{\xi}^{\infty}}+\varepsilon ^{1/2}\|v _{x}^{I,1}\|_{L ^{\infty}}+\|\partial _{x}b _{v}^{\varepsilon}\|_{L ^{\infty}} \right)
   \nonumber \\
   &~\displaystyle \leq c _{0} \|\Phi _{x}^{\varepsilon}\|_{L ^{2}}^{2}\left( \|\langle z \rangle v _{z}^{B,0}\| _{L _{z}^{\infty}}+\|\langle \xi \rangle v _{\xi}^{b,0}\|_{L _{\xi}^{\infty}} \right)+ c(v _{\ast},T) \|\Phi _{x}^{\varepsilon}\|_{L ^{2}}\|\Phi ^{\varepsilon}\|_{L ^{2}}
    \nonumber \\
    &~ \displaystyle \leq \Big( \frac{1}{8}+ K _{1}(T,v _{\ast}) v _{\ast} \Big) \|\Phi _{x}^{\varepsilon}\|_{L ^{2}}^{2}+c(v _{\ast},T) \|\Phi ^{\varepsilon}\|_{L ^{2}}^{2},
\end{align}
where $ K _{1}(T,v _{\ast}):=K _{0}(T,v _{\ast})c_{0}>0 $ is constant with $ K _{0}(v _{\ast},T) $ as in \eqref{some-l-infty-layer-for-v-b-0-stab}. Thus, plugging \eqref{firs-l2-nonlinear}--\eqref{five-l2} into \eqref{L-2-diff-id}, it follows that
\begin{align}\label{fida-l2-final-esti}
\displaystyle  \frac{\mathrm{d}}{\mathrm{d}t}\int _{\mathcal{I} }\vert\Phi ^{\varepsilon} \vert ^{2} \mathrm{d}x +\int _{\mathcal{I} }\vert \Phi _{x}^{\varepsilon}\vert ^{2} \mathrm{d}x \leq \frac{\varepsilon}{4}\int _{\mathcal{I} }\vert V _{x}^{\varepsilon}\vert ^{2}\mathrm{d}x+c(v _{\ast},T) (\|\Phi ^{\varepsilon}\|_{L ^{2}}^{2}+\|V ^{\varepsilon}\|_{L ^{2}}^{2}) +c(v _{\ast},T) \varepsilon ^{1/2},
\end{align}
provided $ c _{0} \delta \leq 1/8 $ in \eqref{firs-l2-nonlinear} and $ K _{1}(T,v _{\ast}) v _{\ast} \leq 1/8 $ in \eqref{five-l2}. To proceed, multiplying the second equation in \eqref{eq-perturbation} by $ V ^{\varepsilon} $ and integrating the resulting equation over $ \mathcal{I} $, we get
\begin{align}\label{V-L2-differ}
\displaystyle \frac{1}{2}\frac{\mathrm{d}}{\mathrm{d}t}\int _{\mathcal{I} }\vert V ^{\varepsilon}\vert ^{2}\mathrm{d}x+ \varepsilon\int _{\mathcal{I} }\vert V _{x}^{\varepsilon}\vert ^{2} \mathrm{d}x&=- \int _{\mathcal{I} }\varepsilon ^{1/2}\Phi _{x}^{\varepsilon}V ^{\varepsilon} V ^{\varepsilon} \mathrm{d}x- \int _{\mathcal{I} }\Phi _{x}^{\varepsilon}V^{A} V ^{\varepsilon} \mathrm{d}x
 \nonumber \\
 & \displaystyle \quad-\int _{\mathcal{I} }(\Phi_{x}^{A}+M)\vert V ^{\varepsilon}\vert ^{2} \mathrm{d}x+ \int _{\mathcal{I} }V ^{\varepsilon}\varepsilon ^{-1/2}\mathcal{R}_{2}^{\varepsilon} \mathrm{d}x= \sum _{i=1}^{4}\mathcal{N}_{i},
\end{align}
where, due to \eqref{apriori-assumption}, integration by parts and the Cauchy-Schwarz inequality,
\begin{align}\label{N-1}
\displaystyle \mathcal{N}_{1} &=  \varepsilon ^{1/2}\int _{\mathcal{I} }\Phi ^{\varepsilon}V _{x}^{\varepsilon}V ^{\varepsilon} \mathrm{d}x \leq \frac{\varepsilon}{8}\int _{\mathcal{I} }\vert V _{x}^{\varepsilon}\vert ^{2}\mathrm{d}x+ c(v _{\ast},T) \|\Phi ^{\varepsilon}\|_{L ^{\infty}}^{2}\int _{\mathcal{I} }\vert V ^{\varepsilon}\vert ^{2} \mathrm{d}x \nonumber \\
 &\displaystyle\leq \frac{\varepsilon}{8}\int _{\mathcal{I} }\vert V _{x}^{\varepsilon}\vert ^{2}\mathrm{d}x+c(v _{\ast},T) \|V ^{\varepsilon}\|_{L ^{2}}^{2},
\end{align}
provided $ \delta<1 $. For $ \mathcal{N}_{2} $ and $ \mathcal{N}_{3} $, the estimates in \eqref{esti-approx-fida-2} along with the Cauchy-Schwarz inequality yields that
\begin{align*}
\displaystyle \mathcal{N}_{2}& \leq  \|V ^{A}\|_{L ^{\infty}}\|\Phi _{x}^{\varepsilon}\|_{L ^{2}}\|V ^{\varepsilon}\|_{L ^{2}} \leq \frac{1}{8} \|\Phi _{x}^{\varepsilon}\|_{L ^{2}}^{2} +c(v _{\ast},T)\|V ^{\varepsilon}\|_{L ^{2}}  ^{2},\ \ \ \mathcal{N}_{3} \leq c(v _{\ast},T) \|V ^{\varepsilon}\|_{L ^{2}}^{2}.
\end{align*}
For the last term $ \mathcal{N}_{4} $, by the Cauchy-Schwarz inequality, one has $ \mathcal{N}_{4} \leq \|V ^{\varepsilon}\|_{L ^{2}} ^{2}+c(v _{\ast},T) \varepsilon ^{-1}\|\mathcal{R}_{2}^{\varepsilon}\| _{L ^{2}} ^{2}$. Inserting the estimates on $ \mathcal{N}_{i}~(1 \leq i \leq 4) $ into \eqref{V-L2-differ}, we get by virtue of Lemma \ref{lem-R-2-ESTI} that
\begin{align}\label{V-l2-ineq}
\displaystyle  &\displaystyle \frac{1}{2}\frac{\mathrm{d}}{\mathrm{d}t}\int _{\mathcal{I} }\vert V ^{\varepsilon}\vert ^{2}\mathrm{d}x+ \varepsilon\int _{\mathcal{I} }\vert V _{x}^{\varepsilon}\vert ^{2} \mathrm{d}x \leq \frac{1}{4}\int _{\mathcal{I} }\vert \Phi _{x}^{\varepsilon}\vert ^{2}\mathrm{d}x +c(v _{\ast},T)\|V ^{\varepsilon}\|_{L ^{2}}  ^{2}+c(v _{\ast},T) \varepsilon ^{1/2}.
\end{align}
Combining \eqref{V-l2-ineq} with \eqref{fida-l2-final-esti}, we obtain that
\begin{align*}
&\displaystyle \frac{\mathrm{d}}{\mathrm{d}t}\int _{\mathcal{I} } \left( \vert V ^{\varepsilon}\vert ^{2} +\vert \Phi ^{\varepsilon}\vert ^{2}\right) \mathrm{d}x+ \int _{\mathcal{I} }\left( \vert \Phi _{x}^{\varepsilon}\vert ^{2}+ \varepsilon\vert V _{x}^{\varepsilon}\vert ^{2}\right)  \mathrm{d}x
 \leq c(v _{\ast},T)\int _{\mathcal{I} } \left( \vert V ^{\varepsilon}\vert ^{2} +\vert \Phi ^{\varepsilon}\vert ^{2}\right) \mathrm{d}x+c(v _{\ast},T) \varepsilon ^{1/2},
\end{align*}
which along with the Gronwall inequality entails for any $ t \in [0,T] $ that
\begin{gather*}
\displaystyle  \|V ^{\varepsilon}(\cdot,t)\|_{L ^{2}}^{2}+ \|\Phi ^{\varepsilon}(\cdot,t)\|_{L ^{2}}^{2}+ \int _{0}^{t}\left(\varepsilon \|V _{x}^{\varepsilon}\| _{L ^{2}}^{2}+\|\Phi _{x}^{\varepsilon}\|_{L ^{2}}^{2} \right) \mathrm{d}\tau \leq c(v _{\ast},T) \varepsilon ^{1/2}.
\end{gather*}
The proof of Lemma \ref{lem-L2-pertur} is complete.
\end{proof}

We proceed to establish $ H ^{1} $ estimate for $ (\Phi ^{\varepsilon}, V ^{\varepsilon}) $.
\begin{lemma}\label{lem-V-x-per}
Under the conditions of Lemma \ref{lem-L2-pertur}, it holds that
\begin{align}\label{con-fida-v-x-lem}
\displaystyle  \varepsilon\|\Phi _{x}^{\varepsilon}(\cdot,t)\|_{L ^{2}}^{2}+\varepsilon\|V _{x}^{\varepsilon}(\cdot,t)\|_{L ^{2}}^{2}+\int _{0}^{t}\left( \varepsilon\|\Phi _{\tau}^{\varepsilon}\| _{ L ^{2}}^{2}+\|V _{\tau}^{\varepsilon}\| _{ L ^{2}}^{2}\right) \mathrm{d}\tau \leq c(v _{\ast},T) \varepsilon ^{1/2},\ \ \ \forall t \in (0,T],
\end{align}
where the constant $ c(v _{\ast},T)>0 $ is independent of $ \varepsilon $ and $ \delta $.

\end{lemma}
\begin{proof}
  Multiplying the second equation in \eqref{eq-perturbation} by $ V _{t}^{\varepsilon} $ and integrating the resulting equation over $ \mathcal{I} $, we have
\begin{align}\label{V-x-diff-iden}
&\displaystyle  \frac{\varepsilon}{2}\frac{\mathrm{d}}{\mathrm{d}t}\int _{\mathcal{I}}\left\vert V _{x}^{\varepsilon}\right\vert ^{2}  \mathrm{d}x+\int _{\mathcal{I} }\left\vert V _{t} ^{\varepsilon}\right\vert ^{2} \mathrm{d}x
 \nonumber \\
 &~\displaystyle = -\int _{\mathcal{I}}(\Phi_{x}^{A}+M)V ^{\varepsilon}V _{t}^{\varepsilon}\mathrm{d}x - \int _{\mathcal{I} }\Phi _{x}^{\varepsilon}V ^{A} V _{t}^{\varepsilon} \mathrm{d}x-\varepsilon ^{1/2} \int _{\mathcal{I}}\Phi _{x}^{\varepsilon}V ^{\varepsilon}V_{t}^{\varepsilon}\mathrm{d}x
 + \varepsilon ^{-1/2}\int _{\mathcal{I}}\mathcal{R}_{2}^{\varepsilon} V_{t}^{\varepsilon}\mathrm{d}x,
\end{align}
where, due to $ \|\Phi _{x}^{A}\|_{L _{T} ^{\infty}L ^{\infty}} \leq c(v _{\ast},T) $ from \eqref{esti-approx-fida-2a} and the Cauchy-Schwarz inequality,
\begin{gather}\label{v-x-sti-one}
\displaystyle   -\int _{\mathcal{I}}(\Phi_{x}^{A}+M)V ^{\varepsilon}V _{t}^{\varepsilon}\mathrm{d}x  \leq \|\Phi_{x}^{A}+M\|_{L ^{\infty}} \|V ^{\varepsilon}\|_{L ^{2}} \|V _{t}^{\varepsilon}\|_{L ^{2}} \leq \frac{1}{4}\|V _{t}^{\varepsilon}\|_{L ^{2}}^{2}+ c(v _{\ast},T)\|V ^{\varepsilon}\|_{L ^{2}}^{2}.
\end{gather}
By \eqref{esti-approx-fida-2b} and the Cauchy-Schwarz inequality, we get
\begin{align}\label{v-x-esti-first}
\displaystyle   -\int _{\mathcal{I} }\Phi _{x}^{\varepsilon}V ^{A} V _{t}^{\varepsilon} \mathrm{d}x \leq \|V ^{A}\|_{L ^{\infty}}\|\Phi _{x}^{\varepsilon}\|_{L ^{2}}\|V _{t}^{\varepsilon}\|_{L ^{2}} \leq \frac{1}{4}\|V _{t} ^{\varepsilon}\|_{L ^{2}}^{2}+c(v _{\ast},T) \|\Phi _{x}^{\varepsilon}\|_{L ^{2}}^{2}.
\end{align}
Thanks to \eqref{con-L2-PERT}, the Cauchy-Schwarz inequality, the Sobolev inequality \eqref{Sobolev-infty} and Lemma \ref{lem-R-2-ESTI}, we deduce that
\begin{align*}
\displaystyle -\varepsilon ^{1/2} \int _{\mathcal{I}}\Phi _{x}^{\varepsilon}V ^{\varepsilon}V_{t}^{\varepsilon}\mathrm{d}x &\leq \frac{1}{8}\int _{\mathcal{I}} \vert V _{t}^{\varepsilon}\vert ^{2} \mathrm{d}x+c(v _{\ast},T) \varepsilon \|V ^{\varepsilon}\|_{L ^{\infty}}^{2}\int _{\mathcal{I}}\vert \Phi _{x}^{\varepsilon}\vert ^{2}\mathrm{d}x
 \nonumber \\
 & \displaystyle \leq \frac{1}{8}\int _{\mathcal{I}} \vert V _{t}^{\varepsilon}\vert ^{2} \mathrm{d}x+c(v _{\ast},T) \varepsilon \|V ^{\varepsilon}\|_{L ^{2}}\|V _{x}^{\varepsilon}\|_{L ^{2}}\int _{\mathcal{I}} \vert \Phi _{x}^{\varepsilon}\vert ^{2}\mathrm{d}x
  \nonumber \\
   &\displaystyle\leq \frac{1}{8}\int _{\mathcal{I}} \vert V _{t}^{\varepsilon}\vert ^{2} \mathrm{d}x+c(v _{\ast},T) \varepsilon ^{5/4}\|V _{x}^{\varepsilon}\|_{L ^{2}}  \int _{\mathcal{I}} \vert \Phi _{x}^{\varepsilon}\vert ^{2}\mathrm{d}x
    \nonumber \\
    & \displaystyle \leq \frac{1}{8}\int _{\mathcal{I}} \vert V _{t}^{\varepsilon}\vert ^{2} \mathrm{d}x+c(v _{\ast},T) \varepsilon \|V _{x}^{\varepsilon}\| _{L ^{2}}^{2}\|\Phi _{x}^{\varepsilon}\|_{L ^{2}}^{2}+c(v _{\ast},T) \varepsilon ^{3/2}\|\Phi _{x}^{\varepsilon}\|_{L ^{2}}^{2},
\end{align*}
and that
\begin{gather}\label{v-x-esti-last}
\displaystyle  \varepsilon ^{-1/2}\int _{\mathcal{I}}\mathcal{R}_{2}^{\varepsilon} V_{t}^{\varepsilon}\mathrm{d}x \leq \frac{1}{8}\int _{\mathcal{I}} \vert V _{t}^{\varepsilon}\vert ^{2} \mathrm{d}x+c(v _{\ast},T)\varepsilon ^{-1}\|\mathcal{R}_{2}^{\varepsilon}\|_{L ^{2}}^{2} \leq \frac{1}{8}\int _{\mathcal{I}} \left\vert V _{t}^{\varepsilon}\right\vert ^{2} \mathrm{d}x+c(v _{\ast},T) \varepsilon ^{1/2}.
\end{gather}
With \eqref{v-x-sti-one}--\eqref{v-x-esti-last} and the fact $ 0< \varepsilon<1 $, we thus update \eqref{V-x-diff-iden} as
\begin{align*}
\displaystyle  \frac{\mathrm{d}}{\mathrm{d}t}\int _{\mathcal{I}} \varepsilon\vert V _{x}^{\varepsilon}\vert ^{2}  \mathrm{d}x+ \|V _{t} ^{\varepsilon}\|_{L ^{2}}^{2}\leq c(v _{\ast},T) \|\Phi _{x}^{\varepsilon}\| _{L ^{2}}^{2}(\varepsilon \|V _{x}^{\varepsilon}\| _{L ^{2}}^{2})+C\|\Phi _{x}^{\varepsilon}\| _{L ^{2}}^{2}+c(v _{\ast},T) \varepsilon ^{1/2},
\end{align*}
which followed by an integration over $ [0,t] $ for any $ t \in (0,T] $ gives
\begin{align}\label{esti-v-x-final}
\displaystyle  \varepsilon\|V _{x} ^{\varepsilon}(\cdot,t)\|_{L ^{2}}^{2}+ \int _{0}^{t}\|V _{\tau} ^{\varepsilon}\|_{L ^{2}}^{2}\mathrm{d}\tau \leq c(v _{\ast},T)\int _{0} ^{t} \|\Phi _{x}^{\varepsilon}\| _{L ^{2}}^{2}(\varepsilon \|V _{x}^{\varepsilon}\| _{L ^{2}}^{2}) \mathrm{d}\tau+c(v _{\ast},T) \varepsilon ^{1/2},
\end{align}
where we have used Lemma \ref{lem-L2-pertur}. Applying the Gronwall inequality to \eqref{esti-v-x-final}, we then obtain that
\begin{align}\label{con-V-x-L2}
\displaystyle   \varepsilon\|V _{x} ^{\varepsilon}(\cdot,t)\|_{L ^{2}}^{2}+ \int _{0}^{t}\|V _{\tau} ^{\varepsilon}\|_{L ^{2}}^{2}\mathrm{d}\tau \leq c(v _{\ast},T) \varepsilon ^{1/2},
\end{align}
where we have used \eqref{con-L2-PERT}. Now let us turn to the estimate on $ \Phi _{x}^{\varepsilon} $. Taking the $ L ^{2} $ inner product of the first equation in \eqref{eq-perturbation} with $ \Phi _{t}^{\varepsilon} $, followed by the integration by parts, we have
\begin{align}\label{fida-x-diff-id}
\displaystyle \frac{1}{2} \frac{\mathrm{d}}{\mathrm{d}t}\int _{\mathcal{I}}\vert \Phi _{x}^{\varepsilon}\vert ^{2}\mathrm{d}x+ \int _{\mathcal{I}}\vert \Phi _{t}^{\varepsilon}\vert ^{2} \mathrm{d}x
&\displaystyle =-\int _{\mathcal{I}} \Phi _{t}^{\varepsilon}\Phi _{x}^{\varepsilon}V_{x}^{A}   \mathrm{d}x -\int _{\mathcal{I}} \varepsilon ^{1/2}  \Phi _{t}^{\varepsilon}\Phi _{x}^{\varepsilon}V _{x}^{\varepsilon}\mathrm{d}x
 \nonumber \\
 &\displaystyle \quad- \int _{\mathcal{I}} V_{x}^{\varepsilon}\Phi _{t}^{\varepsilon}(\Phi_{x}^{A}+M) \mathrm{d}x+\int _{\mathcal{I}} \varepsilon ^{-1/2}\mathcal{R}_{1}^{\varepsilon}\Phi _{t}^{\varepsilon} \mathrm{d}x=:\sum _{i=1}^{4}\mathcal{Q}_{i}.
\end{align}
Next we estimate $ \mathcal{Q}_{i}~(1 \leq i \leq 4) $. From \eqref{esti-approx-fida-2} and the Cauchy-Schwarz inequality, we deduce that
\begin{gather*}
\displaystyle \mathcal{Q}_{1} \leq \|V _{x}^{A}\|_{L ^{\infty}}\|\Phi _{t}^{\varepsilon}\|_{L ^{2}}\|\Phi _{x}^{\varepsilon}\|_{L ^{2}} \leq \frac{1}{8}\|\Phi _{t}^{\varepsilon}\|_{L ^{2}}^{2}+c(v _{\ast},T) \varepsilon ^{-1}\|\Phi _{x}^{\varepsilon}\|_{L ^{2}}^{2}, \\
\displaystyle \mathcal{Q}_{3} \leq \|\Phi _{x}^{A}+M\|_{L ^{\infty}}\|V _{x}^{\varepsilon}\|_{L ^{2}}\|\Phi _{t}\|_{L ^{2}} \leq \frac{1}{8}\|\Phi _{t}^{\varepsilon}\|_{L ^{2}}^{2}+c(v _{\ast},T) \|V _{x}^{\varepsilon}\|_{L ^{2}}^{2}.
\end{gather*}
By the Cauchy-Schwarz inequality, we get
\begin{align}\label{Q-2-esti}
\displaystyle  \mathcal{Q}_{2}   \leq \frac{1}{8}\int _{\mathcal{I}} \vert \Phi _{t}^{\varepsilon}\vert ^{2} \mathrm{d}x+C \varepsilon \int _{\mathcal{I}} \vert \Phi _{x}\vert ^{2}\vert V _{x}^{\varepsilon}\vert ^{2}\mathrm{d}x  \leq \frac{1}{8}\int _{\mathcal{I}} \vert \Phi _{t}^{\varepsilon}\vert ^{2} \mathrm{d}x+ c _{0} \varepsilon \|V _{x}^{\varepsilon}\| _{L ^{\infty}} ^{2} \|\Phi _{x}^{\varepsilon}\|_{L ^{2}}^{2}.
\end{align}
Similarly, we have
\begin{align}\label{Q-4-esti}
\displaystyle \displaystyle \mathcal{Q} _{4} \leq \frac{1}{8}\int _{\mathcal{I}} \vert \Phi _{t}^{\varepsilon}\vert ^{2} \mathrm{d}x+c _{0} \varepsilon ^{-1} \|\mathcal{R}_{1}^{\varepsilon} \| _{L ^{2}}^{2} \leq \frac{1}{8}\int _{\mathcal{I}} \vert \Phi _{t}^{\varepsilon}\vert ^{2} \mathrm{d}x+c(v _{\ast},T) \varepsilon ^{1/2},
\end{align}
where we have used Lemma \ref{lem-R-1-VE}. Therefore, inserting the estimates on $ \mathcal{Q} _{i}~(0 \leq i  \leq 4) $ into \eqref{fida-x-diff-id}, we get
\begin{align*}
&\displaystyle   \frac{1}{2} \frac{\mathrm{d}}{\mathrm{d}t}\int _{\mathcal{I}}\vert \Phi _{x}^{\varepsilon}\vert ^{2}\mathrm{d}x+ \frac{1}{2}\int _{\mathcal{I}}\vert \Phi _{t}^{\varepsilon}\vert ^{2} \mathrm{d}x
\nonumber \\
 &~\displaystyle
 \leq c(v _{\ast},T)(\varepsilon\|V _{x}^{\varepsilon}\|_{ L ^{\infty}}^{2})\|\Phi _{x}^{\varepsilon}\|_{L ^{2}}^{2}+c(v _{\ast},T)( \varepsilon ^{-1}\|\Phi _{x}^{\varepsilon}\|_{L ^{2}}^{2}+\|V_{x}^{\varepsilon}\|_{L ^{2}}^{2}) +c(v _{\ast},T) \varepsilon ^{1/2}.
\end{align*}
Integrating the above inequality over $ (0,t) $ for any $ t \in [0,T] $, we arrive at
\begin{align}\label{fida-x-esti}
\displaystyle  \|\Phi _{x}^{\varepsilon}(\cdot,t)\|_{L ^{2}}+ \int _{0}^{t}\|\Phi _{\tau}^{\varepsilon}\|_{L ^{2}}^{2}\mathrm{d}\tau &\leq c(v _{\ast},T) \int _{0}^{t}\varepsilon\|V _{x}^{\varepsilon}\|_{L ^{\infty}}^{2}\|\Phi _{x}^{\varepsilon}\|_{L ^{2}}^{2}\mathrm{d}\tau+c(v _{\ast},T) \varepsilon ^{-1/2},
\end{align}
where we have used \eqref{con-L2-PERT} and $ 0< \varepsilon<1 $. To close the estimate, it now suffices to show that
\begin{align}\label{v-l-infty}
\displaystyle \varepsilon\int _{0}^{T}\|V _{x} ^{\varepsilon}\|_{L ^{\infty}}^{2}\mathrm{d}t\leq  c(v _{\ast},T)
\end{align}
for some constant $ c(v _{\ast},T)>0 $ independent of $ \varepsilon $ and $ \delta $. Indeed, if \eqref{v-l-infty} holds true, then by the Gronwall inequality, we get from \eqref{fida-x-esti} that for any $ t \in [0,T] $,
\begin{align*}
\displaystyle  \|\Phi _{x}^{\varepsilon}(\cdot,t)\|_{L ^{2}}+ \int _{0}^{t}\|\Phi _{\tau}^{\varepsilon}\|_{L ^{2}}^{2}\mathrm{d}\tau \leq c(v _{\ast},T) \varepsilon ^{-1/2}.
\end{align*}
This along with \eqref{con-V-x-L2} yields \eqref{con-fida-v-x-lem}. To prove \eqref{v-l-infty}, we first derive from the second equation in \eqref{eq-perturbation} that
\begin{gather}\label{V-xx-L2-single}
\displaystyle  \makebox[-0.5pt]{~}
\displaystyle \varepsilon ^{2}\|V _{xx}^{\varepsilon}\|_{L ^{2}}^{2} \leq  \|V _{t}^{\varepsilon}\|_{L ^{2}}^{2}+\varepsilon \|V ^{\varepsilon}\|_{L ^{\infty}}^{2}\|\Phi _{x}^{\varepsilon}\|_{L ^{2}}^{2}+c(v _{\ast},T)\|\Phi _{x}^{\varepsilon}\| _{L ^{2}}^{2} +c(v _{\ast},T)\|V ^{\varepsilon}\|_{L ^{2}}^{2}+\varepsilon ^{-1} \|\mathcal{R}_{2}^{\varepsilon}\| ^{2},
\end{gather}
where \eqref{esti-approx-fida-2} has been used. Therefore,
\begin{align}\label{V-xx-muti}
\displaystyle \varepsilon ^{2}\int _{0} ^{T}\|V _{xx}^{\varepsilon}\|_{L ^{2}}^{2}\mathrm{d}t &\leq \int _{0}^{T}\|V _{t}^{\varepsilon}\|_{L ^{2}}^{2}\mathrm{d}t+\int _{0}^{T}\|\Phi _{x}^{\varepsilon}\|_{L ^{2}}^{2}\mathrm{d}t + c(v _{\ast},T) \int _{0}^{T}\|\Phi _{x}^{\varepsilon}\| _{L ^{2}}^{2} \mathrm{d}t
 \nonumber \\
 & \displaystyle \quad + c(v _{\ast},T)\int _{0}^{T}\|V ^{\varepsilon}\|_{L ^{2}}^{2}\mathrm{d}t+ \varepsilon ^{-1}\int _{0}^{T}\|\mathcal{R}_{2}^{\varepsilon}\|_{L ^{2}}^{2}\mathrm{d}\tau
  \nonumber \\
  & \displaystyle \leq c(v _{\ast},T)\varepsilon ^{1/2},
\end{align}
where we have used \eqref{esti-approx-fida-2b}, \eqref{con-L2-PERT}, \eqref{con-V-x-L2}, Lemma \ref{lem-R-2-ESTI} and the following fact
\begin{gather}\label{V-linfty-single}
\displaystyle \|V ^{\varepsilon}\| _{L ^{\infty}} ^{2} \leq c(v _{\ast},T) \left( \|V ^{\varepsilon}\|_{L ^{2}}^{2} +\|V ^{\varepsilon}\| _{L ^{2}}\|V _{x}^{\varepsilon}\|_{L ^{2}}  \right)  \leq c(v _{\ast},T) \varepsilon ^{1/2}+c(v _{\ast},T)\leq c(v _{\ast},T) ,
\end{gather}
due to \eqref{con-L2-PERT}, \eqref{con-V-x-L2}, $ 0< \varepsilon<1 $ and the Sobolev inequality \eqref{Sobolev-infty}. Therefore, we utilize the Sobolev inequality \eqref{Sobolev-infty} again, along with \eqref{con-L2-PERT}, to derive that
\begin{align*}
\displaystyle \int _{0}^{T}\|V _{x}^{\varepsilon}\|_{L ^{\infty}}^{2}\mathrm{d}t &\leq c(v _{\ast},T) \int _{0}^{T} \|V _{x}^{\varepsilon}\|_{L ^{2}}\|V _{xx}^{\varepsilon}\|_{L ^{2}}\mathrm{d}t +c(v _{\ast},T) \int _{0}^{T}\|V _{x}^{\varepsilon}\|_{L ^{2}}^{2}\mathrm{d}t
 \nonumber \\
 & \displaystyle \leq c(v _{\ast},T) \varepsilon ^{1/2}\int _{0}^{T}\|V _{xx}^{\varepsilon}\|_{ L ^{2}}^{2}\mathrm{d}t+c(v _{\ast},T) \varepsilon ^{-1/2}\int _{0}^{T}\|V _{x}^{\varepsilon}\|_{L ^{2}}^{2}\mathrm{d}t \leq c(v _{\ast},T) \varepsilon ^{-1},
\end{align*}
where the constant $ c(v _{\ast},T)>0 $ is independent of $ \varepsilon $ and $ \delta $. This gives \eqref{v-l-infty}. Thus we finish the proof of Lemma \ref{lem-V-x-per}.
\end{proof}
As a direct consequence of Lemmas \ref{lem-L2-pertur} and \ref{lem-V-x-per}, we have the following corollary.
\begin{corollary}\label{cor-in-lower-order}
Assume the conditions of Lemmas \ref{lem-L2-pertur} and \ref{lem-V-x-per} hold. Then for any solution $ (\Phi ^{\varepsilon},V ^{\varepsilon}) $ to \eqref{eq-perturbation} on $ [0,T] $ satisfying \eqref{apriori-assumption}, we have
\begin{align}\label{fida-xx-multi}
\displaystyle \int _{0}^{T}\left( \varepsilon ^{1/2} \|\Phi _{xx}^{\varepsilon}\|_{L ^{2}}^{2}+\| \Phi _{x}^{\varepsilon}\|_{L ^{\infty}}^{2}+\varepsilon ^{3/2}\|V _{xx} ^{\varepsilon}\|_{L ^{2}}^{2} \right)\mathrm{d}t \leq c(v _{\ast},T),
\end{align}
where $ c(v _{\ast},T)>0 $ is a constant depending on $ T $ but independent of $ \varepsilon $ and $ \delta $.
\end{corollary}
\begin{proof}
  The estimate on $ V _{xx}^{\varepsilon} $ follows from \eqref{V-xx-muti} directly. We now show estimates on $ \Phi _{xx} ^{\varepsilon}$ and $ \Phi _{x}^{\varepsilon} $. By the first equation in \eqref{eq-perturbation}, we have
  \begin{gather}\label{fida-xx-single}
  \displaystyle  \|\Phi _{xx}^{\varepsilon}\|_{L ^{2}}^{2} \leq \|\Phi _{t}^{\varepsilon}\|_{L ^{2}}^{2}+\varepsilon \|\Phi _{x}^{\varepsilon}\|_{L ^\infty}^{2}\|V _{x}^{\varepsilon}\|_{L ^{2}}^{2}+\|V _{x}^{A}\|_{L ^{\infty}}^{2}\|\Phi _{x}^{\varepsilon}\|_{L ^{2}}^{2}+c(v _{\ast},T) \|V _{x}^{\varepsilon}\|_{L ^{2}}^{2}+\varepsilon ^{-1}\|\mathcal{R}_{1}^{\varepsilon}\|_{L ^{2}}^{2},
  \end{gather}
  where we have used \eqref{esti-approx-fida-2a}. Therefore we derive that
  \begin{align}
   \displaystyle \int _{0}^{T}\|\Phi _{xx}^{\varepsilon}\|_{L ^{2}}^{2}\mathrm{d}t& \leq  \int _{0}^{T}\|\Phi _{t}^{\varepsilon}\|_{L ^{2}}^{2}\mathrm{d}t + \varepsilon\sup _{t \in [0,T]}\|\Phi _{x}^{\varepsilon}\|_{L ^{2}}^{2} \int _{0}^{T}\|V _{x}^{\varepsilon}\|_{ L^{\infty}}^{2}\mathrm{d}t+ c(v _{\ast},T) \varepsilon ^{-1} \int _{0}^{T}\|\Phi _{x}^{\varepsilon}\|_{L ^{2}}^{2}\mathrm{d}t
    \nonumber \\
    & \displaystyle \quad+c(v _{\ast},T) \int _{0}^{T}\|V _{x}^{\varepsilon}\|_{L ^{2}}^{2}\mathrm{d}t + \varepsilon ^{-1}\int _{0}^{T}\|\mathcal{R}_{1}^{\varepsilon}\|_{ L ^{2}}^{2}\mathrm{d}t
     \nonumber \\
           & \displaystyle \leq c(v _{\ast},T) \varepsilon ^{-1/2}+c(v _{\ast},T) \varepsilon ^{1/2} \leq c(v _{\ast},T) \varepsilon ^{-1/2}
        \end{align}
 for some constant $ c(v _{\ast},T)>0 $ depending on $ T $ but independent of $ \varepsilon $ and $ \delta $,
        where we have used \eqref{esti-approx-fida-2}, \eqref{con-L2-PERT}, \eqref{con-V-x-L2}, \eqref{v-l-infty}, $ 0< \varepsilon<1 $ and Lemma \ref{lem-R-1-VE}. This along with \eqref{con-L2-PERT} and the Sobolev inequality \eqref{Sobolev-infty} further entails for $ 0< \varepsilon<1 $ that
        \begin{align*}
         \displaystyle \int _{0}^{T}\| \Phi _{x}^{\varepsilon}\|_{L ^{\infty}}^{2}\mathrm{d}t &\leq c(v _{\ast},T) \int _{0}^{T}\left( \| \Phi _{x}^{\varepsilon}\|_{L ^{2}}^{2}+\| \Phi _{x}^{\varepsilon}\|_{L  ^{2}}\| \Phi _{xx}^{\varepsilon}\|_{L ^{2}} \right)\mathrm{d}t
          \nonumber \\
            & \displaystyle \leq c(v _{\ast},T)(1+\varepsilon ^{-1/2})\int _{0}^{T}\| \Phi _{x}^{\varepsilon}\|_{L ^{2}}^{2}\mathrm{d}t+c(v _{\ast},T) \varepsilon ^{1/2} \int _{0}^{T}\| \Phi _{xx}^{\varepsilon}\|_{L ^{2}}^{2}\mathrm{d}t
             \nonumber \\
             & \displaystyle\leq c(v _{\ast},T).
         \end{align*}
          The proof is complete.
\end{proof}

\subsubsection{Higher-order estimates} 
\label{sub:higher_order_estimates}
To prove the convergence result in Theorem \ref{thm-stabi-refor}, we derive some higher-order estimates for $ (\Phi ^{\varepsilon},V ^{\varepsilon}) $ in this subsection.
\begin{lemma}\label{lem-FIDA-V-XX}
Assume that the conditions of Lemmas \ref{lem-L2-pertur} and \ref{lem-V-x-per} hold. Then it holds for any $ t \in (0,T] $ that
\begin{align}\label{fida-t-single}
&\displaystyle   \|\Phi _{t} ^{\varepsilon}(\cdot,t)\|_{L ^{2}}^{2}+\|V _{t}^{\varepsilon}(\cdot,t)\|_{L ^{2}}^{2}+\varepsilon ^{1/2}\|\Phi _{xx}^{\varepsilon}(\cdot,t)\|_{ L ^{2}}^{2}+\varepsilon\|V _{xx}^{\varepsilon}(\cdot,t)\|_{L ^{2}}^{2}
 \nonumber \\
 & \displaystyle \quad + \int _{0}^{t} \left( \|\Phi _{x \tau}^{\varepsilon}\|_{L ^{2}}^{2}+\varepsilon\|V _{x \tau}\|_{L ^{2}}^{2}\right) \mathrm{d}\tau  \leq c(v _{\ast},T) \varepsilon ^{-1/2},\ \ \ t \in (0,T],
\end{align}
where $ c(v _{\ast},T)>0 $ is a constant independent of $ \varepsilon $ and $ \delta $.
\end{lemma}
\begin{proof}
Differentiating the equations in \eqref{eq-perturbation} with respect to $ t $, we end up with
\begin{gather}\label{high-order-eq}
\displaystyle\displaystyle \displaystyle  \begin{cases}
\displaystyle \Phi _{tt}^{\varepsilon}=\Phi _{xxt}^{\varepsilon}-\varepsilon ^{\frac{1}{2}}\Phi _{xt}^{\varepsilon}V _{x}^{\varepsilon} -\varepsilon ^{\frac{1}{2}}\Phi _{x}^{\varepsilon}V _{xt}^{\varepsilon} -\Phi _{xt}^{\varepsilon} V_{x}^{A}-\Phi _{x}^{\varepsilon}V_{xt}^{A}- V _{xt}^{\varepsilon}(\Phi_{x}^{A}+M)- V _{x}^{\varepsilon}\Phi_{xt}^{A}+\varepsilon ^{-\frac{1}{2}}\partial _{t}\mathcal{R}_{1}^{\varepsilon},\\[1mm]
  \displaystyle V _{tt}^{\varepsilon}=\varepsilon V _{xxt}^{\varepsilon}-\varepsilon ^{\frac{1}{2}}\Phi _{xt}^{\varepsilon}V ^{\varepsilon}-\varepsilon ^{\frac{1}{2}}\Phi _{x}^{\varepsilon}V _{t}^{\varepsilon}-\Phi _{xt}^{\varepsilon}V^{A}-\Phi _{x}^{\varepsilon}V_{t}^{A}-\Phi_{xt}^{A}V ^{\varepsilon}-(\Phi_{x}^{A}+M)V _{t}^{\varepsilon}
  +\varepsilon ^{-\frac{1}{2}}\partial _{t}\mathcal{R}_{2}^{\varepsilon}.
    \end{cases}\makebox[-2.5pt]{~}
\end{gather}
Multiplying the first equation in \eqref{high-order-eq} by $ \Phi _{t}^{\varepsilon} $, and then integrating the resulting equation over $ \mathcal{I} $, we get after using the integration by parts that
\begin{align}\label{fida-t-diff-iden}
&\displaystyle \frac{1}{2}\frac{\mathrm{d}}{\mathrm{d}t}\int _{\mathcal{I}} \vert \Phi _{t}^{\varepsilon}\vert ^{2} \mathrm{d}x
+\int _{\mathcal{I}} \vert \Phi _{xt}^{\varepsilon}\vert ^{2} \mathrm{d}x
 \nonumber \\
&~\displaystyle =-\varepsilon ^{1/2}\int _{\mathcal{I}} \Phi _{xt}^{\varepsilon}V_{x}^{\varepsilon}\Phi _{t}^{\varepsilon} \mathrm{d}x - \int _{\mathcal{I}} \varepsilon ^{1/2}\Phi _{x}^{\varepsilon}V _{xt}^{\varepsilon}\Phi _{t}^{\varepsilon} \mathrm{d}x - \int _{\mathcal{I}} \Phi _{x}^{\varepsilon}V_{xt}^{A}\Phi _{t}^{\varepsilon} \mathrm{d}x -  \int _{\mathcal{I}}  V _{x}^{\varepsilon}\Phi_{xt}^{A}\Phi _{t}^{\varepsilon} \mathrm{d}x
 \nonumber \\
 &~\displaystyle \quad + \int _{\mathcal{I}}  \varepsilon ^{-1/2}\partial _{t}\mathcal{R}_{1}^{\varepsilon}\Phi _{t}^{\varepsilon} \mathrm{d}x- \int _{\mathcal{I}} \Phi _{xt}^{\varepsilon}V_{x}^{A}\Phi _{t}^{\varepsilon} \mathrm{d}x - \int _{\mathcal{I}}  V _{xt}^{\varepsilon}(\Phi_{x}^{A}+M) \Phi _{t}^{\varepsilon} \mathrm{d}x=:\sum _{i=1}^{7}\mathcal{H}_{i},
\end{align}
where, in view of the Cauchy-Schwarz inequality, $ \mathcal{ H}_{i}~(1 \leq i \leq 5) $ can be estimated as follows:
\begin{gather}\label{h-esti}
 \displaystyle \begin{cases}
 	\mathcal{H}_{1} \leq \frac{1}{4}\|\Phi _{xt}^{\varepsilon}\|_{L ^{2}}^{2} +c _{0}\varepsilon \|V _{x}^{\varepsilon}\|_{L ^{\infty}}^{2}\|\Phi _{t}^{\varepsilon}\|_{L ^{2}}^{2},\  \displaystyle \mathcal{H}_{2} \leq \frac{\varepsilon}{16}\|V_{xt}^{\varepsilon}\|_{L ^{2}}^{2}+c _{0}\|\Phi _{x}^{\varepsilon}\|_{L ^{\infty}}^{2}\|\Phi _{t}^{\varepsilon}\|_{L ^{2}}^{2},\\[2mm]
 	\displaystyle\mathcal{H}_{3} \leq c _{0} \|V _{xt}^{A}\|_{L ^{\infty}}\left( \varepsilon ^{-1/2}\|\Phi _{x}^{\varepsilon}\|_{L ^{2}} ^{2}+\varepsilon ^{1/2}\|\Phi _{t}^{\varepsilon}\|_{L ^{2}}^{2}\right),	 \ \mathcal{H}_{4} \leq c _{0} \|\Phi _{xt}^{A}\|_{L ^{\infty}}\left(  \|V _{x}^{\varepsilon}\|_{L ^{2}}^{2}+\|\Phi _{t}^{\varepsilon}\|_{L ^{2}}^{2} \right),\\[2mm]
  	\displaystyle\mathcal{H}_{5} \leq  \|\Phi _{t}^{\varepsilon}\|_{L ^{2}}^{2}+ c _{0}\varepsilon ^{-1}\|\partial _{t}\mathcal{R}_{1}^{\varepsilon}\|_{L  ^{2}}^{2}.
 	\end{cases}
 \end{gather}
 For $ \mathcal{H}_{6} $ and $ \mathcal{H}_{7} $, it follows from \eqref{con-vfi-v-I-0-regula}, \eqref{vfi-I-1-V-i-1}, \eqref{approximate-va}, \eqref{l-infty-v-bd-stat}, \eqref{some-l-infty-layer-for-v-b-0-stab}, \eqref{Sobolev-z-xi}, the integration by parts and the Hardy inequality \eqref{hardy} that
 \begin{align}\label{H-6}
 \displaystyle  \displaystyle  \mathcal{H}_{6} &=-\int _{\mathcal{I}}\Phi _{xt}^{\varepsilon}\varepsilon ^{-1/2}( v _{z}^{B,0}+v _{\xi}^{b,0})\Phi _{t}^{\varepsilon}\mathrm{d}x - \int _{\mathcal{I}}\Phi _{xt}^{\varepsilon} ( v _{x}^{I,0}+v _{z}^{B,1}+v _{\xi}^{b,1}+\varepsilon ^{1/2}v _{x}^{I,1}+\partial _{x} b _{v}^{\varepsilon}) \Phi _{t}^{\varepsilon}\mathrm{d}x
  \nonumber \\
  &\displaystyle \leq \left\|\frac{\Phi _{t}^{\varepsilon}}{x(1-x)}\right\|_{L ^{2}}\|\Phi _{xt}^{\varepsilon}\|_{L ^{2}}\left( \left\|\frac{x(1-x)}{\varepsilon ^{1/2}}v _{z}^{B,0}\right\|_{L ^{\infty}}+ \left\|\frac{x(1-x)}{\varepsilon ^{1/2}}v _{\xi}^{b,0}\right\|_{L ^{\infty}} \right)
  \nonumber \\
  &\displaystyle \quad+ \|\Phi _{t}^{\varepsilon}\|_{L ^{2}}\|\Phi _{xt}^{\varepsilon}\|_{L ^{2}} \left( \|v _{x}^{I,0}\|_{L ^{\infty}}+\|v _{z}^{B,1}\|_{L _{z}^{\infty}}+\|v _{\xi}^{b,1}\|_{L _{\xi}^{\infty}}+\varepsilon ^{1/2}\|v _{x}^{I,1}\|_{L ^{\infty}}+\|\partial _{x}b _{v}^{\varepsilon}\|_{L ^{\infty}} \right)
   \nonumber \\
   &\displaystyle \leq c_{0} \|\Phi _{xt}^{\varepsilon}\|_{L ^{2}}^{2}\left( \|\langle z \rangle v _{z}^{B,0}\| _{L _{z}^{\infty}}+\|\langle \xi \rangle v _{\xi}^{b,0}\|_{L _{\xi}^{\infty}} \right)+ c(v _{\ast},T) \|\Phi _{xt}^{\varepsilon}\|_{L ^{2}}\|\Phi _{t}^{\varepsilon}\|_{L ^{2}}
    \nonumber \\
    & \displaystyle \leq \Big( \frac{1}{8}+ K _{1}(T,v _{\ast}) v _{\ast} \Big) \|\Phi _{xt}^{\varepsilon}\|_{L ^{2}}^{2}+ c(v _{\ast},T)\|\Phi _{t}^{\varepsilon}\|_{L ^{2}}^{2}
      \end{align}
      with $ K _{1}(T,v _{\ast}) >0$ as in \eqref{five-l2}, and that
 \begin{align}\label{H-7}
 \displaystyle  \displaystyle \mathcal{H}_{7} &=- \int _{\mathcal{I}}V _{t}^{\varepsilon}\Phi _{t} ^{\varepsilon}\Phi _{xx}^{A}\mathrm{d}x -\int _{\mathcal{I}} (\Phi_{x}^{A}+M)V  _{t}^{\varepsilon}\Phi _{xt}^{\varepsilon} \mathrm{d}x
 \nonumber \\
 &\displaystyle \leq
 - \varepsilon ^{1/2}\int _{\mathcal{I}}V  _{t}^{\varepsilon}\Phi_{t} ^{\varepsilon}(\varphi _{xx}^{B,1}+\varphi _{xx}^{b,1})\mathrm{d}x -\int _{\mathcal{I}}V  _{t}^{\varepsilon}\Phi_{t} ^{\varepsilon}\left( \varphi _{xx}^{I,0}+\varepsilon ^{1/2}\varphi _{xx}^{I,1}+\varepsilon \varphi _{xx}^{B,2}+\varepsilon \varphi _{xx}^{b,2}+\partial _{x}^{2} b _{\varphi}^{\varepsilon}\right) \mathrm{d}x
  \nonumber \\
  &\displaystyle \quad + \|\Phi_{x}^{A}+M\|_{L ^{\infty}} \|V  _{t}^{\varepsilon}\|_{L ^{2}}\|\Phi _{xt}^{\varepsilon}\|_{L ^{2}}
  \nonumber \\
  &\displaystyle \leq \|V  _{t}^{\varepsilon}\|_{L ^{2}} \left\|\frac{\Phi_{t} ^{\varepsilon}}{x(1-x)}\right\|_{L ^{2}}\left( \left\|\frac{x(1-x)}{\varepsilon ^{1/2}}\varphi _{zz}^{B,1}\right\|_{L ^{\infty}}+\left\|\frac{x(1-x)}{\varepsilon ^{1/2}}\varphi _{\xi \xi}^{b,1}\right\|_{L ^{\infty}} \right)
   + c(v _{\ast},T) \|V  _{t}^{\varepsilon}\|_{L ^{2}}\|\Phi _{xt}^{\varepsilon}\|_{L ^{2}}
   \nonumber \\
   &\displaystyle \quad+ \|V  _{t}^{\varepsilon}\|_{L ^{2}}\|\Phi _{t}^{\varepsilon}\|_{L ^{2}}\left( \|\varphi _{xx}^{I,0}\|_{L ^{\infty}}+\|\varphi _{zz}^{B,2}\|_{L _{z}^{\infty}}+\|\varphi _{\xi \xi}^{b,2}\|_{L _{\xi}^{\infty}}+\varepsilon ^{1/2}\|\varphi _{xx}^{I,1}\|_{L ^{\infty}}+\|\partial _{x}^{2}b _{\varphi}^{\varepsilon}\|_{L ^{\infty}} \right)
     \nonumber \\
     & \displaystyle \leq c _{0}\|V  _{t}^{\varepsilon}\|_{L ^{2}}\|\Phi _{xt}^{\varepsilon}\|_{L ^{2}} \left( \|\langle z \rangle \varphi _{zz}^{B,1}\|_{L _{z}^{\infty}}+\|\langle \xi \rangle \varphi _{\xi \xi}^{b,1}\|_{L _{\xi}^{\infty}} \right)
      \nonumber \\
      & \displaystyle \quad
     + c(v _{\ast},T) \|V  _{t}^{\varepsilon}\|_{L ^{2}}\left( \|\Phi _{t}^{\varepsilon}\|_{L ^{2}}+\|\Phi _{xt}^{\varepsilon}\|_{L ^{2}}\right)
     \leq \frac{1}{8}\|\Phi _{xt}\|_{L ^{2}}^{2}+c(v _{\ast},T) (\|\Phi _{t}^{\varepsilon}\|_{L ^{2}}^{2}+\|V_{t} ^{\varepsilon}\|_{L ^{2}}^{2}),
 \end{align}
 where the constant $ C>0 $ is independent of $ \varepsilon $ and $ \delta $. Plugging \eqref{h-esti}--\eqref{H-7} into \eqref{fida-t-diff-iden} followed by an integration over $ (0,t) $ for any $ t \in (0,T] $, it follows that
\begin{align}\label{fida-t-final-diff-per}
&\displaystyle \|\Phi _{t} ^{\varepsilon}(\cdot,t)\|_{L ^{2}}^{2}+ \int _{0}^{t} \|\Phi _{x \tau}^{\varepsilon}\|_{L ^{2}}^{2}\mathrm{d}\tau
 \nonumber \\
 &~\displaystyle\leq \frac{\varepsilon}{4}\int _{0}^{t}\|V _{x \tau}\| _{L ^{2}}^{2}\mathrm{d}\tau+c(v _{\ast},T) \int _{0}^{t}\left(  \varepsilon \|V _{x}^{\varepsilon}\|_{L ^{\infty}}^{2}+\|\Phi _{x}^{\varepsilon}\|_{L ^{\infty}}^{2}\right) \|\Phi _{\tau} ^{\varepsilon}(\cdot,\tau)\|_{L ^{2}}^{2}\mathrm{d}\tau+c(v _{\ast},T) \varepsilon ^{-1/2},
\end{align}
provided that $ K _{1}(v _{\ast},T) v _{\ast} \leq 1/16 $ in \eqref{H-6}, where we have used \eqref{esti-approx-fida-2}, $ 0< \varepsilon <1 $ and Lemmas \ref{lem-R-1-VE}, \ref{lem-L2-pertur} and \ref{lem-V-x-per}.

To proceed, multiplying the second equation in \eqref{high-order-eq} by $ V _{t}^{\varepsilon} $, we get after integrating the resulting equation over $ \mathcal{I} $ that
\begin{align}\label{V-t-est-diff-iden}
&\displaystyle \frac{1}{2}\frac{\mathrm{d}}{\mathrm{d}t} \int _{\mathcal{I}} \left\vert V _{t}^{\varepsilon}\right\vert ^{2} \mathrm{d}x+ \varepsilon \int _{\mathcal{I}} \left\vert V _{xt}^{\varepsilon}\right\vert ^{2} \mathrm{d}x + \int _{\mathcal{I}} \left( \Phi_{x}^{A}+M \right)\left\vert V _{t}^{\varepsilon}\right\vert ^{2}  \mathrm{d}x
 \nonumber \\
 &~\displaystyle = - \varepsilon ^{1/2}\int _{\mathcal{I}} \Phi _{xt}^{\varepsilon}V ^{\varepsilon}V _{t}^{\varepsilon} \mathrm{d}x -\int _{\mathcal{I}} \varepsilon ^{1/2}\Phi _{x}^{\varepsilon}V _{t}^{\varepsilon} V _{t}^{\varepsilon}\mathrm{d}x - \int _{\mathcal{I}} \Phi _{xt}^{\varepsilon}V^{A} V _{t}^{\varepsilon} \mathrm{d}x - \int _{\mathcal{I}} \Phi _{x}^{\varepsilon}V_{t}^{A}V _{t}^{\varepsilon} \mathrm{d}x
  \nonumber \\
  & ~\displaystyle \quad - \int _{\mathcal{I}} \Phi_{xt}^{A} V ^{\varepsilon}V _{t}^{\varepsilon} \mathrm{d}x + \varepsilon ^{-1/2}\int _{\mathcal{I}} \partial _{t}\mathcal{R}_{2}^{\varepsilon}V _{t}^{\varepsilon} \mathrm{d}x=:\sum _{i=1}^{6}\mathcal{L}_{i},
\end{align}
where, similar to \eqref{h-esti}, $ \mathcal{L}_{i}\,(1 \leq i \leq 6) $ enjoy the following estimates:
\begin{align}\label{L-ESTI-fINAL}
\begin{cases}	
\displaystyle \mathcal{L}_{1} \leq \frac{1}{4}\|\Phi _{xt}^{\varepsilon}\|_{ L ^{2}}^{2}+ c(v _{\ast},T) \varepsilon \|V ^{\varepsilon}\|_{L ^{\infty}}^{2}\|V _{t}^{\varepsilon}\|_{L ^{2}}^{2},\ \ \ \mathcal{L}_{2}  \leq  \varepsilon ^{1/2} \|\Phi _{x}^{\varepsilon}\|_{L ^{\infty}}\|V _{t}\|_{L ^{2}}^{2},\\[2mm]
\displaystyle \mathcal{L}_{3} \leq \frac{1}{4}\|\Phi _{xt}^{\varepsilon}\|_{ L ^{2}}^{2}+c(v _{\ast},T) \|V ^{A}\|_{L ^{\infty}}^{2}\|V _{t}^{\varepsilon}\|_{L ^{2}}^{2},\ \ \ \mathcal{L}_{4} \leq \|V _{t} ^{A}\|_{L ^{\infty}} \left( \|\Phi _{x}^{\varepsilon}\| _{L ^{2}}^{2}+\|V _{t}^{\varepsilon}\|_{L ^{2}}^{2}\right),\\[2mm]
\displaystyle \mathcal{L}_{5} \leq  \|\Phi _{xt}^{A}\|_{L ^{\infty}}\left( \|V ^{\varepsilon}\|_{L ^{2}}^{2}+\|V _{t}^{\varepsilon}\|_{L ^{2}}^{2} \right),\ \ \ \mathcal{L }_{6} \leq \|V _{t}^{\varepsilon}\|_{L ^{2}}^{2}+ c(v _{\ast},T) \varepsilon ^{-1}\|\partial _{t}\mathcal{R}_{2}^{\varepsilon}\|_{L ^{2}}^{2}.
\displaystyle
\end{cases}
\end{align}
Therefore, we integrate \eqref{V-t-est-diff-iden} over $ (0,t)\subset (0,T] $ to get
\begin{align}\label{V-t-final-diff-pert}
\displaystyle  \|V _{t}^{\varepsilon}(\cdot,t)\|_{L ^{2}}^{2}+ \int _{0}^{t} \int _{\mathcal{I}}\varepsilon \left\vert V _{x \tau}^{\varepsilon}\right\vert ^{2} \mathrm{d}x\mathrm{d}\tau & \leq c(v _{\ast},T) \int _{0}^{t} \left(1+ \varepsilon \|V ^{\varepsilon}\|_{L ^{\infty}}^{2}+\varepsilon ^{1/2}\|\Phi _{x}^{\varepsilon}\|_{L ^{\infty}} \right)\|V _{\tau}^{\varepsilon}(\cdot,\tau)\| _{L ^{2}}^{2}\mathrm{d}\tau
 \nonumber \\
 &\displaystyle \quad+\frac{1}{2}\int _{0}^{t}\|\Phi _{xt}^{\varepsilon}\|_{L ^{2}}^{2}\mathrm{d}\tau +c(v _{\ast},T) \varepsilon ^{1/2},
\end{align}
where we have used \eqref{esti-approx-fida-2}, \eqref{con-L2-PERT}, \eqref{con-fida-v-x-lem}, $ 0< \varepsilon <1 $ and Lemma \ref{lem-R-2-ESTI}. Combining \eqref{V-t-final-diff-pert} with \eqref{fida-t-final-diff-per}, we arrive at
\begin{align}\label{Final-tim-deriv-esti}
&\displaystyle  \|\Phi _{t} ^{\varepsilon}(\cdot,t)\|_{L ^{2}}^{2}+\|V _{t}^{\varepsilon}(\cdot,t)\|_{L ^{2}}^{2}+ \int _{0}^{t} \left( \|\Phi _{x \tau}^{\varepsilon}\|_{L ^{2}}^{2}+\varepsilon\|V _{x \tau}^{\varepsilon}\|_{L ^{2}}^{2} \right) \mathrm{d}\tau
 \nonumber \\
 &~\displaystyle \leq c(v _{\ast},T) \varepsilon ^{-1/2} +c(v _{\ast},T) \int _{0}^{t}\left(  \varepsilon \|V _{x}^{\varepsilon}\|_{L ^{\infty}}^{2}+\|\Phi _{x}^{\varepsilon}\|_{L ^{\infty}}^{2}\right) \|\Phi _{\tau} ^{\varepsilon}(\cdot,\tau)\|_{L ^{2}}^{2}\mathrm{d}\tau
  \nonumber \\
  &~\displaystyle \quad+c(v _{\ast},T) \int _{0}^{t} \left(1+ \varepsilon \|V ^{\varepsilon}\|_{L ^{\infty}}^{2}+\varepsilon ^{1/2}\|\Phi _{x}^{\varepsilon}\|_{L ^{\infty}} \right)\|V _{\tau}^{\varepsilon}(\cdot,\tau)\| _{L ^{2}}^{2}\mathrm{d}\tau,
\end{align}
Applying the Gronwall inequality to \eqref{Final-tim-deriv-esti}, alongside \eqref{v-l-infty}, \eqref{V-linfty-single} and \eqref{fida-xx-multi}, we get that
\begin{align*}
\displaystyle  \|\Phi _{t} ^{\varepsilon}(\cdot,t)\|_{L ^{2}}^{2}+\|V _{t}^{\varepsilon}(\cdot,t)\|_{L ^{2}}^{2}+ \int _{0}^{t} \left( \|\Phi _{x \tau}^{\varepsilon}\|_{L ^{2}}^{2}+\varepsilon\|V _{x \tau}^{\varepsilon}\|_{L ^{2}}^{2}\right) \mathrm{d}\tau  \leq c(v _{\ast},T) \varepsilon ^{-1/2}.
\end{align*}
This along with \eqref{esti-approx-fida-2}, \eqref{con-L2-PERT}, \eqref{con-fida-v-x-lem}, \eqref{V-xx-L2-single} and Lemmas \ref{lem-R-1-VE}, \ref{lem-R-2-ESTI} further entails that
\begin{align*}
\displaystyle \varepsilon ^{3/2}\|V _{xx}^{\varepsilon}\|_{L _{T}^{\infty}L ^{2}}^{2} \leq c(v _{\ast},T).
\end{align*}
It now remains to derive the estimate for $ \Phi _{xx}^{\varepsilon} $. Multiplying the first equation in \eqref{high-order-eq} by $ \Phi _{xx}^{\varepsilon} $, followed by an integration over $ \mathcal{I} $, we get
\begin{align}\label{esti-fida-xx-single}
\displaystyle \frac{1}{2}\frac{\mathrm{d}}{\mathrm{d}t}\int _{\mathcal{I}}\vert \Phi _{xx}^{\varepsilon}\vert^{2} \mathrm{d}x &= \int _{\mathcal{I}}\Phi _{tt}^{\varepsilon}\Phi _{xx}^{\varepsilon}\mathrm{d}x +\varepsilon ^{1/2}\int _{\mathcal{I}} \Phi _{xt}^{\varepsilon}V_{x}^{\varepsilon}\Phi _{xx}^{\varepsilon} \mathrm{d}x + \int _{\mathcal{I}} \varepsilon ^{1/2}\Phi _{x}^{\varepsilon}V _{xt}^{\varepsilon}\Phi _{xx}^{\varepsilon} \mathrm{d}x  + \int _{\mathcal{I}} \Phi _{x}^{\varepsilon}V_{xt}^{A}\Phi _{xx}^{\varepsilon} \mathrm{d}x
 \nonumber \\
 &~\displaystyle \quad+  \int _{\mathcal{I}}  V _{x}^{\varepsilon}\Phi_{xt}^{A}\Phi _{xx}^{\varepsilon} \mathrm{d}x-\int _{\mathcal{I}}  \varepsilon ^{-1/2}\partial _{t}\mathcal{R}_{1}^{\varepsilon}\Phi _{xx}^{\varepsilon} \mathrm{d}x+ \int _{\mathcal{I}} \Phi _{xt}^{\varepsilon}V_{x}^{A}\Phi _{xx}^{\varepsilon} \mathrm{d}x
  \nonumber \\
  &~\displaystyle \quad+ \int _{\mathcal{I}}  V _{xt}^{\varepsilon}(\Phi_{x}^{A}+M) \Phi _{xx}^{\varepsilon} \mathrm{d}x:= \int _{\mathcal{I}}\Phi _{tt}^{\varepsilon}\Phi _{xx}^{\varepsilon}\mathrm{d}x +\sum _{i=1}^{7}\hat{\mathcal{H}}_{i},
\end{align}
where, thanks to integration by parts, we have
\begin{align}\label{fid-tt}
\displaystyle  \int _{\mathcal{I}}\Phi _{tt}^{\varepsilon}\Phi _{xx}^{\varepsilon}\mathrm{d}x= \frac{\mathrm{d}}{\mathrm{d}t} \int _{\mathcal{I}}\Phi _{t}^{\varepsilon}\Phi _{xx}^{\varepsilon}\mathrm{d}x -\int _{\mathcal{I}}\Phi _{t}^{\varepsilon}\Phi _{xxt}^{\varepsilon}= \frac{\mathrm{d}}{\mathrm{d}t} \int _{\mathcal{I}}\Phi _{t}^{\varepsilon}\Phi _{xx}^{\varepsilon}\mathrm{d}x+\int _{\mathcal{I}}\vert \Phi _{xt}^{\varepsilon}\vert ^{2}\mathrm{d}x .
\end{align}
For $ \hat{\mathcal{ H}}_{i}~(1 \leq i \leq 7) $, we get by the Cauchy-Schwarz inequality that
\begin{align}
\displaystyle  \begin{cases}
	\hat{\mathcal{H}}_{1} \leq c _{0}(\|\Phi _{xt}^{\varepsilon}\|_{L ^{2}}^{2}+\varepsilon \|V _{x}^{\varepsilon}\|_{L ^{\infty}}^{2}\|\Phi _{xx}^{\varepsilon}\|_{L ^{2}}^{2}),\ \ \hat{\mathcal{H}}_{2} \leq c _{0}( \varepsilon \|V _{xt}^{\varepsilon}\|_{L ^{2}}^{2} + \|\Phi _{x}^{\varepsilon}\|_{L ^{\infty}}^{2}\|\Phi _{xx}^{\varepsilon}\|_{L ^{2}}^{2}), \\[2pt]
	\displaystyle \hat{\mathcal{H}}_{3} \leq c _{0}( \|V _{xt}^{A}\|_{L ^{\infty}}^{2}\|\Phi _{x}^{\varepsilon}\|_{L ^{2}}^{2}+\|\Phi _{xx}^{\varepsilon}\|_{L ^{2}}^{2}), \ \ \hat{\mathcal{H}}_{4} \leq c _{0}( \|\Phi _{xt}^{A}\|_{L ^{\infty}}^{2}\|V _{x}^{\varepsilon}\|_{L ^{2}}^{2}+ \|\Phi _{xx}^{\varepsilon}\|_{L ^{2}}^{2})
	 \nonumber \\[2pt]
	 \displaystyle \hat{\mathcal{H}}_{5} \leq c _{0}( \|\Phi _{xx}^{\varepsilon}\|_{L ^{2}}^{2}+ \varepsilon ^{-1}\|\partial _{t}\mathcal{R}_{1}^{\varepsilon}\|_{L ^{2}}^{2}),\ \ \hat{\mathcal{H}}_{6} \leq  \|V _{x}^{A}\|_{L ^{\infty}}\left( \|\Phi _{xt}^{\varepsilon}\|_{L ^{2}}^{2}+\|\Phi _{xx}^{\varepsilon}\|_{L ^{2}}^{2} \right)
	  \nonumber \\[2pt]
	  \displaystyle   \hat{\mathcal{H}}_{7} \leq c _{0}( \varepsilon ^{1/2}\|\Phi _{x}^{A}+M\|_{L ^{\infty}}^{2}\|V _{xt}^{\varepsilon}\|_{L ^{2}}^{2}+\varepsilon ^{-1/2}\|\Phi _{xx}^{\varepsilon}\|_{L ^{2}}^{2}).
 \end{cases}
\end{align}
Inserting \eqref{fid-tt} and estimates on $ \hat{\mathcal{H}}_{i}\,(1 \leq i \leq 7) $ into \eqref{esti-fida-xx-single}, followed by an integration in $ t $, we get
\begin{align*}
&\displaystyle  \|\Phi _{xx}^{\varepsilon}(\cdot,t)\|_{L ^{2}}^{2}
 \nonumber \\
 &\displaystyle~ \leq  \int _{\mathcal{I}}\Phi _{t}^{\varepsilon}\Phi _{xx}^{\varepsilon}\mathrm{d}x+ c(v _{\ast},T) \varepsilon ^{-1/2} +c(v _{\ast},T) \varepsilon ^{-1/2}\int _{0}^{t}\left( \|\Phi _{x \tau}^{\varepsilon}\|_{L ^{2}}^{2}+\|\Phi _{xx}^{\varepsilon}\|_{L ^{2}}^{2} \right)\mathrm{d}\tau
 \nonumber \\
 & \displaystyle~ \quad+c(v _{\ast},T) \int _{0}^{t}\Big( \varepsilon ^{1/2}\|V _{x \tau}^{2}\|_{L ^{2}}^{2}+\varepsilon ^{-1/2}\|\Phi _{xx}^{\varepsilon}\|_{L ^{2}}^{2} \Big)\mathrm{d}\tau +c(v _{\ast},T) \int _{0}^{t}\left( \varepsilon \|V _{x}^{\varepsilon}\|_{L ^{\infty}}^{2}+\|\Phi _{x}^{\varepsilon}\|_{L ^{\infty}}^{2} \right) \|\Phi _{xx}^{\varepsilon}\|_{L ^{2}}^{2}\mathrm{d}\tau
 \nonumber \\
 & \displaystyle~ \leq \frac{1}{2}\|\Phi _{xx}^{\varepsilon}(\cdot,t)\|_{L ^{2}}^{2}+c(v _{\ast},T) \|\Phi _{t}^{\varepsilon}\|_{L ^{2}}^{2}+c(v _{\ast},T) \varepsilon ^{-1}+c(v _{\ast},T) \int _{0}^{t}\left( \varepsilon \|V _{x}^{\varepsilon}\|_{L ^{\infty}}^{2}+\|\Phi _{x}^{\varepsilon}\|_{L ^{\infty}}^{2} \right) \|\Phi _{xx}^{\varepsilon}\|_{L ^{2}}^{2}\mathrm{d}\tau
  \nonumber \\
  & \displaystyle~ \leq \frac{1}{2}\|\Phi _{xx}^{\varepsilon}(\cdot,t)\|_{L ^{2}}^{2}+c(v _{\ast},T) \varepsilon ^{-1}+c(v _{\ast},T) \int _{0}^{t}\left( \varepsilon \|V _{x}^{\varepsilon}\|_{L ^{\infty}}^{2}+\|\Phi _{x}^{\varepsilon}\|_{L ^{\infty}}^{2} \right) \|\Phi _{xx}^{\varepsilon}\|_{L ^{2}}^{2}\mathrm{d}\tau ,
\end{align*}
where we have used \eqref{esti-approx-fida-2}, \eqref{con-L2-PERT}, \eqref{con-fida-v-x-lem}, \eqref{fida-xx-multi} and Lemma \ref{lem-R-1-VE}. That is,
\begin{align*}
\displaystyle \|\Phi _{xx}^{\varepsilon}(\cdot,t)\|_{L ^{2}}^{2}   \leq c(v _{\ast},T) \varepsilon ^{-1}+c(v _{\ast},T)\int _{0}^{t}\left( \varepsilon \|V _{x}^{\varepsilon}\|_{L ^{\infty}}^{2}+\|\Phi _{x}^{\varepsilon}\|_{L ^{\infty}}^{2} \right) \|\Phi _{xx}^{\varepsilon}\|_{L ^{2}}^{2}\mathrm{d}\tau,
\end{align*}
which along with \eqref{fida-xx-multi} and the Gronwall inequality gives
\begin{align*}
\displaystyle  \|\Phi _{xx}^{\varepsilon}(\cdot,t)\|_{L ^{2}}^{2} \leq c(v _{\ast},T) \varepsilon ^{-1}, \ \ t \in (0,T]
\end{align*}
for some constant $ c(v _{\ast},T)>0 $ independent of $ \varepsilon $ and $ \delta $, and thus ends the proof of Lemma \ref{lem-FIDA-V-XX}.
\end{proof}
With Lemma \ref{lem-FIDA-V-XX}, we can get an improved estimate for $ \Phi _{x}^{\varepsilon} $.
\begin{corollary}\label{cor-fida-x-modified}
Assume the conditions in Lemmas \ref{lem-L2-pertur}--\ref{lem-FIDA-V-XX} hold. Let $ (\Phi ^{\varepsilon},V ^{\varepsilon}) $ be the solution of the problem \eqref{eq-perturbation} on $ [0,T] $ satisfying \eqref{apriori-assumption}. Then we have
\begin{align}\label{con-fida-x-modified}
\displaystyle  \|\Phi _{x}^{\varepsilon}(\cdot,t)\|_{L ^{2}}^{2} +\int _{0}^{t}\int _{\mathcal{I}}\vert \Phi _{\tau}^{\varepsilon}\vert ^{2} \mathrm{d}x \mathrm{d}\tau\leq c(v _{\ast},T), \ \ t \in (0,T],
\end{align}
where $ c(v _{\ast},T)>0 $ is a constant independent of $ \varepsilon $ and $ \delta $.

\end{corollary}
\begin{proof}
  Recalling \eqref{fida-x-diff-id}, \eqref{Q-2-esti} and \eqref{Q-4-esti}, we have
  \begin{align}\label{fida-x-diff-modified}
&\displaystyle \frac{1}{2} \frac{\mathrm{d}}{\mathrm{d}t}\int _{\mathcal{I}}\vert \Phi _{x}^{\varepsilon}\vert ^{2}\mathrm{d}x+ \int _{\mathcal{I}}\vert \Phi _{t}^{\varepsilon}\vert ^{2} \mathrm{d}x
 \nonumber \\
&~\displaystyle =-\int _{\mathcal{I}} \Phi _{t}^{\varepsilon}\Phi _{x}^{\varepsilon}V_{x}^{A}   \mathrm{d}x
- \int _{\mathcal{I}} V_{x}^{\varepsilon}\Phi _{t}^{\varepsilon}(\Phi_{x}^{A}+M) \mathrm{d}x -\int _{\mathcal{I}} \varepsilon ^{1/2}  \Phi _{t}^{\varepsilon}\Phi _{x}^{\varepsilon}V _{x}^{\varepsilon}\mathrm{d}x+\int _{\mathcal{I}} \varepsilon ^{-1/2}\mathcal{R}_{1}^{\varepsilon}\Phi _{t}^{\varepsilon} \mathrm{d}x
  \nonumber \\
  &~ \displaystyle \leq -\int _{\mathcal{I}} \Phi _{t}^{\varepsilon}\Phi _{x}^{\varepsilon}V_{x}^{A}   \mathrm{d}x - \int _{\mathcal{I}} V_{x}^{\varepsilon}\Phi _{t}^{\varepsilon}(\Phi_{x}^{A}+M) \mathrm{d}x +  \frac{1}{4}\int _{\mathcal{I}} \left\vert \Phi _{t}^{\varepsilon}\right\vert ^{2} \mathrm{d}x \nonumber \\
   & \displaystyle~ \quad+c(v _{\ast},T)\varepsilon \|V _{x}^{\varepsilon}\| _{L ^{\infty}} ^{2} \|\Phi _{x}^{\varepsilon}\|_{L ^{2}}^{2}+c(v _{\ast},T) \varepsilon ^{1/2}
   \nonumber \\
   &~\displaystyle =: \hat{\mathcal{Q}}_{1}+\hat{\mathcal{Q}}_{2}+  \frac{1}{4}\int _{\mathcal{I}} \left\vert \Phi _{t}^{\varepsilon}\right\vert ^{2} \mathrm{d}x+c(v _{\ast},T) \varepsilon \|V _{x}^{\varepsilon}\| _{L ^{\infty}} ^{2} \|\Phi _{x}^{\varepsilon}\|_{L ^{2}}^{2}+c(v _{\ast},T) \varepsilon ^{1/2},
\end{align}
where, thanks to \eqref{con-vfi-v-I-0-regula}, \eqref{vfi-I-1-V-i-1}, \eqref{approximate-va}, \eqref{l-infty-v-bd-stat}, \eqref{some-l-infty-layer-for-v-b-0-stab}, \eqref{Sobolev-z-xi}, integration by parts and the Hardy inequality \eqref{hardy}, $ \hat{Q}_{1} $ and $ \hat{\mathcal{Q}}_{2} $ enjoy the following estimates:
\begin{align*}
\displaystyle \hat{\mathcal{Q}}_{1}&=  -\int _{\mathcal{I}} \Phi _{t}^{\varepsilon}\Phi _{x}^{\varepsilon}\varepsilon ^{-1/2}\left( v _{z}^{B,0}+v _{\xi}^{b,0}\right)   \mathrm{d}x - \int _{\mathcal{I}}  \Phi _{t}^{\varepsilon}\Phi _{x}^{\varepsilon}\left( v _{x}^{I,0}+v _{z}^{B,1}+v _{\xi}^{b,1}+ \varepsilon ^{1/2} v _{x}^{I,1}+\partial _{x} b _{v}^{\varepsilon} \right) \mathrm{d}x
 \nonumber \\
   &\displaystyle \leq \left\|\frac{\Phi _{t}^{\varepsilon}}{x(1-x)}\right\|_{L ^{2}}\|\Phi _{x}^{\varepsilon}\|_{L ^{2}}\left( \left\|\frac{x(1-x)}{\varepsilon ^{1/2}}v _{z}^{B,0}\right\|_{L ^{\infty}}+ \left\|\frac{x(1-x)}{\varepsilon ^{1/2}}v _{\xi}^{b,0}\right\|_{L ^{\infty}} \right)
  \nonumber \\
  &\displaystyle \quad+ \|\Phi _{t}^{\varepsilon}\|_{L ^{2}}\|\Phi _{x}^{\varepsilon}\|_{L ^{2}} \left( \|v _{x}^{I,0}\|_{L ^{\infty}}+\|v _{z}^{B,1}\|_{L _{z}^{\infty}}+\|v _{\xi}^{b,1}\|_{L _{\xi}^{\infty}}+\varepsilon ^{1/2}\|v _{x}^{I,1}\|_{L ^{\infty}}+\|\partial _{x}b _{v}^{\varepsilon}\|_{L ^{\infty}} \right)
   \nonumber \\
   &\displaystyle \leq c _{0} \|\Phi _{xt}^{\varepsilon}\|_{L ^{2}} \|\Phi _{x}^{\varepsilon}\|_{L ^{2}}\left( \|\langle z \rangle v _{z}^{B,0}\| _{L _{z}^{\infty}}+\|\langle \xi \rangle v _{\xi}^{b,0}\|_{L _{\xi}^{\infty}} \right)+c(v _{\ast},T) \|\Phi _{x}^{\varepsilon}\|_{L ^{2}}\|\Phi _{t}^{\varepsilon}\|_{L ^{2}}
    \nonumber \\
    & \displaystyle \leq \frac{1}{8} \|\Phi _{t}^{\varepsilon}\|_{L ^{2}}^{2}+c(v _{\ast},T) \varepsilon ^{1/2} \|\Phi _{xt}^{\varepsilon}\|_{L ^{2}}^{2}+c(v _{\ast},T) (1+\varepsilon ^{-1/2})\|\Phi _{x}^{\varepsilon}\|_{L ^{2}}^{2},
\end{align*}
\begin{align*}
\displaystyle  \hat{\mathcal{Q}}_{2}&=  \int _{\mathcal{I}}V ^{\varepsilon}\Phi _{t} ^{\varepsilon}\Phi _{xx}^{A}\mathrm{d}x +\int _{\mathcal{I}} (\Phi_{x}^{A}+M)V  ^{\varepsilon}\Phi _{xt}^{\varepsilon} \mathrm{d}x
 \nonumber \\
 &\displaystyle \leq \varepsilon ^{1/2}\int _{\mathcal{I}}V  ^{\varepsilon}\Phi_{t} ^{\varepsilon}(\varphi _{xx}^{B,1}+\varphi _{xx}^{b,1})\mathrm{d}x + \int _{\mathcal{I}}V  ^{\varepsilon}\Phi_{t} ^{\varepsilon}\left( \varphi _{xx}^{I,0}+\varepsilon ^{1/2}\varphi _{xx}^{I,1}+\varepsilon \varphi _{xx}^{B,2}+\varepsilon \varphi _{xx}^{b,2}+\partial _{x}^{2} b _{\varphi}^{\varepsilon}\right) \mathrm{d}x
  \nonumber \\
  &\displaystyle \quad + \|\Phi_{x}^{A}+M\|_{L ^{\infty}} \|V  ^{\varepsilon}\|_{L ^{2}}\|\Phi _{xt}^{\varepsilon}\|_{L ^{2}}
  \nonumber \\
  &\displaystyle \leq \|V  ^{\varepsilon}\|_{L ^{2}} \left\|\frac{\Phi_{t} ^{\varepsilon}}{x(1-x)}\right\|_{L ^{2}}\left( \left\|\frac{x(1-x)}{\varepsilon ^{1/2}}\varphi _{zz}^{B,1}\right\|_{L ^{\infty}}+\left\|\frac{x(1-x)}{\varepsilon ^{1/2}}\varphi _{\xi \xi}^{B,1}\right\|_{L ^{\infty}} \right)
  +c(v _{\ast},T) \|V  ^{\varepsilon}\|_{L ^{2}}\|\Phi _{xt}^{\varepsilon}\|_{L ^{2}}
   \nonumber \\
   &\displaystyle \quad+c  _{0} \|V  ^{\varepsilon}\|_{L ^{2}}\|\Phi _{t}^{\varepsilon}\|_{L ^{2}}\left( \|\varphi _{xx}^{I,0}\|_{L ^{\infty}}+\varepsilon ^{1/2}\|\varphi _{xx}^{I,1}\|_{L ^{\infty}}+\|\varphi _{zz}^{B,2}\|_{L _{z}^{\infty}}+\|\varphi _{\xi \xi}^{b,2}\|_{L _{\xi}^{\infty}}+\|\partial _{x}^{2}b _{\varphi}^{\varepsilon}\|_{L ^{\infty}} \right)
         \nonumber \\
     & \displaystyle \leq c _{0} \|V  ^{\varepsilon}\|_{L ^{2}}\|\Phi _{xt}^{\varepsilon}\|_{L ^{2}} \left( \|\langle z \rangle \varphi _{zz}^{B,1}\|_{L _{z}^{\infty}}+\|\langle \xi \rangle \varphi _{\xi \xi}^{b,1}\|_{L _{\xi}^{\infty}} \right)
           +c(v _{\ast},T) \|V  ^{\varepsilon}\|_{L ^{2}}\left( \|\Phi _{t}^{\varepsilon}\|_{L ^{2}}+\|\Phi _{xt}^{\varepsilon}\|_{L ^{2}}\right) \nonumber \\
      & \displaystyle
     \leq \frac{1}{8}\|\Phi _{t}^{\varepsilon}\|_{L ^{2}}^{2}+c(v _{\ast},T) \varepsilon ^{1/2} \|\Phi _{xt}^{\varepsilon}\|_{L ^{2}}^{2}+c(v _{\ast},T) (1+\varepsilon ^{-1/2})\|V ^{\varepsilon}\|_{L ^{2}}^{2}.
\end{align*}
Therefore, we update \eqref{fida-x-diff-modified} as
\begin{align}\label{fida-xx-modi-final}
&\displaystyle   \frac{1}{2} \frac{\mathrm{d}}{\mathrm{d}t}\int _{\mathcal{I}}\vert \Phi _{x}^{\varepsilon}\vert ^{2}\mathrm{d}x+ \frac{1}{2}\int _{\mathcal{I}}\vert \Phi _{t}^{\varepsilon}\vert ^{2} \mathrm{d}x
 \nonumber \\
 & \displaystyle~ \leq  c(v _{\ast},T) \varepsilon \|V _{x}^{\varepsilon}\| _{L ^{\infty}} ^{2} \|\Phi _{x}^{\varepsilon}\|_{L ^{2}}^{2}+c(v _{\ast},T) \varepsilon ^{1/2}+c(v _{\ast},T) \varepsilon ^{1/2} \|\Phi _{xt}^{\varepsilon}\|_{L ^{2}}^{2} \nonumber \\
  &~\displaystyle \quad+c(v _{\ast},T)\varepsilon ^{-1/2}(\|\Phi _{x}^{\varepsilon}\|_{L ^{2}}^{2}+\|V ^{\varepsilon}\|_{L ^{2}}^{2}).
\end{align}
Integrating \eqref{fida-xx-modi-final} with respect to $ t $ gives
\begin{align*}
\displaystyle  \int _{\mathcal{I}}\vert \Phi _{x}^{\varepsilon}\vert ^{2}(\cdot,t)\mathrm{d}x+\int _{0}^{t}\vert \Phi _{\tau}^{\varepsilon}\vert ^{2}\mathrm{d}x \mathrm{d}\tau \leq c(v _{\ast},T)+ c(v _{\ast},T) \int _{0}^{t} \varepsilon \|V _{x}^{\varepsilon}\| _{L ^{\infty}} ^{2} \|\Phi _{x}^{\varepsilon}\|_{L ^{2}}^{2}\mathrm{d}\tau, \ \ t \in [0,T]
\end{align*}
for some constant $ c(v _{\ast},T) >0$ independent of $ \varepsilon $ and $ \delta $, where we have used \eqref{con-L2-PERT}, \eqref{fida-t-single} and $ 0< \varepsilon<1 $. This alongside the Gronwall inequality and \eqref{v-l-infty} immediately implies \eqref{con-fida-x-modified}, and the proof is complete.
\end{proof}
\begin{remark}
In view of \eqref{con-L2-PERT} and \eqref{con-fida-x-modified}, the a priori assumption  \eqref{apriori-assumption} is verified. Indeed, from \eqref{con-L2-PERT}, \eqref{con-fida-x-modified} and the Sobolev inequality \eqref{Sobolev-infty}, we get for $ 0< \varepsilon<1 $ that
\begin{align}\label{FIDA-X-INFTY-INCOR}
\displaystyle \sup _{t \in [0,T]} \|\Phi ^{\varepsilon}(\cdot, t)\|_{L ^{\infty}}^{2} &\leq c(v _{\ast},T) \sup _{t \in [0,T]}\left( \|\Phi ^{\varepsilon}(\cdot,t)\|_{L ^{2}}^{2}+\|\Phi _{x}^{\varepsilon}(\cdot,t)\|_{L ^{2}}\|\Phi ^{\varepsilon}(\cdot,t)\|_{L ^{2}} \right)
 \nonumber \\
 & \leq c(v _{\ast},T) \left( \varepsilon ^{1/2} +\varepsilon ^{1/4}\right)  \leq c(v _{\ast},T)  \varepsilon ^{1/4},
\end{align}
where the constant $ c(v _{\ast},T)>0 $ is independent of $ \varepsilon $ and $ \delta $. Furthermore, if we take $ \delta= \frac{\delta _{1}}{2} $ with $ \delta _{1} $ as in Lemma \ref{lem-L2-pertur}, then we have $ \sup _{t \in [0,T]}\|\Phi ^{\varepsilon}(\cdot,t)\|_{L ^{\infty}} \leq c(v _{\ast},T) \varepsilon ^{1/8}< \frac{\delta}{2} $ provided
$ c(v _{\ast},T) \varepsilon ^{1/8} \leq \delta _{1}/4 $. Hence, all the estimates in Lemmas \ref{lem-L2-pertur}--\ref{lem-FIDA-V-XX} and Corollaries \ref{cor-in-lower-order} and \ref{cor-fida-x-modified} exactly hold true with the constant $ c(v _{\ast},T) $ independent of $ \varepsilon $.
\end{remark}

\subsection{Proof of Proposition \ref{prop-pertur}} 
\label{sub:proof_of_propo}

Thanks to the analysis and results in the preceding subsection, we know that for any $ T >0$ such that $ C _{1}(T , v _{\ast})v _{\ast}< 1/16 $ with $ K_{1}(T,v _{\ast}) $ presented in \eqref{five-l2} and \eqref{H-6}, the solution $ (\Phi ^{\varepsilon},V ^{\varepsilon}) $ satisfies
for any $ t \in [0,T] $,
\begin{gather}\label{esti1-proof-pro}
\displaystyle  \|\Phi^{\varepsilon}(\cdot,t)\|_{L ^{2}}^{2}+ \varepsilon\|\Phi _{x}^{\varepsilon}(\cdot,t)\|_{L ^{2}}^{2}+\varepsilon ^{3/2}\|\Phi _{xx}^{\varepsilon}\|_{L ^{2}}^{2}+\varepsilon ^{\ell}\|\partial _{x}^{\ell}V ^{\varepsilon}(\cdot,t)\|_{L ^{2}}^{2} \leq c(v _{\ast},T) \varepsilon ^{1/2}
\end{gather}
and
\begin{eqnarray}\label{esti2-proof-pro}
\begin{aligned}
&\displaystyle \int _{0}^{t}\left( \|\Phi _{x}^{\varepsilon}\|_{L ^{2}}^{2}+ \varepsilon ^{1/2}\|\Phi _{\tau} ^{\varepsilon}\|_{L ^{2}}^{2}+\varepsilon \|\Phi _{x \tau}^{\varepsilon}\|_{L ^{2}}^{2}+\varepsilon \|V _{x}^{2}\|_{L ^{2}}^{2}
+\|V _{\tau}^{\varepsilon}\|_{L ^{2}}^{2}+\varepsilon ^{5/2}\|V _{x \tau}\|_{L ^{2}}^{2} \right) \mathrm{d}\tau\\
&\ \ \ \ \ \  \leq c(v _{\ast},T) \varepsilon ^{1/2},
\end{aligned}
\end{eqnarray}
where $ \ell=0,1,2 $, $ c(v _{\ast},T) >0$ is a constant depending on $ T $ but independent of $ \varepsilon $. In particular, since $ K _{1}(T,v _{\ast}) $ is increasing in $ T $,
if $ v _{\ast} $ is fixed, then there exists an increasing function $ \phi(\cdot,v _{\ast})=K _{1} ^{-1}(\cdot, v _{\ast}) $ such that $ K _{1}(T, v _{\ast})v _{\ast}\leq 1/16 $ provided $ T \leq \phi(\frac{1}{16 v _{\ast}},v _{\ast})=:T _{0} $. Then the estimates \eqref{esti1-proof-pro} and \eqref{esti2-proof-pro} hold for any $t \in [0,T _{0}]  $. This along with the local existence result and the continuation argument implies that the problem \eqref{eq-perturbation} admits a unique solution $ (\Phi ^{\varepsilon}, V ^{\varepsilon}) \in L ^{\infty}(0,T _{0};H ^{2}\times H ^{2}) $ satisfying \eqref{esti1-proof-pro} and \eqref{esti2-proof-pro}. In what follows, we shall show that $ T _{0} \rightarrow \infty $ as $ v _{\ast}\rightarrow 0 $. To achieve this, without loss of generality, we first assume that $ v _{\ast} \leq 1 $. Then we may strengthen the condition $ K_{1}(T, v _{\ast})v _{\ast}\leq 1/16 $ for \eqref{five-l2} and \eqref{H-6} as $ K _{1}(T, 1)v _{\ast}\leq 1/16 $. Here we write $ K _{1}(T):=K _{1}(T,1) $ for simplicity. Clearly, since $ K _{1}(T, v _{\ast})$ is increasing in $ v _{\ast} $, we have $ K _{1}(T, v _{\ast})v _{\ast}\leq 1/16 $ as long as $ K_{1}(T)v _{\ast}\leq 1/16 $. Therefore we know that the estimates \eqref{esti1-proof-pro} and \eqref{esti2-proof-pro} hold for any $ t \leq K _{1}^{-1}(\frac{1}{16 v _{\ast}})=:T _{0} $ with $ T _{0}\rightarrow \infty $ as $ v _{\ast} \rightarrow 0 $ due to the increasing monotonicity of $ K _{1}^{-1}(\cdot) $. This completes the proof of Proposition \ref{prop-pertur}. \hfill $ \square $

\subsection{Proof of Theorem \ref{thm-stabi-refor}} 
\label{sub:proof_of_theorem_ref}
From Proposition \ref{prop-pertur}, we know that for any $ v _{\ast} >0$, there exist constants $ T _{0}>0 $ and $ \varepsilon _{0}>0 $ such that for any $ \varepsilon \in (0,\varepsilon _{0}) $, the problem \eqref{refor-eq}--\eqref{BD-POSITIVE-VE} admits a unique solution $ (\varphi ^{\varepsilon},v ^{\varepsilon})\in L ^{\infty}(0,T _{0} ;H ^{2}\times H ^{2}) $. To finish the proof of Theorem \ref{thm-stabi-refor}, now it remains only to show the estimates in \eqref{conver-thm1}. Recalling \eqref{ansaz} and Lemma \ref{lemf-vfi-V-I-1}, it suffices to show the estimates of $ \mathcal{E}_{1}^{\varepsilon} $ and $ \mathcal{E} _{2}^{\varepsilon} $ stated in \eqref{R-1-R-2}. Thanks to \eqref{some-l-infty-layer-for-stab}, \eqref{l-infty-v-bd-stat}, \eqref{b-v-ve-esti}, \eqref{b-v-ve-l-2-esti}, Lemma \ref{lemf-vfi-V-I-1} and the fact that $ \partial _{x} b _{\varphi} $ is independent of $ x $, there holds that
 \begin{gather}
 \displaystyle
 \displaystyle \|\partial _{z}^{l}\varphi ^{B,2}\|_{ L _{T}^{\infty}L _{z}^{\infty}}+\|\partial _{\xi}^{l}\varphi ^{b,2}\|_{ L _{T}^{\infty}L _{\xi}^{\infty}}+
  \|\varphi _{x}^{I,1}\|_{L _{T}^{\infty} L ^{\infty}} +\varepsilon ^{-1}\|\partial _{x}^{l}b _{\varphi} ^{\varepsilon}\|_{L _{T}^{\infty}L ^{\infty}} \leq c(v _{\ast},T),\label{vfi-B-l2-proofthm}\\
  \displaystyle \| v ^{B,1}\|_{ L _{T}^{\infty}L _{z}^{\infty}}+\|v ^{b,1}\|_{ L _{T}^{\infty}L _{\xi}^{\infty}}+
  \|v^{I,1}\|_{L _{T}^{\infty} L ^{\infty}}+ \varepsilon ^{-1}\|b _{v} ^{\varepsilon}\|_{L _{T}^{\infty}L ^{\infty}} \leq c(v _{\ast},T),
 \end{gather}
 where $ l=0,1 $, the constant $ c(v _{\ast},T)>0 $ is independent of $ \varepsilon $. Furthermore, from \eqref{FIDA-X-INFTY-INCOR}, \eqref{Sobolev-infty}, \eqref{Sobolev-modified}, Lemmas \ref{lem-L2-pertur}--\ref{lem-FIDA-V-XX} and Corollary \ref{cor-fida-x-modified}, we get
\begin{gather}
\displaystyle  \|\Phi ^{\varepsilon}\| _{L _{T} ^{\infty}L ^{\infty}} \leq c(v _{\ast},T)\left(\|\Phi ^{\varepsilon}\|_{L _{T}^{\infty}L ^{2}}+\|\Phi ^{\varepsilon}\|_{L _{T}^{\infty}L ^{2}}^{1/2}\|\Phi _{x}^{\varepsilon}\|_{L _{T}^{\infty}L ^{2}}^{1/2} \right)\leq c(v _{\ast},T) \varepsilon ^{1/8},\\
\displaystyle \|\Phi _{x}^{\varepsilon}\| _{L _{T}^{\infty}L ^{\infty}} \leq c(v _{\ast},T)\left(\|\Phi _{x}^{\varepsilon}\|_{L _{T}^{\infty}L ^{2}} +\|\Phi _{x}^{\varepsilon}\|_{L _{T}^{\infty}L ^{2}}^{1/2}\|\Phi _{xx}^{\varepsilon}\|_{L _{T}^{\infty}L ^{2}}^{1/2} \right)  \leq c(v _{\ast},T) \varepsilon ^{- 1/4}
\end{gather}
and
\begin{gather}\label{V-ep-proof}
\displaystyle  \|V ^{\varepsilon}\| _{L _{T}^{\infty}L ^{\infty}} \leq \sqrt{2}\|V ^{\varepsilon}\|_{L _{T}^{\infty}L ^{2}}^{1/2}\|V _{x}^{\varepsilon}\|_{L _{T}^{\infty}L ^{2}}^{1/2} \leq c(v _{\ast},T)
\end{gather}
for some constant $ c(v _{\ast},T)>0 $ independent of $ \varepsilon $. Therefore we get from \eqref{R-1-precise}, \eqref{R-2-pres}, \eqref{vfi-B-l2-proofthm}--\eqref{V-ep-proof} that
\begin{align*}
\displaystyle  \|\mathcal{E} _{1}^{\varepsilon}\|_{L _{T}^{\infty}L ^{\infty}} &\leq c_{0}\varepsilon \left( \|\varphi ^{B,2}(z,t)\|_{L _{T}^{\infty}L ^{\infty}}+\|\varphi ^{b,2}(\xi,t)\| _{L _{T}^{\infty}L ^{\infty}}\right)
 \nonumber \\
 & \displaystyle \quad +c _{0}\varepsilon ^{1/2} \|\Phi ^{\varepsilon}(x,t)\|_{L _{T}^{\infty}L ^{\infty}} +c_{0}\|b _{\varphi}^{\varepsilon}\|_{L _{T}^{\infty}L ^{\infty}}
 \nonumber \\
 & \displaystyle \leq c(v _{\ast},T) \varepsilon ^{5/8},
\end{align*}
\begin{align*}
\displaystyle  \|\mathcal{E} _{2}^{\varepsilon}\|_{L _{T}^{\infty}L ^{\infty}}& \leq c _{0}\varepsilon ^{1/2}\Big(\| v ^{I,1}(x,t)\|_{L _{T}^{\infty}L ^{\infty}}+\|v ^{B,1}(z,t)\|_{L _{T}^{\infty}L ^{\infty}}+\|v ^{b,1}(\xi,t)\|_{L _{T}^{\infty}L ^{\infty}}\Big)
 \nonumber \\
 & \displaystyle \quad+ c _{0}\varepsilon ^{1/2} \|V(x,t)\|_{L _{T}^{\infty}L ^{\infty}}
 +c_{0}\|b _{v}^{\varepsilon}(x,t)\|_{L _{T}^{\infty}L ^{\infty}}
  \nonumber \\
  & \displaystyle \leq c(v _{\ast},T) \varepsilon ^{1/2}
\end{align*}
and
\begin{align}\label{R-1-x-esti-proof}
\displaystyle \displaystyle  \|\partial _{x} \mathcal{E} _{1}^{\varepsilon}\|_{L _{T}^{\infty}L ^{\infty}} &\leq c_{0}\varepsilon ^{1/2} \left( \|\varphi _{z}^{B,2}(z,t)\|_{L _{T}^{\infty}L _{z} ^{\infty}}+\|\varphi  _{\xi}^{b,2}(\xi,t)\| _{L _{T}^{\infty}L _{\xi}^{\infty}}\right)
 \nonumber \\
 & \displaystyle \quad +c_{0}\varepsilon ^{1/2} \|\Phi  _{x}^{\varepsilon}(x,t)\|_{L _{T}^{\infty}L ^{\infty}} +c _{0}\|\partial _{x} b _{\varphi}^{\varepsilon}\|_{L _{T}^{\infty}L ^{\infty}}
 \nonumber \\
 & \displaystyle \leq c(v _{\ast},T)\varepsilon ^{1/4},
\end{align}
where $ 0< \varepsilon<1 $ has been used. Combining the above estimates on $ \mathcal{E} _{i}^{\varepsilon}\,(i=1,2) $, we get \eqref{conver-thm1}, and thus finish the proof of Theorem \ref{thm-stabi-refor}. \hfill $ \square $

\subsection{Proof of Theorem \ref{thm-original}}
Theorem \ref{thm-original} follows directly from Theorem \ref{thm-stabi-refor} except the estimate \eqref{u-converg}. To prove \eqref{u-converg}, we first notice from \eqref{ansaz} that
\begin{gather*}
\displaystyle  \varphi _{x}^{\varepsilon}= \varphi _{x}^{I,0}+\varphi _{z}^{B,1}(z,t)+\varphi _{\xi}^{b,1}(\xi,t)+\varepsilon ^{1/2} \varphi _{x}^{I,1} +\partial _{x}\mathcal{E} _{1}^{\varepsilon},
\end{gather*}
which implies that
\begin{align}\label{u-expan-thm}
\displaystyle  \displaystyle  \displaystyle u ^{\varepsilon}&=u ^{I,0}+  u^{B,0}+u ^{b,0} +\varepsilon ^{1/2} \varphi _{x}^{I,1} +\partial _{x}\mathcal{E} _{1}^{\varepsilon}.
\end{align}
On the other hand, from \eqref{vfi-B-l2-proofthm} and \eqref{R-1-x-esti-proof}, we have $ \|\varepsilon ^{1/2} \varphi _{x}^{I,1}\|_{L _{T}^{\infty}L ^{\infty}} +\|\partial _{x}\mathcal{E} _{1}^{\varepsilon}\|_{L _{T}^{\infty}L ^{\infty}}\leq c(v _{\ast},T) \varepsilon ^{1/4} $ with the constant $ c(v _{\ast},T)>0 $ independent of $ \varepsilon $. This along with \eqref{u-expan-thm} gives rise to \eqref{u-converg}. \hfill $ \square $

\vspace{4mm}

\appendix
\renewcommand{\appendixname}{Appendix~\Alph{section}}
\section{Local existence result on \texorpdfstring{$ v ^{B,0} $}{v-B0}} 
\label{sub:appendix_a}
In this section, we detail the proof of local existence and uniqueness of solutions to the problem \eqref{first-bd-layer-pro} for the leading-order boundary layer profile $ v ^{B,0}  $. Equivalently, we study the reformulated problem \eqref{eq-for-esti-v-B-0}, i.e.,
\begin{align}\label{eq-for-esti-v-B-0-appendix}
 \displaystyle \begin{cases}
  \displaystyle \vartheta_{t}=\vartheta_{zz} -\overline{u ^{I,0}}    {\mathop{\mathrm{e}}}^{\vartheta+ \phi}(\vartheta+ \phi)- \overline{u ^{I,0}}\,v ^{I,0}(0,t)\left(    {\mathop{\mathrm{e}}}^{\vartheta+ \phi}-   1) \right)+\varrho,\\
\displaystyle \vartheta(0,t)=0,\ \ \vartheta(+\infty,t)=0,\\
\displaystyle \vartheta(z,0)=0.
 \end{cases}
  \end{align}
The solution space for the problem reads
\begin{gather*}
\mathcal{X}_{T}=\left\{ u \in L _{T}^{2}L _{z}^{2} \vert\, \partial _{t}^{l} u \vert _{t=0}=\theta _{l},\   \partial _{t}^{k}u \in L _{T}^{2} H _{z} ^{6-2k},\ l=0,1,2,\, k=0,1,2,3 \right\}
\end{gather*}
for some $ T>0 $, where $\theta _{0}\equiv 0 $, and $ \theta _{l}:= \partial _{t}^{l}\vartheta \vert _{t=0}~(l=1,2) $ are determined by $ u _{0} $, $ v _{0} $ and $ \vartheta (z,0) $ through the equation $ \eqref{eq-for-esti-v-B-0-appendix}_{1} $. By \eqref{compatibility-vfi}, we know that the initial datum is compatible up to order two. We shall divide the proof into three steps in the following.

{\bf Step 1:} \emph{Linearization.} Given $ \omega \in \mathcal{X}_{T} $, we first consider the following linearized problem for \eqref{eq-for-esti-v-B-0-appendix}:
\begin{align}\label{line-eq-appen}
\displaystyle
\begin{cases}
  \displaystyle v _{t}= v _{zz}- \overline{u ^{I,0}}    v- \overline{u ^{I,0}}    {\mathop{\mathrm{e}}}^{w+\phi}\phi -\overline{u ^{I,0}}(     {\mathop{\mathrm{e}}}^{w+ \phi}-1 ) (v ^{I,0}(0,t)+w)+ \varrho=:v _{zz} -\overline{u ^{I,0}}    v+F+ \varrho ,\\
  \displaystyle v(0,t)=0,\ \ v(+\infty,t)=0,\\
  \displaystyle v(z,0)=0,
\end{cases}
\end{align}
where $ \overline{u ^{I,0}} $ is as in \eqref{eq-for-esti-v-B-0}. Let $ V=     {\mathop{\mathrm{e}}}^{ \int _{0}^{t}\overline{u ^{I,0}}\mathrm{d}s}v   $. Then $ V $ satisfies
\begin{align*}
\begin{cases}
  \displaystyle V _{t}= V _{zz}+F {\mathop{\mathrm{e}}}^{ \int _{0}^{t}\overline{u ^{I,0}}\mathrm{d}s}+ \varrho  {\mathop{\mathrm{e}}}^{ \int _{0}^{t}\overline{u ^{I,0}}\mathrm{d}s} ,\\
  \displaystyle V(0,t)=0,\ \ V(+\infty,t)=0,\\
  \displaystyle V(z,0)=0,
\end{cases}
\end{align*}
which can be solved explicitly by the reflection method:
\begin{align*}
\displaystyle  V&= \int _{0}^{t}\int _{0}^{\infty}\Gamma(z-y,t- \tau)\left[ F(y,\tau) {\mathop{\mathrm{e}}}^{ \int _{0}^{\tau}\overline{u ^{I,0}}\mathrm{d}s}+ \varrho(y, \tau)  {\mathop{\mathrm{e}}}^{ \int _{0}^{\tau}\overline{u ^{I,0}}\mathrm{d}s
}\right] \mathrm{d}y \mathrm{d}\tau
 \nonumber \\
 & \displaystyle \quad -\int _{0}^{t}\int _{-\infty}^{0}\Gamma(z-y,t- \tau)\left[ F(-y,\tau) {\mathop{\mathrm{e}}}^{ \int _{0}^{\tau}\overline{u ^{I,0}}\mathrm{d}s}+ \varrho(-y, \tau)  {\mathop{\mathrm{e}}}^{ \int _{0}^{\tau}\overline{u ^{I,0}}\mathrm{d}s
}\right] \mathrm{d}y \mathrm{d}\tau,
\end{align*}
where $ \Gamma(z,t)=\frac{1}{\sqrt{4 \pi t}}		{\mathop{\mathrm{e}}}^{- \frac{z ^{2}}{4t}} $ is the heat kernel. Hence one can recover $ v $ from the above identity along with the definition of $ V $. The uniqueness of solutions to the problem \eqref{line-eq-appen} is standard, so the details are omitted here.

{\bf Step 2:} \emph{A priori estimates.} We shall show for \eqref{line-eq-appen} that there exists a suitably large constant $ K >0$ and a small $ T _{0}>0 $ such that if
\begin{gather}\label{w-bound-regularity}
\displaystyle \sum _{k=0}^{3} \|\partial _{t}^{k} w\|_{L _{T}^{2}H _{z}^{6-2k}}^{2} \leq 2 K \ \ \ \mbox{for\ }T \leq T _{0},
\end{gather}
then it holds that
\begin{align}\label{invariant-esti}
\displaystyle  \sum _{k=0}^{3} \|\partial _{t}^{k} v\|_{L _{T}^{2}H _{z}^{6-2k}}^{2} \leq K,\ \ \ \forall\, T \leq T _{0}.
\end{align}
First, by the Sobolev embedding theorem and \eqref{w-bound-regularity}, we have
\begin{gather}\label{w-infty-sobo}
 \displaystyle \partial _{t} ^{k} w \in C([0,T]; H _{z}^{5-2k}),\ \ \ k=0,1,2.
 \end{gather}
 Furthermore, for any $ t \in [0,T] $ and for $ k=0,1,2 $, it follows that
 \begin{align}\label{w-infty-sobo-2}
 \displaystyle \|\partial _{t}^{k}w(\cdot,t)\|_{H ^{4-2 k}}&= \Big\|\theta _{k} + \int _{0}^{t} \partial _{\tau}^{k+1}w(\cdot,\tau) \mathrm{d}\tau \Big\|_{H ^{4-2k}}  \leq C \|\theta _{k}\|_{H ^{4-2k}}+ \int _{0}^{T} \| \partial _{t}^{k+1}w(\cdot,t) \| _{H ^{4-2k}} \mathrm{d} t
  \nonumber \\
  & \displaystyle  \leq \hat{C} _{0} + T ^{1/2}\left( \int _{0}^{T}\| \partial _{t}^{k+1}w(\cdot,t) \| _{H ^{4-2k}} ^{2} \mathrm{d}t \right) ^{1/2}\\
  & \leq \hat{C} _{0}+ \tilde{C} K ^{1/2} T ^{1/2},
 \end{align}
 where $ \hat{C} _{0} $ is a positive constant depending only on the initial data $ \theta _{j} $. Hereafter $ \tilde{C} >0$ is a generic constant independent of $ T $. With \eqref{con-v-I-0}, \eqref{vfi-bd-I-0}, \eqref{w-bound-regularity}, \eqref{w-infty-sobo} and \eqref{w-infty-sobo-2}, similar to the proof of \eqref{V-B-0-exp} and \eqref{con-exp-v-B-0-final}, we proceed to derive for $ k=0,1,2 $ that
\begin{align}\label{F-ESTI}
&\displaystyle  \left\|\partial _{t}^{k}F\right\|_{L _{T}^{2}H _{z}^{4-2k}}^{2}
 \nonumber \\
 & \displaystyle~\leq \sum _{j=0}^{k}\left\{ \|\partial _{t}^{j}\overline{u ^{I,0}} \partial _{t}^{k-j}({\mathop{\mathrm{e}}}^{w+\phi}\phi)\|_{ L _{T}^{2} H _{z}^{4-2k}}^{2} + \|\partial _{t}^{j}\overline{u ^{I,0}} \partial _{t}^{k-j}[w(     {\mathop{\mathrm{e}}}^{w+ \phi}-1 )]\|_{ L _{T} ^{2} H _{z}^{4-2k}}^{2}\right.
  \nonumber \\
  & \displaystyle~ \qquad\quad\  \left.+  \|\partial _{t}^{j}(\overline{u ^{I,0}} v ^{I,0}(0,t))\partial _{t}^{k-j}(     {\mathop{\mathrm{e}}}^{w+ \phi}-1 )\|_{ L _{T}^{2} H _{z}^{4-2k}}^{2} \right\}
   \nonumber \\
   & \displaystyle~ \leq \tilde{C}T\sum _{j=0}^{k}\|\partial _{t}^{j}\overline{u ^{I,0}}\|_{L ^{\infty}(0,T)}^{2}\left( \| \partial _{t}^{k-j}({\mathop{\mathrm{e}}}^{w+\phi}\phi)\|_{ L _{T}^{\infty} H _{z}^{4-2k}}^{2} + \|\partial _{t}^{k-j}(     {\mathop{\mathrm{e}}}^{w+ \phi}-1 )\|_{ L _{T}^{\infty} H _{z}^{4-2k}}^{2} \right)
    \nonumber \\
    & \displaystyle~ \quad+ \tilde{C}T\sum _{j=0}^{k} \|\partial _{t}^{j}\overline{u ^{I,0}} (v ^{I,0}(0,t))\|_{L ^{\infty}(0,T)}^{2} \|\partial _{t}^{k-j}[w(     {\mathop{\mathrm{e}}}^{w+ \phi}-1 )]\|_{ L _{T}^{\infty}H _{z}^{4-2k}}^{2}
    \nonumber \\
    & \displaystyle~\leq \tilde{C}     K{\mathop{\mathrm{e}}}^{\tilde{C}K} T(1+ \sum _{j=0}^{k} \|\partial _{t}^{j} w\|_{L _{T}^{\infty}H _{z}^{4-2k}}^{2}) \leq \tilde{C}  TK{\mathop{\mathrm{e}}}^{\tilde{C}K}(1+K T )
\end{align}
where we have used the fact $     {\mathop{\mathrm{e}}}^{w+\phi} \leq     {\mathop{\mathrm{e}}}^{\tilde{C}K} $ due to \eqref{w-bound-regularity}, the constant $ \tilde{C}>0 $ may depend on $ \overline{u ^{I,0}} $, $ v ^{I,0} $ and $ \phi $, but independent of $ T $ and $ K $. Thanks to \eqref{con-v-I-0}, we get for $ k=0,1,2 $ that
\begin{align}\label{VRHO-L2}
\displaystyle   \left\|\partial _{t}^{k} \varrho\right\|_{ L _{T}^{2}H _{z}^{4-2k}} ^{2}& \leq \tilde{C}T+ \tilde{C} \sum _{j=0}^{k} \|\partial _{t}^{j} v  ^{I,0}(0,t)\|_{L ^{2}(0,T)} ^{2} \leq \tilde{C}T,
\end{align}
where $ \tilde{C}>0 $ is a constant independent of $ K $ and $ T $. By a procedure similar to the one in the proof of Lemma \ref{lem-v-B-0}, one can deduce for $ T \leq 1 $ that
\begin{align*}
&\displaystyle   \sum _{k=0}^{3} \|\partial _{t}^{k} V\|_{L _{T}^{2}H _{z}^{6-2k}}^{2}
 \nonumber \\
  &~\displaystyle\leq \tilde{C}+ \sum _{k=0}^{2} \left( \left\|\partial _{t}^{k}\left( F {\mathop{\mathrm{e}}}^{ \int _{0}^{t}\overline{u ^{I,0}}\mathrm{d}s} \right) \right\|_{L _{T}^{2}H _{z}^{4-2k}} ^{2}+ \left\|\partial _{t}^{k}\left(  \varrho  {\mathop{\mathrm{e}}}^{ \int _{0}^{t}\overline{u ^{I,0}}\mathrm{d}s} \right)  \right\|_{L _{T}^{2}H _{z}^{4-2k}}^{2} \right)
 \nonumber \\
 & ~\displaystyle \leq   \tilde{C} 		\sum _{k=0}^{2}{\mathop{\mathrm{e}}}^{\tilde{C}T}\left( \left\| \partial _{t}^{k}F\right\|_{L _{T}^{2}H _{z}^{4-2k}} ^{2} +\|\partial _{t}^{k}\varrho\|_{L _{T}^{2}H _{z}^{4-2k}} ^{2}\right)
  \nonumber \\
  & ~\displaystyle \leq \tilde{C}+ \tilde{C}T  \Big(1+K ^{2}{\mathop{\mathrm{e}}}^{\tilde{C}T+\tilde{C}K}\Big),
\end{align*}
where we have used \eqref{F-ESTI} and \eqref{VRHO-L2}, the constant $ \tilde{C} >0 $ is independent of $ K $ and $ T $. In view of the definition of $ V $, there holds that
\begin{align*}
\displaystyle  \sum _{k=0}^{2} \|\partial _{t}^{k} v\|_{L _{T}^{2}H _{z}^{6-2k}}^{2} \leq \hat{C} _{1}+\tilde{C}T \Big(1+K ^{2}{\mathop{\mathrm{e}}}^{CT+CK}\Big)\end{align*}
for some constants $ \hat{C} _{1}  $ and $ \tilde{C} $ independent of $ K $ and $ T $. Hence, we get
\begin{align}\label{con-invarian}
\displaystyle   \sum _{k=0}^{2} \|\partial _{t}^{k} v\|_{L _{T}^{2}H _{z}^{6-2k}}^{2} \leq 2 \hat{C} _{1}=:K,
\end{align}
provided
\begin{gather*}
\displaystyle T \leq \min \left\{ 1, \left[\tilde{C}+ (2 \hat{C} _{1})^{2} 		{\mathop{\mathrm{e}}}^{\tilde{C}+2\tilde{C} \hat{C} _{1}} \right]^{-1}  \right\}=:T _{0}.
\end{gather*}
This gives \eqref{invariant-esti}.

{\bf Step 3:} \emph{Contraction.} Denote
\begin{gather*}
\displaystyle \mathcal{Y}_{T}:=\left.\Big\{ u \in \mathcal{X}_{T} \right\vert\, \sum _{k=0}^{3} \|\partial _{t}^{k} u\|_{L _{T}^{2}H _{z}^{6-2k}}^{2} \leq K  \Big\}
\end{gather*}
with $ K  $ as in \eqref{con-invarian}. In the previous steps, we have proved for  $ T \leq T _{0} $ that the solution map $ \Theta:~ \mathcal{Y}_{T}\rightarrow \mathcal{Y}_{T} $ for the linearized problem \eqref{line-eq-appen} is well-defined. To prove the existence of solutions to \eqref{eq-for-esti-v-B-0-appendix}, it now suffices to show the contraction of $ \Theta $ in the norm $ \|\cdot\|_{C(0,T ;L _{z}^{2})} $ for suitably small $ T>0 $. For any $ w _{1},\, w _{2}   \in \mathcal{Y}_{T}$, denote $ v _{i}=\Theta(w _{i})~(i=1,2) $ and
\begin{gather*}
\displaystyle W=w _{1}-w _{2},\ \ V=v _{1}-v _{2}.
\end{gather*}
Then we have from \eqref{line-eq-appen} that
\begin{align*}
\begin{cases}
  \displaystyle V _{t}= V _{zz}-\overline{u ^{I,0}} V -\overline{u ^{I,0}}		{\mathop{\mathrm{e}}}^{\phi}\left(     {\mathop{\mathrm{e}}}^{w _{1}}-    {\mathop{\mathrm{e}}}^{w _{2}} \right)(\phi+v ^{I,0}(0,t)+w _{1})-\overline{u ^{I,0}}  (  {\mathop{\mathrm{e}}}^{w _{2}+\phi}-1)W,\\
  \displaystyle V(0,t)=0,\ \ \ V(+\infty,t)=0,\\
  \displaystyle V(z,0)=0.
\end{cases}
\end{align*}
The standard $ L ^{2} $ estimate implies that
\begin{align*}
\displaystyle \frac{\mathrm{d}}{\mathrm{d}t}\left\|V\right\|_{L _{z}^{2}} ^{2}+\int _{\mathbb{R}_{+}}\left( V _{z}^{2}+V ^{2} \right)  \mathrm{d}z \leq \tilde{C} 		{\mathop{\mathrm{e}}}^{\tilde{C}K}\|W\|_{L ^{2}}^{2}.
 \end{align*}
It thus holds that
 \begin{align*}
 \displaystyle \sup _{t \in [0,T]}\left\|V\right\|_{L _{z}^{2}}^{2} + \|V\|_{L _{T}^{2}H _{z}^{1}}^{2} \leq \tilde{C} 	T	{\mathop{\mathrm{e}}}^{K}\sup _{t \in [0,T]} \|W\|_{ L _{z}^{2}}^{2} \leq \frac{1}{2}\sup _{t \in [0,T]} \left( \|W\|_{ L _{z}^{2}}^{2}+\|W\| _{L _{T}^{2}H _{z}^{1}}^{2}\right) ,
 \end{align*}
 provided
 \begin{gather*}
 \displaystyle T \leq \min \left\{ T _{0}, \frac{1}{2}\left[ \tilde{C} 		{\mathop{\mathrm{e}}}^{\tilde{C}K} \right]^{-1} \right\}=:T _{1}.
 \end{gather*}
 Hence the desired contraction of $ \Theta $ is proved.

 Finally, based on the analysis in Steps 1 to 3, we conclude that the problem \eqref{eq-for-esti-v-B-0-appendix} admits a solution $ \vartheta\in \mathcal{Y}_{T _{1}} $. The uniqueness of the solution is standard, so we omit the details here. The proof is complete.

\vspace{4mm}

\section{Proof of \texorpdfstring{\eqref{g-esti}}{(3.122)}.} 
\label{sub:appendix_b}
We shall prove that
\begin{gather}\label{g-to-prove-appendix}
 \displaystyle  \|\langle z \rangle ^{l} \partial _{t}^{k}g \|_{L _{T}^{2}H _{z}^{4-2k}} \leq c(v _{\ast},T) \ \ \ \mbox{for}\ k=0,1,2,
 \end{gather}
 where the constant is independent of $ \varepsilon $ and $ \delta $. For this, we first split the function $ g $ into three parts
\begin{gather*}
\displaystyle g=\mathcal{J}_{1}+\mathcal{J}_{2}+\mathcal{J}_{3},
\end{gather*}
where
\begin{align*}
\displaystyle \mathcal{J} _{1}&= ({\mathop{\mathrm{e}}}^{v ^{B,0}}-1)\int _{z}^{\infty} \eta v ^{I,1}(0,t)\partial _{y}\left[(\varphi _{x}^{I,0}(0,t)+M+\varphi _{y}^{B,1}) {\mathop{\mathrm{e}}}^{-v ^{B,0}}\right] \mathrm{d}s (v ^{I,0}(0,t)+v ^{B,0})
 \nonumber \\
 &\displaystyle\quad-\int _{z}^{\infty} \eta v ^{I,1}(0,t)\partial _{y}\left[( \varphi_{x} ^{I,0}(0,t)+M+\varphi _{y}^{B,1}) {\mathop{\mathrm{e}}}^{-v ^{B,0}}\right] \mathrm{d}y (v ^{I,0}(0,t)+v ^{B,0}),
  \nonumber \\
  \displaystyle \mathcal{J}_{2}&=({\mathop{\mathrm{e}}}^{v ^{B,0}}-1)\int _{z}^{\infty}\left[  v _{y}^{B,0}(\varphi _{xx} ^{I,0}(0,t)y+\varphi _{x} ^{I,1}(0,t))+\varphi _{y}^{B,1}v _{x}^{I,0}(0,t) \right]   {\mathop{\mathrm{e}}}^{-v ^{B,0}} \mathrm{d}y (v ^{I,0}(0,t)+v ^{B,0})  \nonumber \\
   & \displaystyle \quad +\int _{z}^{\infty}\left[  v _{y}^{B,0}(\varphi _{xx}^{I,0}(0,t)y+\varphi _{x}^{I,1}(0,t))+\varphi _{y}^{B,1}v _{x}^{I,0}(0,t) \right]   {\mathop{\mathrm{e}}}^{-v ^{B,0}} \mathrm{d}y(v ^{I,0}(0,t)+v ^{B,0}),
    \nonumber \\
\displaystyle \mathcal{J}_{3}& =  (\varphi _{x} ^{I,0}(0,t)+M) \eta(z)v ^{I,1}(0,t)+\eta(z)v _{t}^{I,1}(0,t)- \varphi _{z}^{B,1}(v _{x}^{I,0}(0,t)z+v ^{I,1}(0,t))
  \nonumber \\
  & \displaystyle \quad -(\varphi _{xx}^{I,0}(0,t)z+\varphi _{x}^{I,1}(0,t))v ^{B,0} + \eta(z)v ^{I,1}(0,t) (\varphi _{x}^{I,0}(0,t)+M+\varphi _{z}^{B,1})(v ^{I,0}(0,t)+v ^{B,0})
   \nonumber \\
   & \displaystyle \quad - \eta''(z)v ^{I,1}(0,t)+\eta(z)v ^{I,1}(0,t)\varphi _{z}^{B,1}.
\end{align*}
Thanks to \eqref{con-vfi-v-I-0-regula}, \eqref{vfi-bd-I-0}, \eqref{exp-v-B-0-linfty}, \eqref{con-exp-v-B-0-final}, \eqref{con-v -I-1}, \eqref{some-l-infty-layer-for-vfi-b-1-stab}, \eqref{l-infty-v-bd-stat} and the H\"older inequality, we get for $ k=0,1,2$ and $ l \in \mathbb{N} $ that
\begin{align}\label{pre-J-1}
 &\displaystyle \left\| \langle z \rangle  ^{l}\partial _{t}^{k} \int _{z}^{\infty} \eta v ^{I,1}(0,t)\partial _{y}\left[(\varphi _{x} ^{I,0}(0,t)+M+\varphi _{y}^{B,1}) {\mathop{\mathrm{e}}}^{-v ^{B,0}}\right] \mathrm{d}y \right\| _{L _{T}^{2}H _{z} ^{4-2k}}^{2}
  \nonumber \\
  &~ \displaystyle \leq \left\| \langle z \rangle  ^{l}\partial _{t}^{k} \int _{z}^{\infty} \eta v ^{I,1}(0,t)\partial _{y}\left[(\varphi _{x}^{I,0}(0,t)+M+\varphi _{y}^{B,1}) {\mathop{\mathrm{e}}}^{-v ^{B,0}}\right] \mathrm{d}y \right\| _{L _{T}^{2}L _{z} ^{2}}^{2}
   \nonumber \\
   &~\displaystyle \quad + c(v _{\ast},T)\left\| \langle z \rangle  ^{l}\partial _{t}^{k}\left[ \eta v ^{I,1}(0,t)\partial _{z}\Big((\varphi _{x} ^{I,0}(0,t)+M+\varphi _{z}^{B,1}) {\mathop{\mathrm{e}}}^{-v ^{B,0}} \Big)   \right]  \right\| _{L _{T}^{2}H _{z} ^{3-2k}}^{2}
    \nonumber \\
    & ~\displaystyle \leq c(v _{\ast},T) \sum _{j=0}^{k}\left( \| \langle z \rangle ^{l+2}\partial _{t}^{j}v ^{B,0}\|_{L _{T}^{2}H_{z}^{1}}^{2}+\|\langle z \rangle ^{l+2}\partial _{t}^{j}\varphi ^{B,1}\| _{L _{T}^{2}H _{z}^{2}}^{2}\right) \int _{\mathbb{R}_{+}} \langle y \rangle ^{-2} \mathrm{d}y
     \nonumber \\
     & \displaystyle ~\quad+   c(v _{\ast},T) \sum _{j=0}^{k}\left( \| \langle z \rangle ^{l+2}\partial _{t}^{j}v ^{B,0}\|_{L _{T}^{2}H_{z}^{4-2k}}^{2}+\|\langle z \rangle ^{l+2}\partial _{t}^{j}\varphi ^{B,1}\| _{L _{T} ^{2}H _{z}^{5-2k}}^{2}\right) \int _{\mathbb{R}_{+}} \langle y \rangle ^{-2} \mathrm{d}y
      \nonumber \\
                 &~  \leq c(v _{\ast},T),
 \end{align}
 where we have used $ \partial _{t}^k v^{I,1}(0,t) \in L ^{\infty}(0,T)\,( k=0,1,2 )$ due to \eqref{con-v -I-1} and taken the space-time $ L ^{\infty} $ norms for the terms involving lower-order spatial derivatives which are bounded according to \eqref{some-l-infty-layer-for-vfi-b-1-stab} and \eqref{l-infty-v-bd-stat}. Therefore, $ \mathcal{J}_{1} $ can be estimated as follows:
 \begin{align}\label{J-1-esti}
 &\displaystyle  \left\|\langle z \rangle ^{l}\partial _{t}^{k}\mathcal{J}_{1}\right\|_{L _{T}^{2}H _{z}^{4-2k}}
 \nonumber \\
 &~\displaystyle \leq c(v _{\ast},T)   \left\| \langle z \rangle  ^{l}\partial _{t}^{k} \int _{z}^{\infty} \eta v ^{I,1}(0,t)\partial _{y}\left[(\varphi _{x}^{I,0}(0,t)+M+\varphi _{y}^{B,1}) {\mathop{\mathrm{e}}}^{-v ^{B,0}}\right] \mathrm{d}y \right\| _{L _{T}^{2}H _{z} ^{4-2k}}^{2}
   \nonumber \\
   &~ \displaystyle \qquad\times \left( \|    {\mathop{\mathrm{e}}}^{v ^{B,0}}-1\|_{L _{T}^{\infty}H _{z}^{4-2k}}^{2}+1 \right)  \left( 1+\|v ^{ B,0}\|_{L _{T}^{\infty}H _{z}^{4-2k}}^{2} \right)
   \nonumber \\
   & ~\displaystyle \leq c(v _{\ast},T)
 \end{align}
 for $ k=0,1,2 $, where we have used \eqref{product-law}, \eqref{con-exp-v-B-0-final}, \eqref{some-l-infty-layer-for-vfi-b-1-stab} and \eqref{l-infty-v-bd-stat}. By similar arguments as proving \eqref{pre-J-1}, we get
 \begin{align*}
 &\displaystyle  \left\|\langle z \rangle ^{l}\partial _{t}^{k} \int _{z}^{\infty}\left[  v _{y}^{B,0}(\varphi _{xx} ^{I,0}(0,t)y+\varphi _{x} ^{I,1}(0,t))+\varphi _{y}^{B,1}v _{x}^{I,0}(0,t) \right]   {\mathop{\mathrm{e}}}^{-v ^{B,0}} \mathrm{d}y \right\|_{L _{T}^{2}H _{z}^{4-2k}}^{2}
  \nonumber \\
  &~\displaystyle \leq c(v _{\ast},T)\sum _{j=0}^{k}\|\langle z \rangle ^{l} \int _{z}^{\infty}\left[ \partial _{t}^{k-j}( v _{y}^{B,0}    {\mathop{\mathrm{e}}}^{-v ^{B,0}}) \partial _{t}^{j}(\varphi _{xx}^{I,0}(0,t)y+\varphi _{x}^{I,1}(0,t))\mathrm{d}y\right\|_{L _{T}^{2}H _{z}^{4-2k}}^{2}
   \nonumber \\
   &~\displaystyle \quad +c(v _{\ast},T) \sum _{j=0}^{k} \left\|\langle z \rangle ^{l} \int _{z}^{\infty} \partial _{t}^{j}v _{x}^{I,0}(0,t) \partial _{t}^{k-j}(\varphi _{y}^{B,1}    {\mathop{\mathrm{e}}}^{-v ^{B,0}})\mathrm{d}y\right\|_{L _{T}^{2}H _{z}^{4-2k}}^{2}
    \nonumber \\
       & ~ \displaystyle \leq c(v _{\ast},T) \int _{\mathbb{R}_{+}}\langle y \rangle  ^{-2} \mathrm{d}z \sum _{i,j=0}^{k}\| \langle z \rangle ^{l+3}\partial _{t}^{i}v ^{B,0}\|_{L _{T}^{\infty}H _{z}^{5-2k}}^{2} \int _{0}^{T}\left( \left\vert \partial _{t}^{j}\varphi _{xx}^{I,0}(0,t)\right\vert ^{2}+\left\vert \partial _{t}^{j}\varphi _{x}^{I,1}\right\vert ^{2}\right)\mathrm{d}t
        \nonumber \\
        & ~\displaystyle \quad+ c(v _{\ast},T)\sum _{i,j=0}^{k}\left( \| \langle z \rangle ^{l+2}\partial _{t}^{i}v ^{B,0}\|_{L _{T}^{\infty}H _{z}^{4-2k}}^{2}+\|\langle z \rangle ^{l+2}\partial _{t}^{i}\varphi ^{B,1}\|_{L _{T}^{\infty}H _{z} ^{4-2k}}^{2} \right) \int _{\mathbb{R}_{+}}\langle y \rangle  ^{-2} \mathrm{d}z
         \nonumber \\
         &~\displaystyle \quad \qquad~\qquad \qquad \times \int _{0}^{T} \left\vert \partial _{t}^{j}\partial _{x}v^{I,0}(0,t)\right\vert ^{2}\mathrm{d}t
              \nonumber \\
              &~ \displaystyle \leq c(v _{\ast},T),
 \end{align*}
 due to \eqref{con-vfi-v-I-0-regula}, \eqref{con-vfi-I-1}, \eqref{some-l-infty-layer-for-vfi-b-1-stab}, \eqref{l-infty-v-bd-stat} and the H\"older inequality. Therefore, for $ \mathcal{J}_{2} $, we get
 \begin{align}\label{J-2-esti}
 &\displaystyle \left\|\langle z \rangle ^{l}\mathcal{J}_{2}\right\|_{L _{T}^{2}H _{z}^{4-2k}}
  \nonumber \\
  &~\displaystyle \leq c(v _{\ast},T)  \left\|\langle z \rangle ^{l}\partial _{t}^{k} \int _{z}^{\infty}\left[  v _{y}^{B,0}(\varphi _{xx}^{I,0}(0,t)y+\varphi _{x}^{I,1}(0,t))+\varphi _{y}^{B,1}v _{x}^{I,0}(0,t) \right]   {\mathop{\mathrm{e}}}^{-v ^{B,0}} \mathrm{d}y \right\|_{L _{T}^{2}H _{z}^{4-2k}}^{2}
  \nonumber \\
  & \displaystyle \quad \quad\times  \left( \|    {\mathop{\mathrm{e}}}^{v ^{B,0}}-1\|_{L _{T}^{\infty}H _{z}^{4-2k}}^{2}+1 \right)\left( 1+\|v ^{ B,0}\|_{L _{T}^{\infty}H _{z}^{4-2k}}^{2} \right) \leq c(v _{\ast},T).
 \end{align}
 Now let us turn to $ \mathcal{J}_{3} $. With \eqref{vfi-bd-I-0}, \eqref{con-vfi-I-1}, \eqref{con-v -I-1}, \eqref{some-l-infty-layer-for-vfi-b-1-stab}, \eqref{l-infty-v-bd-stat} and the fact that $ \eta $ is a smooth function with compact support, we deduce for $ k=0,1,2 $ that
 \begin{align*}
  \displaystyle  \left\|\langle z \rangle ^{l}\partial _{t}^{k} \mathcal{J}_{3}\right\|_{L _{T}^{2}H _{z}^{4-2k}}^{2}  & \leq c(v _{\ast},T)  \sum _{i,j=0}^{k}  \left( 1+\|\langle z \rangle\partial _{t}^{j}\varphi ^{B,1}\|_{L _{T}^{\infty}H _{z}^{5-2k}}^{2}+\|\langle z \rangle\partial _{t} ^{j}v ^{B,0}\|_{L _{T}^{\infty}H _{z}^{4-2k}}^{2} \right)
   \nonumber \\
   &~\displaystyle \quad \qquad \quad\times \left( 1+ \int _{0}^{T}(\vert \partial _{t}^{i}  v ^{I,1}(0,t)\vert ^{2}+\vert \partial _{t}^{i}  \varphi _{x} ^{I,1}(0,t)\vert ^{2}+ \vert \partial _{t}^{2}\varphi _{xx}^{I,0}(0,t)\vert ^{2}) \mathrm{d}t\right)
    \nonumber \\
    &\displaystyle\leq c(v _{\ast},T).
     \end{align*}
     This combined with \eqref{J-1-esti} and \eqref{J-2-esti} gives \eqref{g-to-prove-appendix}. The proof of Appendix \ref{sub:appendix_b} is complete.

 \section{Some analytic tools.}\label{Appen-Analysis}
 In this section, we  collect some basic results used in this paper, which include some Sobolev-type inequalities and an embedding theorem on space-time Sobolev spaces. Let us begin with the Sobolev inequalities.

\begin{lemma}[{\cite[Page 236]{brezis-ham-book}}]
Let $ p>1$. Then for any $ \epsilon>0 $, there exists a positive constant $ C=C(\epsilon,p) $ such that
\begin{gather}\label{l-infty-Sobolev}
\displaystyle \|h\|_{L ^{\infty}(\mathcal{I})} \leq \epsilon \|h _{x}\|_{L ^{p}(\mathcal{I})}+ C\|h\|_{L ^{1}(\mathcal{I})}
\end{gather}
for any $ h \in W ^{1,p}(\mathcal{I}) $.
\end{lemma}
\begin{lemma}\label{lem-}
For any $ h \in H ^{1}(\mathcal{I}) $, it holds that
\begin{gather}\label{Sobolev-infty}
\displaystyle \|h\|_{ L ^{\infty}(\mathcal{I})}\leq C \left( \|h\|_{L ^{2}}+\|h\|_{L ^{2}}^{1/2}\|h _{x}\|_{L ^{2}}^{1/2} \right)
\end{gather}
where $ C>0 $ is a constant independent of $ h $.
\end{lemma}
We also remark that if $ h \in H _{0}^{1}(\mathcal{I}) $, then
\begin{align}\label{Sobolev-modified}
\displaystyle \|h\|_{L ^{\infty}} \leq \sqrt{2}\|h\|_{L ^{2}}^{1/2}\|h _{x}\|_{L ^{2}}^{1/2}\ \  \mbox{and}\ \  \|h\|_{L ^{\infty}} \leq C \|h _{x}(\cdot,t)\|_{L ^{2}},
\end{align}
and that if $ h \in H _{z}^{1}\,(\mathrm{resp.}~ H _{\xi}^{1} ) $, then
\begin{align}\label{Sobolev-z-xi}
\displaystyle \|h\|_{L _{z}^{\infty}} \leq C \|h \|_{L _{z}^{2}}^{1/2}\|h _{z}\|_{L _{z}^{2}}^{1/2} \leq C \|h\|_{H _{z}^{1}} \,(\mathrm{resp.}~ \|h\|_{L _{\xi}^{\infty}} \leq C \|h \|_{L _{z}^{2}}^{1/2}\|h _{\xi}\|_{L _{\xi}^{2}}^{1/2} \leq C \|h\|_{H _{\xi}^{1}}),
\end{align}
where the constant $ C >0$ is independent of $ h $.\\

Next, we introduce the Hardy's inequality.
 \begin{lemma}[cf. {\cite[Page 233]{brezis-ham-book}}]\label{lem-hardy}
 Let $ u \in W _{0}^{1,p}(\mathcal{I}) $ with $ 1<p< \infty $. Then
 \begin{gather}\label{hardy}
 \displaystyle \left\|\frac{u}{x(1-x)}\right\|_{L ^{p}(\mathcal{I})} \leq C _{p}\|u _{x}\|_{L ^{p}(\mathcal{I})},
 \end{gather}
 where $ C _{p}>0 $ is a constant depending only on $ p $.

 \end{lemma}

The following embedding theorem is also frequently used in our analysis.
\begin{proposition}[cf. \mbox{\cite[Lemma 1.2]{temam-book-2001}}]\label{prop-embeding-spacetime}
Let $ V $, $ H $ and $ V' $ be three Hilbert spaces satisfying $ V \subset H \subset V ' $ with $ V' $ being the dual of $ V$. If a function $ u $ belongs to $ L ^{2} (0,T;V)$ and its time derivatives $ u _{t} $ belongs to $ L ^{2}(0,T;V') $, then
\begin{gather*}
\displaystyle u \in C([0,T];H) \ \mbox{ and }\ \|u\|_{L ^{\infty}(0,T;H)} \leq C \left( \|u\|_{L ^{2}(0,T;V)}+\|u _{t}\|_{L ^{2}(0,T;V')} \right),
\end{gather*}
where the constant $ C>0 $ depends on $ T $ but independent of $ u $.

\end{proposition}
\begin{remark}
Proposition \ref{prop-embeding-spacetime} implies the following fact for any $ m \in \mathbb{N} $\:,
\begin{gather*}
\displaystyle   \left\{ u \vert\, u \in L ^{2}(0,T;X ^{m+2}), u _{t}\in L ^{2}(0,T;X ^{m}) \right\}\hookrightarrow C([0,T];X ^{m+1})\ \mbox{continuously},
\end{gather*}
where $ X ^{m}:=H ^{m}$, $ H _{z}^{m}$ or $ H _{\xi}^{m}  $.
\end{remark}
Finally, by the change of variables in \eqref{bd-layer-variable}, for any $ G _{1}(z,t) \in H _{z}^{m} $ and $ G _{2}(\xi,t) \in H _{\xi}^{m} $ with $ m \in \mathbb{N} $, we have the following inequalities
\begin{subequations}\label{inte-transfer}
\begin{gather}
\displaystyle \left\|\partial _{x}^{m}G _{1}\left(\frac{x}{\varepsilon ^{1/2}},t \right)\right\|_{L ^{2}}= \varepsilon ^{\frac{1-2m}{4}}\|\partial _{z}^{m}G _{1}(z,t)\|_{L _{z}^{2}},\ \ \  \left\|\partial _{x}^{m}G _{1}\left(z,t \right)\right\|_{L ^{\infty}}= \varepsilon ^{-\frac{m}{2}}\|\partial _{z}^{m}G _{1}(z,t)\|_{L _{z}^{\infty}}, \label{z-transfer}\\
\displaystyle
\displaystyle \left\|\partial _{x}^{m}G _{2}\left(\frac{x-1}{\varepsilon ^{1/2}},t \right)\right\|_{L ^{2}}= \varepsilon ^{\frac{1-2m}{4}}\|\partial _{\xi}^{m}G _{2}(\xi,t)\|_{L _{\xi}^{2}},\ \ \ \left\|\partial _{x}^{m}G _{2}\left(\xi,t \right)\right\|_{L ^{\infty}}= \varepsilon ^{-\frac{m}{2}}\|\partial _{\xi}^{m}G _{2}(\xi,t)\|_{L _{\xi}^{\infty}}. \label{xi-transfer}
\end{gather}
\end{subequations}

 \vspace{4mm}

 \section*{Acknowledgement} 
JAC was partially supported by the Advanced Grant Nonlocal-CPD (Nonlocal PDEs for Complex Particle Dynamics: Phase Transitions, Patterns and Synchronization) of the European Research Council Executive Agency (ERC) under the European Union’s Horizon 2020 research and innovation programme (grant agreement No. 883363) and the EPSRC grant EP/V051121/1. GH is partially supported by the National Natural Science Foundation (grant No. 12201221), the Guangdong Basic and Applied Basic Research Foundation (grant No. 2021A1515111038), and the CAS AMSS-POLYU Joint Laboratory of Applied Mathematics postdoctoral fellowship scheme. ZAW was supported in part by the Hong Kong RGC GRF grant No. PolyU 153031/17P and an internal grant ZZPY.

\vspace{4mm}

\bibliographystyle{mysiam}

\end{document}